\documentclass[twoside,12pt]{article}
\DeclareMathSizes{20}{30}{16}{12}
\usepackage{float}
\usepackage{amsmath}
\usepackage{mathtools}
\usepackage{mathrsfs}
\usepackage{stmaryrd}
\usepackage{bbm}
\usepackage{yfonts}
\usepackage{amsfonts}
\usepackage{tcolorbox}
\usepackage{tikz-cd}
\usepackage[utf8]{inputenc}

\usepackage{caption}
\usepackage{indentfirst}
\usepackage[Symbol]{upgreek}
\usepackage{enumerate}   
\usepackage{upref}

\usepackage{graphicx}
\usepackage[margin=0.8 in]{geometry}
\usepackage{fancyhdr}
\usepackage{hyperref}
\hypersetup{
    colorlinks,
    citecolor=black,
    filecolor=black,
    linkcolor=black,
    urlcolor=black
}
\newcommand{\N}{\mathbb{N}}
\newcommand{\Z}{\mathbb{Z}}

\newcommand{\ra}{\rightarrow}

\newcommand{\calA}{\mathcal{A}}
\newcommand{\calI}{\mathcal{I}}
\newcommand{\calJ}{\mathcal{J}}
\newcommand{\calE}{\mathcal{E}}
\newcommand{\calK}{\mathcal{K}}
\newcommand{\calC}{\mathcal{C}}
\newcommand{\calT}{\mathcal{T}}
\newcommand{\calD}{\mathcal{D}}

\newcommand{\calR}{\mathcal{R}}
\newcommand{\calB}{\mathcal{B}}

\newcommand{\calM}{\mathcal{M}}

\newcommand{\calP}{\mathcal{P}}
\newcommand{\calQ}{\mathcal{Q}}
\newcommand{\calV}{\mathcal{V}}

\newcommand{\op}{\operatorname}
\newcommand{\w}{\widehat}

\newcommand{\Ox}{\mathcal{O}}
\newcommand{\F}{\mathcal{F}}
\newcommand{\calG}{\mathcal{G}}

\newcommand{\ov}{\overline}

\newcommand{\frakA}{\mathfrak{A}}
\newcommand{\frakB}{\mathfrak{B}}
\newcommand{\frakX}{\mathfrak{X}}
\newcommand{\frakE}{\mathfrak{E}}
\newcommand{\frakF}{\mathfrak{F}}
\newcommand{\frakU}{\mathfrak{U}}

\newcommand{\frakS}{\mathcal{S}}
\newcommand{\frakY}{\mathfrak{Y}}
\newcommand{\frakT}{\mathfrak{T}}
\newcommand{\frakZ}{\mathfrak{Z}}

\newcommand{\str}{\text{strict}}

\long\def\/*#1*/{}

\usepackage{unicode-math}
\usepackage{fontspec}

\usepackage{sectsty}
\sectionfont{\centering}

\usepackage[toc,page]{appendix}
\usepackage{amsthm}

\newtheorem*{proof*}{Proof}

\newtheorem{theorem}[subsection]{Theorem}
\newtheorem{theorem*}{Theorem}
\newtheorem{proposition}[subsection]{Proposition}
\newtheorem{corollaire}[subsection]{Corollary}
\newtheorem{lemma}[subsection]{Lemma}

\theoremstyle{definition}

\newtheorem{definition}[subsection]{Definition}
\newtheorem{remark}[subsection]{Remarks}

\theoremstyle{definition}
\newtheorem{parag}[subsection]{}

\numberwithin{equation}{section} 

\makeatletter
\renewcommand{\theequation}{%
  \ifnum\value{subsubsection}>0
    \thesubsubsection.\arabic{equation}%
  \else
    \thesubsection.\arabic{equation}%
  \fi
}
\makeatother

\title{Indexed Logarithmic Cartier Transform}
\author{Sami Fersi}
\date{}

\AtEndDocument{\bigskip{\footnotesize%
  \textsc{Sami Fersi, Laboratoire Alexander Grothendieck, Institut des Hautes Études Scientifiques, 35 Route de Chartres, 91440 Bures-sur-Yvette, France} \par
  \textit{E-mail address}: \texttt{fersi@ihes.fr} \par
}}
\bibliography{./references/ref}
\begin{document}
\maketitle

\begin{abstract}
This article is the third of a series of three articles whose goal is to generalize the Cartier transform of Ogus and Vologodsky to the logarithmic setting. 
In the second one \cite{SF2}, we generalized a topos-theoretic version of this transform, due to Oyama. Let $W$ be the ring of Witt vectors of a perfect field of positive characteristic $p$ and equip $\op{Spf}W$ with the trivial log structure. Let $\frakS$ be an fs log $p$-adic formal scheme, locally of finite type and log flat over $\op{Spf}W$ and denote its special fiber by $S.$ For a log smooth morphism of logarithmic schemes $X \ra S,$ we constructed crystalline-like ringed topoi $\calE'$ and $\underline{\calE}$ and subcategories of crystals of quasi-coherent modules $\calC'$ and $\underline{\calC},$ equivalent respectively, under some lifting assumption, to modules with Higgs fields and integrable connections, both satisfying certain nilpotence conditions, and a morphism of topoi $\underline{\calE} \ra \calE'.$ We proved that the pullback functor of this morphism of topoi preserves quasi-coherent crystals and induces a fully faithful functor $\calC' \ra \underline{\calC}.$ Since the Frobenius morphism is not, in general, flat in the log smooth setting, it is not clear that this functor is essentially surjective. To address this issue, we refine the topoi and crystals mentioned above by endowing them with an indexed structure, inspired by Lorenzon’s indexed extension of Cartier descent to smooth logarithmic schemes. Using the Azumaya property of the indexed algebra of logarithmic differential operators, we then obtain an equivalence between the corresponding categories of indexed crystals, thereby generalizing the Cartier transform.
\end{abstract}

\tableofcontents

\section{Introduction}

\begin{parag}
This article is the third in a series of three articles whose goal is to generalize the Cartier transform of Ogus and Vologodsky to the logarithmic case. In the first one \cite{SF1}, we generalized a local version of the Cartier transform, due to Shiho \cite{Shiho}, to the log smooth setting. In the second one \cite{SF2}, we generalized a topos-theoretic version of this transform, due to Oyama \cite{Oyama}. More precisely, for a perfect field $k$ of positive characteristic with ring of Witt vectors $W(k),$ let $\frakS$ be an fs log $p$-adic formal scheme, locally of finite type and log flat over $\op{Spf}W(k),$ $S$ its special fiber and $X\ra S$ a log smooth morphism of fs logarithmic schemes. We constructed two ringed topoi $\calE'$ and $\underline{\calE}$ as well as subcategories of crystals of quasi-coherent modules $\calC'$ and $\underline{\calC}$ of these respective topoi. Let $X\ra X'$ be the exact relative Frobenius (\cite{SF1} 3.1). As Oyama did in the classical smooth setting, we proved that, under a lifting assumption of $X$ and $X'$ to log smooth formal schemes over $\frakS,$ the categories $\calC'$ and $\underline{\calC}$ are equivalent respectively to modules equipped with Higgs fields and modules equipped with integrable connections, both satisfying certain nilpotence conditions. We also constructed, without assuming any lifting conditions, a morphism of ringed topoi $\varphi:\underline{\calE} \ra \calE'.$ We proved that the inverse image functor $\varphi^{-1}$ preserves crystals and hence induces a functor $\calC' \ra \underline{\calC}$ extending the usual Cartier transform. We then proved that this functor is fully faithful. In the smooth case, faithfully flat descent is crucial in proving that this functor is essentially surjective. Because of the lack of faithful flatness of the relative Frobenius, using the same argument is not possible in the log smooth case. That is why we introduce \emph{indexed algebras and modules} and shift to an \emph{Azumaya algebra approach}. For a morphism of logarithmic schemes $f:X\ra S,$ Lorenzon defines \emph{indexed algebras} $\calA_X$ and $\calB_{X/S}$ \cite{Lor2000}. These indexed algebras are topos-theoretic versions of graded algebras in the sense that each local section possesses an \emph{index}, adding two local sections is only possible if they have the same index and multiplying them adds the corresponding indices. The indexed algebra $\calA_X$ is naturally equipped with a connection $d_{\calA_X}$ and $\calB_{X/S}$ is taken to be its sheaf of horizontal sections. Lorenzon proves an \emph{indexed Cartier descent theorem}: if $f$ is log smooth, then locally free $\calB_{X/S}$-modules of finite type are equivalent to locally free $\calA_X$-modules of finite type, equipped with an integrable connection, which has vanishing $p$-curvature and is compatible with $d_{\calA_X}.$ Ohkawa and Schepler used these indexed algebras to introduce an \emph{indexed algebra of differential operators} $\widetilde{\calD}_{X/S}$ and prove that it is an Azumaya algebra over its center (\cite{Schepler} and \cite{Ohkawa}). We review their work in more detail and then construct indexed versions of Oyama crystals that represent $\calB_{X/S}$-modules with Higgs fields and $\calA_X$-modules with integrable connections. We finally construct an equivalence of categories between them.
\end{parag}

\begin{parag}
We describe our work in more details.
Let $(X,\Ox)$ be a ringed topos, $\calI$ a sheaf of monoids of $X,$ $p_1,p_2:\calI^2 \ra \calI$ the canonical projections and $\sigma:\calI^2 \ra \calI$ the addition map. We have a localization morphism
$
j_{\calI}:X_{/\calI} \ra X.
$
The morphisms $p_1,p_2$ and $\sigma$ also induce morphisms of ringed topoi, abusively denoted by
$$
p_1,p_2,\sigma: \left (X_{/\calI^2},\Ox_{\calI^2} \right ) \ra \left (X_{/\calI},\Ox_{\calI} \right ),
$$
where $\Ox_{\calI}=j_{\calI}^{-1}\Ox$ and $\Ox_{\calI^2}=j_{\calI^2}^{-1}\Ox.$
An $\calI$-indexed algebra is then an $\Ox_{\calI}$-module equipped with a multiplication
$$
\calA \boxtimes \calA :=p_1^*\calA \otimes_{\Ox_{\calI^2}}p_2^*\calA \ra  \sigma^*\calA,
$$
satisfying certain associativity, commutativity and unity properties \eqref{inddef}. Similarly, given an $\calI$-indexed algebra $\calA,$ an $\calI$-indexed $\calA$-module is an $\Ox_{\calI}$-module $\calE$ equipped with a morphism
$$
\calA \boxtimes \calE \ra \sigma^*\calE,
$$
satisfying certain properties. We then introduce indexed versions of the sheaf of endomorphisms \eqref{inddef2} and the tensor product \eqref{indtensprod}. In other words, given an $\calI$-indexed algebra $\calA$ and $\calI$-indexed $\calA$-modules $\calE$ and $\F,$ we define an $\calI$-indexed algebra
$
\mathscr{End}_{\calA}\left (\calE,\F\right)
$
and an $\calI$-indexed $\calA$-module
$
\calE \circledast_{\calA} \F.
$
\end{parag}

\begin{parag}
For a logarithmic scheme $X,$ set $\upmu=\ov{\calM}_X^{gp}.$ The logarithmic structure on $X$ gives rise to an exact sequence of abelian groups
$$
0 \ra \Ox_X^* \xrightarrow{\lambda} \calM_X^{gp} \ra \upmu \ra 0,
$$
which, in turn, gives an invertible $\Ox_{X,\upmu}$-module
$$
\calA_X=\left (\calM^{gp} \times \Ox_{X,\upmu} \right )_{/\sim},
$$
where $(m,ab) \sim (\lambda(a)+m,b)$ for any local sections $a,$ $b$ and $m$ of $\Ox_{X,\upmu}^*,$ $\Ox_{X,\upmu}$ and $\calM_X^{gp}$ respectively. We equip this module $\calA_X$ with a structure of $\upmu$-indexed algebra. Given a morphism of logarithmic schemes $X\ra S,$ Lorenzon defines a canonical connection
$$
d_{\calA_X}:\calA_X \ra \calA_X\otimes_{\Ox_{X,\upmu}} \omega^1_{X/S,\upmu}
$$
and denotes by $\calB_{X/S}$ its kernel, which inherits a structure of $\upmu$-indexed algebra from $\calA_X$ \eqref{Bgroupe2}. He also proves that, if $X\ra S$ is log smooth, then $\calA_X$ is a locally free $\upmu$-indexed $\calB_{X/S}$-module of rank $p^d,$ where $d$ is the rank of $\omega^1_{X/S}$ \eqref{thmlor}.
\end{parag}

\begin{parag}
For a morphism of logarithmic schemes $X\ra S,$ we set $\upmu=\ov{\calM}_X^{gp}$ and $\widetilde{\calD}_{X/S}=\calA_X \otimes_{\Ox_{X,\upmu}} \calD_{X/S,\upmu}.$ In section 4, we recall that $\widetilde{\calD}_{X/S}$ has a structure of $\upmu$-indexed algebra and we recall a result by Schepler and Ohkawa about the Azumaya property of $\widetilde{\calD}_{X/S}.$ More precisely, we consider a log smooth morphism $f:\frakX \ra \frakS$ of fs logarithmic $p$-adic formal schemes such that $\frakS$ is log flat and locally of finite type over $\op{Spf}W(k).$ Let $X\ra S$ be its special fiber. If we consider $\calD_{X/S}$ as an algebra over $S^{\bullet}\calT_{X'/S}$ via the $p$-curvature map (\cite{SF2} 4), then $\widetilde{\calD}_{X/S}$ is a $\upmu$-indexed Azumaya algebra over its center $\calB_{X/S}\otimes_{\Ox_{X',\upmu}}S^{\bullet}\calT_{X/S,\upmu}.$
If we suppose that $F_1:X \ra X'$ lifts to an $\frakS$-morphism of fs $p$-adic logarithmic formal schemes $F:\frakX \ra \frakX'$ such that $\frakX'$ is log smooth over $\frakS,$ Ohkawa constructs a specific splitting of $\widetilde{\calD}_{X/S}$ over $\w{S}^{\bullet}\calT_{X'/S,\upmu},$ where $\w{S}^{\bullet}\calT_{X'/S}$ is the completion with respect to the ideal $\bigoplus_{n\ge 1}S^{n}\calT_{X'/S}.$ We deduce from that the following proposition:
\begin{proposition}[\ref{propKoko1717}]
Let $f:\frakX \ra \frakS$ be a log smooth morphism of fs $p$-adic logarithmic formal schemes such that $\frakS$ is log flat and locally of finite type over $\op{Spf}W(k).$ We suppose that $F_1:X \ra X'$ lifts to an $\frakS$-morphism of fs $p$-adic logarithmic formal schemes $F:\frakX \ra \frakX'$ such that $\frakX'$ is log smooth over $\frakS.$ Denote by $X\ra S$ the special fiber of $f.$
Let $\widehat{\Gamma}^{\bullet}\calT_{X'/S}$ be the completed PD-algebra of $\calT_{X'/S}$ and $\widehat{\Gamma}=\left (\widehat{\Gamma}^{\bullet}\calT_{X'/S} \right )_{\upmu}.$
Then we have a Morita equivalence:
$$
\Phi_{\upmu}:\begin{Bmatrix}\mathrm{Locally\ PD}\text{-}\mathrm{nilpotent}\\
\upmu\text{-}\mathrm{indexed\ }\left (\calB_{X/S}\otimes_{\Ox_{X',\upmu}}\widehat{\Gamma}\right )\\ \text{-}\mathrm{modules} \end{Bmatrix} \xrightarrow{\sim} \begin{Bmatrix} \mathrm{Locally\ PD}\text{-}\mathrm{nilpotent\ }\upmu\text{-}\mathrm{indexed}\\ \left ((\calB_{X/S}\otimes_{\Ox_{X',\upmu}}\widehat{\Gamma})\circledast_{\calB_{X/S}\otimes_{\Ox_{X',\upmu}}S}\widetilde{\calD}_{X/S} \right ) \\ \text{-}\mathrm{modules}\end{Bmatrix}.
$$
\end{proposition}
\end{parag}

\begin{parag}
In what follows, we consider the ring of Witt vectors $W(k)$ of a perfect field $k$ of characteristic $p.$ If a gothic letter $\frakX$ denotes a logarithmic $p$-adic formal scheme, the corresponding roman letter $X$ will denote its special fiber. We equip $\op{Spf}W(k)$ with the trivial logarithmic structure and we consider a log flat and locally of finite type fs logarithmic $p$-adic formal scheme $\frakS$ over $\op{Spf}W(k).$ We denote by $S$ its special fiber and consider a log smooth morphism $X\ra S$ of fs logarithmic schemes which is of Cartier type.
In section 5, we recall the Oyama sites $\calE(X/\frakS)$ and $\underline{\calE}(X/\frakS)$ and the subcategories of crystals of the associated ringed topoi (\cite{SF2} 8). We have a continuous and cocontinuous functor $\rho:\underline{\calE}(X/\frakS) \ra \calE(X'/\frakS)$ \eqref{era3rhoShiho} that induces a morphism of ringed topoi
$$
C_{X/\frakS}:\widetilde{\underline{\calE}}(X/\frakS) \ra \widetilde{\calE}(X'/\frakS).
$$
The pullback functor of $C_{X/\frakS}$ preserves crystals.

We also recall that, under a lifting assumption of $X,$ there exist Hopf algebras $\calR'_1$ and $\calQ_1$ on $X'$ and $X$ respectively, such that the subcategories of crystals of $\widetilde{\calE}(X'/\frakS)$ and $\widetilde{\underline{\calE}}(X/\frakS)$ are equivalent to modules on $X$ equipped with stratifications with respect to $\calR'_1$ and $\calQ_1$ respectively \eqref{equivRQ}. 
\end{parag}

\begin{parag}
Next, in section 6, we define indexed versions of these Oyama topoi. The functors
$$
u':\begin{array}[t]{clc}
\calE(X'/\frakS) & \ra & \text{ét}_{/X'} \\
(U,\frakT,u) & \mapsto & U
\end{array},\ 
\underline{u}:\begin{array}[t]{clc}
\underline{\calE}(X/\frakS) & \ra & \text{ét}_{/X} \\
(U,\frakT,u) & \mapsto & U
\end{array}
$$
induce morphisms of topoi
\begin{equation}
u':\widetilde{\calE}(X'/\frakS) \ra X'_{\text{ét}},\ \underline{u}:\widetilde{\underline{\calE}}(X/\frakS) \ra X_{\text{ét}}.
\end{equation}
We set
$
\upmu_X=\ov{\calM}_X^{gp},\ \upmu'=u'^{-1}\upmu_X,\ \underline{\upmu}=\underline{u}^{-1}\upmu_X.
$
The localized topoi
$
\widetilde{\calE}(X'/\frakS)_{/\upmu'}$ and $\widetilde{\underline{\calE}}(X/\frakS)_{/\underline{\upmu}}
$
are naturally equipped with the rings of 
$
\widetilde{\calE}(X'/\frakS)$ and $\widetilde{\underline{\calE}}(X/\frakS)$ localized respectively at $\upmu'$ and $\underline{\upmu}.$
These ringed topoi will play the role of Oyama topoi in the indexed framework. Just as in the nonindexed case, the data of a module $\F$ of $\widetilde{\calE}(X'/\frakS)_{/\upmu'}$ (resp. $\widetilde{\underline{\calE}}(X/\frakS)_{/\underline{\upmu}}$) is equivalent to the data, for every object $(U,\frakT,u),$ of a module $\F_{(U,\frakT,u)}$ of $\left (U_{\text{ét}/\ov{\calM}_U^{gp}},\left (u_*\Ox_T \right )_{\ov{\calM}_U^{gp}} \right ),$ and, for every morphism $(f,g):(U_1,\frakT_1,u_1) \ra (U_2,\frakT_2,u_2),$ a morphism
$
\widetilde{f}_{\upmu}^*\F_{(U_2,\frakT_2,u_2)} \ra \F_{(U_1,\frakT_1,u_1)},
$
satisfying certain conditions \eqref{indlindescentdata2}.

A local section $s\in \Gamma(U,\upmu_X)$ naturally induces morphisms of topoi
\begin{equation}
\alpha_s':\widetilde{\calE}(U'/\frakS)  \ra  \widetilde{\calE}(X'/\frakS)_{/\upmu'},\ 
\underline{\alpha}_s:\widetilde{\underline{\calE}}(U/\frakS) \ra  \widetilde{\underline{\calE}}(X/\frakS)_{/\underline{\upmu}}.
\end{equation}
We then define a crystal of $\widetilde{\calE}(X'/\frakS)_{/\upmu'}$ (resp. $\widetilde{\underline{\calE}}(X/\frakS)_{/\underline{\upmu}}$) to be a module $\F$ of $\widetilde{\calE}(X'/\frakS)_{/\upmu'}$ (resp. $\widetilde{\underline{\calE}}(X/\frakS)_{/\underline{\upmu}}$) such that 
$
\alpha_s'^{-1}\F\ (\text{resp.}\ \underline{\alpha}_s^{-1}\F)
$
is a crystal for all local sections $s$ of $\upmu_X.$
\end{parag}

\begin{parag}
Similarly, we can define subsites $\calE_{\mathrm{strict}}(X'/\frakS)_{/\upmu'}$ and $\underline{\calE}_{\mathrm{strict}}(X/\frakS)_{/\underline{\upmu}}$ of $\calE(X'/\frakS)_{/\upmu'}$ and $\underline{\calE}(X/\frakS)_{/\underline{\upmu}}$ respectively, consisting of objects $(U,\frakT,u,s)$ such that $u$ is strict. We then prove, in \ref{equivShihoind}, that the crystals on these subsites extend naturally to crystals on $\calE_{\mathrm{strict}}(X'/\frakS)_{/\upmu'}$ and $\underline{\calE}_{\mathrm{strict}}(X/\frakS)_{/\underline{\upmu}}$ respectively:
\begin{proposition}[\ref{equivShihoind}]\label{prop19}
Let $\frakS$ be an fs logarithmic $p$-adic formal scheme, log flat and locally of finite type over $\op{Spf}W(k).$ Let $X$ be a log smooth fs logarithmic scheme over $S.$
The restriction functor
$$
\widetilde{\calE}(X'/\frakS)_{/\upmu'} \ra \widetilde{\calE}_{\mathrm{strict}}(X'/\frakS)_{/\upmu'},\quad
\left ( \text{resp.}\ \widetilde{\underline{\calE}}(X/\frakS)_{/\underline{\upmu}} \ra \widetilde{\underline{\calE}}_{\mathrm{strict}}(X/\frakS) \right )_{/\underline{\upmu}}
$$
induces an equivalence between the subcategories of crystals
$$
\begin{Bmatrix}\text{Crystals\ of\ }\widetilde{\calE}(X'/\frakS)_{/\upmu'} \end{Bmatrix} \xrightarrow{\sim} \begin{Bmatrix}\text{Crystals\ of\ }\widetilde{\calE}_{\mathrm{strict}}(X'/\frakS)_{/\upmu'} \end{Bmatrix}
$$
$$
\left (\text{resp.}\ \begin{Bmatrix}\text{Crystals\ of\ }\widetilde{\underline{\calE}}(X/\frakS)_{/\underline{\upmu}} \end{Bmatrix} \xrightarrow{\sim} \begin{Bmatrix}\text{Crystals\ of\ }\widetilde{\underline{\calE}}_{\mathrm{strict}}(X/\frakS)_{/\underline{\upmu}} \end{Bmatrix}
\right )
$$
\end{proposition}
We can define a crystal $\frakA_X$ of $\widetilde{\underline{\calE}}_{\mathrm{strict}}(X/\frakS)_{/\underline{\upmu}}$ satisfying, for any object $(U,\frakT,u)$ of $\underline{\calE}_{\mathrm{strict}}(X/\frakS),$
$$
\frakA_{X,(U,\frakT,u)}=\calA_T,
$$
where $T$ is the special fiber of $\frakT.$
This crystal then extends, by \ref{prop19}, to a crystal $\frakA_X$ of $\widetilde{\underline{\calE}}(X/\frakS)_{/\underline{\upmu}}.$
Consider the assumption
\begin{itemize}
\item[$(H)$] The morphism $X\ra S$ lifts to a log smooth morphism $f:(\frakX,Q) \ra (\frakS,P)$ of framed fs logarithmic $p$-adic formal schemes and $X'$ lifts to an fs framed logarithmic $p$-adic formal scheme $(\frakX',Q')$ over $(\frakS,P),$ such that $\frakX' \ra \frakS$ is log smooth.
\end{itemize}
We then prove, in \ref{equivRQind} and under the hypothesis $(H),$ that there exists a canonical equivalence of categories between the category of crystals of $\widetilde{\underline{\calE}}(X/\frakS)_{/\underline{\upmu}}$ and the category of $\Ox_{X,\upmu_X}$-modules equipped with a $\calQ_{1,\upmu_X}$-stratification. We also prove that, under $(H),$ there exists a canonical equivalence of categories between the category of crystals of $\widetilde{\calE}(X'/\frakS)_{/\upmu'}$ and the category of $\Ox_{X',\upmu_X}$-modules equipped with an $\calR_{1,\upmu_X}'$-stratification. In addition, these equivalences preserve structures of $\upmu_X$-indexed algebras and $\upmu_X$-indexed modules.
Let $q_1,q_2:Q_{\frakX} \ra \frakX$ and $q_1',q_2':R_{\frakX',1} \ra \frakX'$ be the canonical projections. The projections $q_1$ and $q_2$ are strict and so they induce, by (\cite{Lor2000} I 3.1 page 273), isomorphisms of $\upmu_X$-indexed algebras
$c_i:q_i^*\calA_X \xrightarrow{\sim} \calA_{Q_1}.$
We get a $\calQ_{1,\upmu_X}$-stratification
\begin{equation}\label{intro2}
\epsilon_{\calA_X,Q}:q_2^*\calA_X \xrightarrow{c_2} \calA_{Q_1} \xrightarrow{c_1^{-1}} q_1^*\calA_X.
\end{equation}
Under the hypothesis $(H),$ the crystal $\frakA_X$ previously defined, corresponds to $\left (\calA_X,\epsilon_{\calA_X,Q} \right ).$
Now consider the $\upmu_X$-indexed $\Ox_{X',\upmu_X}$-module $\calB_{X/S}.$ The isomorphism that exchanges factors $q_2'^*\calB_{X/S} \xrightarrow{\sim} q_1'^*\calB_{X/S}$ is an $\calR_{1,\upmu_X}'$-stratification on $\calB_{X/S}$ and so it corresponds to a crystal $\frakB_X$ of $\widetilde{\calE}(X'/\frakS)_{/\upmu'}.$ We check in \ref{propBglue} that the definition of this crystal $\frakB_X$ doesn't depend on the local assumption $(H)$ and can hence be defined globally as a crystal of $\widetilde{\calE}(X'/\frakS)_{/\upmu'}$ without the need of the assumption $(H).$ The crystals $\frakA_X$ and $\frakB_X$ are naturally equipped with structures of indexed algebras \eqref{equivRQind}.
We will be interested in crystals of $\upmu'$-indexed $\frakB_X$-modules and $\underline{\upmu}$-indexed $\frakA_X$-modules. Under the hypothesis $(H),$ these crystals correspond to modules with \emph{admissible} stratifications:
\begin{enumerate}
\item Let $\calE$ be a $\upmu_X$-indexed $\calA_{X}$-module equipped with a $\calQ_{1,\upmu_X}$-stratification
$$\epsilon:\calQ_{1,\upmu_X}\otimes_{\Ox_{X,\upmu_X}}\calE \ra \calE\otimes_{\Ox_{X,\upmu_X}}\calQ_{1,\upmu_X}.$$
The stratification $\epsilon$ is said to be \emph{admissible} if it is compatible with $\epsilon_{\calA_X,Q}$ \eqref{intro2}.
\item A $\upmu_X$-indexed $\calB_{X/S}$-module $\calE'$ equipped with an $\calR_{1,\upmu_X}'$-stratification
$\epsilon':\calR_{1,\upmu_X}'\otimes_{\Ox_{X',\upmu_X}}\calE' \ra \calE'\otimes_{\Ox_{X',\upmu_X}}\calR_{1,\upmu_X}'$
is said to be \emph{admissible} if $\epsilon'$ is $\calB_{X/S}$-linear.
\end{enumerate}
We prove, in \ref{equivRQAB}, that the canonical equivalence of categories between the category of crystals of $\widetilde{\underline{\calE}}(X/\frakS)_{/\underline{\upmu}}$ (resp. $\widetilde{\calE}(X'/\frakS)_{/\upmu'}$) and the category of $\Ox_{X,\upmu_X}$-modules equipped with a $\calQ_{1,\upmu_X}$-stratification (resp. $\Ox_{X',\upmu_X}$-modules equipped with an $\calR_{1,\upmu_X}'$-stratification), induces a canonical equivalence of categories between the full subcategory of crystals of $\upmu$-indexed $\frakA_X$-modules of $\widetilde{\underline{\calE}}(X/\frakS)_{/\underline{\upmu}}$ (resp. crystals of $\upmu'$-indexed $\frakB_X$-modules of $\widetilde{\calE}(X'/\frakS)_{/\upmu'}$) and the full subcategory of $\upmu_X$-indexed $\calA_X$-modules equipped with an admissible $\calQ_{1,\upmu_X}$-stratification (resp. $\upmu_X$-indexed $\calB_{X/S}$-modules equipped with an admissible $\calR_{1,\upmu_X}'$-stratification).
\end{parag}

\begin{parag}
In addition to the hypothesis $(H),$ suppose that the exact relative Frobenius $F_1:X\ra X'$ lifts to an $(\frakS,P)$-morphism of framed logarithmic formal schemes $F:(\frakX,Q) \ra (\frakX',Q'),$ such that $\frakX'\ra \frakS$ is log smooth. The lifting $F$ induces a morphism of Hopf algebras $\calV:\calR_1' \ra \calQ_1$ \eqref{paragIran6}. In \ref{propkraz1934}, we construct a commutative diagram
\begin{equation}
\begin{tikzcd}
\begin{Bmatrix}\text{Crystals\ of\ }\widetilde{\calE}(X'/\frakS)_{/\upmu'} \end{Bmatrix} \ar{r}{C^{-1}_{X/\frakS,\upmu}} \ar[swap, sloped]{d}{\sim} & \begin{Bmatrix}\text{Crystals\ of\ }\widetilde{\underline{\calE}}(X/\frakS)_{/\underline{\upmu}} \end{Bmatrix}  \ar[sloped]{d}{\sim} \\
\begin{Bmatrix}\Ox_{X',\upmu_X}\text{-modules\ with\ an}\\ \calR_{1,\upmu_X}'\text{-stratification} \end{Bmatrix} \ar{r}{\Psi_{\upmu}} & \begin{Bmatrix}\Ox_{X,\upmu_X}\text{-modules\ with\ a}\\ \calQ_{1,\upmu_X}\text{-stratification} \end{Bmatrix},
\end{tikzcd}
\end{equation}
where the vertical equivalences were given above, $C_{X/\frakS,\upmu}$ is induced by the morphism of topoi $C_{X/\frakS}$ and $\Psi_{\upmu}(\calE',\epsilon') = \left (F_1^*\calE',\calV^*\epsilon'\right )$ for an $\Ox_{X',\upmu_X}$-module $\calE'$ equipped with an $\calR_{1,\upmu_X}'$-stratification $\epsilon'.$
Finally, we prove, always under the assumption of the lifting of the Frobenius, the following theorem:
\begin{theorem}[\ref{THMXX}]
Let $(\frakS,P)$ be an fs framed logarithmic $p$-adic formal scheme which is log flat and locally of finite type over $\op{Spf}W(k)$ and denote its special fiber by $S.$ Let $X$ be an fs logarithmic scheme which is log smooth and of Cartier type over $S.$ Suppose that the exact relative Frobenius $F_1:X\ra X'$ lifts to an $(\frakS,P)$-morphism of framed logarithmic formal schemes $F:(\frakX,Q) \ra (\frakX',Q'),$ such that $\frakX\ra \frakS$ and $\frakX'\ra \frakS$ are log smooth. Then, there exists a commutative diagram
\begin{equation}
\begin{tikzcd}
\begin{Bmatrix}\mathrm{Crystals\ of\ }\upmu'\text{-}\mathrm{indexed}\\ \frakB_X\text{-}\mathrm{modules\ of\ }\widetilde{\calE}(X'/\frakS)_{/\upmu'} \end{Bmatrix} \ar{r}{C_{X/\frakS,\upmu}^{-1}} \ar[swap,sloped]{d}{\sim} & \begin{Bmatrix}\mathrm{Crystals\ of\ }\underline{\upmu}\text{-}\mathrm{indexed}\\ \frakA_X\text{-}\mathrm{modules\ of\ }\widetilde{\underline{\calE}}(X/\frakS)_{/\underline{\upmu}} \end{Bmatrix} \ar[sloped]{d}{\sim}\\
\begin{Bmatrix}\upmu_X\text{-}\mathrm{indexed\ }\calB_{X/S}\text{-}\mathrm{modules\ with}\\ \mathrm{an\ admissible}\\ \calR_{1,\upmu_X}'\text{-}\mathrm{stratification} \end{Bmatrix} \ar[swap,sloped]{d}{\sim} \ar{r} & \begin{Bmatrix}\upmu_X\text{-}\mathrm{indexed\ }\calA_X\text{-}\mathrm{modules\ with}\\ \mathrm{an\ admissible}\\Q_{1,\upmu_X}\text{-}\mathrm{stratification} \end{Bmatrix} \ar[sloped]{d}{\sim} \\
\begin{Bmatrix}\mathrm{Locally\ PD}\text{-}\mathrm{nilpotent\ } \upmu_X\text{-}\mathrm{indexed}\\\left (\left (\w{\Gamma}^{\bullet}\calT_{X'/S}\right )_{\upmu_X}\otimes_{\Ox_{X',\upmu_X}}\calB_{X/S} \right )\text{-}\mathrm{modules} \end{Bmatrix} \ar{r}{\Phi_{\upmu}} & \begin{Bmatrix}\mathrm{Locally\ PD}\text{-}\mathrm{nilpotent\ }\upmu_X\text{-}\mathrm{indexed}\\\widetilde{\calD}_{X/S}^{\gamma}\text{-}\mathrm{modules} \end{Bmatrix}.
\end{tikzcd}
\end{equation}
Since $\Phi_{\upmu}$ is an equivalence of categories, the functor $C_{X/\frakS,\upmu}^{-1}$ is also an equivalence of categories.
\end{theorem}
Finally, we deduce an equivalence of categories between crystals of $\frakB_X$-modules and crystals of $\frakA_X$-modules:
\begin{theorem}[\ref{THMXXYY}]
Let $\frakS$ be an fs logarithmic $p$-adic formal scheme which is log flat and locally of finite type over $\op{Spf}W(k)$ and denote its special fiber by $S.$ Let $X$ be an fs logarithmic scheme which is log smooth and of Cartier type over $S.$ Then, we have an equivalence of categories
$$
\begin{Bmatrix}\mathrm{Crystals\ of\ }\upmu'\text{-}\mathrm{indexed}\\ \frakB_X\text{-}\mathrm{modules\ of\ }\widetilde{\calE}(X'/\frakS)_{/\upmu'} \end{Bmatrix}  \ra \begin{Bmatrix}\mathrm{Crystals\ of\ }\underline{\upmu}\text{-}\mathrm{indexed}\\ \frakA_X\text{-}\mathrm{modules\ of\ }\widetilde{\underline{\calE}}(X/\frakS)_{/\underline{\upmu}} \end{Bmatrix}.
$$
\end{theorem}
Note that, in this last theorem, we do not require any lifting assumption neither on the exact relative Frobenius nor on $X.$
\end{parag}

\textbf{Acknowledgement.} This article is the third part of my thesis prepared at Université Paris-Saclay and IHES. I express my greatest gratitude to my PhD advisor Ahmed Abbes for introducing me to this topic, his guidance and his patience. I also thank Atsushi Shiho and Daxin Xu for their helpful comments and suggestions.

\section{Notations}

\begin{parag}
For an integer $d\ge 1$ and $I,J\in \N^d,$ we set $I!=\prod_{i=1}^dI_i!$ and $\begin{pmatrix}I \\ J \end{pmatrix}=\prod_{i=1}^d\begin{pmatrix}I_i \\ J_i \end{pmatrix}.$
\end{parag}

\begin{parag}
For integers $1\le i\le d,$ we denote by $\epsilon_{d,i}$ the multi-index of $\N^d$ whose all coefficients are zero except for the $i$th coefficient which equals $1.$ If no confusion arises, we simply denote it by $\epsilon_i.$
\end{parag}

\begin{parag}\label{not23}
For a prime number $p,$ an integer $n\ge 1$ and a logarithmic $p$-adic formal scheme $\frakX,$ we denote by $\frakX_n$ the logarithmic scheme obtained from $\frakX$ by reduction modulo $p^n$ (\cite{SF1} 7.8). The corresponding roman letter $X$ will denote $\frakX_1.$
\end{parag}

\begin{parag}
Let $p$ be a prime number. For morphisms of fs logarithmic schemes (resp. fs logarithmic $p$-adic formal schemes) $X\ra S$ and $Y\ra S,$ we denote by $X\times_S^{\op{log}}Y$ their fiber product in the category of fs logarithmic schemes (resp. fs logarithmic $p$-adic formal schemes).
\end{parag}

\section{Indexed objects}

\subsection*{Indexed algebras and modules}

\begin{parag}\label{indpar1}
In this section, we review the theory of indexed objects introduced by Lorenzon in \cite{Lor2000}. We fix a site $C$ in which all fiber products are representable and consider the associated topos $X$ and a sheaf $\calI$ of $X.$ We denote by $C_{/\calI}$ the category of objects of $C$ over $\calI$ and consider the functor
\begin{equation}
j_{\calI}:\begin{array}[t]{clc} C_{/\calI} & \ra & C \\ (A\ra \calI) & \mapsto & A. \end{array}
\end{equation}
We equip $C_{/\calI}$ with the topology induced by $j_{\calI}$ (\cite{SGA43} exposé III 3.1). Then $j_{\calI}$ is continuous and cocontinuous (\cite{SGA43} exposé III 5.2) and so it induces three functors
\begin{equation}\label{indeq1}
j_{\calI !}:\widetilde{C_{/\calI}}\ra \widetilde{C},\ j_{\calI}^*:\widetilde{C}\ra \widetilde{C_{/\calI}},\ j_{\calI *}:\widetilde{C_{/\calI}}\ra \widetilde{C}
\end{equation}
such that $(j_{\calI}^*,j_{\calI*})$ is a morphism of topoi, and $j_{\calI !}$ is a left adjoint of $j_{\calI}^*$ (\cite{SGA43} exposé III 5.3). By (\cite{SGA43} exposé III 5.4), the functor $j_{\calI !}$ factors as
$\widetilde{C_{/\calI}} \xrightarrow{e} \widetilde{C}_{/\calI} \ra \widetilde{C}$
where $e$ is an equivalence of categories and $\widetilde{C}_{/\calI}\ra \widetilde{C}$ is the canonical functor. By identifying $\widetilde{C_{/\calI}}$ with $\widetilde{C}_{/\calI}$ via $e,$ the three functors given above (\ref{indeq1}) correspond to functors:
\begin{equation}\label{indeq2}
j_{\calI !}:\widetilde{C}_{/\calI}\ra \widetilde{C},\ j_{\calI}^*:\widetilde{C}\ra \widetilde{C}_{/\calI},\ j_{\calI *}:\widetilde{C}_{/\calI}\ra \widetilde{C}
\end{equation}
We denote simply by $j_{\calI}:\widetilde{C}_{/\calI} \ra \widetilde{C}$ the morphism $(j_{\calI}^*,j_{\calI*}),$ called the localisation morphism of $\widetilde{C}$ at $\calI.$
By (\cite{SGA43} exposé III 5.4), the functor $j_{\calI}^*$ is given, for any sheaf $\F$ of $\widetilde{C},$ by 
\begin{equation}\label{indeq8}
j_{\calI}^*(\F)=(\F\times \calI \ra \calI),
\end{equation}
where $\F\times \calI \ra \calI$ is the canonical projection.

A morphism $f:\calJ\ra \calI$ in $\widetilde{C}$ defines an object of $\widetilde{C}_{/\calI}$ and so we have a functor
$(C_{/\calI})_{/f}  \ra  C_{/\calI},\quad (g\ra f) \mapsto  g$
that we canonically identify with the functor abusively denoted by
$
C_{/\calJ}  \ra  C_{/\calI},\quad
g  \mapsto  f\circ g.
$
Note that the topology on $C_{/\calJ}$ induced by $f$ is the same as the topology induced by the canonical functor $j_{\calJ}:C_{/\calJ} \ra C.$ Indeed, $j_{\calI}\circ f=j_{\calJ}.$ Since all fiber products in $C$ are representable, by (\cite{SGA43} Exposé II 1.3.1), any topology is determined by a pretopology. By (\cite{SGA43} Exposé III 3.3), for a family $F=(Y_i\ra Y)_{i\in I}$ of morphisms of $C_{/\calJ},$ $F$ is a covering family for the topology induced by $j_{\calJ}$ on $C_{/\calJ}$ if and only if $j_{\calJ}(F)=(j_{\calJ}(Y_i)\ra j_{\calJ}(Y))_{i\in I}$ is a covering if and only if $f(F)=(f(Y_i)\ra f(Y))_{i\in I}$ is a covering for the topology induced by $j_{\calI}$ on $C_{/\calI}.$

We get three functors
\begin{equation}\label{indeq3}
f_!:\widetilde{C}_{/\calJ}\ra \widetilde{C}_{/\calI},\ f^*:\widetilde{C}_{/\calI} \ra \widetilde{C}_{/\calJ},\ f_*:\widetilde{C}_{/\calJ} \ra \widetilde{C}_{/\calI}
\end{equation}
such that $(f^*,f_*)$ is the morphism of localisation of $\widetilde{C}_{/\calI}$ at $f$ and $f_!$ is left adjoint to $f^*.$

If $U$ is an object of $C$ and $i\in\calI(U),$ then, by (\ref{indeq3}), we get three functors
\begin{equation}\label{indeq6}
i_!:\widetilde{C}_{/U}\ra \widetilde{C}_{/\calI},\ i^*:\widetilde{C}_{/\calI} \ra \widetilde{C}_{/U},\ i_*:\widetilde{C}_{/U} \ra \widetilde{C}_{/\calI}
\end{equation} 
From now on, for every object $\F$ of $\widetilde{C}_{/\calI},$ we denote $i^*\F$ by $\F_i.$
\end{parag}

\begin{parag}\label{indfiber73}
Suppose that $C$ is the small étale site on a scheme $S,$ $\calI$ an étale sheaf on $S$ and $i\in\calI(U)$ a section of $\calI$ over an étale $S$-scheme $U.$ Denote by $\calI_{\text{ét}}$ the topos $S_{\text{ét}/\calI.}$ We have a canonical equivalence of topoi $U_{\text{ét}}\xrightarrow{\sim} S_{\text{ét}/U}.$ Then, by (\ref{indeq6}), we get three morphisms
\begin{equation}\label{indeq4}
i_!:U_{\text{ét}}\ra \calI_{\text{ét}},\ i^*:\calI_{\text{ét}} \ra U_{\text{ét}},\ i_*:U_{\text{ét}} \ra \calI_{\text{ét}}
\end{equation}
If $\F\ra \calI$ is an object of $\calI_{\text{ét}},$ then $i^*(\F\ra \calI)=\F\times_{\calI}U.$ We will denote it by $\F_i$ or $\F_{|U}.$
\end{parag}

\begin{remark}\label{indremark73kraz} 
Let $X$ be a topos, $f:\calJ\ra \calI$ a morphism of sheaves of $X$ and equip $X$ with its canonical topology and $X_{/\calI}$ with the topology induced by the functor $j_{\calI}:X_{/\calI} \ra X.$ Then $\widetilde{X}$ is canonically equivalent to $X$ (\cite{SGA43} Exposé IV 1.2) and by applying the results of \ref{indpar1} to the site $C=X,$ we get functors
\begin{equation}
j_{\calI !}:X_{/\calI}\ra X,\ j_{\calI}^*:X\ra X_{/\calI},\ j_{\calI *}:X_{/\calI}\ra X,
\end{equation}
\begin{equation}
f_!:X_{/\calJ}\ra X_{/\calI},\ f^*:X_{/\calI} \ra X_{/\calJ},\ f_*:X_{/\calJ} \ra X_{/\calI}.
\end{equation}
\end{remark}

\begin{parag}\label{indremark73}
Let $C$ be a site, $X$ the associated topos and $\calI$ a sheaf of $X.$ Let $\calT$ be the split fibered $\left (X_{/\calI} \right )$-category defined, for every sheaf $\F$ of $X_{/\calI},$ by the localized category $\left (X_{/\calI} \right )_{/\F}$ and, for every morphism $\varphi:\F \ra \calG,$ by the inverse image functor of the morphism of topoi
$
\left (X_{/\calI} \right )_{/\F} \ra \left (X_{/\calI} \right )_{/\calG},
$
induced by $\varphi.$ By (\cite{Giraud} II 3.4.4), $\calT$ is a stack. If $i$ is a local section of $\calI$ over an object $U$ of $C,$ then we have a canonical isomorphism of topoi
$
\left (X_{/\calI} \right )_{/i}\xrightarrow{\sim} X_{/U}.
$
The data of a sheaf $\F$ of $X_{/\calI}$ is then, equivalent to the data, for every local section $i$ of $\calI,$ of a sheaf $\F_i$ of $X_{/U},$ and, for every morphism $g:V\ra U$ of $C$ and every local section $i$ of $\calI$ over $U,$ of an isomorphism
$
c_{i,g}:g^{-1}\F_{i} \xrightarrow{\sim} \F_{i_{|V}},
$
satisfying cocycle conditions and such that $c_{i,\op{Id}_U}=\op{Id}_{\F_i}.$ In particular, the data of a morphism $f:\F \ra \calG$ in $X_{/\calI}$ is equivalent to the data, for every section $i\in\calI(U)$ for an object $U$ of $C,$ of a morphism $f_i:\F_i \ra \calG_i$ satisfying the following compatibility condition: for all morphisms $U\ra V$ of $C$ and sections $i\in \calI(U)$ and $j\in \calI(V)$ such that $j_{|U}=i,$ the following diagram is commutative:
$$
\begin{tikzcd}
\F_{j|U} \ar{r}{f_{j|U}} \ar[swap,sloped]{d}{\sim} & \calG_{j|U}\ar[sloped]{d}{\sim} \\
\F_i\ar{r}{f_i} & \calG_i
\end{tikzcd}
$$
where the vertical arrows are the restriction morphisms, which are isomorphisms.
\end{parag}

\begin{parag}\label{indpar2}
Let $(X,\Ox)$ be a ringed topos and $f:\calJ\ra \calI$ a morphism of $X$ and consider the functors
\begin{equation}\label{indeq5}
j_{\calI}^*:X\ra X_{/\calI},\ j_{\calI *}:X_{/\calI}\ra X,
\end{equation}
\begin{equation}\label{indeq7}
f^*:X_{/\calI} \ra X_{/\calJ},\ f_*:X_{/\calJ} \ra X_{/\calI},
\end{equation}
defined in \ref{indremark73kraz}. Then $(X,j_{\calI}^*\Ox)$ and $(X,j_{\calJ}^*\Ox)$ are ringed topoi and $(j_{\calI}^*,j_{\calI *})$ and $(f^*,f_*)$ are morphisms of ringed topoi. Note that there is no difference between $j_{\calI}^*$ and $j_{\calI}^{-1}$ (resp. $f^*$ and $f^{-1}$). In the rest of this section, we denote $j_{\calI}^*\Ox$ by $\Ox_{\calI}.$ Denote by $\boldsymbol{\op{Mod}}_{\Ox}$ (resp. $\boldsymbol{\op{Mod}}_{\Ox_{\calI}}$) the category of $\Ox$-modules (resp. $\Ox_{\calI}$-modules). The functor $j_{\calI}^*:\boldsymbol{\op{Mod}}_{\Ox}\ra \boldsymbol{\op{Mod}}_{\Ox_{\calI}}$ (resp. $f^*:\boldsymbol{\op{Mod}}_{\Ox_{\calJ}}\ra \boldsymbol{\op{Mod}}_{\Ox_{\calI}}$) has a left adjoint, that we denote by $j_{\calI !}$ (resp. $f_!$). Note that $j_{\calI !}$ and $f_!$ are different from the functors $j_{\calI !}$ and $f_!$ defined in \ref{indpar1}. Otherwise stated, we will only use the notions for modules. 

Consider the canonical projections $p_1:\calI \times \calJ \ra \calI$ and $p_2:\calI\times \calJ \ra \calJ.$ Let $\calE$ (resp. $\F$) be an $\Ox_{\calI}$-module (resp. $\Ox_{\calJ}$-module). We denote by $\calE \boxtimes \F$ \emph{the exterior tensor product of $\calE$ and $\F$}, i.e. the $\Ox_{\calI\times \calJ}$-module defined as follows:
\begin{equation}
\calE \boxtimes \F := p_1^*\calE \otimes_{\Ox_{\calI\times \calJ}}p_2^*\F.
\end{equation}
\end{parag}

\begin{proposition}\label{indpropshriek}
Let $C$ be a site, $X$ the associated topos, $\Ox$ a sheaf of rings on $C$ and $f:\calJ\ra \calI$ a morphism of sheaves of $X.$ Then:
\begin{enumerate}
\item For every $\Ox_{\calI}$-module $\F,$ $j_{\calI !}\F$ is the sheaf associated to the presheaf
$U \mapsto \bigoplus_{u\in \op{Hom}(U,\calI)}\F(u),$
where we view $\F$ as a sheaf on the localised site $C_{/\calI}.$
\item For every $\Ox_{\calJ}$-module $\F,$ $f_!\F$ is the sheaf associated to the presheaf
$(i:U\ra \calI) \mapsto \bigoplus_{\substack{j\in \op{Hom}(U,\calJ)\\ f\circ j=i}}\F(j)$
where $U$ is an object of $C$ and $\F$ is viewed as a sheaf on the localised site $C_{/\calJ}.$
\item Suppose that $\calI$ is a sheaf of monoids and let $\sigma:\calI^2 \ra \calI$ be the addition map and $\F$ an $\Ox_{\calI^2}$-module. Then $\sigma_!\F$ is the sheaf associated to the presheaf
$(s:U \ra \calI) \mapsto \bigoplus_{\substack{u,v\in \Gamma(U,\calI)\\ u+v=s}}\F((u,v)).$
\end{enumerate}
\end{proposition}

\begin{proof}
We just prove the first assertion as the second is a special case of the first and the third is an application of the second.
For any $\Ox_{\calI}$-module $\F$ of $X_{/\calI},$ denote by $v(\F)$ the sheaf associated to the presheaf
$U \mapsto \bigoplus_{u\in \op{Hom}_C(U,\calI)}\F(u).$
We want to prove that $v$ is left adjoint to $j_{\calI}^*.$ Let $\F$ (resp. $\calG$) be an $\Ox_{\calI}$-module (resp. $\Ox$-module). Then the data of a morphism $g:v(\F)\ra \calG$ is equivalent to the data of compatible morphisms $$g_U:v(\F)(U) \ra \calG(U)$$
for every object $U$ of $C.$ This is equivalent to the data of compatible morphisms
$\F(u)\ra \calG(U)$
for every object $U$ of $C$ and every morphism $u:U \ra \calI.$ Since $(j_{\calI}^*\calG)(u)=\calG(U)$ (\ref{indeq8}), we conclude that this is equivalent to the data of compatible morphisms
$\F(u) \ra (j_{\calI}^*\calG)(u)$
for every object $u$ of $C_{/\calI}.$
\end{proof}

\begin{proposition}\label{indpropshriekoh}
Let $C$ be a site, $X$ the associated topos and $\Ox$ a sheaf of rings on $C.$
\begin{enumerate} 
\item Let $\calI$ and $\calJ$ be sheaves of $X.$ Denote by $p:\calI\times \calJ \ra \calI$ and $q:\calI\times \calJ \ra \calJ$ the canonical projections. Then, for every $\Ox_{\calI}$-module $\F$ of $X_{/\calI},$ there exists a canonical isomorphism
$$q_!p^*\F\xrightarrow{\sim}j_{\calJ}^*j_{\calI !}\F.$$
\item Let $f:\calJ \ra \calI$ and $g:\calK \ra \calI$ be morphisms of sheaves of $X.$ Denote by $p:\calJ\times_{\calI} \calK \ra \calJ$ and $q:\calJ\times_{\calI} \calK \ra \calK$ the canonical projections. Then, for every $\Ox_{\calJ}$-module $\F$ of $X_{/\calJ},$ there exists a canonical isomorphism
$$q_!p^*\F\xrightarrow{\sim}g^*f_!\F.$$
\end{enumerate}
\end{proposition}

\begin{proof}
We have the following commutative diagram of ringed topoi
$$
\begin{tikzcd}
X_{/\calI\times\calJ}\ar{r}{p}\ar{d}{q}\ar{dr}{j_{\calI\times\calJ}} & X_{/\calI}\ar{d}{j_{\calI}} \\
X_{/\calJ}\ar{r}{j_{\calJ}} & X
\end{tikzcd}
$$
The adjunction morphism
$p_!p^*\F \ra \F$
induces a morphism
$j_{\calJ !}q_! p^*\F=j_{\calI !}p_!p^*\F\ra j_{\calI !}\F.$
By adjunction, it yields a morphism
$
q_! p^*\F=j_{\calI !}p_!p^*\F\ra j_{\calJ}^*j_{\calI !}\F.
$
We check that it is an isomorphism using \ref{indpropshriek}. On one hand, $q_!p^*\F$ is the sheaf on $C_{/\calJ}$ associated to the presheaf
$$(j:U\ra \calJ) \mapsto \bigoplus\limits_{\substack{s:U\ra \calI\times \calJ \\ q\circ s=j}}(p^*\F)(s)=\bigoplus\limits_{i:U\ra \calI}\F(i).$$
On the other hand, $j_{\calJ}^*j_{\calI !}\F$ is the sheaf
$(j:U\ra \calJ) \mapsto (j_{\calI !}\F)(U)$ and thus it is associated to the presheaf
$(j:U\ra \calJ) \mapsto \bigoplus\limits_{i:U\ra \calI}\F(i).$
The second assertion is proved similary by considering the commutative diagram of ringed topoi
$$
\begin{tikzcd}
X_{/\calJ\times_{\calI}\calK}\ar{r}{p}\ar{d}{q}\ar{dr}{j_{\calJ\times_{\calI}\calK}} & X_{/\calJ}\ar{d}{f} \\
X_{/\calK}\ar{r}{g} & X_{/\calI}
\end{tikzcd}
$$
\end{proof}

\begin{corollaire}\label{indpropshriek2}
Let $C$ be a site, $X$ the associated topos, $\Ox$ a sheaf of rings on $C$ and $f:\calJ \ra \calI$ a morphism of sheaves of $X.$ For every module $\F$ of $X_{/\calJ}$ and every section $k\in \Gamma(U,\calI)$ over an object $U$ of $C,$ the sheaf $(j_{\calI !}\F)_k$ of $X_{/U}$ is associated to the presheaf
$$(i:V\ra U) \mapsto \bigoplus\limits_{\substack{j\in \Gamma(V,\calJ)\\ f\circ j=k\circ i}}\F(j),$$
where $V$ is an object of $C.$ 
\end{corollaire}

\begin{proof}
By the second assertion of \ref{indpropshriekoh} applied to $g=k,$ we have a canonical isomorphism
$q_!p^*\F\xrightarrow{\sim}(f_!\F)_k,$
where $p:U\times_{\calI}\calJ \ra \calJ$ and $q:U\times_{\calI}\calJ\ra U$ are the canonical projections. By \ref{indpropshriek}, the sheaf $q_!p^*\F$ is associated to the presheaf
$(i:V\ra U) \mapsto \bigoplus\limits_{\substack{s:V\ra U\times_{\calI}\calJ\\ p\circ s=i}}(p^*\F)(s).$
\end{proof}

\begin{parag}
Keep the same notations of \ref{indpar2} and suppose $\calI$ is a sheaf of commutative monoids. Consider the canonical projections 
$p_{ij}:\calI^3\ra \calI^2,\ 1\le i<j\le 3,\ q_i:\calI^3\ra \calI,\ 1\le i\le 3.$
For any $\Ox_{\calI}$-modules $\calE,$ $\F$ and $\calG,$ we have canonical isomorphisms of $\Ox_{\calI^3}$-modules
\begin{equation}
(\calE \boxtimes \F )\boxtimes \calG \xrightarrow{\sim} q_1^*\calE \otimes_{\Ox_{\calI^3}}q_2^*\F \otimes_{\Ox_{\calI^3}}q_3^*\calG \xleftarrow{\sim} \calE \boxtimes (\F \boxtimes \calG).
\end{equation}
We identify the triple exterior tensor products $(\calE \boxtimes \F) \boxtimes \calG$ and $\calE \boxtimes (\F \boxtimes \calG)$ via this canonical isomorphism and we denote them simply by $\calE\boxtimes \F \boxtimes \calG.$
We can now define $\calI$-indexed algebras and $\calI$-indexed modules:
\end{parag}

\begin{definition}\label{inddef}
Let $(X,\Ox)$ be a ringed topos, $\calI$ a sheaf of commutative monoids of $X,$ $p_1,p_2:\calI^2\ra \calI,$ $q_1,q_2,q_3:\calI^3\ra \calI$ and $p_{ij}:\calI^3\ra \calI^2,\ 1\le i<j\le 3$ the canonical projections, $\sigma:\calI^2\ra \calI$ and $\tau:\calI^3\ra \calI$ the addition maps, $e$ the final object of $X$ and $0:e\ra \calI$ the zero section and $\nu:\calI \ra e$ the canonical morphism.
\begin{enumerate}
\item An \emph{$\calI$-indexed $\Ox$-algebra} is an $\Ox_{\calI}$-module $\calA$ equipped with a $\Ox_{\calI^2}$-linear map
$$\pi:\calA \boxtimes \calA \ra \sigma^*\calA$$
and a global section $1_{\calA}:e\ra \calA_0,$ where $\calA_0$ is the $\Ox$-module $0^*\calA,$ such that the following diagrams are commutative:
\begin{enumerate}
\item Associativity:
$$
\begin{tikzcd}
\calA \boxtimes \calA \boxtimes \calA \ar{r}{\pi \boxtimes \op{Id}_{\calA}} \ar[swap]{d}{\op{Id}_{\calA} \boxtimes \pi} & (\sigma \times\op{Id}_{\calI})^*(\calA \boxtimes \calA) \ar{d}{(\sigma \times\op{Id}_{\calI})^*\pi} \\
(\op{Id}_{\calI}\times \sigma)^*(\calA \boxtimes \calA) \ar{r}{(\op{Id}_{\calI}\times \sigma)^*\pi} & \tau^*\calA
\end{tikzcd}
$$
\item Unity:
$$
\begin{tikzcd}
\calA \ar{r}{\op{Id}_{\calA}\otimes \nu^*1_{\calA}} \ar{d}[swap]{\nu^*1_{\calA}\otimes \op{Id}_{\calA}} \ar{dr}{\op{Id}_{\calA}} & \calA \otimes_{\Ox_{\calI}}\nu^*\calA_0=(\op{Id}_{\calI}\times 0)^*(\calA \boxtimes \calA) \ar{d}{(\op{Id}_{\calI}\times 0)^*\pi} \\  \nu^*\calA_0\otimes_{\Ox_{\calI}}\calA=(0\times \op{Id}_{\calI})^*(\calA \boxtimes \calA) \ar[swap]{r}{(0\times \op{Id}_{\calI})^*\pi} & \calA
\end{tikzcd}
$$
where the identification $\calA \otimes_{\Ox_{\calI}}\nu^*\calA_0=(\op{Id}_{\calI}\times 0)^*(\calA \boxtimes \calA)$ follows from the commutative diagram
$$
\begin{tikzcd}
\calI=\calI\times e \ar{r}{\nu} \ar{d}{\op{Id}_{\calI}\times 0} \ar[swap,bend right=50]{dd}{\op{Id}_{\calI}} & e\ar{d}{0} \\
 \calI^2\ar{r}{p_2} \ar{d}{p_1} & \calI \\ \calI &  
\end{tikzcd}
$$
and $\nu^*\calA_0\otimes_{\Ox_{\calI}}\calA=(0\times \op{Id}_{\calI})^*(\calA \boxtimes \calA)$ follows from a similar diagram.
\end{enumerate}
We say that $\calA$ is \emph{commutative} if the diagram
$
\begin{tikzcd}
\calA \boxtimes \calA \ar{r}{\pi} \ar{d}{s} & \sigma^*\calA \\
\calA \boxtimes \calA \ar{ur}{\pi} &
\end{tikzcd}
$
where $s$ is the isomorphism $x\boxtimes y\mapsto y\boxtimes x,$ is commutative.
\item Let $\calA$ be an $\calI$-indexed $\Ox$-algebra. The center of $\calA$ is the subsheaf $\calC$ of $\calA$ consisting of local sections $a$ of $\calA$ such that $\pi(p_1^*a\otimes p_2^*b)=\pi(p_1^*b\otimes p_2^*a)$ for every local section $b$ of $\calA.$ It is immediate that $\calC$ is an $\calI$-indexed $\Ox$-subalgebra of $\calA.$
\item Let $\calA$ and $\calB$ be two $\calI$-indexed $\Ox$-algebras. A morphism of $\calI$-indexed $\Ox$-algebras $f:\calA \ra \calB$ is a $\Ox_{\calI}$-linear morphism $f:\calA \ra \calB$ such that the following diagram is commutative:
$$
\begin{tikzcd}
\calA \boxtimes \calA \ar{r} \ar{d}{f\boxtimes f} & \sigma^*\calA \ar{d}{\sigma^*f} \\
\calB\boxtimes \calB \ar{r} & \sigma^*\calB
\end{tikzcd}
$$
where the horizontal arrows are the multiplication morphisms.
\item If $f:\calA \ra \calB$ is a morphism of $\calI$-indexed $\Ox$-algebras, we say that $\calB$ is an \emph{$\calI$-indexed $\calA$-algebra.}
\item Let $\calA$ be an $\calI$-indexed $\Ox$-algebra. An \emph{$\calI$-indexed left $\calA$-module} is an $\Ox_{\calI}$-module $\calE$ equipped with a $\Ox_{\calI^2}$-linear map
$\rho:\calA \boxtimes \calE \ra \sigma^*\calE,$
such that the following diagrams are commutative
\begin{enumerate}
\item Associativity:
$$
\begin{tikzcd}
\calA\boxtimes \calA \boxtimes \calE \ar{r}{\op{Id}_{\calA}\boxtimes \rho} \ar{d}{\pi \boxtimes \op{Id}_{\calE}} & \calA \boxtimes \sigma^*\calE = (\op{Id}_{\calI}\times \sigma)^*(\calA \boxtimes \calE) \ar{d}{(\op{Id}_{\calI}\times \sigma)^*\rho} \\ 
\sigma^*\calA \boxtimes \calE = (\sigma \times \op{Id}_{\calI})^*(\calA \boxtimes \calE) \ar{r}{(\sigma \times \op{Id}_{\calI})^*\rho} & \tau^*\calE
\end{tikzcd}
$$
\item Unity:
$$
\begin{tikzcd}
\calE \ar{dr}{\op{Id}_{\calE}} \ar[swap]{d}{\nu^*1_{\calA}\otimes \op{Id}_{\calE}} & \\
\nu^*\calA_0\otimes_{\Ox_{\calI}}\calE = (0\times \op{Id}_{\calI})^* (\calA \boxtimes \calE) \ar[swap]{r}{(0\times \op{Id}_{\calI})^*\rho} & \calE
\end{tikzcd}
$$
\end{enumerate}
We also define a \emph{$\calI$-indexed right $\calA$-module} as an $\Ox_{\calI}$-module $\calE$ equipped with a $\Ox_{\calI^2}$-linear map
$\calE \boxtimes \calA \ra \sigma^*\calE$
satisfying similar conditions. Unless explicitly stated, we only consider indexed left module structures.

\item Let $\calA$ be an $\calI$-indexed $\Ox$-algebra and $\calE$ and $\F$ two $\calI$-indexed $\calA$-modules. A \emph{morphism of $\calI$-indexed $\calA$-modules $f:\calE \ra \F$} is a $\Ox_{\calI}$-linear morphism
$f:\calE \ra \F$
such that the following diagram is commutative
$$\begin{tikzcd}
\calA \boxtimes \calE \ar{r}{\op{Id}_{\calA} \boxtimes f} \ar{d} & \calA \boxtimes \F \ar{d} \\
\sigma^*\calE \ar{r}{\sigma^*f} & \sigma^* \F
\end{tikzcd}$$ 
where the vertical arrows are the maps defining the $\calA$-module structures on $\calE$ and $\F.$

For an $\calI$-indexed $\Ox$-algebra $\calA,$ we denote by $\boldsymbol{\op{Mod}}_{\calA}$ the category of $\calI$-indexed $\calA$-modules. The category $\boldsymbol{\op{Mod}}_{\calA}$ is an abelian subcategory of $\boldsymbol{\op{Mod}}_{\Ox_{\calI}}.$
\end{enumerate}
\end{definition}

\begin{remark}\label{era2indrem}
Let $(X,\Ox)$ be a ringed topos on a site $C,$ $e$ the final object of $X$ and $\calI$ a sheaf of monoids of $X.$ Recall the notation \eqref{indfiber73}. For an object $U$ of $C,$ we denote by $\Ox_U$ the restriction of $\Ox$ to $U.$ By \ref{indremark73}, an $\calI$-indexed algebra structure on an $\Ox_{\calI}$-module $\calA$ is equivalent to the data of a global section $1:e\ra \calA_0$ and, for sections $s,t\in \Gamma(U,\calI)$ of $\calI$ over an object $U$ of $C,$ of an $\Ox_U$-linear morphism
$$\calA_s\otimes_{\Ox_U}\calA_t\ra \calA_{s+t},$$
such that, for any sections $r,s,t\in \Gamma(U,\calI)$ over an object $U$ of $C,$ the following properties are satisfied:
\begin{enumerate}
\item The diagram
$$
\begin{tikzcd}
\calA_r\otimes_{\Ox_U}\calA_s\otimes_{\Ox_U}\calA_t \ar{r}\ar{d} & \calA_{r+s}\otimes_{\Ox_U}\calA_t \ar{d} \\
\calA_r\otimes_{\Ox_U}\calA_{s+t} \ar{r} & \calA_{r+s+t}
\end{tikzcd}
$$
is commutative.
\item The composed morphisms
$\calA_s \xrightarrow{x\mapsto x\otimes 1} \calA_s\otimes_{\Ox_U}\calA_0 \ra \calA_s$ and
$\calA_s \xrightarrow{x\mapsto 1\otimes x} \calA_0\otimes_{\Ox_U}\calA_s \ra \calA_s$
are both equal to the identity.
\end{enumerate}
\end{remark}

\begin{proposition}\label{krazzzz}
Let $(X,\Ox)$ be a ringed topos on a site $C,$ $\calI$ a sheaf of monoids of $X,$ $\sigma:\calI^2 \ra \calI$ the addition morphism, $\calA$ an $\calI$-indexed algebra, $\calE$ an $\calI$-indexed $\calA$-module and $\F$ an $\Ox_{\calI}$-module. Then, there is a canonical $\calI$-indexed $\calA$-algebra structure on $\sigma_!\left (\calE \boxtimes \F\right ).$
\end{proposition}

\begin{proof}
For an object $U$ of $C$ and local sections $s,t\in \Gamma(U,\calI),$ the sheaf $\left ( \sigma_! \left ( \calE \boxtimes \F \right ) \right )_s$ is, by \ref{indpropshriek} (3), the sheaf on $C_{/U}$ associated to the presheaf 
$$
(V\ra U) \mapsto \bigoplus_{\substack{u,v\in \calI(V) \\ u+v=s}}\calE_u(V) \otimes_{\Ox(V)}\F_v(V).
$$
For all morphisms $V\ra U$ of $C,$ we have a canonical morphism
$$
\calA_s(V) \otimes_{\Ox(V)} \bigoplus_{\substack{u,v\in \calI(V) \\ u+v=t}}\calE_u(V) \otimes_{\Ox(V)}\F_v(V) \ra \bigoplus_{\substack{\alpha,\beta \in \calI(V) \\ \alpha+\beta=s+t}} \calE_{\alpha}(V)\otimes_{\Ox(V)}\F_{\beta}(V),
$$
induced by the multiplications
$
\calA_s\otimes_{\Ox_U}\calE_u \ra \calE_{s+u}.
$
By \ref{indremark73}, these morphisms define the desired structure.
\end{proof}

\subsection*{Localization of indexed objects}

\begin{parag}\label{defresind}
Let $X$ be a topos, $\calI$ and $S$ sheaves of $X$ and $\F$ a sheaf of $X_{/\calI}.$ Consider the commutative diagram of topoi
$$
\begin{tikzcd}
(X_{/S})_{/\calI_{|S}}\ar{r}{\sim}\ar[swap]{dr}{j_{\calI_{|S}}} & X_{/S\times\calI} \ar{r}{j_{S\times \calI/\calI}} \ar{d}{j_{S\times\calI/S}} & X_{/\calI} \ar{d}{j_{\calI}} \\
 & X_{/S} \ar{r}{j_S} & X
\end{tikzcd}
$$
We call \emph{localization of $\F$ at $S$} the sheaf $j_{S\times \calI/\calI}^*\F$ of $(X_{/S})_{\calI_{|S}}$ and we denote it by $\F_{|S}.$
We will also denote the composition $(X_{/S})_{\calI_{|S}} \xrightarrow{\sim} X_{/S\times\calI} \ra X_{/\calI}$ by $j_{S\times \calI/\calI}.$
\end{parag}

\begin{proposition}\label{Locprop1}
Let $(X,\Ox)$ be a ringed topos, $\calI$ a sheaf of monoids of $X,$ $S$ a sheaf of sets of $X,$ $\calA$ an $\calI$-indexed $\Ox$-algebra and $\calE$ an $\calI$-indexed $\calA$-module. Then $\calA_{|S}$ (resp. $\calE_{|S}$) is canonically equipped with an $\calI_{|S}$-indexed $\Ox_S$-algebra structure (resp. an $\calI_{|S}$-indexed $\calA_{|S}$-module structure).
\end{proposition}

\begin{proof}
Let $p_1,p_2:\calI^2\ra \calI$ and $p_{1,S},p_{2,S}:\calI_{|S}^2\ra \calI_{|S}$ be the canonical projections and $\sigma:\calI^2\ra \calI$ and $\sigma_S:\calI_{|S}^2\ra \calI_{|S}$ the addition maps. We have the following commutative diagrams of topoi
$$
\begin{tikzcd}
(X_{/S})_{/\calI^2_{|S}} \ar{r}{\sim} \ar[bend right=-30]{rr}{j_{S\times \calI^2/\calI^2}} \ar{d}{p_{i,S}} & X_{/S\times \calI^2} \ar{r} \ar{d} & X_{/\calI^2} \ar{d}{p_i} \\
(X_{/S})_{/\calI_{|S}} \ar{r}{\sim} \ar[bend right=30]{rr}{j_{S\times\calI/\calI}} & X_{S\times \calI} \ar{r} & X_{/\calI}
\end{tikzcd}
\quad
\begin{tikzcd}
(X_{/S})_{/\calI^2_{|S}} \ar{r}{\sim} \ar[bend right=-30]{rr}{j_{S\times \calI^2/\calI^2}} \ar{d}{\sigma_{S}} & X_{/S\times \calI^2} \ar{r} \ar{d} & X_{/\calI^2} \ar{d}{\sigma} \\
(X_{/S})_{/\calI_{|S}} \ar{r}{\sim} \ar[bend right=30]{rr}{j_{S\times\calI/\calI}} & X_{S\times \calI} \ar{r} & X_{/\calI}
\end{tikzcd}
$$
Applying $j_{S\times \calI^2/\calI^2}^*$ to the morphism
$\calA\boxtimes \calA \ra \sigma^*\calA$
defining the $\calI$-indexed algebra structure on $\calA,$ we get
$\calA_{|S}\boxtimes \calA_{|S}\ra \sigma_{S}^*\calA_{|S}.$
Similarly, applying $j_{S\times \calI^2/\calI^2}^*$ to the morphism
$\calA\boxtimes \calE \ra \sigma^*\calE$
defining the $\calI$-indexed $\calA$-module structure on $\calE,$ we get
$\calA_{|S}\boxtimes \calE_{|S} \ra \sigma_S^*\calE_{|S}.$
These two morphisms define the desired structures.
\end{proof}

\begin{parag}\label{inddef2}
Let $(X,\Ox)$ be a ringed topos on a site $C,$ $\calI$ a sheaf of commutative monoids of $X,$ $\calA$ an $\calI$-indexed $\Ox$-algebra and $\calE$ and $\F$ two $\calI$-indexed $\calA$-modules. Denote by $p_1,p_2:\calI^2 \ra \calI$ the canonical projections and by $\sigma:\calI^2 \ra \calI$ the addition map. Equip $p_{1*}\mathscr{Hom}_{\Ox_{\calI^2}}(p_2^*\calE,\sigma^*\F)$ with the $\Ox_{\calI}$-module structure induced by the canonical morphism of rings $\Ox_{\calI} \ra p_{1*}\Ox_{\calI^2}.$
We denote by $\mathscr{Hom}_{\calA}(\calE,\F)$ the sub-$\Ox_{\calI}$-module of $p_{1*}\mathscr{Hom}_{\Ox_{\calI^2}}(p_2^*\calE,\sigma^*\F)$ defined as follows: let $s\in \Gamma(U,\calI)$ be a section of $\calI$ over an object $U$ of $C$ and consider the morphism $\sigma_s:\calI_{|U} \ra \calI_{|U},\ x\mapsto x+s.$ The $\Ox(U)$-module $p_{1*}\mathscr{Hom}_{\Ox_{\calI^2}}(p_2^*\calE,\sigma^*\F)(s)=\mathscr{Hom}_{\Ox_{\calI^2}}(p_2^*\calE,\sigma^*\F)(s\times \op{Id}_{\calI})$ is the set of $\Ox_U$-linear morphisms $f:(s\times \op{Id}_{\calI})^*p_2^*\calE \ra (s\times \op{Id}_{\calI})^*\sigma^*\F.$
Consider the commutative diagrams of topoi:
$$
\begin{tikzcd}
X_{/U\times \calI} \ar{r}{s\times \op{Id}_{\calI}}  \ar[swap]{dr}{j_{U\times \calI/\calI}} & X_{/\calI^2} \ar{d}{p_2} & & X_{U\times\calI} \ar{r}{s\times \op{Id}_{\calI}} \ar{d}{\sigma_s} & X_{/\calI^2}\ar{d}{\sigma} \\
 & X_{/\calI} & & X_{U\times \calI} \ar{r}{j_{U\times \calI/\calI}} & X_{/\calI}
\end{tikzcd}
$$ 
Define $\mathscr{Hom}_{\calA}(\calE,\F)(s)$ to be the subset of $\left (p_{1*}\mathscr{Hom}_{\Ox_{\calI^2}}(p_2^*\calE,\sigma^*\F)\right )(s)$ consisting of the morphisms of $\calI_{|U}$-indexed $\calA_{|U}$-modules $\calE_{|U} \ra \sigma_s^*\F_{|U}.$ This is clearly a subsheaf of $\Ox_{\calI}$-modules of $p_{1*}\mathscr{Hom}_{\Ox_{\calI^2}}(p_2^*\calE,\sigma^*\F).$

If $\calE=\F,$ we denote $\mathscr{Hom}_{\calA}(\calE,\F)$ simply by $\mathscr{End}_{\calA}(\calE).$
\end{parag}

\begin{parag}\label{indremtakriz1}
Keep the same hypothesis of \ref{inddef2}.
\begin{enumerate}
\item Let $U$ be an object of $C$ and $s\in \Gamma(U,\calI).$ Then $\mathscr{Hom}_{\calA}(\calE,\F)_s$ is the $\Ox_U$-module of $X_{/U}$ defined as follows: it sends an object $v:V\ra U$ of $C_{/U}$ to the set of morphisms of $\calI_{|V}$-indexed $\calA_{|V}$-modules $\calE_{|V} \ra \sigma_{s_{|V}}^*\F_{|V}.$

\item Let $U$ an object of $C,$ $s\in \Gamma(U,\calI)$ and $a\in \Gamma(U,\calA_s).$ The section $a$ can be considered as an element of $\mathscr{Hom}_{\calA}(\F,\F)(s)$ as follows: for any object $V\ra U$ of $C_{/U}$ and any section $t\in \Gamma(V,\calI_{|U})=\Gamma(V,\calI),$ consider the composition
$$
\begin{array}[t]{clclc}
\F_t & \ra & \calA_s\otimes_{\Ox_U}\F_t & \ra & \F_{t+s}\\
x& \mapsto & a\otimes x& \mapsto & ax.
\end{array}
$$
By \ref{era2indrem}, these compositions define a morphism
\begin{equation}\label{era2a}
a_{\F}:\F_{|U}\ra \sigma_s^*\F_{|U}.
\end{equation}

\item $\mathscr{Hom}_{\calA}(\calE,\F)$ has an $\calI$-indexed $\calA$-module structure defined as follows: Let $s,t\in \Gamma(U,\calI)$ be sections of $\calI$ over an object $U$ of $C.$ We will define a morphism
$$\calA_s\otimes_{\Ox_U}\mathscr{Hom}_{\calA}(\calE,\F)_t \ra \mathscr{Hom}_{\calA}(\calE,\F)_{s+t}.$$
Let $V \ra U$ be an object of $C_{/U},$ $a\in \Gamma(V,\calA_s)$ and $f\in \Gamma(V,\mathscr{Hom}_{\calA}(\calE,\F)_t).$ We set
$$af:\calE_{|V} \xrightarrow{f}\sigma_{t_{|V}}^*\F_{|V} \xrightarrow{\sigma_{t_{|V}}^*a_{\F}}\sigma_{(s+t)_{|V}}^*\F_{|V},$$
where $a$ is defined in (\ref{era2a}). Note that, by the $\calA_{|V}$-linearity of $f,$ the morphism $af$ is equal to the composition
$
\calE_{|V} \xrightarrow{a_{\calE}} \sigma_{s_{|V}}^* \calE_{|V} \xrightarrow{\sigma_{s_{|V}}^*f}\sigma^*_{(s+t)_{|V}}\F_{|V}.
$ 
\end{enumerate}
\end{parag}

\begin{definition}\label{inddef33}
Let $(X,\Ox)$ be a ringed topos, $\calI$ a sheaf of commutative monoids of $X,$ $\calA$ an $\calI$-indexed algebra and $\calE$ and $\F$ two $\calI$-indexed $\calA$-modules. We call \emph{internal Hom functor of $\calI$-indexed $\calA$-modules} the sheaf $\mathscr{Hom}_{\calA}(\calE,\F)$ defined in \ref{inddef2}, which is naturally equipped with an $\calI$-indexed $\calA$-module structure in \ref{indremtakriz1} (3).
\end{definition}

\begin{parag}\label{algstrendleft}
The following is a generalization of \ref{indremtakriz1} (3). Let $(X,\Ox)$ be a ringed topos on a site $C,$ $\calI$ a sheaf of monoids of $X,$ $p_1,p_2:\calI^2 \ra \calI$ the canonical projections, $\sigma:\calI^2\ra \calI$ the addition map, $\calA \ra \calB$ a morphism of $\calI$-indexed algebras and $\calE$ and $\F$ two $\calI$-indexed $\calB$-modules. Let $s,t\in \Gamma(U,\calI)$ be two sections of $\calI$ over an object $U$ of $C$ and $\sigma_s:\calI_{|U} \ra \calI_{|U},\ x\mapsto s+x.$ As in \ref{indremtakriz1}, a section $b\in \calB(s)$ defines morphisms
$b_{\calE}:\calE_{|U} \ra \sigma_s^*\calE_{|U}$ and
$b_{\F}:\F_{|U} \ra \sigma_s^*\F_{|U}.$
We have two morphisms
$$
\varphi^g(s,t):\begin{array}[t]{clc}
\calB(s) \otimes_{\Ox(U)}\mathscr{Hom}_{\calA}(\calE,\F)(t) & \ra & \mathscr{Hom}_{\calA}(\calE,\F)(s+t) \\
b \otimes f & \mapsto & f \circ b_{\calE},
\end{array}
$$
$$
\varphi^d(s,t):\begin{array}[t]{clc}
\calB(s) \otimes_{\Ox(U)}\mathscr{Hom}_{\calA}(\calE,\F)(t) & \ra & \mathscr{Hom}_{\calA}(\calE,\F)(s+t) \\
b \otimes f & \mapsto & b_{\F} \circ f,
\end{array}
$$
where the compositions are defined as follows:
\begin{alignat*}{2}
f\circ b_{\calE} &: \calE_{|U} \xrightarrow{b_{\calE}} \sigma_s^*\calE_{|U} \xrightarrow{\sigma_s^*f} \sigma_{s+t}^*\F_{|U}, \quad b_{\F}\circ f &: \calE_{|U} \xrightarrow{f} \sigma_t^*\F_{|U} \xrightarrow{\sigma_t^*b_{\F}} \sigma_{s+t}^*\F_{|U}.
\end{alignat*}
These morphisms, for every $(s,t),$ define two morphisms
$$
\varphi^g,\varphi^d:\calB \boxtimes \mathscr{Hom}_{\calA}(\calE,\F) \ra \sigma^*\mathscr{Hom}_{\calA}(\calE,\F).
$$
These define two structures of $\calI$-indexed $\calB$-modules on $\mathscr{Hom}_{\calA}(\calE,\F).$ We call them \emph{the left} (resp. \emph{right}) \emph{$\calI$-indexed $\calB$-module structure on $\mathscr{Hom}_{\calA}(\calE,\F).$}
\end{parag}

\begin{parag}\label{indtensprod}
Let $(X,\Ox)$ be a ringed topos over a site $C,$ $\calI$ a sheaf of commutative monoids of $X,$ $\calA$ an $\calI$-indexed algebra, $\calE$ and $\F$ two $\calI$-indexed $\calA$-modules (\ref{inddef}), $p_1,p_2:\calI^2\ra \calI,$ $q_1,q_2,q_3:\calI^3\ra \calI$ and $p_{ij}:\calI^3\ra \calI^2,\ 1\le i<j\le 3$ the canonical projections and $\sigma:\calI^2\ra \calI$ and $\tau:\calI^3\ra \calI$ the addition maps. The $\calA$-module structures on $\calE$ and $\F$ induce morphisms
$\calA\boxtimes \calE \boxtimes \F \ra (\sigma^*\calE) \boxtimes \F$
and
$\calA\boxtimes \calE \boxtimes \F \ra \calE \boxtimes (\sigma^*\F).$
The commutativity of the diagrams
$$
\begin{tikzcd}
\calI^3 \ar{r}{p_{12}} \ar[swap]{d}{\sigma \times \op{Id}_{\calI}} & \calI^2\ar{d}{\sigma} & & \calI^3\ar{dr}{q_3}\ar[swap]{d}{\sigma\times \op{Id}_{\calI}} & \\
\calI^2\ar{r}{p_1} & \calI & & \calI^2\ar{r}{p_2} & \calI
\end{tikzcd}
$$
implies that
$(\sigma^*\calE) \boxtimes \F=(\sigma \times \op{Id}_{\calI})^*(\calE \boxtimes \F).$
Similarly, we have
$\calE \boxtimes (\sigma^*\F)=(\op{Id}_{\calI}\times\sigma)^*(\calE \boxtimes \F).$
Thus we have morphisms
\begin{alignat}{2}
\calA\boxtimes \calE \boxtimes \F \ra (\sigma \times \op{Id}_{\calI})^*(\calE \boxtimes \F), \label{indeqqq1} \\
\calA\boxtimes \calE \boxtimes \F \ra (\op{Id}_{\calI}\times\sigma)^*(\calE \boxtimes \F). \label{indeqqq2}
\end{alignat}
Since $\tau=\sigma \circ (\sigma \times \op{Id}_{\calI})=\sigma \circ (\op{Id}_{\calI}\times \sigma),$ we get
$$\tau_!=\sigma_! \circ (\sigma \times \op{Id}_{\calI})_!=\sigma_! \circ (\op{Id}_{\calI}\times \sigma)_!.$$
By adjunction and applying $\sigma_!$ on both sides of (\ref{indeqqq1}) and (\ref{indeqqq2}), we get two morphisms
\begin{equation}\label{indeq10}
\tau_!(\calA \boxtimes \calE \boxtimes \F) \rightrightarrows \sigma_!(\calE \boxtimes \F).
\end{equation}
We denote by $\calE \circledast_{\calA} \F$ the coequalizer of these two morphisms. We then have an exact sequence of $\Ox_{\calI}$-modules
\begin{equation}\label{exseqindtens}
\tau_!(\calA \boxtimes \calE \boxtimes \F) \ra \sigma_!(\calE \boxtimes \F) \ra \calE \circledast_{\calA}\F \ra 0.
\end{equation}
\end{parag}

\begin{proposition}
Keep the same hypothesis of \ref{indtensprod}. Then the $\calI$-indexed $\calA$-module structures on $\calE$ and $\F$ induce an $\calI$-indexed $\calA$-module structure on $\calE\circledast_{\calA}\F.$
\end{proposition}

\begin{proof}
By \ref{indpropshriek} and \ref{indpropshriek2}, the two morphisms (\ref{indeq10}) are given explicitly as follows: for a section $l\in \calI(U)$ over an object $U$ of $C,$ the morphisms
$\tau_!(\calA \boxtimes \calE \boxtimes \F)_l \rightrightarrows \sigma_!(\calE \boxtimes \F)_l$
are induced by
\begin{equation}\label{indeq14}
\begin{array}[t]{clc}
\bigoplus\limits_{\substack{i,j,k\in \calI(V) \\ i+j+k=l}}\calA_i(V)\otimes_{\Ox(V)}\calE_i(V)\otimes_{\Ox(V)}\F_k(V) & \ra & \bigoplus\limits_{\substack{r,s\in \calI(V) \\ r+s=l}}\calE_r(V)\otimes_{\Ox(V)} \F_s(V) \\
a \otimes x \otimes y & \mapsto & (ax) \otimes y \\
a \otimes x \otimes y & \mapsto & x \otimes (ay)
\end{array}
\end{equation}
where $V\ra U$ is an object of $C_{/U},$ $a \in \calA_i(V),$ $x\in \calE_j(V)$ and $y\in \F_k(V).$

The adjunction morphism
$\calE \boxtimes \F \ra \sigma^*\sigma_!(\calE\boxtimes \F)$
induces a morphism
$p_{23}^*(\calE\boxtimes \F) \ra  p_{23}^*\sigma ^* \sigma_!(\calE\boxtimes \F)=(\op{Id}_{\calI}\times \sigma)^*p_2^*\sigma_!(\calE\boxtimes \F).$
Again, by adjunction we obtain a morphism
\begin{equation}\label{indeq9}
(\op{Id}_{\calI}\times \sigma)_!p_{23}^*(\calE\boxtimes \F) \ra p_2^*\sigma_!(\calE\boxtimes \F).
\end{equation}
By \ref{indpropshriekoh}, it is an isomorphism.
Its inverse
$p_2^*\sigma_!(\calE\boxtimes \F) \ra (\op{Id}_{\calI}\times \sigma)_!p_{23}^*(\calE\boxtimes \F)$
induces a morphism
$\calA \boxtimes \sigma_!(\calE \boxtimes \F) \ra p_1^*\calA \otimes_{\Ox_{\calI^2}} (\op{Id}_{\calI}\times \sigma)_!p_{23}^*(\calE \boxtimes \F).$
By (\cite{SGA43} Exposé IV 12.11.b), there exists a canonical isomorphism
\begin{alignat*}{2}
p_1^*\calA \otimes_{\Ox_{\calI^2}} (\op{Id}_{\calI}\times \sigma)_!p_{23}^*(\calE \boxtimes \F) &\xrightarrow{\sim}  (\op{Id}_{\calI}\times \sigma)_!\left ( (\op{Id}_{\calI}\times \sigma)^*p_1^*\calA \otimes_{\Ox_{\calI^3}}p_{23}^*(\calE \boxtimes \F)\right ) \\
&\xrightarrow{\sim} (\op{Id}_{\calI}\times \sigma)_!(\calA \boxtimes \calE \boxtimes \F).
\end{alignat*}
It follows that we have a morphism
$\calA \boxtimes \sigma_!(\calE\boxtimes \F) \ra (\op{Id}_{\calI}\times \sigma)_!(\calA\boxtimes \calE\boxtimes \F)$
and then a morphism
$\sigma_!(\calA \boxtimes \sigma_!(\calE\boxtimes \F)) \ra \tau_!(\calA \boxtimes \calE\boxtimes \F).$
Composing with one of the morphisms (\ref{indeq10}) and then the canonical projection, we get a morphism
$$\sigma_!(\calA \boxtimes \sigma_!(\calE\boxtimes \F)) \ra \calE \circledast_{\calA}\F$$
which induces by adjunction
\begin{equation}\label{indeq13}
\calA \boxtimes \sigma_!(\calE \boxtimes \F)\ra \sigma^*(\calE \circledast_{\calA}\F).
\end{equation}
For sections $k,l\in \calI(U)$ over an object $U$ of $C,$ the fiber of (\ref{indeq13}) at $(k,l)$ is induced by
$$
\begin{array}[t]{clc}
\calA_k(U) \otimes_{\Ox(U)} \bigoplus\limits_{i+j=l}\calE_i(U) \otimes_{\Ox(U)}\F_j(U) & \ra & \bigoplus\limits_{r+s=k+l}\calE_r(U)\otimes_{\Ox(U)}\F_s(U) / \sim \\
a \otimes x\otimes y & \mapsto & (ax)\otimes y=x\otimes (ay)
\end{array}
$$
where $a\in \calA_k(U),$ $x\in \calE_i(U)$ and $y\in \F_j(U)$ and $\sim$ is the equalizer of the morphisms (\ref{indeq14}).
This morphism factors through the quotient of
$$\calA_k(U) \otimes_{\Ox(U)} \bigoplus\limits_{i+j=l}\calE_i(U) \otimes_{\Ox(U)}\F_j(U)$$
by the submodule generated by sections
$a \otimes ((bx)\otimes y-x\otimes(by)),$
where $a\in \calA_k(U),$ $b\in \calA_{\alpha}(U),$ $x\in \calE_{\beta}(U),$ $y\in \F_{\gamma}(U)$ and $\alpha,\beta,\gamma \in \calI(U)$ such that $\alpha+\beta+\gamma=i+j.$
It follows that the morphism (\ref{indeq13}) factors through
$\calA \boxtimes (\calE\circledast_{\calA}\F).$
The resulting morphism satisfies the conditions of (\ref{inddef} 3) and thus gives the desired structure.
The $\calI$-indexed $\calA$-module $\calE\circledast_{\calA}\F$ is called \emph{the tensor product of $\calE$ and $\F$ over $\calA.$} 
\end{proof}

\begin{proposition}\label{indproptensprod}
Keep the same hypothesis of \ref{indtensprod} and consider a section $k\in \Gamma(U,\calI)$ of $\calI$ over an object $U$ of $C.$ By construction, $(\calE \circledast_{\calA}\F)_k$ is the sheaf associated to the the presheaf
$$
\begin{array}[t]{clc}
C_{/U} & \mapsto & \op{Sets} \\
V & \mapsto & \bigoplus\limits_{\substack{i,j\in \calI(V) \\ i+j=k}}\calE_i(V)\otimes_{\Ox(V)}\F_j(V)/\sim,
\end{array}
$$
where $(ax)\otimes y\sim x\otimes (ay)$ for every section $a\in\calA_r(V),$ $x\in \calE_s(V),$ $y\in \F_t(V)$ and $r,s,t\in \calI(V)$ such that $r+s+t=k.$
\end{proposition}

\begin{remark}\label{indrem1}
Keep the same hypothesis of \ref{indtensprod}. 
\begin{enumerate}
\item Let $f:\calA \ra \calA'$ be a morphism of $\calI$-indexed algebras. The composition
$$\calA\boxtimes \calA' \xrightarrow{f\boxtimes \op{Id}_{\calA'}} \calA' \boxtimes \calA' \xrightarrow{\pi_{\calA'}} \sigma^*\calA'$$
defines an $\calI$-indexed $\calA$-module structure on $\calA'.$
The morphism 
$$\pi_{\calA'}:\calA' \boxtimes \calA' \ra \sigma^*\calA'$$
defining the indexed algebra structure on $\calA',$ induces a morphism
$$\sigma_!(\calA' \boxtimes \calA') \ra \calA'$$
and then an $\calA$-linear morphism
\begin{equation}\label{indrem11}
v:\calA' \circledast_{\calA}\calA' \ra \calA'.
\end{equation}
We call $v$ the \emph{$\calI$-indexed codiagonal map.}
\item Let $\calA \ra \calA'$ be a morphism of $\calI$-indexed algebras. Denote by $e$ the final object of $X,$ $0:e\ra \calI$ the zero section of $\calI$ and $\nu:\calI \ra e$ the canonical morphism. The morphisms
$$\calA' \xrightarrow{\op{Id}_{\calA'}\times 1_{\calA'}} \calA' \otimes_{\Ox_{\calI}} \nu^*\calA'_0=(\op{Id}_{\calI}\times 0)^*(\calA' \boxtimes \calA')$$
$$\calA' \xrightarrow{1_{\calA'}\times \op{Id}_{\calA'}} \nu^*\calA'_0\otimes_{\Ox_{\calI}} \calA'=(0\times \op{Id}_{\calI})^*(\calA' \boxtimes \calA')$$
induce two morphisms
$$(\op{Id}_{\calI}\times 0)_! \calA' \ra \calA' \boxtimes \calA'$$
$$(0\times \op{Id}_{\calI})_! \calA' \ra \calA' \boxtimes \calA'$$
and then two morphisms
\begin{equation}\label{indeq17}
\calA'=\sigma_!(\op{Id}_{\calI}\times 0)_!\calA' \rightrightarrows \sigma_!(\calA' \boxtimes \calA').
\end{equation}
Composing with the canonical projection
$$\sigma_!(\calA' \boxtimes \calA') \ra \calA' \circledast_{\calA}\calA',$$
we get two $\calA$-linear morphisms
\begin{equation}\label{indrem12}
\pi_1,\pi_2:\calA' \ra \calA' \circledast_{\calA} \calA'.
\end{equation}
These morphisms are induced by the morphisms of presheaves
$$
\begin{array}[t]{clclc}
\calA'_k(U) & \ra & \calA'_k(U)\otimes_{\Ox(U)}\calA'_0(U) & \subset & \bigoplus\limits_{\substack{i,j\in \calI(U) \\ i+j=k}}\calA'_i(U) \otimes_{\Ox(U)}\calA'_j(U) \\
x & \mapsto & x \otimes 1 & &
\end{array}
$$
and
$$
\begin{array}[t]{clclc}
\calA'_k(U) & \ra & \calA'_0(U)\otimes_{\Ox(U)}\calA'_k(U) & \subset & \bigoplus\limits_{\substack{i,j\in \calI(U) \\ i+j=k}}\calA'_i(U) \otimes_{\Ox(U)}\calA'_j(U) \\
x & \mapsto & 1 \otimes x & &
\end{array}
$$
\end{enumerate}
\end{remark}

\begin{lemma}\label{lemtop15}
Let $u:(T',\Ox_{T'})\ra (T,\Ox_T)$ be a morphism of ringed topoi, $I$ a sheaf of $T,$ $I'=u^{-1}I$ and $j_I$ and $j_{I'}$ the localization morphisms. Equip $T_{/I}$ and $T'_{/I'}$ with the rings $j_I^{-1}\Ox_T$ and $j_{I'}^{-1}\Ox_{T'}$ respectively. Consider the commutative diagram of ringed topoi
$$
\begin{tikzcd}
T'_{/I'} \ar{r}{u'} \ar[swap]{d}{j_{I'}} & T_{/I} \ar{d}{j_I} \\
T' \ar{r}{u} & T,
\end{tikzcd}
$$
where $u'$ is induced by $u.$
\begin{enumerate}
\item For every sheaf $\F'$ of $T',$ there exists a canonical isomorphism
$$j_I^*u_*\F' \xrightarrow{\sim} u'_*j_{I'}^*\F'.$$
\item For every sheaf $\F$ of $T_{/I},$ there exists a canonical isomorphism
$$j_{I'!}u'^*\F \xrightarrow{\sim} u^*j_{I!}\F.$$
\end{enumerate}
The proposition is true if we consider topoi instead of ringed topoi and we replace $u^*,$ $u'^*,$ $j_{\calI}^*$ and $j_{\calI'}^*$ by $u^{-1},$ $u'^{-1},$ $j_{\calI}^{-1}$ and $j_{\calI'}^{-1}$ respectively.
\end{lemma}

\begin{proof}
The adjunction morphism
$\F' \ra j_{I'*}j_{I'}^*\F'$
induces
$u_*\F' \ra u_*j_{I'*}j_{I'}^*\F'=j_{I*}u'_*j_{I'}^*\F'.$
We then get a canonical morphism
\begin{equation}\label{eqtop11}
j_I^*u_*\F' \ra u'_*j_{I'}^*\F'.
\end{equation}
This morphism is an isomorphism. Indeed, if $U$ is an object of $T$ then
$$
\left (j_I^*u_*\F' \right )\left (U_{/I} \right )=\left (u_*\F'\right )(U)=\F'(u^{-1}U),
$$
$$
\left (u'_*j_{I'}^*\F' \right )\left (U_{/I} \right )=\left (j_{I'}^*\F' \right )\left ( (u^{-1}U)_{/I'}\right )=\F'(u^{-1}U).
$$
The second assertion follows from the first as follows: for any object $\calG$ of $T',$ we have
\begin{alignat*}{2}
\op{Hom}_{T'}(j_{I'!}u'^*\F,\calG) &= \op{Hom}_{T_{/I}}(\F,u'_*j_{I'}^*\calG )= \op{Hom}_{T_{/I}}(\F,j_I^*u_*\calG) &= \op{Hom}_{T'}(u^*j_{I!}\F,\calG).
\end{alignat*}
The proof for topoi instead of ringed topoi is similar.
\end{proof}

\begin{proposition}\label{ren3}
Let $(X,\Ox)$ be a ringed topos, $\calI$ a sheaf of monoids of $X,$ $S$ a sheaf of $X$, $\calA$ an $\calI$-indexed algebra and $\calE$ and $\F$ two $\calI$-indexed $\calA$-modules. There exists a canonical isomorphism
$$
\left (\calE_{|S} \right ) \circledast_{\calA_{|S}} \left (\F_{|S} \right ) \xrightarrow{\sim} \left (\calE \circledast_{\calA}\F \right )_{|S}  .
$$
\end{proposition}

\begin{proof}
Denote by $\sigma:\calI^2 \ra \calI$ and $\tau:\calI^3 \ra \calI$ the addition maps and $p_1,p_2:\calI^2 \ra \calI$ the canonical projections. We have the commutative diagram of topoi
$$
\begin{tikzcd}
\left (X_{/S} \right )_{/\calI_{|S}^2} \ar{r}{j_{\calI^2\times S/\calI^2}} \ar[swap]{d}{\sigma_{|S}} & X_{/\calI^2} \ar{d}{\sigma} \\
\left (X_{/S} \right )_{/\calI_{|S}} \ar{r}{j_{\calI\times S/\calI}} & X_{/\calI}.
\end{tikzcd}
$$
By \ref{lemtop15}, there exists a canonical isomorphism
$
\left (\sigma_{|S} \right )_! \left ( \left ( \calE \boxtimes \F \right )_{|S} \right ) \xrightarrow{\sim} \left ( \sigma_!(\calE \boxtimes \F) \right )_{|S}.
$
By the commutative diagram
$$
\begin{tikzcd}
\left (X_{/S} \right )_{/\calI_{|S}^2} \ar{r}{j_{\calI^2\times S/\calI^2}} \ar[swap]{d}{p_{i|S}} & X_{/\calI^2} \ar{d}{p_i} \\
\left (X_{/S} \right )_{/\calI_{|S}} \ar{r}{j_{\calI\times S/\calI}} & X_{/\calI},
\end{tikzcd}
$$
we get a canonical isomorphism
$
\left (\calE_{|S} \right ) \boxtimes \left (\F_{|S} \right ) \xrightarrow{\sim} \left ( \calE \boxtimes \F \right )_{|S}.
$
It follows that we have a canonical isomorphism
$
\left (\sigma_{|S} \right )_! \left (\calE_{|S}  \boxtimes \F_{|S}  \right ) \xrightarrow{\sim} \left ( \sigma_!(\calE \boxtimes \F) \right )_{|S}.
$
Similarly, we have a canonical isomorphism
$
\left (\tau_{|S} \right )_! \left ( \calA_{|S} \boxtimes \calE_{|S} \boxtimes \F_{|S} \right ) \xrightarrow{\sim} \left ( \tau_!(\calA \boxtimes \calE \boxtimes \F) \right )_{|S}.
$
By definition of $\calE \circledast_{\calA} \F$ and the right exactness of $\left (\tau_{|S} \right )_!$ and the restriction functor $\calM \mapsto \calM_{|S,}$ these isomorphisms fit into a commutative diagram
$$
\begin{tikzcd}
\left (\tau_{|S} \right )_! \left ( \calA_{|S} \boxtimes \calE_{|S} \boxtimes \F_{|S} \right ) \ar[sloped]{d}{\sim} \ar{r} & \left (\sigma_{|S} \right )_! \left ( \calE_{|S} \boxtimes \F_{|S} \right ) \ar{r} \ar[sloped]{d}{\sim} & \left (\calE_{|S} \right ) \circledast_{\calA_{|S}} \left (\F_{|S} \right ) \ar{r}  \ar[sloped,dashed]{d}{\sim} & 0 \\
\left ( \tau_!(\calA \boxtimes \calE \boxtimes \F) \right )_{|S} \ar{r} & \left ( \sigma_!(\calE \boxtimes \F) \right )_{|S} \ar{r} & \left (\calE \circledast_{\calA} \F\right )_{|S} \ar{r} & 0,
\end{tikzcd}
$$
where the rows are induced by \eqref{exseqindtens} and hence exact.
Hence the existence of the third vertical isomorphism.
\end{proof}

\begin{proposition}
Let $C$ be a site, $(X,\Ox)$ a ringed topos on $C,$ $\calI$ a sheaf of monoids of $X,$ $\calB \ra \calA$ a morphism of $\calI$-indexed algebras and $\calE,$ $\F$ and $\calG$ three $\calI$-indexed $\calA$-modules. Denote by $p_1,p_2:\calI^2 \ra \calI$ the canonical projections and by $\sigma:\calI^2 \ra \calI$ the addition map. Consider the morphism
\begin{equation}\label{takcompadj}
\mathscr{Hom}_{\calA}(\calE,\F) \boxtimes \mathscr{Hom}_{\calA}(\F,\calG) \ra \sigma^*\mathscr{Hom}_{\calA}(\calE,\calG)\end{equation}
defined, for every sections $s,t\in \Gamma(U,\calI)$ of $\calI$ over an object $U$ of $C,$ by
\begin{equation}
\begin{array}[t]{clc}
\mathscr{Hom}_{\calA}(\calE,\F)(s) \otimes_{\Ox(U)}\mathscr{Hom}_{\calA}(\F,\calG)(t) & \ra & \mathscr{Hom}_{\calA}(\calE,\calG)(s+t) \\
f\otimes g & \mapsto & g\circ f,
\end{array}
\end{equation}
where $g\circ f$ is the following composition
\begin{equation}\label{era2compshift}
g\circ f:\calE_{|U} \xrightarrow{f} \sigma_s^*\F_{|U} \xrightarrow{\sigma_s^*g} \sigma_{s+t}^*\calG_{|U}.
\end{equation}
The morphism \eqref{takcompadj} induces a morphism
\begin{equation}\label{era3comptakriz}
\mathscr{Hom}_{\calA}(\calE,\F) \circledast_{\calB} \mathscr{Hom}_{\calA}(\F,\calG) \ra \mathscr{Hom}_{\calA}(\calE,\calG).
\end{equation}
In addition, if $\calE=\F=\calG,$ the operation \eqref{era2compshift} defines an $\calI$-indexed algebra structure on $\mathscr{End}_{\calA}(\calE).$
\end{proposition}

\begin{proof}
The morphism \eqref{takcompadj} induces, by adjunction,
\begin{equation}\label{eqa}
\sigma_!\left (\mathscr{Hom}_{\calA}(\calE,\F) \boxtimes \mathscr{Hom}_{\calA}(\F,\calG) \right ) \ra \mathscr{Hom}_{\calA}(\calE,\calG).
\end{equation}
By \ref{indpropshriek} (3), $\sigma_!\left (\mathscr{Hom}_{\calA}(\calE,\F) \boxtimes \mathscr{Hom}_{\calA}(\F,\calG) \right )$ is the sheaf associated to the presheaf
$$(s:U\ra \calI) \mapsto \bigoplus_{\substack{u,v\in \Gamma(U,\calI) \\ u+v=s}}\mathscr{Hom}_{\calA}(\calE,\F)(u) \otimes_{\Ox(U)} \mathscr{Hom}_{\calA}(\F,\calG)(v).$$
The morphism \eqref{eqa} is then induced, for every sections $u,v,s\in \Gamma(U,\calI)$ such that $u+v=s,$ by
$$
\varphi_{u,v}:\begin{array}[t]{clc}
\mathscr{Hom}_{\calA}(\calE,\F)(u) \otimes_{\Ox(U)} \mathscr{Hom}_{\calA}(\F,\calG)(v) & \ra & \mathscr{Hom}_{\calA}(\calE,\calG)(s) \\
f \otimes g & \mapsto & g\circ f.
\end{array}
$$
By \ref{indremtakriz1} (3), for any sections $u,v,w\in \Gamma(U,\calI),$ $f\in \mathscr{Hom}_{\calA}(\calE,\F)(u),$ $g\in\mathscr{Hom}_{\calA}(\F,\calG)(v)$ and $a\in \calB(w),$ we have
$$\varphi_{u+w,v}((af)\otimes g)=\varphi_{u,v+w}(f\otimes (ag)).$$
The result then follows by \ref{indproptensprod}.
The assertion about the $\calI$-indexed algebra structure on $\mathscr{End}_{\calA}(\calE)$ is clear.
\end{proof}

\begin{proposition}\label{indprop710}
Let $(X,\Ox)$ be a ringed topos and $\calI$ a sheaf of monoids of $X.$
\begin{enumerate}
\item For every $\calI$-indexed algebra $\calA$ and $\calI$-indexed $\calA$-module $\calE,$ there exists a canonical isomorphism
$$\calA\circledast_{\calA}\calE \xrightarrow{\sim} \calE.$$
\item For every $\Ox$-algebra $\calA$ over $X$ and $\calI$-indexed algebra $\calB,$ the $\Ox_{\calI}$-module $j_{\calI}^*\calA\otimes_{\Ox_{\calI}}\calB$ is naturally an $\calI$-indexed algebra and the canonical morphism
$$\calB \ra j_{\calI}^*\calA \otimes_{\Ox_{\calI}}\calB$$
is a morphism of $\calI$-indexed algebras.
\item For every $\calI$-indexed algebra $\calA,$ $\Ox$-module $\calE$ and $\calI$-indexed $\calA$-modules $\F$ and $\calG,$ the $\Ox_{\calI}$-module $(j_{\calI}^*\calE)\otimes_{\Ox_{\calI}}\F$ is naturally an $\calI$-indexed $\calA$-module and there exists a canonical isomorphism
\begin{equation}\label{indiso1}
\left ( j_{\calI}^*\calE \otimes _{\Ox_{\calI}}\F \right ) \circledast_{\calA} \calG \xrightarrow{\sim} j_{\calI}^*\calE \otimes_{\Ox_{\calI}} \left ( \F \circledast_{\calA}\calG \right ).
\end{equation}
\end{enumerate}
\end{proposition}

\begin{proof}
For the first assertion, the morphism
$\calA \boxtimes \calE \ra \sigma^*\calE$
defining the indexed $\calA$-module structure on $\calE,$ induces by adjunction
\begin{equation}\label{indeq16}
\sigma_!(\calA \boxtimes \calE) \ra \calE.
\end{equation}
By the description given in \ref{indproptensprod}, (\ref{indeq16}) induces a morphism
\begin{equation}\label{indeq20}
\calA \circledast_{\calA}\calE \rightarrow \calE.
\end{equation}
Let $k\in \calI(U)$ be a section over an object $U$ of $C.$ The morphism
$\left (\calA \circledast_{\calA}\calE \right )_k\ra \calE_k,$
induced by (\ref{indeq20}), is defined by the morphism of presheaves $\varphi$ given, for every object $V$ of $C_{/U},$ by
$$
\varphi(V):\begin{array}[t]{clc}
\bigoplus\limits_{\substack{i,j\in \calI(V)\\ i+j=k}}\calA_i(V)\otimes \calE_j(V) /\sim& \ra & \calE_k(V) \\
a\otimes x & \mapsto & ax
\end{array}$$
for $a\in \calA_i(V),\ x\in \calE_j(V)$ and where $\sim$ is the relation defined in \ref{indproptensprod}.
The morphism $\varphi$ is clearly surjective. Now let $n$ be a positive integer, $i_1,\hdots,i_n,j_1,\hdots,j_n \in \calI(V)$ such that $i_1+j_1=i_2+j_2=\hdots=i_n+j_n=k,$ $a_{l}\in \calA_{i_l}(V),$ $x_{l}\in \calE_{j_l}(V)$ for all $1\le l\le n,$ such that 
$$\varphi(a_1\otimes x_1+\hdots+a_n\otimes x_n)=0.$$
Then
$\sum_{l=1}^na_{l}x_{l}=0.$
But
$\sum_{l=1}^na_{l}\otimes x_{l} \sim 1\otimes \left (\sum_{l=1}^n a_{l}x_{l} \right ).$
This proves that $\varphi$ is injective.

For the second assertion, the indexed algebra structure on $j_{\calI}^*\calA\otimes_{\Ox_{\calI}}\calB$ is given by the composition
\begin{alignat*}{2}
(j_{\calI}^*\calA\otimes_{\Ox_{\calI}}\calB)\boxtimes (j_{\calI}^*\calA\otimes_{\Ox_{\calI}}\calB) & \xrightarrow{\sim} j_{\calI^2}^*(\calA\otimes_{\Ox}\calA)\otimes_{\Ox_{\calI^2}}(\calB\boxtimes \calB) \\
& \xrightarrow{j_{\calI^2}^*\pi_{\calA}\otimes \pi_{\calB}} (j_{\calI^2}^*\calA)\otimes_{\Ox_{\calI^2}}\sigma^*\calB\\
& \xrightarrow{\sim} \sigma^*(j_{\calI}^*\calA\otimes_{\Ox_{\calI}}\calB),
\end{alignat*}
where the first and third arrows are the canonical isomorphisms and $\pi_{\calA}$ and $\pi_{\calB}$ are the multiplication maps defining the algebra structures on $\calA$ and $\calB$ respectively. The fact that the canonical morphism
$\calB \ra j_{\calI}^*\calA \otimes_{\Ox_{\calI}}\calB$
is a morphism of $\calI$-indexed algebras follows from the obvious commutativity of the following diagram, for any local sections $i$ and $j$ of $\calI$ over an object $U,$
$$
\begin{tikzcd}
\calB_i \otimes_{\Ox_U}\calB_j \ar{r}\ar{d} & (\calA_{|U}\otimes_{\Ox_U}\calB_i)\otimes_{\Ox_U}(\calA_{|U}\otimes_{\Ox_U}\calB_j)\ar{d} \\
\calB_{i+j}\ar{r} & \calA_{|U}\otimes_{\Ox_U}\calB_{i+j}
\end{tikzcd}
$$
For the third assertion, the $\calI$-indexed $\calA$-module structure on $j_{\calI}^*\calE \otimes_{\Ox_{\calI}}\F$ is given by the composition
\begin{alignat*}{2}
\calA \boxtimes (j^*_{\calI}\calE \otimes_{\Ox_{\calI}}\F) &\xrightarrow{\sim} (\calA\boxtimes \F)\otimes_{\Ox_{\calI^2}}j_{\calI^2}^*\calE \\
& \xrightarrow{\pi \otimes \op{Id}_{j_{\calI^2}^*\calE}} (\sigma^*\F) \otimes_{\Ox_{\calI^2}}j_{\calI^2}^*\calE \\
& \xrightarrow{\sim} \sigma^*\left (j_{\calI}^*\calE \otimes_{\Ox_{\calI}}\F \right ),
\end{alignat*}
where the first and third arrows are the canonical isomorphisms and $\pi$ is the multiplication map defining the indexed module structure on $\F.$
The canonical isomorphism
$$(j_{\calI}^*\calE \otimes_{\Ox_{\calI}}\F)\boxtimes \calG \xrightarrow{\sim} j_{\calI^2}^*\calE \otimes_{\Ox_{\calI^2}}(\F \boxtimes \calG)$$
induces an isomorphism
$$\sigma_!\left ((j_{\calI}^*\calE \otimes_{\Ox_{\calI}}\F)\boxtimes \calG \right )\xrightarrow{\sim} \sigma_!\left (j_{\calI^2}^*\calE \otimes_{\Ox_{\calI^2}}(\F \boxtimes \calG)\right ).$$
Composing it with the canonical isomorphism (\cite{SGA43} Exposé IV 12.11.b)
$$\sigma_!\left (j_{\calI^2}^*\calE \otimes_{\Ox_{\calI^2}}(\F \boxtimes \calG)\right ) \xrightarrow{\sim} j_{\calI}^*\calE \otimes_{\Ox_{\calI}} \sigma_!(\F \boxtimes \calG),$$
we get a canonical isomorphism
$$\sigma_!\left ((j_{\calI}^*\calE \otimes_{\Ox_{\calI}}\F)\boxtimes \calG \right )\xrightarrow{\sim}j_{\calI}^*\calE \otimes_{\Ox_{\calI}} \sigma_!(\F \boxtimes \calG)$$
which induces the desired isomorphism.
\end{proof}

\begin{proposition}\label{propisoAkaza}
Let $(X,\Ox)$ be a ringed topos, $\calI$ a sheaf of monoids of $X,$ $\calA$ an $\Ox$-algebra, $\calA'$ an $\calA$-algebra, $\calB$ an $\calI$-indexed $\Ox$-algebra and $\calE$ an $\calI$-indexed $(\calA_{\calI}\otimes_{\Ox_{\calI}}\calB)$-module. Then $\calE$ is canonically equipped with an $\calA_{\calI}$-module structure and there exists a canonical isomorphism of $\calI$-indexed $(\calA'_{\calI}\otimes_{\Ox_{\calI}}\calB)$-modules
\begin{equation}\label{propisoAkazahh}
\calA'_{\calI}\otimes_{\calA_{\calI}}\calE \xrightarrow{\sim} (\calA'_{\calI}\otimes_{\Ox_{\calI}}\calB)\circledast_{\calA_{\calI}\otimes_{\Ox_{\calI}}\calB}\calE.
\end{equation}
\end{proposition}

\begin{proof}
Let $p_1,p_2:\calI^2\ra \calI$ be the canonical projections, $\sigma:\calI^2\ra \calI$ the addition map and $i_2:\calI \ra \calI^2,\ s\mapsto (0,s).$ The canonical map
$$(\calA_{\calI}\otimes_{\Ox_{\calI}}\calB)\boxtimes \calE \ra \sigma^*\calE$$
induces
\begin{equation}\label{eq7241}
i_2^*p_1^*(\calA_{\calI}\otimes_{\Ox_{\calI}}\calB)\otimes_{\Ox_{\calI}}i_2^*p_2^*\calE \ra i_2^*\sigma^*\calE.
\end{equation}
Let $e$ be the final object of $X.$ By the commutativity of the diagrams
$$
\begin{tikzcd}
\calI \ar{r}{i_2} \ar{d} & \calI^2\ar{d}{p_1} & & \calI \ar{r}{i_2} \ar[swap]{dr}{\op{Id}_{\calI}} & \calI^2 \ar{d}{p_2} & & \calI \ar{r}{i_2} \ar[swap]{dr}{\op{Id}_{\calI}} & \calI^2 \ar{d}{\sigma} \\
e \ar{r}{0} & \calI & & & \calI & & & \calI
\end{tikzcd}
$$
the equality \eqref{eq7241} is equivalent to
$(\calA_{\calI} \otimes_{\Ox_{\calI}}\calB_{0,\calI})\otimes_{\Ox_{\calI}}\calE \ra \calE.$
Composing this morphism with the morphism induced by
$\calA_{\calI} \ra \calA_{\calI}\otimes_{\Ox_{\calI}} \calB_{0,\calI},\ x\mapsto x\otimes 1,$
we get the $\calA_{\calI}$-module structure on $\calE.$

For any section $s\in  \Gamma(U,\calI)$ over an object $U,$ consider the composition of
$$\begin{array}[t]{clc}
\calA'_U\otimes_{\Ox_U}\calE_s & \ra & \calA'_U\otimes_{\Ox_U}\calB_0\otimes_{\Ox_U}\calE_s \\
a\otimes x & \mapsto & a\otimes 1 \otimes x
\end{array}$$
with the canonical morphism
$\calA'_U\otimes_{\Ox_U}\calB_0\otimes_{\Ox_U}\calE_s \ra \left ((\calA'_{\calI}\otimes_{\Ox_{\calI}}\calB)\circledast_{\calA_{\calI}\otimes_{\Ox_{\calI}}\calB}\calE \right )_s.$
These compositions induce a morphism
\begin{equation}\label{eq1Akaza}
\calA'_{\calI}\otimes_{\calA_{\calI}}\calE \ra (\calA'_{\calI}\otimes_{\Ox_{\calI}}\calB)\circledast_{\calA_{\calI}\otimes_{\Ox_{\calI}}\calB}\calE.
\end{equation}
Conversely, the canonical morphism
$(\calA_{\calI}\otimes_{\Ox_{\calI}}\calB)\boxtimes \calE \ra \sigma^*\calE$
induces a morphism
$$\calA'_{\calI^2}\otimes_{\calA_{\calI^2}}((\calA_{\calI}\otimes_{\Ox_{\calI}}\calB)\boxtimes \calE)= (\calA'_{\calI}\otimes_{\Ox_{\calI}}\calB)\boxtimes \calE \ra \calA'_{\calI^2}\otimes_{\calA_{\calI^2}}\sigma^*\calE=\sigma^*(\calA'_{\calI}\otimes_{\calA_{\calI}}\calE).$$
By adjunction, we get
$\sigma_!\left ( (\calA'_{\calI}\otimes_{\Ox_{\calI}}\calB)\boxtimes \calE\right ) \ra \calA'_{\calI}\otimes_{\calA_{\calI}}\calE.$
By checking on fibers, this morphism yields
\begin{equation}\label{eq2Akaza}
(\calA'_{\calI}\otimes_{\Ox_{\calI}}\calB)\circledast_{\calA_{\calI}\otimes_{\Ox_{\calI}}\calB}\calE \ra \calA'_{\calI}\otimes_{\calA_{\calI}}\calE.
\end{equation}
Again, by checking on fibers, we prove that (\ref{eq1Akaza}) and (\ref{eq2Akaza}) are inverse to each other.
\end{proof}

\begin{proposition}\label{kraz827}
Let $(X,\Ox)$ be a ringed topos, $\calI$ a sheaf of monoids of $X,$ $\calA$ an $\Ox$-algebra, $\calB$ an $\calI$-indexed $\Ox$-algebra, $\calB'$ an $\calI$-indexed $\calB$-algebra and $\calE$ an $\calI$-indexed $\calA_{\calI}\otimes_{\Ox_{\calI}}\calB$-module. Then $\calE$ is canonically equipped with an $\calI$-indexed $\calB$-module structure and with an $\calA_{\calI}$-module structure and there exists a canonical isomorphism of $\calI$-indexed $(\calB'\otimes_{\Ox_{\calI}}\calA_{\calI})$-modules
$$\calB'\circledast_{\calB}\calE \xrightarrow{\sim} (\calB'\otimes_{\Ox_{\calI}}\calA_{\calI})\circledast_{\calB\otimes_{\Ox_{\calI}}\calA_{\calI}}\calE.$$
\end{proposition}

\begin{proof}
Let $p_1,p_2:\calI^2\ra \calI$ be the canonical projections and $\sigma:\calI^2\ra \calI$ the addition map.
The fact that $\calE$ is naturally equipped with an $\calA_{\calI}$-module structure follows from proposition \ref{propisoAkaza}. The $\calI$-indexed $\calB$-module structure on $\calE$ comes from the composition
$$\calB \boxtimes \calE \ra (\calA_{\calI}\otimes_{\Ox_{\calI}}\calB)\boxtimes \calE \ra \sigma^*\calE,$$
where the first arrow is induced by 
$\calB \ra \calA_{\calI}\otimes_{\Ox_{\calI}} \calB,\ b\mapsto 1\otimes b$
and the second one comes from the $\calI$-indexed $(\calA_{\calI}\otimes_{\Ox_{\calI}}\calB)$-module structure on $\calE.$

The canonical morphism
$\calA_{\calI}\otimes_{\Ox_{\calI}}\calE \ra \calE$
induces a morphism
$\calB' \boxtimes (\calA_{\calI} \otimes_{\Ox_{\calI}}\calE)=(\calB'\otimes_{\Ox_{\calI}}\calA_{\calI})\boxtimes \calE \ra \calB' \boxtimes \calE$
and then
$\sigma_!((\calB'\otimes_{\Ox_{\calI}}\calA_{\calI})\boxtimes \calE) \ra \sigma_!(\calB' \boxtimes \calE).$
By checking on fibers, we deduce a morphism
\begin{equation}\label{eq3Akaza}
(\calB'\otimes_{\Ox_{\calI}}\calA_{\calI})\circledast_{\calB\otimes_{\Ox_{\calI}}\calA_{\calI}} \calE \ra \calB' \circledast_{\calB} \calE.
\end{equation}
Conversely, the morphism
$\calB' \ra \calB'\otimes_{\Ox_{\calI}}\calA_{\calI},\ b\mapsto b\otimes 1$
induces
$\sigma_! (\calB' \boxtimes \calE ) \ra \sigma_!\left ( (\calB' \otimes_{\Ox_{\calI}}\calA_{\calI})\boxtimes \calE \right )$
and then
\begin{equation}\label{eq4Akaza}
\calB' \circledast_{\calB} \calE \ra (\calB'\otimes_{\Ox_{\calI}}\calA_{\calI})\circledast_{\calB\otimes_{\Ox_{\calI}}\calA_{\calI}} \calE.
\end{equation}
By checking on fibers, we prove that \eqref{eq3Akaza} and \eqref{eq4Akaza} are inverse to each other.
\end{proof}

\begin{proposition}\label{BCmatrix}
Let $(X,\Ox)$ be a ringed topos, $\calI$ a sheaf of monoids of $X,$ $\calA \ra \calB$ a morphism of $\calI$-indexed $\Ox$-algebras and $\calE$ an $\calI$-indexed $\calA$-module. There exists a canonical morphism of $\calI$-indexed $\calB$-algebras
\begin{equation}\label{erafisoprime}
\calB\circledast_{\calA}\mathscr{End}_{\calA}(\calE) \ra \mathscr{End}_{\calB}(\calB \circledast_{\calA} \calE).
\end{equation}
\end{proposition}

\begin{proof}
Let $\sigma:\calI^2 \ra \calI$ be the addition morphism and $p_1,p_2:\calI^2\ra \calI$ the projection morphisms. Let $s,t\in \Gamma(U,\calI)$ be sections of $\calI$ over an object $U.$ We define a morphism
$$\varphi_{(s,t)}:\calB_s\otimes_{\Ox_U}\mathscr{End}_{\calA}(\calE)_t \ra \mathscr{End}_{\calB}\left (\calB\circledast_{\calA}\calE\right )_{s+t}$$ as follows:
Let $V \ra U$ be a morphism in the underlying site of $X,$ $b\in \Gamma(V,\calB_s)$ and $f\in \Gamma \left (V,\mathscr{End}_{\calA}(\calE)_t\right ).$ By \ref{indremtakriz1}, $f$ is a morphism of $\calI_{|V}$-indexed $\calA_{|V}$-modules
$$f:\calE_{|V} \ra \sigma_{t_{|V}}^*\calE_{|V},$$
where $\sigma_{t_{|V}}:\calI_{|V} \ra \calI_{|V}$ is addition by $t_{|V}.$
Also, by \ref{indremtakriz1}, the morphism $b$ induces a morphism of $\calI_{|V}$-indexed $\calB_{|V}$-modules
$$b:\calB_{|V} \ra \sigma_{s_{|V}}^*\calB_{|V}.$$
We then get an $\Ox_{\calI_{|V}^2}$-linear morphism
$$
b\boxtimes f:\calB_{|V} \boxtimes \calE_{|V} \ra \left (\sigma_{s_{|V}}^*\calB_{|V} \right )\boxtimes \left (\sigma_{t_{|V}}^*\calE_{|V} \right )=\sigma_{\left (s_{|V},t_{|V} \right )}^*\left ( \calB_{|V}\boxtimes \calE_{|V}\right ).
$$
We then get an $\Ox_{\calI_{|V}}$-linear morphism
\begin{equation}\label{ren1}
\sigma_{|V!} \left (b\boxtimes f\right ): \sigma_{|V!} \left (\calB_{|V} \boxtimes \calE_{|V}\right ) \ra \sigma_{|V!} \left (\sigma_{\left (s_{|V},t_{|V} \right )}^*\left ( \calB_{|V}\boxtimes \calE_{|V}\right ) \right ).
\end{equation}
By the commutativity of diagram
$$
\begin{tikzcd}
\left (X_{/U} \right )_{/\calI_{|U}^2} \ar{r}{\sigma_{(s,t)}} \ar[swap]{d}{\sigma_{|U}} & \left (X_{/U} \right )_{/\calI_{|U}^2} \ar{d}{\sigma_{|U}} \\
\left (X_{/U} \right )_{/\calI_{|U}} \ar{r}{\sigma_{s+t}} & \left (X_{/U} \right )_{/\calI_{|U}},
\end{tikzcd}
$$
we get a canonical morphism
$$
\sigma_{|V!} \left (\sigma_{\left (s_{|V},t_{|V} \right )}^*\left ( \calB_{|V}\boxtimes \calE_{|V}\right ) \right ) \ra \sigma_{s_{|V}+t_{|V}}^*\sigma_{|V!}\left ( \calB_{|V}\boxtimes \calE_{|V}\right ).
$$
Composing this with \eqref{ren1}, we get an $\Ox_{\calI_{|V}}$-linear morphism
\begin{equation}\label{ren2}
\sigma_{|V!} \left (b\boxtimes f\right ): \sigma_{|V!} \left (\calB_{|V} \boxtimes \calE_{|V}\right ) \ra \sigma_{s_{|V}+t_{|V}}^*\left (\sigma_{|V!}\left ( \calB_{|V}\boxtimes \calE_{|V}\right ) \right ).
\end{equation}
By \ref{indproptensprod}, \eqref{ren2} induces a $\calB_{|V}$-linear morphism
$$
b\circledast_{\calA_{|V}}f: \calB_{|V}\circledast_{\calA_{|V}}\calE_{|V} \ra \sigma_{s_{|V}+t_{|V}}^*\left (\calB_{|V}\circledast_{\calA_{|V}}\calE_{|V}\right ).
$$
By \ref{ren3}, we get
$$
\left (b\circledast_{\calA_{|V}}f:\left (\calB\circledast_{\calA}\calE\right )_{|V} \ra \sigma_{s_{|V}+t_{|V}}^*\left (\calB\circledast_{\calA}\calE\right )_{|V} \right ) \in \mathscr{End}_{\calB}\left (\calB\circledast_{\calA}\calE\right )_{s+t}(V).
$$
We set $\varphi_{(s,t)}(b\otimes f)=b\circledast_{\calA_{|V}}f.$ By \ref{indremark73}, the morphisms $\varphi_{(s,t)}$ define a morphism
$$
\calB\boxtimes \mathscr{End}_{\calA}(\calE) \ra \sigma^*\mathscr{End}_{\calB}(\calB\circledast_{\calA}\calE),
$$
and then, by adjunction,
\begin{equation}\label{ren4}
\sigma_!\left ( \calB\boxtimes \mathscr{End}_{\calA}(\calE) \right ) \ra \mathscr{End}_{\calB}(\calB\circledast_{\calA}\calE).
\end{equation}
Again, by \ref{indremark73},  \eqref{ren4} induces
$$
\varphi:\calB\circledast_{\calA}\mathscr{End}_{\calA}(\calE) \ra \mathscr{End}_{\calB}(\calB \circledast_{\calA} \calE).
$$
\end{proof}

\begin{proposition}\label{Koko829}
Let $f:Y\ra X$ be a morphism of ringed topoi, $\calI$ a sheaf of monoids of $X,$ $\calA$ an $\calI$-indexed algebra of $X$ and $\calE$ and $\F$ two $\calI$-indexed $\calA$-modules. There exist a canonical isomorphism of $f^{-1}\calI$-indexed $f_{/\calI}^{-1}\calA$-modules
$$
f_{/\calI}^{-1}\calE \circledast_{f_{/\calI}^{-1}\calA}f_{/\calI}^{-1}\F \xrightarrow{\sim} f_{/\calI}^{-1}\left ( \calE \circledast_{\calA}\F \right )
$$
and a canonical isomorphism of $f^{-1}\calI$-indexed $f^*\calA$-modules
$$
f_{/\calI}^*\calE \circledast_{f_{/\calI}^*\calA}f_{/\calI}^*\F \xrightarrow{\sim} f_{/\calI}^*\left ( \calE \circledast_{\calA}\F \right ).
$$
\end{proposition}

\begin{proof}
Denote by $\sigma:\calI^2 \ra \calI$ and $\tau:\calI^3\ra \calI$ the addition morphisms. By \ref{lemtop15} and the commutative diagram
$$
\begin{tikzcd}
Y_{/f^{-1}\calI^2} \ar{r}{f_{/\calI^2}} \ar[swap]{d}{f^{-1}\sigma} & X_{/\calI^2} \ar{d}{\sigma} \\
Y_{/f^{-1}\calI} \ar[swap]{r}{f_{/\calI}} & X_{/\calI},
\end{tikzcd}
$$
we get a canonical isomorphism
$$
\left (f^{-1}\sigma\right )_!f_{/\calI^2}^{-1} \left (\calE \boxtimes \F \right ) \xrightarrow{\sim} f_{/\calI}^{-1}\sigma_!\left ( \calE\boxtimes \F\right ).
$$
We similarly have a canonical isomorphism
$$
\left (f^{-1}\tau\right )_!f_{/\calI^3}^{-1} \left (\calA \boxtimes \calE \boxtimes \F \right ) \xrightarrow{\sim} f_{/\calI}^{-1}\tau_!\left ( \calA \boxtimes \calE\boxtimes \F\right ).
$$
By definition of $\calE \circledast_{\calA} \F$ and the exactness of $f_{/\calI}^{-1},$ these isomorphisms fit into a commutative diagram
$$
\begin{tikzcd}
\left (f^{-1}\tau\right )_!f_{/\calI^3}^{-1} \left (\calA \boxtimes \calE \boxtimes \F \right ) \ar[sloped]{d}{\sim} \ar{r} & \left (f^{-1}\sigma\right )_!f_{/\calI^2}^{-1} \left (\calE \boxtimes \F \right ) \ar{r} \ar[sloped]{d}{\sim} & \left (f_{/\calI}^{-1}\calE \right ) \circledast_{f_{/\calI}^{-1}\calA} \left (f_{/\calI}^{-1}\F \right )   \ar[sloped,dashed]{d}{\sim} \ra 0 \\
f^{-1}_{/\calI}\tau_!(\calA \boxtimes \calE \boxtimes \F)  \ar{r} & f^{-1}_{/\calI}\sigma_!(\calE \boxtimes \F)  \ar{r} & f^{-1}_{/\calI}\left (\calE \circledast_{\calA} \F\right ) \ra 0,
\end{tikzcd}
$$
where the rows are exact and induced by \eqref{exseqindtens}.
Hence the existence of the third vertical isomorphism. The proof of the second isomorphism is similar.
\end{proof}

\begin{parag}
We end this subsection by defining local freeness for indexed modules. We introduce the following notation: let $(X,\Ox)$ be a ringed topos on a site $C,$ $\calI$ a sheaf of monoids of $X,$ $\calA$ an $\calI$-indexed algebra and $\calE$ an $\calI$-indexed $\calA$-module. Let $m\in \Gamma(U,\calI)$ be a section of $\calI$ over an object $U$ of $C$ and consider the commutative diagram of topoi
\begin{equation}\label{diagtakriz22}
\begin{tikzcd}
\left (X_{/U} \right )_{/\calI_{|U}} \ar{dr}{j_{\calI_{|U}}} \ar{r}{\sim} & X_{/U\times \calI} \ar{r}{j_{U\times\calI/\calI}} \ar{d}{j_{\calI\times U/U}} & X_{/\calI} \ar{d}{j_{\calI}} \\
 & X_{/U} \ar{r}{j_U} & X
\end{tikzcd}
\end{equation}
We denote by $\sigma_m:\calI_{|U} \ra \calI_{|U},\ x\mapsto x+m$ and $\calE(m)=\sigma_m^*\calE_{|U},$ where $\calE_{|U}$ is the $\calI_{|U}$-indexed $\calA_{|U}$-module defined in \ref{Locprop1}. The sheaf $\calE(m)$ is clearly an $\calI_{|U}$-indexed $\calA_{|U}$-module.
\end{parag}

\begin{definition}\label{deflocfreeind}
Let $X$ be a scheme, $\calI$ a sheaf of abelian groups on $X,$ $\sigma:\calI^2 \ra \calI$ the addition map, $\calA$ an $\calI$-indexed algebra and $\calE$ an $\calI$-indexed $\calA$-module.
\begin{enumerate}
\item We say that $\calE$ is a \emph{free $\calI$-indexed $\calA$-module} (resp. \emph{free $\calI$-indexed $\calA$-module of finite rank}) if there exist sections (resp. finitely many sections) $m_{i}\in \Gamma(X,\calI), i\in I$ satisfying
$$\calE \cong \bigoplus_{i\in I} \calA(m_{i}).$$
In this case, a \emph{local basis} of the $\calI$-indexed $\calA$-module $\calE$ is a family of local sections $(e_i)_{i\in I}$ of $\prod_{i\in I}\calE_{s_i},$ where the $s_i$'s are local sections of $\calI$ over an étale $X$-scheme $U,$ satisfying: for all $V\ra U$ étale and all local sections $s\in \Gamma (U,\calI)$ and $x\in \Gamma(V,\calE_s),$ there exists unique local sections $(a_i)_{i\in I}$ of $\prod_{i\in I}\Gamma(V,\calA_{s-s_i})$ (all equal to zero but finitely many) such that $x=\sum_{i\in I}a_ie_i,$ where $a_ie_i$ denotes the image of $a_i\otimes e_i$ by $\calA_{s-s_i} \otimes_{\Ox_U} \calE_{s_i} \ra \calE_s.$
\item We say that $\calE$ is a \emph{locally free $\calI$-indexed $\calA$-module} (resp. \emph{locally free $\calI$-indexed $\calA$-module of finite rank}) if there exists an étale covering $(U_i)_{i\in I}$ of $X$ such that for every $i\in I,$ the $\calI_{|U_i}$-indexed $\calA_{|U_i}$-module $\calE_{|U_i}$ is free (resp. free of finite rank).
\end{enumerate}
\end{definition}

\begin{proposition}\label{propfreetakriz}
Let $X$ be a scheme, $\calI$ a sheaf of monoids on $X,$ $\calA \ra \calA'$ a morphism of $\calI$-indexed algebras and $\calE$ an $\calI$-indexed $\calA$-module. If $\calE$ is a locally free $\calI$-indexed $\calA$-module then $\calA' \circledast_{\calA}\calE$ is a locally free $\calI$-indexed $\calA'$-module. In addition, if $(e_i)_{i\in I}$ is a basis for the $\calI$-indexed $\calA$-module $\calE,$ where $e_i$ is over a section $s_i$ of $\calI,$ then $(1\otimes e_i)_{i\in I}$ is a basis for the $\calI$-indexed $\calA'$-module $\calA'\circledast_{\calA}\calE.$ Here, $1\otimes e_i$ is considered as a section of $\calA'\circledast_{\calA}\calE$ via the morphism $\calA'_0\otimes_{\Ox_X}\calE_{s_i} \ra (\calA' \circledast_{\calA}\calE)_{s_i}$ \eqref{indproptensprod}.
\end{proposition}

\begin{proof}
It is sufficient to prove that, for any section $m\in \Gamma(X,\calI),$
$\calA(m)\circledast_{\calA}\calA' \cong \calA'(m).$
This is a special case of the lemma \ref{ren5} below.
\end{proof}

\begin{lemma}\label{ren5}
Let $X$ be a scheme, $\calI$ a sheaf of monoids on $X,$ $\calA$ an $\calI$-indexed algebra, $\calE$ an $\calI$-indexed $\calA$-module and $m\in \Gamma(X,\calI).$ Then there exists a canonical isomorphism
$$\calA(m)\circledast_{\calA}\calE \xrightarrow{\sim} \calE(m).$$
\end{lemma}

\begin{proof}
Let $p_1,p_2:\calI^2 \ra \calI$ be the canonical projections, $\sigma:\calI^2 \ra \calI$ the addition map, $\sigma_m:\calI \ra \calI,\ x\mapsto x+m$ the addition by $m,$ $\sigma_{(m,0)}:\calI^2 \ra \calI^2,\ x\mapsto x+(m,0)$ the addition by $(m,0)$ and $\pi:\calA \boxtimes \calE \ra \sigma^*\calE$ the canonical map. We have the following commutative diagrams
$$
\begin{tikzcd}
\calI^2 \ar{r}{\sigma_{(m,0)}} \ar{d}{p_1} & \calI^2 \ar{d}{p_1} & \calI^2 \ar{r}{\sigma_{(m,0)}} \ar{dr}{p_2} & \calI^2 \ar{d}{p_2} & \calI^2 \ar{r}{\sigma_{(m,0)}} \ar{d}{\sigma} & \calI^2 \ar{d}{\sigma} \\
\calI \ar{r}{\sigma_m} & \calI & & \calI & \calI \ar{r}{\sigma_m} & \calI
\end{tikzcd}
$$
We then obtain a morphism
$\sigma_{(m,0)}^*\pi: \calA(m)\boxtimes \calE \ra \sigma^*\calE(m).$
By adjunction, we get a morphism
$\sigma_!(\calA(m) \boxtimes \calE) \ra \calE(m).$
By checking on fibers, we prove it induces an isomorphism
$\calA(m)\circledast_{\calA}\calE \xrightarrow{\sim} \calE(m).$
\end{proof}

\begin{proposition}
Let $X$ be a scheme, $\calI$ a sheaf of abelian groups of $X_{\text{ét}},$ $\calA \ra \calB$ a morphism of $\calI$-indexed algebras and $\calE$ a locally free $\calI$-indexed $\calA$-module of finite rank \eqref{deflocfreeind}. The canonical morphism \eqref{erafisoprime} is then an isomorphism:
\begin{equation}\label{erafiso}
\calB\circledast_{\calA}\mathscr{End}_{\calA}(\calE) \xrightarrow{\sim} \mathscr{End}_{\calB}(\calB \circledast_{\calA} \calE).
\end{equation}
\end{proposition}

\begin{proof}
If $U\ra X$ is an étale morphism, then $\left (\calB \circledast_{\calA}\mathscr{End}_{\calA}(\calE) \right )_{|U}=\calB_{|U}\circledast_{\calA_{|U}}\mathscr{End}_{\calA}(\calE)_{|U}$ by \ref{ren3} and $\mathscr{End}_{\calA}(\calE)_{|U}=\mathscr{End}_{\calA_{|U}}\left (\calE_{|U} \right )$ by definition of $\mathscr{End}_{\calA}$ (\ref{inddef2}). We can then suppose that $\calE$ is a free $\calI$-indexed $\calA$-module. It follows that there exists $s_1,\hdots,s_d\in \Gamma(X,\calI)$ and an isomorphism
$$
\calE\xrightarrow{\sim} \bigoplus_{i=1}^d\calA(s_i).
$$
The result then follows from \ref{ren5} and the canonical isomorphisms
$$
\bigoplus_{i,j=1}^d\calA(s_j-s_i) \xrightarrow{\sim} \mathscr{End}_{\calA}\left (\bigoplus_{i=1}^d\calA(s_i) \right ),\ \bigoplus_{i,j=1}^d\calB(s_j-s_i) \xrightarrow{\sim} \mathscr{End}_{\calB}\left (\bigoplus_{i=1}^d\calB(s_i) \right ).
$$
\end{proof}

\subsection*{Indexed connections}

\begin{parag}\label{indparag77}
In this subsection, we fix a morphism $f:X\ra S$ of logarithmic schemes or logarithmic formal schemes and an étale sheaf $\calI$ on $X.$ Applying the results of the previous section to the small étale topos $X_{\text{ét}},$ we get three functors
\begin{equation}\label{indeq71}
j_{\calI !}:\calI_{\text{ét}}\ra X_{\text{ét}},\ j_{\calI}^*:X_{\text{ét}}\ra \calI_{\text{ét}},\ j_{\calI *}:\calI_{\text{ét}}\ra X_{\text{ét}}
\end{equation}
In the rest of this article and for any étale sheaf $\F$ on $X$ (resp. any morphism $f:\F \ra \calG$ of étale sheaves on $X$), we denote $j_{\calI}^*\F$ (resp. $j_{\calI}^*f$) by $\F_{\calI}$ (resp. $f_{\calI}$). In particular, we denote $j_{\calI}^*\Ox_X$ by $\Ox_{X,\calI}$ and, for any positive integer $i,$ we denote the sheaf $j_{\calI}^*\left (\omega^i_{X/S}\right )$ by $\omega^i_{X/S,\calI}.$
\end{parag}

\begin{definition}\label{indprop42}
Let $\calE$ be an $\Ox_{X,\calI}$-module and $\lambda$ a global section of $\Ox_{X,\calI}.$ A \emph{logarithmic $\lambda$-connection on $\calE$} is an $(f^{-1}\Ox_S)_{\calI}$-linear morphism $\nabla:\calE\ra \calE\otimes_{\Ox_{X,\calI}}\omega^1_{X/S,\calI}$ satisfying the modified Leibniz rule
$$\nabla(ax)=a\nabla(x)+\lambda x\otimes d_{\calI}a,$$
for all local sections $a$ and $x$ of $\Ox_{X,\calI}$ and $\calE$ respectively and where $d:\Ox_X\ra \omega^1_{X/S}$ denotes the universal derivation. 
Logarithmic $1$-connections are called \emph{logarithmic connections.}
\end{definition} 

\begin{proposition}
Let $\calE$ be an $\Ox_{X,\calI}$-module and $\lambda$ a global section of $\Ox_{X,\calI}.$ A $\lambda$-connection $\nabla$ on $\calE$ induces, for any positive integer $i,$ a morphism
$$\nabla^i:\calE\otimes_{\Ox_{X,\calI}}\omega^i_{X/S,\calI}\ra \calE\otimes_{\Ox_{X,\calI}}\omega^{i+1}_{X/S,\calI}$$
defined for any local sections $x$ and $\omega$ of $\calE$ and $\omega^i_{X/S,\calI},$ respectively, by
$$\nabla^i(x\otimes \omega)=\nabla(x)\wedge \omega+\lambda x\otimes d_{\calI}\omega$$
where $\nabla(x)\wedge \omega$ denotes the image of $\nabla(x)\otimes \omega$ by the canonical morphism $$\calE\otimes_{\Ox_{X,\calI}} \omega^1_{X/S,\calI}\otimes_{\Ox_{X,\calI}}\omega^i_{X/S,\calI}\ra \calE\otimes_{\Ox_{X,\calI}}\omega^{i+1}_{X/S,\calI}.$$
\end{proposition}

\begin{proof}
See (\cite{Lor2000} II 1.3) for the case $\lambda=1.$ The case for an arbitrary $\lambda$ is similar.
\end{proof}

\begin{definition}\label{indhiggs}
Let $\calE$ be an $\calI$-indexed $\Ox_{X}$-module, $\lambda$ a global section of $\Ox_{X,\calI}$ and $\nabla:\calE\ra \calE\otimes_{\Ox_{X,\calI}}\omega^1_{X/S,\calI}$ a $\lambda$-connection on $\calE.$ The \emph{curvature of $\nabla$,} denoted by $K(\nabla),$ is the composition $\nabla^1\circ \nabla$ (where $\nabla^1$ is defined in \ref{indprop42}). The curvature $K(\nabla)$ is $\Ox_{X,\calI}$-linear and we say that the $\lambda$-connection $\nabla$ is \emph{integrable} if $K(\nabla)=0.$ Integrable $0$-connections are called \emph{Higgs fields.} We denote by $\boldsymbol{\op{MIC^{ind}}(X/S)}$ (resp. $\boldsymbol{\op{HIG^{ind}}(X/S)}$) the category of $\Ox_{X,\calI}$-modules equipped with an integrable connection (resp. Higgs field).
\end{definition}

\begin{remark}
\begin{enumerate}
\item In the rest of this article, we will always drop the term "logarithmic" and refer to logarithmic connections and logarithmic Higgs fields simply by connections and Higgs fields.

\item Let $\calE$ be an $\Ox_{X,\calI}$-module and $\lambda$ a global section of $\Ox_{X,\calI}.$ By \ref{indremark73}, the data of a $\lambda$-connection $$\nabla : \calE \ra \calE \otimes_{\Ox_{X,\calI}} \omega^1_{X/S,\calI}$$
is equivalent to the data of a compatible family of $\lambda$-connections
$$\nabla_i:\calE_i \ra \calE_i \otimes_{\Ox_U}\omega^1_{U/S}$$
for every section $i\in \calI(U).$ Furthermore, $K(\nabla)_i=K(\nabla_i)$ and $\nabla$ is integrable if and only if $\nabla_i$ is integrable for every section $i$ of $\calI.$

\item Let $(\calE_1,\nabla_1)$ and $(\calE_2,\nabla_2)$ be two $\Ox_{X,\calI}$-modules equipped with $\lambda$-connections, where $\lambda$ is a global section of $\Ox_{X,\calI}.$ Then $\calE=\calE_1\otimes_{\Ox_{X,\calI}}\calE_2$ is canonically equipped with a $\lambda$-connection $\nabla$ defined by
$$
\nabla(x_1\otimes x_2)=\nabla_1(x_1)\otimes x_2+x_1\otimes \nabla_2(x_2),
$$
for local sections $x_1$ and $x_2$ of $\calE_1$ and $\calE_2$ respectively.
\end{enumerate}
\end{remark}

\subsection*{Indexed algebra associated to a logarithmic structure}

\begin{parag}\label{paragAX}
Let $X$ be a logarithmic scheme or logarithmic formal scheme.
We denote by $\calI$ the sheaf $\ov{\calM}^{gp}_X$ and by $\sigma:\calI^2\ra \calI$ the addition map and consider the functors introduced in (\ref{indeq71}):
\begin{equation}
j_{\calI !}:\calI_{\text{ét}}\ra X_{\text{ét}},\ j_{\calI}^*:X_{\text{ét}}\ra \calI_{\text{ét}},\ j_{\calI *}:\calI_{\text{ét}}\ra X_{\text{ét}}.
\end{equation}
The exact sequence of abelian groups
$0 \ra \Ox_X^* \xrightarrow{\alpha_X^{-1}} \calM_X^{gp} \ra \calI \ra 0$
turns $\calM_X^{gp}$ into an $\Ox_{X,\calI}^*$-torsor of $\calI_{\text{ét}}.$ To this torsor corresponds the invertible $\Ox_{X,\calI}$-module
\begin{equation}\label{Agroupe}
\calA_X=\Ox_{X,\calI}\times_{\calI}\calM_X^{gp}/\sim,
\end{equation}
where $\sim$ is the relation
\begin{equation}\label{eq8422c}
(a,m)\sim (a',m') \iff \exists m''\in \calM_X^{\times},\begin{cases} m+m''=m' \\ a'\alpha_X(m'')=a\end{cases}
\end{equation}
for any local sections $a,a'$ of $\Ox_{\calI,\ov{m}}$ and $m,m'$ of $\calM_X^{gp}$ such that $\ov{m}=\ov{m'}$ in $\ov{\calM}_X^{gp}=\calI.$

For a local section $m\in\calM_X^{gp}(U)$ over an étale $X$-scheme $U,$ the $\Ox_U$-module
\begin{equation}\label{eq8422b}
\calA_{X,\ov{m}}=\Ox_U\times \calM_{X,\ov{m}}^{gp}/\sim
\end{equation}
admits a basis
\begin{equation}\label{eqAbase}
e_m=(1,m).
\end{equation}
For any local sections $m,m'\in\calM_X^{gp}(U),$ we consider the morphism
\begin{equation}
\begin{array}[t]{clc}
\calA_{X,\ov{m}}\otimes_{\Ox_U}\calA_{X,\ov{m'}} & \ra & \calA_{X,\ov{m}+\ov{m'}}\\
a e_m\otimes a'e_{m'} & \mapsto & aa'e_{m+m'}
\end{array}
\end{equation}
These morphisms are well-defined and compatible and so, by \ref{era2indrem}, they define a morphism
\begin{equation}\label{Agroupe2}
\pi:\calA_X\boxtimes \calA_X \ra \sigma^*\calA_X
\end{equation}
which defines an $\calI$-indexed $\Ox_X$-algebra structure on $\calA_X$ (\ref{inddef} (2)).
\end{parag}

\begin{proposition}[\cite{Lor2000} 1.6 and 1.7]\label{dAX}
Let $X\ra S$ be a morphism of logarithmic schemes or logarithmic formal schemes. We denote by $\calI$ the sheaf $\ov{\calM}^{gp}_X.$ Then, the $\calI$-indexed $\Ox_X$-algebra $\calA_X$ \eqref{Agroupe} admits a unique connection \eqref{indprop42}
$$d_{\calA_X}:\calA_X \ra \calA_X \otimes_{\Ox_{X,\calI}}\omega^1_{X/S,\calI}$$
given by 
\begin{equation}\label{eq8431Koko}
d_{\calA_X}(e_m)=e_m\otimes \op{dlog}m,
\end{equation}
for any local section $m$ of $\calM_X^{gp}.$ In addition, $d_{\calA_X}$ is integrable and its $p$-curvature vanishes.
\end{proposition}

\begin{definition}\label{Bgroupe}
Let $X\ra S$ be a morphism of logarithmic schemes or logarithmic formal schemes, $\calI$ the sheaf $\ov{\calM}^{gp}_X,$ $p_1,p_2:\calI^2\ra \calI$ the canonical projections and $\sigma:\calI^2\ra \calI$ the addition map. 
A connection $\nabla$ on an $\calI$-indexed $\calA_X$-module $\calE$ (\ref{indprop42}) is said to be \emph{admissible} if the map $\calA_X\boxtimes \calE \ra \sigma^*\calE$ defining the $\calA_X$-module structure on $\calE$ is horizontal with respect to the connections
$$
\begin{array}[t]{clc}
\calA_X \boxtimes \calE & \ra & (\calA_X\boxtimes \calE) \otimes_{\Ox_{X,\calI^2}} \omega^1_{X/S,\calI^2}\\
a \otimes x & \mapsto & \left ( p_1^*d_{\calA_X} \right )(a)\otimes x+a\otimes \left (p_2^*\nabla \right )(x)
\end{array}
$$
and $\sigma^*\nabla.$
\end{definition}

\begin{parag}\label{Bgroupe2}
Let $X\ra S$ be a morphism of logarithmic schemes or logarithmic formal schemes and $\calI$ the sheaf $\ov{\calM}^{gp}_X.$ Then the canonical connection $d_{\calA_X}$ is admissible. Indeed, for any local sections $s,t:U\ra \calM_X^{gp}$ over an étale $X$-scheme $U,$
$$
\left (\pi \otimes \op{Id}_{\omega^1_{U/S}} \right ) \left (d_{\calA_X}(e_s)\otimes e_t+e_s\otimes d_{\calA_X}(e_t) \right )=d_{\calA_X}(e_{s+t}),$$
where $\pi$ is the multiplication morphism \eqref{Agroupe2}. So, for any local section $a,b\in \Gamma(U,\Ox_X),$
\begin{alignat*}{2}
\left (\pi \otimes \op{Id}_{\omega^1_{U/S}} \right ) \left (d_{\calA_X}(ae_s)\otimes (be_t)+(ae_s)\otimes d_{\calA_X}(be_t) \right ) &= e_{s+t}\otimes \left ( d(ab)+ab \op{dlog}(s+t) \right ) \\
&= d_{\calA_X} \left (abe_{s+t} \right ).
\end{alignat*}
and so $\pi:\calA_X \boxtimes \calA_X \ra \sigma^*\calA_X$ is horizontal with respect to the connections $\sigma^*d_{\calA_X}$ and $\op{Id}_{\calA_X}\boxtimes d_{\calA_X}+d_{\calA_X}\boxtimes \op{Id}_{\calA_X}.$
It follows that the kernel of the connection $d_{\calA_X}$ is an $\calI$-indexed $\Ox_X$-subalgebra of $\calA_X.$ We denote it by $\calB_{X/S}.$
\end{parag}

\begin{proposition}\label{lemBres}
Let $X\ra S$ be a morphism of logarithmic schemes, $U\ra X$ an étale and strict morphism, $\calI=\ov{\calM}_X^{gp}$ and $\calI_U=\calI_{|U}.$
Let $u:U_{\text{ét}/\calI_{\cal_U}} \ra X_{\text{ét}/\calI}$ the morphism of ringed topoi induced by $U\ra X.$ Then, there exists a canonical isomorphism of $\calI$-indexed algebras
$$
u^*\calB_{X/S} \xrightarrow{\sim} \calB_{U/S}.
$$
\end{proposition}

\begin{proof}
The morphism $U\ra X$ is étale and strict so we have a canonical isomorphism
$$
u^*\omega^1_{X/S,\calI} \xrightarrow{\sim} \omega^1_{U/S,\calI_U}.
$$
We have the following commutative diagram of $\Ox_{U',\calI_U}$-modules:
$$
\begin{tikzcd}
0 \ar{r} & u^*\calB_{X/S} \ar{r} \ar[dashed]{d} & u^*\calA_X \ar[sloped]{d}{\sim} \ar{r}{u^*d_{\calA_X}} & u^*\calA_{X} \otimes_{\Ox_{U,\calI_U}} \omega^1_{U/S,\calI} \ar[sloped]{d}{\sim} \\
0 \ar{r} & \calB_{U/S} \ar{r} & \calA_U \ar{r}{d_{\calA_U}} & \calA_U \otimes_{\Ox_{U,\calI_U}} \omega^1_{U/S,\calI}.
\end{tikzcd}
$$
The lines are exact by definition of $\calB_{X/S}$ and $\calB_{U/S}$ and the left exactness of $u^*$ (which comes from the flatness of $U\ra X$). The vertical isomorphisms are given in (\cite{Lor2000} I 3.1 page 273). Hence the existence of the dashed arrow and the fact that it is an isomorphism.
\end{proof}

We end this section by recalling the following result of Lorenzon: 
\begin{theorem}[\cite{Lor2000} 2.6]\label{thmlor}
Let $f:X\ra S$ be a log smooth morphism of fine logarithmic schemes of characteristic $p.$ Denote by $\calI$ the sheaf $\ov{\calM}^{gp}_X.$ Then $\calA_X$ is an $\calI$-indexed $\calB_{X/S}$-module locally free of finite rank. More precisely, if $m_1,\hdots,m_d \in \Gamma(X,\calM_X)$ are local coordinates for $X/S$ and, for every $I=(I_1,\hdots,I_d)\in \Z^d,$ $m_I=\sum_{i=1}^dI_im_i\in \Gamma(X,\calM_X^{gp})$ then $(e_{m_I})_{I\in \{0,\hdots p-1\}^d}$ is a basis for the $\calI$-indexed $\calB_{X/S}$-module $\calA_X.$
\end{theorem}

\section{The Azumaya algebra of logarithmic differential operators}

\begin{parag}
In this section, we consider a log smooth morphism $f_1:X\ra S$ of fs logarithmic schemes of characteristic $p.$ Set
$
\upmu = \ov{\calM}_X^{gp}.
$
Denote by $F_1:X\ra X'$ the exact relative Frobenius (\cite{SF1} 3.1) and by $P_0$ the PD envelope of the diagonal immersion $X\ra X\times_S^{\op{log}}X$ (\cite{Kat89} 5.3). We identify the étale sites of $X,$ $X'$ and $P_0$ via the universal homeomorphism $F_1$ and the nilimmersion $X\ra P_0$ (\cite{SP} 04DZ). Denote by $\calP_0$ the structural ring of $P_0$ and by $\ov{\calI}_{0}$ the PD ideal of $\calP_0.$ Denote  by $p_1,p_2:P_0 \ra X$ the canonical projections. Following (\cite{SF1} 4.1), set $\calP^n_{0}=\calP_0/\ov{\calI}_{0}^{[n+1]},$ $\calD^n_{X/S}=\mathscr{Hom}_{\Ox_X}(\calP_0^n,\Ox_X)$ for any integer $n\ge 0$ and
$\calD_{X/S}=\lim\limits_{\substack{\longrightarrow \\ n\ge 0}} \calD^n_{X/S},$
where $\calP^n_{0}$ is considered as an $\Ox_X$-module via the first projection $p_1:P_0 \ra X.$

Consider the $p$-curvature morphism (\cite{SF2} 4.1)
\begin{equation}\label{Takriz2}
\psi:S^{\bullet}\calT_{X'/S} \ra \calD_{X/S},
\end{equation}
where $S^{\bullet} \calT_{X'/S}$ is the symmetric algebra of $\calT_{X'/S}=\mathscr{Hom}_{\Ox_{X'}}(\omega^1_{X'/S},\Ox_{X'}).$
Recall that the image of $\psi$ is a sub $\Ox_{X'}$-algebra of the center of $\calD_{X/S}$ (\cite{SF2} 4.1) We consider $\calD_{X/S}$ as an algebra over $S^{\bullet} \calT_{X'/S}$ via this morphism.
Finally, for any $\Ox_X$-module $\F,$ denote by $\F_{\upmu}$ the image of $\F$ by the functor $j^*_{\upmu}:X_{\text{ét}} \ra X_{\text{ét}/\upmu}$ defined in \ref{indpar2} and set
\begin{equation}\label{Dtilda1713}
\widetilde{\calD}^n_{X/S}=\calA_X \otimes_{\Ox_{X,\upmu}}\calD^n_{X/S,\upmu},\quad \widetilde{\calD}_{X/S}=\calA_X \otimes_{\Ox_{X,\upmu}}\calD_{X/S,\upmu},
\end{equation}
where $\calA_X$ is defined in \ref{paragAX}. We recall results by Ohkawa on the Azumaya algebra property of $\widetilde{\calD}_{X/S}.$ For a more details, see (\cite{SFthesis} 17).
\end{parag}

\begin{proposition}[\cite{Ohkawa} 4.2]\label{era3parag142}
There exists a unique $\upmu$-indexed algebra structure on $\widetilde{\calD}_{X/S}$ such that the maps $(a\otimes 1)(b\otimes 1)=(ab)\otimes 1$ for all local sections $a$ and $b$ of $\calA_X$ and $(1\otimes \partial)(1\otimes \partial')=1\otimes (\partial \partial')$ for all local sections $\partial$ and $\partial'$ of $\calD_{X/S},$ and such that, for any local coordinates $m_1,\hdots m_d$ of $X\ra S$ (\cite{SF1} 4.4), any local sections $a$ and $\partial$ of $\calA_X$ and $\calD_{X/S}$ and $I\in \N^d,$ we have the relations $a\otimes \partial = (a\otimes 1)(1\otimes \partial)$ and
$$
(1\otimes \partial_I)(a\otimes 1)=\sum_{J\le I} \begin{pmatrix} I\\ J \end{pmatrix} (\partial_{I-J}\cdot a)\otimes \partial_J,
$$
where $\partial_I$ are the differential operators associated to $m_1,\hdots,m_d$ as in (\cite{SF1} 4.4) and $\partial_{I-J}\cdot a$ is the action given by the connection $d_{\calA_X}$ \eqref{dAX}.
\end{proposition}

\begin{proposition}[\cite{Ohkawa} 4.16]\label{Dcenter}
The center of $\widetilde{\calD}_{X/S}$ is $\calB_{X/S}\otimes_{\Ox_{X',\upmu}}\left (S^{\bullet}\calT_{X'/S} \right )_{\upmu}.$
\end{proposition}

\begin{proposition}[\cite{Ohkawa} 4.20, \cite{Schepler} 2.16, \cite{SFthesis} 17.8]\label{AzumayaAkaza}
$\widetilde{\calD}_{X/S}$ is a $\upmu$-indexed Azumaya algebra over its center $\calB_{X/S}\otimes_{\Ox_{X',\upmu}}\left (S^{\bullet}\calT_{X'/S} \right )_{\upmu}$ of rank $p^{2d}.$
\end{proposition}

\subsection*{A splitting of the Azumaya algebra $\widetilde{\calD}_{X/S}$}

\begin{parag}
We now recall a specific splitting of the Azumaya algebra $\widetilde{\calD}_{X/S}$ that was constructed by Ohkawa in (\cite{Ohkawa2015} 3). In the remaining of this section, we consider a perfect field $k$ of positive characteristic $p$ and denote its ring of Witt vectors by $W.$ We equip $\op{Spf}W$ with the trivial logarithmic structure. We suppose that $f_1:X\ra S$ lifts to a log smooth morphism $f:\frakX \ra \frakS$ of fs $p$-adic logarithmic formal schemes such that $\frakS$ is log flat and locally of finite type over $\op{Spf}W.$ We suppose that $F_1:X \ra X'$ lifts to an $\frakS$-morphism of fs $p$-adic logarithmic formal schemes $F:\frakX \ra \frakX'$ such that $\frakX'$ is log smooth over $\frakS.$

Let $S=(S^{\bullet}\calT_{X'/S})_{\upmu},$ $\widehat{S}=(\widehat{S}^{\bullet}\calT_{X'/S})_{\upmu}$ and $F_1^*\widehat{S}=(F_1^*\widehat{S}^{\bullet}\calT_{X'/S})_{\upmu}.$
\end{parag}

\begin{theorem}[\cite{Ohkawa2015} 4.3]\label{erafPsi}
There exists an action of $(\calB_{X/S}\otimes_{\Ox_{X',\upmu}}\widehat{S})\circledast_{\calB_{X/S}\otimes_{\Ox_{X',\upmu}}S}\widetilde{\calD}_{X/S}$ on $\calA_X\otimes_{\Ox_{X,\upmu}}F_1^*\widehat{S},$ depending on the lifting $F$ of the relative Frobenius, inducing an isomorphism of $\upmu$-indexed algebras
\begin{alignat*}{2}
\psi:(\calB_{X/S}\otimes_{\Ox_{X',\upmu}}\widehat{S})\circledast_{\calB_{X/S}\otimes_{\Ox_{X',\upmu}}S}\widetilde{\calD}_{X/S} \xrightarrow{\sim} & \mathscr{End}_{\calB_{X/S}\otimes_{\Ox_{X',\upmu}}\widehat{S}}\left ( \calA_X\otimes_{\Ox_{X,\upmu}}F_1^*\widehat{S}\right ).
\end{alignat*}
\end{theorem}

\begin{proof}
We just outline how this isomorphism is obtained. The lifting $F$ induces a morphism $P \ra P',$ where $P$ and $P'$ are the log PD envelopes of $\frakX \ra \frakX \times_{\frakS} \frakX$ and $\frakX' \ra \frakX' \times_{\frakS} \frakX'$ respectively. Denote by $\varphi:\calP' \ra \calP$ the corresponding morphism of structural rings. Then $\varphi(\calP') \subset p\calP.$ Dividing by $p$ and reducing modulo $p,$ we obtain a morphism $\calP'_0=\calP'/(p) \ra \calP_0=\calP/(p).$ The restriction on the PD ideal $\calI'$ of $\calP'_0$ factors through $\calI'^{[2]}.$ We hence obtain a morphism $\omega^1_{X'/S}=\calI'/\calI'^{[2]} \ra \calP_0$ and then a morphism $u:F_1^*\Gamma^{\bullet}\omega^1_{X'/S} \ra \calP_0,$ where $\Gamma^{\bullet}$ denotes the PD algebra. The $\calP_0$-linearization of
$$
F_1^*\Gamma^{\bullet}\omega^1_{X'/S} \xrightarrow{\Delta} F_1^*\Gamma^{\bullet}\omega^1_{X'/S} \otimes_{\Ox_X} F_1^*\Gamma^{\bullet}\omega^1_{X'/S}\xrightarrow{\op{Id} \otimes u}  F_1^*\Gamma^{\bullet}\omega^1_{X'/S} \otimes_{\Ox_X} \calP_0,
$$
where $\Delta$ is the comultiplication map, is a $\calP_0$-stratification on $F_1^*\Gamma^{\bullet}\omega^1_{X'/S}.$ We hence have an action of $\calD_{X/S}$ on $F_1^*\Gamma^{\bullet}\omega^1_{X'/S}.$ By duality, we get an action of $F_1^*\Gamma^{\bullet}\omega^1_{X'/S}$ on $F_1^*\widehat{S}^{\bullet}\calT_{X'/S}.$ Along with the action of $\calD_{X/S,\upmu}$ on $\calA_X$ given by the connection $d_{\calA_X},$ we obtain an action of $\widetilde{\calD}_{X/S}$ on $\calA_X\otimes_{\Ox_{X,\upmu}}F_1^*\widehat{S}.$ This action extends to an action of $(\calB_{X/S}\otimes_{\Ox_{X',\upmu}}\widehat{S})\circledast_{\calB_{X/S}\otimes_{\Ox_{X',\upmu}}S}\widetilde{\calD}_{X/S}$ and that induces the morphism $\psi.$ An argument using the Azumaya algebra nature of $\widetilde{\calD}_{X/S}$ and dimensions proves that $\psi$ is an isomorphism.
\end{proof}

\begin{corollaire}[\cite{Ohkawa2015} 4.4]
The functor $\Psi^{az}$ from the category of $\upmu$-indexed $(\calB_{X/S}\otimes_{\Ox_{X',\upmu}}\widehat{S})$-modules to the category of $\upmu$-indexed left $(\calB_{X/S}\otimes_{\Ox_{X',\upmu}}\widehat{S})\circledast_{\calB_{X/S}\otimes_{\Ox_{X',\upmu}}S}\widetilde{\calD}_{X/S}$-modules, given by
$$\Psi^{az}(\calE')=(\calA_X\otimes_{\Ox_{X,\upmu}}F_1^*\widehat{S}) \circledast_{\calB_{X/S}\otimes_{\Ox_{X',\upmu}}\widehat{S}}  \calE'=\calA_X\circledast_{\calB_{X/S}}\calE'$$
is an equivalence of categories.
\end{corollaire}

\begin{corollaire}\label{cor1713}
Let $\widehat{\Gamma}^{\bullet}\calT_{X'/S}$ be the completed PD-algebra of $\calT_{X'/S}$ and $\widehat{\Gamma}=\left (\widehat{\Gamma}^{\bullet}\calT_{X'/S} \right )_{\upmu}.$ There exists a canonical isomorphism of $\upmu$-indexed algebras
\begin{equation}\label{iso17131}
(\calB_{X/S}\otimes_{\Ox_{X',\upmu}}\widehat{\Gamma})\circledast_{\calB_{X/S}\otimes_{\Ox_{X',\upmu}}S}\widetilde{\calD}_{X/S} \xrightarrow{\sim} \mathscr{End}_{\calB_{X/S}\otimes_{\Ox_{X',\upmu}}\widehat{\Gamma}}\left ( \calA_X\otimes_{\Ox_{X,\upmu}}F_1^*\widehat{\Gamma}\right ).
\end{equation}
\end{corollaire}

\begin{proof}
It is sufficient to base change the isomorphism $\psi$ of \ref{erafPsi} by the canonical morphism
$$\calB_{X/S}\otimes_{\Ox_{X',\upmu}}\widehat{S} \ra \calB_{X/S}\otimes_{\Ox_{X',\upmu}}\widehat{\Gamma}$$
and then apply \eqref{erafiso} and \ref{propisoAkaza}.
\end{proof}

\begin{remark}\label{remark29}
Let $\calE'$ be $\upmu$-indexed $(\calB_{X/S}\otimes_{\Ox_{X',\upmu}}\widehat{S})$-module and let $\calE=\Psi^{az}(\calE').$ The $\widehat{S}$-module structure on $\calE'$ gives rise to a Higgs field $\theta':\calE' \ra \calE'\otimes_{\Ox_{X',\upmu}}\omega^1_{X'/S,\upmu}.$ 

Let us give a local description of the module structure on $\calE$ over $\calD_{X/S,\upmu}$ and $\w{S}.$
Let $s$ and $t$ be local sections of $\calI$ over an étale $X$-scheme $U$ and $a\in \calA_{X,s},$ $x'\in \calE'_t.$ By \ref{indproptensprod}, $a\otimes x'\in \calA_{X,s} \otimes_{\Ox_{U'}} \calE'_t$ can be seen as a local section of
$$\left (\calA_X \circledast_{\calB_{X/S}}  \calE'\right )_{s+t}.$$
Suppose that we have local coordinates $m_1,\hdots,m_d\in \Gamma(X,\calM_X)$ for $X\ra S$ (\cite{SF1} 4.4) that lift to local sections $\widetilde{m}_1,\hdots,\widetilde{m}_d\in \Gamma(\frakX,\calM_{\frakX}).$ For every $1\le i\le d,$ let $m_i'=\pi^{\flat}m_i,$ where $\pi:X'\ra X$ is given in (\cite{SF1} (3.1.2)). Consider the differential operators $\partial_I$ defined in (\cite{SF1} 4.4) and let $(\partial_i')_{1\le i\le d}$ be the dual basis of $(\op{dlog}m_i')_{1\le i\le d}.$ 

We give the actions of $\partial_{\epsilon_i}\in \calD_{X/S}$ and $\partial'^{I} \in S^{\bullet}\calT_{X'/S}$ on $a\otimes x'.$ Recall that, by (\cite{SF1} 9.3), there exist local sections $b_i$ of $\Ox_{\frakX}$ such that $F^{\flat}(\widetilde{m}_i')=p\widetilde{m}_i+\alpha_{\frakX}^{-1}(1+pb_i).$ If we denote by $c_i$ the image of $b_i$ in $\Ox_X.$ 
By (\cite{Ohkawa2015} 4.6), we have
\begin{equation}\label{actionpartialax1}
\partial_{\epsilon_i} \cdot (a\otimes x') = (\partial_{\epsilon_i}\cdot a) \otimes x' +a\otimes \theta'_t(\partial'_{i})(x')+\sum_{j=1}^d\partial_{\epsilon_i}(p_2^{\#}c_j-p_1^{\#}c_j) a \otimes \theta'_t(\partial'_j)( x'),
\end{equation}
where $\partial_{\epsilon_i}\cdot a$ is the action of $\partial_{\epsilon_i}$ on $a$ via the connection $d_{\calA_X}$ and $\theta'_t(\partial_i')$ is the composition $\calE'_t \xrightarrow{\theta'_t} \calE'_t \otimes_{\Ox_{U'}} \omega^1_{U'/S} \xrightarrow{\op{Id} \otimes \partial_i'} \calE'_t.$
Note that $\partial_{\epsilon_i}(p_1^{\#}c_j)=0$ by definition of $\partial_{\epsilon_i},$ so $\partial_{\epsilon_i}(p_2^{\#}c_j-p_1^{\#}c_j)=\partial_{\epsilon_i}(p_2^{\#}c_j).$

Since $d_{\calA_X}$ has vanishing $p$-curvature \eqref{dAX}, the action of $S$ on $\calA_X$ vanishes and so
\begin{equation}\label{actionpartialax2}
\partial_i'^{k} \cdot (a \otimes x') = a \otimes \left ( \partial_i'^{k} \cdot x' \right ).
\end{equation}
\end{remark}

\begin{definition}\label{deflocnilind}
Let $\widehat{\Gamma}$ be as defined in \ref{cor1713}. We say that a $\widehat{\Gamma}$-module $\calE$ is \emph{locally PD-nilpotent} if, for every local section $s\in \Gamma(U,\upmu)$ over an étale $X$-scheme $U,$ the $\widehat{\Gamma}^{\bullet}\calT_{U'/S}$-module $\calE_s$ is locally PD-nilpotent (\cite{SF2} 7.16). If $\calA$ is a $\widehat{\Gamma}$-algebra then an $\calA$-module is said to be \emph{locally PD-nilpotent} if it is so as a $\widehat{\Gamma}$-module.
\end{definition}

\begin{proposition}\label{locPDnilprop}
Let $\widehat{\Gamma}$ be as defined in \ref{cor1713}. Let $\calE$ be a $\widehat{\Gamma}$-module. Suppose that there exists an étale covering $(U_i \ra X)_{i\in I}$ such that the restriction $\calE_{|U_i}$ \eqref{defresind} is locally PD-nilpotent for all $i\in I.$ Then $\calE$ is locally PD-nilpotent.
\end{proposition}

\begin{proof}
Let $s\in \Gamma(U,\upmu)$ be a section over an étale $X$-scheme $U$ and $s_i=s_{|U\times_XU_i}.$ The diagram
$$
\begin{tikzcd}
\left (U_i\times_XU \right )_{\text{ét}} \ar{rr} \ar{d}{s_i} \ar{drr}{s_i} & & U_{\text{ét}} \ar{d}{s} \\
\left (\left (U_i\times_XU \right )_{\text{ét}} \right )_{/\upmu_{|U_i\times_XU}} \ar[swap]{rr}{j_{\upmu\times (U_i\times_XU)/\upmu}} & & X_{/\upmu},
\end{tikzcd}
$$
where the upper arrow is the canonical morphism and the lower morphism is given in \ref{defresind}, is commutative so
$$
\left (\calE_s\right )_{|U_i\times_XU}=\left (\calE_{|U_i\times_XU} \right )_{s_i}=\calE_{s_i}.
$$
The result follows.
\end{proof}

\begin{parag}
Let $\widehat{\Gamma}$ be as defined in \ref{cor1713}. By (\cite{Schepler} 2.2), the isomorphism \eqref{iso17131} induces an equivalence of categories
\begin{equation}\label{ren6}
\begin{array}[t]{clc}
\op{Mod}\left (\calB_{X/S}\otimes_{\Ox_{X',\upmu}}\widehat{\Gamma} \right ) & \xrightarrow{\sim} & \op{Mod} \left ( (\calB_{X/S}\otimes_{\Ox_{X',\upmu}}\widehat{\Gamma})\circledast_{\calB_{X/S}\otimes_{\Ox_{X',\upmu}}S}\widetilde{\calD}_{X/S} \right ) \\
\calE & \mapsto & \left ( \calA_X\otimes_{\Ox_{X,\upmu}}F_1^*\widehat{\Gamma}\right ) \circledast_{\calB_{X/S}\otimes_{\Ox_{X',\upmu}}\widehat{\Gamma}} \calE.
\end{array} 
\end{equation}
The inverse functor is given by
\begin{equation}\label{ren7}
\F \mapsto \mathscr{Hom}_{(\calB_{X/S}\otimes_{\Ox_{X',\upmu}}\widehat{\Gamma})\circledast_{\calB_{X/S}\otimes_{\Ox_{X',\upmu}}S}\widetilde{\calD}_{X/S}} \left (\calA_X\otimes_{\Ox_{X,\upmu}}F_1^*\widehat{\Gamma},\F \right ).
\end{equation}
\end{parag}

\begin{proposition}\label{propKoko1717}
The equivalence \eqref{ren6} induces and equivalence between the full subcategories of locally PD-nilpotent modules:
$$
\Phi_{\upmu}:\begin{Bmatrix}\mathrm{Locally\ PD}\text{-}\mathrm{nilpotent}\\
\upmu\text{-}\mathrm{indexed\ }\left (\calB_{X/S}\otimes_{\Ox_{X',\upmu}}\widehat{\Gamma}\right )\text{-}\\ \mathrm{modules} \end{Bmatrix} \xrightarrow{\sim} \begin{Bmatrix} \mathrm{Locally\ PD}\text{-}\mathrm{nilpotent\ }\upmu\text{-}\mathrm{indexed}\\ \left ((\calB_{X/S}\otimes_{\Ox_{X',\upmu}}\widehat{\Gamma})\circledast_{\calB_{X/S}\otimes_{\Ox_{X',\upmu}}S}\widetilde{\calD}_{X/S} \right )\text{-} \\ \mathrm{modules}\end{Bmatrix}.
$$
\end{proposition}

\begin{proof}
It is sufficient to prove that the functors \eqref{ren6} and \eqref{ren7} preserve locally PD-nilpotent modules.
To simplify notations, we set
$$
\calB_{\Gamma}=\calB_{X/S}\otimes_{\Ox_{X',\upmu}}\widehat{\Gamma},\ \calB_S=\calB_{X/S}\otimes_{\Ox_{X',\upmu}}S,\ \calA_{\Gamma}=\calA_X\otimes_{\Ox_{X,\upmu}}F_1^*\widehat{\Gamma}.
$$
Let $\calE$ be a locally PD-nilpotent object of $\op{Mod}\left (\calB_{\Gamma}\right ).$ Since the property of being locally PD-nilpotent is local \eqref{locPDnilprop} and $\calA_{\Gamma}$ is a locally free $\upmu$-indexed $\calB_{\Gamma}$-module of finite rank \eqref{thmlor}, we can suppose that there exists sections $m_1,\hdots,m_d\in \Gamma(X,\upmu)$ such that
$$
\calA_{\Gamma}=\bigoplus_{i=1}^d\calB_{\Gamma}(m_i).
$$
Then, by \ref{ren5}, we have
$$
\calA_{\Gamma}\circledast_{\calB_{\Gamma}}\calE=\bigoplus_{i=1}^d\calE(m_i).
$$
It is then clear that \eqref{ren6} preserves locally PD-nilpotent modules.

Suppose now that $\F$ is a locally PD-nilpotent module of $\op{Mod}\left (\calB_{\Gamma} \circledast_{\calB_S}\widetilde{\calD}_{X/S} \right ).$ Let $s\in \Gamma(U,\upmu)$ be a section over an étale $X$-scheme, $V$ an étale $U$-scheme and
$$
f\in \mathscr{Hom}_{\calB_{\Gamma}\circledast_{\calB_{S}}\widetilde{\calD}_{X/S}} \left (\calA_{\Gamma},\F \right )_s(V).$$
By replacing $s$ with $s_{|V},$ we can suppose that $V=U.$ By \ref{indremtakriz1} (1), $f$ is a morphism of $\upmu_{|U}$-indexed $\left (\calB_{\Gamma}\circledast_{\calB_{S}}\widetilde{\calD}_{X/S}\right )_{|U}$-modules
$$
f:\calA_{\Gamma|U} \ra \sigma_s^*\F_{|U},
$$
where $\sigma_s:\upmu_{|U} \ra \upmu_{|U}$ is addition by $s.$ After shrinking $U,$ we can suppose that $\calA_{\Gamma|U}$ is free of finite rank over $\calB_{\Gamma|U}.$ Take a basis $(e_i)_{1\le i\le d}$ of the $\upmu_{|U}$-indexed $\calB_{\Gamma|U}$-module $\calA_{\Gamma|U}$ \eqref{deflocfreeind}. There exists local sections $s_i$ of $\upmu$ such that $e_i\in \calA_{\Gamma,s_i}$ for all $i.$ Since $f(e_i)$ is locally annahilated by some power of the augmentation ideal of $\widehat{\Gamma}^{\bullet}\calT_{X'/S},$ we deduce that $\mathscr{Hom}_{\calB_{\Gamma}\circledast_{\calB_{S}}\widetilde{\calD}_{X/S}} \left (\calA_{\Gamma},\F \right )_s$ is locally PD-nilpotent.
\end{proof}

\begin{remark}
Let $\calE'$ be a locally PD-nilpotent $\upmu$-indexed $\left (\calB_{X/S}\otimes_{\Ox_{X',\upmu}}\w{\Gamma}\right )$-module and $\calE=\Phi_{\upmu}(\calE')=(\calA_X\otimes_{\Ox_{X,\upmu}}F_1^*\widehat{\Gamma}) \circledast_{\calB_{X/S}\otimes_{\Ox_{X',\upmu}}\widehat{\Gamma}}  \calE'=\calA_X\circledast_{\calB_{X/S}}\calE'.$ We give a local description of the module structure on $\calE$ over $\calD_{X/S,\upmu}$ and $\w{\Gamma}.$ Keep the assumption and notation of \ref{remark29}.
By \eqref{actionpartialax1} and \eqref{actionpartialax2}, we have
\begin{equation}\label{actionpartialax}
\begin{alignedat}{2}
\partial_{\epsilon_i} \cdot (a\otimes x') &= (\partial_{\epsilon_i}\cdot a) \otimes x' +a\otimes (\partial'_{i}\cdot x')+\sum_{j=1}^d\partial_{\epsilon_i}(p_2^{\#}c_j-p_1^{\#}c_j) a \otimes (\partial'_j\cdot x'), \\
\partial_i'^{[k]} \cdot (a \otimes x') &= a \otimes \left ( \partial_i'^{[k]} \cdot x' \right ).
\end{alignedat}
\end{equation}
 
\end{remark}

\section{Oyama Topoi}

\begin{parag}
In this section, we recall the logarithmic Oyama topoi introduced in (\cite{SF2} 8). Recall the notation \ref{not23}. Let $\frakS$ be a logarithmic formal scheme, log flat and locally of finite type over $\op{Spf}W(k)$ and let $X\ra S$ be a log smooth morphism of logarithmic schemes. 
\end{parag}

\begin{parag}\label{paragIran2}
Let $T$ be a fine logarithmic scheme of characteristic $p$ and $U$ an étale strict logarithmic $X$-scheme. 
Denote by $\underline{T}$ the logarithmic scheme-theoretic image of the Frobenius $F_T$ of $T$ (\cite{SF2} 5.3). Let $u:\underline{T}\ra U$ be a morphism of logarithmic schemes. Recall that $F_T$ factors as $T\xrightarrow{G_T} \underline{T} \xrightarrow{i} T,$ where $i$ is the canonical immersion (\cite{SF2} 5.5). Denote by $U''$ the base change of $U\ra S$ by $F_S$ in the category of fine logarithmic schemes and by $U\ra U'$ the exact relative Frobenius of $U$ with respect to $S.$ Then, by (\cite{SF2} 5.11), there exists unique morphisms $w:T\ra U''$ and $\Delta_u:T\ra U'$ fitting into the following commutative diagrams :
\begin{equation}\label{morphismShiho1}
\begin{tikzcd}
T\ar{r}{G_T} \ar{dr}{w}\ar[bend right=30]{rdd} & \underline{T} \ar{dr}{u} \ar{r}{i} & T\ar[bend right=-30]{dd} \\
 & U''\ar{d}\ar{r} & U\ar{d} \\
 & S \ar{r}{F_S} & S.
\end{tikzcd} \qquad 
\begin{tikzcd}
U \ar{rr} & & U' \ar{d} \\
\underline{T} \ar[swap]{u}{u} \ar{r}{i} & T \ar{ur}{\Delta_{u}} \ar[swap]{r}{w} & U''.
\end{tikzcd}
\end{equation}
\end{parag}

\begin{parag}\label{paragIran3}
The site $\calE(X/\frakS)$ (resp. $\underline{\calE}(X/\frakS)$) is defined as follows :
\begin{itemize}
\item The objects of the underlying category are triplets $(U,\frakT,u)$ where $U$ is an étale and strict $X$-scheme, $\frakT$ is a log flat logarithmic formal scheme over $\frakS$ and $u:T\ra U$ (resp. $u:\underline{T} \ra U$) is an affine morphism (\cite{SF2} 8.1).
\item A morphism $(U_1,\frakT_1,u_1) \ra (U_2,\frakT_2,u_2)$ of the underlying category is a pair $(f,g)$ where $g:U_1\ra U_2$ is an $X$-morphism and $f:\frakT_1,\ra \frakT_2$ is an $\frakS$-morphism satisfying $g\circ u_1=u_2\circ f_1$ (resp. $g\circ u_1=u_2\circ \underline{f_1}$) (\cite{SF2} 8.1).
\item A morphism $(U_1,\frakT_1,u_1) \ra (U_2,\frakT_2,u_2)$ is \emph{cartesian} if $T_1$ is isomorphic to $T_2\times_{U_2}U_1$ (resp. if $T_1 \ra T_2$ and $\Delta_{u_1}:T_1\ra U_1'$ induce an isomorphism $T_1\xrightarrow{\sim} U_1'\times_{U_2',\Delta_{u_2}}T_2$ and $\frakT_1 \ra \frakT_2$ is étale and strict) (\cite{SF2} 8.4).
\item A family of cartesian morphisms $\left ( (U_i,\frakT_i,u_i) \ra (U,\frakT,u) \right )_{i\in I}$ is a covering if $\left (U_i \ra U \right )_{i\in I}$ is an étale covering. This defines a pretopology and we consider the associated topology (\cite{SF2} 8.14).
\end{itemize}
The associated topos, denoted $\widetilde{\calE}(X/\frakS)$ (resp. $\widetilde{\underline{\calE}}(X/\frakS)$), is equipped with the ring $\Ox_{\calE(X/\frakS)}:(U,\frakT,u) \mapsto \Gamma \left (T,\Ox_T\right )$ (resp. $\Ox_{\underline{\calE}(X/\frakS)}:(U,\frakT,u) \mapsto \Gamma \left (T,\Ox_T\right )$).
\end{parag}

\begin{parag}\label{paragIran4}
For an object $(U,\frakT,u)$ of $\calE(X/\frakS)$ (resp. $\underline{\calE}(X/\frakS)$), we have a functor (\cite{SF2} 8.6)
\begin{equation}\label{alphasite}
\alpha_{(U,\frakT,u)}:\text{ét}_{/U} \ra \calE(X/\frakS)\quad \left ( \text{resp.}\ \underline{\calE}(X/\frakS) \right ).
\end{equation}
By (\cite{SF2} 8.19), this functor induces a morphism of topoi, abusively denoted by $\alpha_{(U,\frakT,u)},$
\begin{equation}\label{alphatopos}
\alpha_{(U,\frakT,u)}:U_{\text{ét}} \ra \widetilde{\calE}(X/\frakS)\ \left( \mathrm{resp. }\ \widetilde{\underline{\calE}}(X/\frakS)\right )
\end{equation}
whose inverse image functor is the composition by \eqref{alphasite}.
Also, by (\cite{SF2} 8.9), for a morphism
$$(f,g):(U_1,\frakT_1,u_1) \ra (U_2,\frakT_2,u_2)$$
of $\calE(X/\frakS)$ (resp. $\underline{\calE}(X/\frakS)$), we have a morphism
\begin{equation}
\beta_{(f,g)}:\alpha_{(U_1,\frakT_1,u_1)} \ra \alpha_{(U_2,\frakT_2,u_2)}\circ j_g,
\end{equation}
where $j_g:\text{ét}_{/U_1} \ra \text{ét}_{/U_2}$ is given by the composition by $f:U_1 \ra U_2.$
A morphism $f:(U_1,\frakT_1,u_1) \ra (U_2,\frakT_2,u_2)$ of either of the sites $\calE(X/\frakS)$ or $\underline{\calE}(X/\frakS)$ canonically induces a morphism of topoi
\begin{equation}\label{Kokoftilda}
\widetilde{f}:\left (U_{1,\text{ét}},u_{1*}\Ox_{T_1} \right ) \ra \left (U_{2,\text{ét}},u_{2*}\Ox_{T_2} \right ).
\end{equation}
A module $\F$ of $\widetilde{\calE}(X/\frakS)$ (resp. $\widetilde{\underline{\calE}}(X/\frakS)$) is equivalent to the data, for all objects $(U,\frakT,u),$ of a module $\F_{(U,\frakT,u)}$ of $\left (U_{\text{ét}},u_*\Ox_T\right )$ and, for every morphism $f:(U_1,\frakT_1,u_1) \ra (U_2,\frakT_2,u_2),$ a morphism of modules $c_{\F,f}:\widetilde{f}^*\F_{(U_2,\frakT_2,u_2)} \ra \F_{(U_1,\frakT_1,u_1)}$ satisfying functoriality and isomorphism conditions given in (\cite{SF2} 8.24). Explicitly, the module $\F_{(U,\frakT,u)}$ is obtained from $\F$ by composing with $\alpha_{(U,\frakT,u)}:\text{ét}_{/U} \ra \calE(X/\frakS)\ \left (\text{resp.}\ \underline{\calE}(X/\frakS) \right )$ (\cite{SF2} (8.19.10)). These data are called the descent data of $\F.$
\end{parag}

\begin{definition}[\cite{SF2} 8.25]\label{defcrys}
A module $\F$ of $\widetilde{\calE}(X/\frakS)$ or $\widetilde{\underline{\calE}}(X/\frakS)$ is a \emph{crystal} if $c_{\F,f}$ is an isomorphism for all morphisms $f.$
\end{definition}

\begin{parag}
Suppose that $X\ra S$ is of Cartier type. By (\cite{SF2} 8.32), the functor
\begin{equation}\label{era3rhoShiho}
\rho:\underline{\calE}(X/\frakS) \ra \calE(X'/\frakS),\ (U,\frakT,u) \mapsto (U',\frakT,\Delta_u)
\end{equation}
is continuous and cocontinuous and induces a morphism of ringed topoi
\begin{equation}\label{era3C}
C_{X/\frakS}:\widetilde{\underline{\calE}}(X/\frakS) \ra \widetilde{\calE}(X'/\frakS)
\end{equation}
such that $C_{X/\frakS}^{-1}\F=\F\circ \rho$ for all modules $\F.$ By (\cite{SF2} 8.35), the functor $C_{X/\frakS}^{-1}$ preserves crystals.
\end{parag}

\begin{parag}\label{paragIran5}
If $X\ra S$ lifts to a morphism of framed logarithmic formal schemes $(\frakX,Q) \ra (\frakS,P)$ (\cite{SF1} 8.3) then we can consider the factorization of the diagonal immersion $\frakX \ra \frakX\times_{\frakS}^{\op{log}} \frakX$ into a strict immersion $\Delta:\frakX\ra \frakY$ and a log étale affine morphism $\frakY\ra \frakX\times_{\frakS}^{\op{log}}\frakX$ (\cite{SF1} 8.10). We recall the formal groupoids $R_{\frakX,1}$ and $Q_{\frakX}$ defined in (\cite{SF1} 11.2) and (\cite{SF2} 6.4) respectively. Let $\frakX \ra \frakU \ra \frakY$ be a factorization of $\Delta$ into a closed immersion followed by an open one, $\calI$ the ideal of $\frakX \ra \frakU$ and $\calI^{\#}$ the ideal generated by $a^p$ for local sections $a$ of $\calI.$ Let $Q_{\frakX}$ (resp. $R_{\frakX,1}$) be the dilatation of $\calI^{\#}+(p)$ (resp. $\calI+(p)$) with respect to $p$ and equip them with the log structure pullback of that of $\frakY$ (\cite{SF2} 6.4). The formal schemes $Q_{\frakX}$ and $R_{\frakX,1}$ have natural structures of formal groupoids (\cite{SF2} 6.6 and 6.7). We have the following equivalence of categories:
\end{parag}

\begin{proposition}[\cite{SF2} 8.28]\label{equivRQ}
There exist equivalences of categories
$$
\begin{Bmatrix}
\Ox_X\text{-modules\ equipped\ with} \\ R_{\frakX,1}\text{-stratification}
\end{Bmatrix}
\xrightarrow{\sim}
\begin{Bmatrix}
\text{Crystals\ of\ }\widetilde{\calE}(X/\frakS)
\end{Bmatrix}
$$
$$\begin{Bmatrix}
\Ox_X\text{-modules\ equipped\ with} \\ Q_{\frakX}\text{-stratification}
\end{Bmatrix}
\xrightarrow{\sim}
\begin{Bmatrix}
\text{Crystals\ of\ }\widetilde{\underline{\calE}}(X/\frakS)
\end{Bmatrix}
$$
\end{proposition}

\begin{parag}\label{paragIran6}
If in addition the exact relative Frobenius $F_1:X\ra X'$ lifts to an $(\frakS,P)$-morphism of framed logarithmic formal schemes $F:(\frakX,Q) \ra (\frakX',Q')$ such that $\frakX'$ is log smooth over $\frakS,$ then, by (\cite{SF2} 6.17), $F$ induces a morphism of formal groupoids
\begin{equation}\label{nu}
Q_{\frakX} \ra R_{\frakX',1},
\end{equation}
over $F.$
Denote by $\nu:\calR_1' \ra \calQ_1$ the corresponding morphism of Hopf algebras mod $p$. We then have a functor $\calV$ from the category of $\Ox_{X'}$-modules equipped with an $\calR'_1$-stratification to the category of $\Ox_X$-modules equipped with an $\calQ_1$-stratification, defined by $\calV(\calE',\epsilon')=(F_1^*\calE',\nu^*\epsilon').$ This functor fits into the commutative diagram :
\begin{equation}\label{diag185}
\begin{tikzcd}
\begin{Bmatrix}\text{Crystals\ of\ }\widetilde{\calE}(X'/\frakS) \end{Bmatrix} \ar[swap, sloped]{d}{\sim} \ar{r}{C_{X/\frakS}^{-1}} & \begin{Bmatrix} \text{Crystals\ of\ }\widetilde{\underline{\calE}}(X/\frakS) \end{Bmatrix} \ar[sloped]{d}{\sim} \\
\begin{Bmatrix}\Ox_{X'}\text{-modules\ with\ an}\\ \calR_1'\text{-stratification} \end{Bmatrix} \ar{r}{\calV} & \begin{Bmatrix}\Ox_{X}\text{-modules\ with\ a}\\ \calQ_1\text{-stratification} \end{Bmatrix}.
\end{tikzcd}
\end{equation}
We finally denote by $\calE_{\str}(X/\frakS)$ and $\underline{\calE}_{\str}(X/\frakS)$ the subsites of $\calE(X/\frakS)$ and $\underline{\calE}(X/\frakS)$ respectively consisting of objects $(U,\frakT,u)$ such that $u$ is strict (\cite{SF2} 8.2 and 8.14). We refer to the associated topoi by the \emph{strict topoi}. Similarly to the big topoi, we have descent data for modules of the strict topoi and we can define crystals of these strict topoi. By (\cite{SF2} 8.26), the canonical embedding of the strict sites into the big sites induce equivalences of categories between the category of crystals of $\widetilde{\calE}_{\str}(X/\frakS)$ (resp. $\widetilde{\underline{\calE}}_{\str}(X/\frakS)$) and the category of crystals of $\widetilde{\calE}(X/\frakS)$ (resp. $\widetilde{\underline{\calE}}(X/\frakS)$).
\end{parag}

\begin{remark}\label{remShiho1}
If $(U,\frakT,u)$ is an object of $\calE(X/\frakS)$ (resp. $\underline{\calE}(X/\frakS)$) and if $V$ is an étale an strict $U$-scheme and $(V,\frakT_V,u_V)=\alpha_{(U,\frakT,u)}(V),$ then the construction of $u_V$ shows that $u_V$ is strict if $u$ is strict. It follows that, if $(U,\frakT,u)$ is an object of $\calE_{\mathrm{strict}}(X/\frakS)$ (resp. $\underline{\calE}_{\mathrm{strict}}(X/\frakS)$), then $\alpha_{(U,\frakT,u)}$ \eqref{alphasite} factors through $\calE_{\mathrm{strict}}(X/\frakS)$ (resp. $\underline{\calE}_{\mathrm{strict}}(X/\frakS)$).
\end{remark}

\section{Indexed logarithmic Oyama topoi}

\begin{parag}\label{intro19}
Let $\frakS$ be a logarithmic formal scheme, log flat and locally of finite type over $\op{Spf}W(k),$ and $X\ra S$ a log smooth morphism of logarithmic schemes. For an étale $X$-scheme $U,$ we equip $U$ with the logarithmic structure $\left (\calM_{X}\right )_{|U},$ restriction of $\calM_X$ on $U,$ so that $U\ra X$ is strict.
For every logarithmic scheme $T,$ we denote by $\upmu_T$ the sheaf $\ov{\calM}_T^{gp}.$
\end{parag}

\begin{parag}\label{paragu}
The functors
$$
u_X':\begin{array}[t]{clc}
\calE(X'/\frakS) & \ra & \text{ét}_{/X'} \\
(U,\frakT,u) & \mapsto & U
\end{array},\ 
\underline{u}_X:\begin{array}[t]{clc}
\underline{\calE}(X/\frakS) & \ra & \text{ét}_{/X} \\
(U,\frakT,u) & \mapsto & U
\end{array}
$$
clearly satisfy properties 1 and 2 of (\cite{SF2} 8.18). If $(U,\frakT,u)$ is an object of $\calE(X'/\frakS)$ (resp. $\underline{\calE}(X/\frakS)$) and $(V_i \ra U)_{i\in I}$ is an étale covering, then, by (\cite{SF2} 8.8 (1)),
$$\left (\alpha_{(U,\frakT,u)}(V_i) \ra (U,\frakT,u) \right )_{i\in I}$$
is an étale covering, where $\alpha_{(U,\frakT,u)}$ is given in \eqref{alphasite}. It follows that the functors $u_X'$ and $\underline{u}_X$ satisfy property 3 of (\cite{SF2} 8.18). We conclude that
$u_X'$ and $\underline{u}_X$ are continuous and cocontinuous for the étale topologies and they induce morphisms of topoi, that we abusively denote by
\begin{equation}\label{Kokou}
u_X':\widetilde{\calE}(X'/\frakS) \ra X'_{\text{ét}},\ \underline{u}_X:\widetilde{\underline{\calE}}(X/\frakS) \ra X_{\text{ét}},
\end{equation}
such that the inverse image functors are composition with $u_X'$ and $\underline{u}_X$ respectively. 

We denote by $\mathbbm{T}'$ (resp. $\underline{\mathbbm{T}}$) the split fibered $\widetilde{\calE}(X'/\frakS)$-category (resp. $\widetilde{\underline{\calE}}(X/\frakS)$-category) defined, for every sheaf $\F$ of $\widetilde{\calE}(X'/\frakS)$ (resp. $\widetilde{\underline{\calE}}(X/\frakS)$), by the localized category $\widetilde{\calE}(X'/\frakS)_{/\F}$ (resp. $\widetilde{\underline{\calE}}(X/\frakS)_{/\F}$) and, for every morphism $\varphi:\F \ra \calG,$ by the inverse image functor of the morphism of topoi
\begin{equation}\label{Kokotopoi1}
\widetilde{\calE}(X'/\frakS)_{/\F} \ra \widetilde{\calE}(X'/\frakS)_{/\calG} \left( \text{resp.}\ \widetilde{\underline{\calE}}(X/\frakS)_{/\F} \ra \widetilde{\underline{\calE}}(X/\frakS)_{/\calG}\right )
\end{equation}
induced by $\varphi.$ By (\cite{Giraud} II 3.4.4), $\mathbbm{T}'$ (resp. $\underline{\mathbbm{T}}$) is a stack. We consider the fibered category $\textgoth{T}' \ra X'_{\text{ét}}$ (resp. $\underline{\textgoth{T}} \ra X_{\text{ét}}$) defined by base change of $\mathbbm{T}' \ra \widetilde{\calE}(X'/\frakS)$ (resp. $\underline{\mathbbm{T}} \ra \widetilde{\underline{\calE}}(X/\frakS)$) by the functor $u_X'^{-1}$ (resp. $\underline{u}_X^{-1}$):
\begin{equation}\label{stack1}
\begin{tikzcd}
\textgoth{T}' \ar{r} \ar{d} & \mathbbm{T}' \ar{d} & \underline{\textgoth{T}} \ar{r} \ar{d} & \underline{\mathbbm{T}} \ar{d} \\
X'_{\text{ét}} \ar{r}{u_X'^{-1}} & \widetilde{\calE}(X'/\frakS) & X_{\text{ét}} \ar{r}{\underline{u}_X^{-1}} & \widetilde{\underline{\calE}}(X/\frakS). 
\end{tikzcd}
\end{equation}
By (\cite{Giraud} II 3.1.1), they are stacks.
Let $U$ be an étale $X$-scheme and denote by $U^a$ the associated sheaf of $X_{\text{ét}}.$ By definition, for any object $(V,\frakZ,w)$ of $\widetilde{\underline{\calE}}(X/\frakS),$ we have
$$
\underline{u}_X^{-1}\left (U^a\right )(V,\frakZ,w) = U^a(V)=\op{Hom}_{X}(V,U).
$$
A morphism $(V,\frakZ,w) \ra \underline{u}_X^{-1}(U^a)$ is then equivalent to an $X$-morphism $V\ra U.$ It follows that
$$
\begin{array}[t]{clc}
\underline{\calE}(X/\frakS)_{/\underline{u}_X^{-1}(U^a)} & \ra & \underline{\calE}(U/\frakS) \\
\left ((V,\frakZ,w) \ra \underline{u}_X^{-1}(U^a) \right ) & \mapsto & (V,\frakZ,w)
\end{array}
$$
is an equivalence of categories. This functor clearly satisfies the conditions of (\cite{SF2} 8.18) for the étale topology on the target and the topology on the source induced by that of $\underline{\calE}(X/\frakS).$ It is hence continuous and cocontinuous and induces an equivalence of topoi
\begin{equation}\label{stackequiv1}
\widetilde{\underline{\calE}}(X/\frakS)_{/\underline{u}_X^{-1}(U^a)} \xrightarrow{\sim} \widetilde{\underline{\calE}}(U/\frakS).
\end{equation}
Similarly, if $U$ is an étale $X'$-scheme, then we have an equivalence of topoi
\begin{equation}\label{stackequiv2}
\widetilde{\calE}(X'/\frakS)_{/u_X'^{-1}(U^a)} \xrightarrow{\sim} \widetilde{\calE}(U/\frakS).
\end{equation}
For every morphism of étale $X$-schemes $g:V\ra U,$ the functor
$$
\underline{\rho_g}:\underline{\calE}(V/\frakS) \ra \underline{\calE}(U/\frakS),
$$
defined by composition with $g,$ is clearly continuous and cocontinuous. It then induces a morphism of topoi, abusively denoted
\begin{equation}\label{eqrhoEST1}
\underline{\rho_g}:\widetilde{\underline{\calE}}(V/\frakS) \ra \widetilde{\underline{\calE}}(U/\frakS),
\end{equation}
whose inverse image functor is the composition with the functor $\underline{\rho_g}.$
The morphism of topoi $\underline{\rho_g}$ corresponds to the localization morphism of topoi $\underline{\mathbbm{T}}_{V^a} \ra \underline{\mathbbm{T}}_{U^a}$ \eqref{Kokotopoi1}.
Similarly, if $g:V\ra U$ is a morphism of étale $X'$-schemes, then we have a continuous and cocontinuous functor
$$
\rho_{g}:\calE(V/\frakS) \ra \calE(U/\frakS),
$$
that induces a morphism of topoi, abusively denoted by
\begin{equation}\label{eqrhoEST2}
\rho_{g}:\widetilde{\calE}(V/\frakS) \ra \widetilde{\calE}(U/\frakS),
\end{equation}
and that corresponds to the localization morphism $\mathbbm{T}'_{V^a} \ra \mathbbm{T}'_{U^a}$ \eqref{Kokotopoi1}.
If no ambiguity arises, we will denote $\rho_g$ and $\underline{\rho_g}$ by $\rho_{U,V}$ and $\underline{\rho_{U,V}}$ respectively.
\end{parag}

\begin{parag}\label{paragMuzan194}
Let $\calI_X$ be a sheaf of $X_{\text{ét}}.$ The exact relative Frobenius $F_1:X\ra X'$ is a universal homeomorphism and so we can identify the small étale sites of $X$ and $X'$ via $F_1.$ We hence consider $\calI_X$ also as a sheaf of $X'_{\text{ét}}.$ We set
$$\calI'=u_X'^{-1}\calI_X=\calI_X\circ u_X',\ \underline{\calI}=\underline{u}_X^{-1}\calI_X=\calI_X\circ \underline{u}_X.$$
Let $U$ be an étale $X$-scheme and $s\in \Gamma(U,\calI_X).$
For every object $(V,\frakZ,v)$ of $\calE(X'/\frakS)$ (resp. $\underline{\calE}(X/\frakS)$), we have
$$s_{|V} \in \Gamma(V,\calI_X)=\Gamma\left ( (V,\frakZ,v),\calI'\right )\ \left (\text{resp.}\ \Gamma \left ((V,\frakZ,v),\underline{\calI}\right )\right ).$$
So we have functors
$$
\alpha_s':\begin{array}[t]{clc}
\calE(U'/\frakS) & \ra & \calE(X'/\frakS)_{/\calI'} \\
(V,\frakZ,v) & \mapsto & \left ( \rho_{X',U'}(V,\frakZ,v),s_{|V}\right ),
\end{array}\quad
\underline{\alpha}_s:\begin{array}[t]{clc}
\underline{\calE}(U/\frakS) & \ra & \underline{\calE}(X/\frakS)_{/\underline{\calI}} \\
(V,\frakZ,v) & \mapsto & \left ( \underline{\rho_{X,U}}(V,\frakZ,v),s_{|V}\right ).
\end{array}
$$
Again, by (\cite{SF2} 8.18), $\alpha_s'$ and $\underline{\alpha}_s$ are continuous and cocontinuous so they induce morphisms of topoi
\begin{equation}\label{alphakraz}
\alpha_s':\widetilde{\calE}(U'/\frakS)  \ra  \widetilde{\calE}(X'/\frakS)_{/\calI'},\ 
\underline{\alpha}_s:\widetilde{\underline{\calE}}(U/\frakS) \ra  \widetilde{\underline{\calE}}(X/\frakS)_{/\underline{\calI}},
\end{equation}
such that the inverse image functors are composition by $\alpha_s'$ and $\underline{\alpha}_s$ respectively.

The morphisms of topoi $\alpha'_s$ and $\underline{\alpha}_s$ can be defined equivalently as follows: let $U$ be an étale $X$-scheme (resp. étale $X'$-scheme) and $s\in \Gamma(U,\calI_X).$ The local section $s$ corresponds to a morphism $s:U^a\ra \calI_X$ and so we have a morphism
$$
\underline{u}_X^{-1}(s):\underline{u}_X^{-1}(U^a) \ra \underline{u}_X^{-1}(\calI_X)=\underline{\calI}\ \left (\text{resp.}\ u_X'^{-1}(s):u_X'^{-1}(U^a) \ra u_X'^{-1}(\calI_X)=\calI'\right ).
$$
By \eqref{stackequiv1} (resp. \eqref{stackequiv2}), it induces, by localization, a morphism of topoi:
$$
\underline{\alpha}_s:\widetilde{\underline{\calE}}(U/\frakS) \ra \widetilde{\underline{\calE}}(X/\frakS)_{/\underline{\calI}}\ \left (\text{resp.}\ \alpha'_s:\widetilde{\calE}(U/\frakS)^{-1}(U^a) \ra \widetilde{\calE}(X'/\frakS)_{/\calI'} \right ).
$$
Finally, if $V \ra U$ is an étale morphism of $X$-schemes then the diagrams
$$
\begin{tikzcd}
\calE(V'/\frakS) \ar{r}{\rho_{U',V'}} \ar[swap]{dr}{\alpha_{s_{|V}}'} & \calE(U'/\frakS), \ar{d}{\alpha_s'} & \underline{\calE}(V/\frakS) \ar{r}{\underline{\rho_{U,V}}} \ar[swap]{dr}{\underline{\alpha}_{s_{|V}}} & \underline{\calE}(U/\frakS) \ar{d}{\underline{\alpha}_s} \\
 & \calE(X'/\frakS)_{/\calI'} & & \underline{\calE}(X/\frakS)_{\underline{\calI}}
\end{tikzcd}
$$
are commutative.
\end{parag}

\begin{remark}\label{remShiho194}
If $\calI_X$ is an étale sheaf on $X,$ we can restrict the sheaves $\calI'$ and $\underline{\calI},$ defined in \ref{paragMuzan194}, to the subsites $\calE_{\mathrm{strict}}(X'/\frakS)$ and $\underline{\calE}_{\mathrm{strict}}(X/\frakS)$ respectively. We denote by $\calI'_{\mathrm{strict}}$ and $\underline{\calI}_{\mathrm{strict}}$ these restrictions. We then have localized topoi
$$
\calE_{\mathrm{strict}}(X'/\frakS)_{/\calI'_{\mathrm{strict}}},\ \underline{\calE}_{\mathrm{strict}}(X/\frakS)_{/\underline{\calI}_{\mathrm{strict}}}.
$$
For a local section $s$ of $\calI_X$ over an étale $X$-scheme $U,$ we have functors
\begin{equation}\label{alphakrazShiho}
\underline{\alpha}^{\mathrm{strict}}_s:\widetilde{\underline{\calE}}_{\mathrm{strict}}(U/\frakS) \ra \widetilde{\underline{\calE}}_{\mathrm{strict}}(X/\frakS)_{/\underline{\calI}_{\mathrm{strict}}}\ \left (\text{resp.}\ \alpha'^{\mathrm{strict}}_s:\widetilde{\calE}_{\mathrm{strict}}(U'/\frakS) \ra \widetilde{\calE}_{\mathrm{strict}}(X'/\frakS)_{/\calI'_{\mathrm{strict}}} \right )
\end{equation}
and similar properties to those mentioned in \ref{paragMuzan194}.
\end{remark}

\begin{proposition}\label{indexeddescentdata}
Let $\calI_X$ be a sheaf of $X_{\text{ét}}$ and consider the sheaves $\calI'$ and $\underline{\calI}$ defined in \ref{paragMuzan194}. A presheaf of sets $\F$ on $\calE(X'/\frakS)_{/\calI'}$ (resp. $\underline{\calE}(X/\frakS)_{/\underline{\calI}}$) is equivalent to the data, for every section $s\in \Gamma(U,\calI_X)$ over an étale $X$-scheme $U,$ of a presheaf of sets $\F_s$ on $\calE(U'/\frakS)$ (resp. $\underline{\calE}(U/\frakS)$), such that, for every morphism $g:V\ra U$ of étale $X$-schemes and every section $s\in \Gamma(U,\calI_X),$ we have isomorphisms
$$\rho_{g'}^{-1}\F_s \xrightarrow{\sim} \F_{s_{|V}}\ \left (\text{resp.}\ \underline{\rho_g}^{-1}\F_s \xrightarrow{\sim} \F_{s_{|V}}\right ),$$
satisfying cocycle conditions.
\end{proposition}

\begin{proof}
If $\F$ is a presheaf of sets on $\calE(X'/\frakS)_{/\calI'},$ $U\ra X$ is an étale morphism and $s\in \Gamma(U,\calI_X)$ then we set $\F_s=\F\circ \alpha_s'.$
Conversly, if the presheaves $\F_s$ are given, we define $\F$ by
$$
\F:\begin{array}[t]{clc} 
\calE(X'/\frakS)_{/\calI'} & \ra & \op{Set} \\
(U,\frakT,u,s) & \mapsto & \F_s(\rho_{X',U}(U,\frakT,u)).
\end{array}
$$
If $(U_1,\frakT_1,u_1,s_1) \ra (U_2,\frakT_2,u_2,s_2)$ is a morphism in $\calE(X'/\frakS)_{/\calI'},$ then $s_{2|U_1}=s_1$ and $\rho_{X',U_2}(U_1,\frakT_1,u_1) \ra \rho_{X',U_2}(U_2,\frakT_2,u_2)$ is a morphism in $\calE(U_2/\frakS).$ We thus have a map
$$\F_{s_2}\circ \rho_{X',U_2}(U_2,\frakT_2,u_2) \ra \F_{s_2}\circ \rho_{X',U_2}(U_1,\frakT_1,u_1).$$
Since
$$
\rho_{X',U_2}=\rho_{U_1,U_2}\circ \rho_{X',U_1},
$$
we have
\begin{alignat*}{2}
\F_{s_2} \circ \rho_{X',U_2}=\F_{s_2}\circ \rho_{U_1,U_2}\circ \rho_{X',U_1} &= \left (\rho_{U_1,U_2}^{-1}\F_{s_2}\right )\circ \rho_{X',U_1}\\
&\xrightarrow{\sim}\F_{s_{2|U_1}}\circ \rho_{X',U_1}\\
&=\F_{s_1}\circ \rho_{X',U_1},
\end{alignat*}
we deduce a functorial map
$$\F(U_2,\frakT_2,u_2,s_2) \ra \F(U_1,\frakT_1,u_1,s_1).$$
These two constructions are clearly inverse to each other.
The proof for $\underline{\calE}(X/\frakS)_{/\underline{\calI}}$ is similar.
\end{proof}

\begin{remark}\label{remShiho196}
Proposition \ref{indexeddescentdata} remains true if we replace $\calE(X'/\frakS)_{/\calI'}$ and $\underline{\calE}(X/\frakS)_{\underline{\calI}}$ with the localized sites $\calE_{\mathrm{strict}}(X'/\frakS)_{/\calI'_{\mathrm{strict}}}$ and $\underline{\calE}_{\mathrm{strict}}(X/\frakS)_{/\underline{\calI}_{\mathrm{strict}}}$ respectively, $\calE(U'/\frakS)$ and $\underline{\calE}(U/\frakS)$ with $\calE_{\mathrm{strict}}(U'/\frakS)$ and $\underline{\calE}_{\mathrm{strict}}(U/\frakS)$ respectively and $\underline{\alpha}_s$ and $\alpha'_s$ with $\underline{\alpha}^{\mathrm{strict}}_s$ and $\alpha'^{\mathrm{strict}}_s$ respectively.
\end{remark}

\begin{definition}\label{inddescentdata}
Let $\calI_X$ be a sheaf of $X_{\text{ét}}$ and consider the sheaves $\calI'$ and $\underline{\calI}$ defined in \ref{paragMuzan194} and the sheaves $\calI'_{\mathrm{strict}}$ and $\underline{\calI}_{\mathrm{strict}}$ defined in \ref{remShiho194}. For a presheaf $\F$ on $\calE(X'/\frakS)_{/\calI'},$ $\underline{\calE}(X/\frakS)_{/\underline{\calI}},$ $\calE_{\mathrm{strict}}(X'/\frakS)_{/\calI'_{\mathrm{strict}}}$ or $\underline{\calE}_{\mathrm{strict}}(X/\frakS)_{/\underline{\calI}_{\mathrm{strict}}},$ we call \emph{descent data associated to $\F$} the presheaves $\F_s$ given in \ref{indexeddescentdata} and \ref{remShiho196}.
\end{definition}

\begin{proposition}\label{iranlem145}
Let $\calI_X$ be a sheaf of $X_{\text{ét}}$ and consider the sheaves $\calI'$ and $\underline{\calI}$ defined in \ref{paragMuzan194}. Let $\F$ be a presheaf on $\calE(X'/\frakS)_{/\calI'}$ (resp. $\underline{\calE}(X/\frakS)_{/\underline{\calI}}$) and $(\F_s)$ the associated descent data \eqref{inddescentdata}.  Then the presheaf $\F$ is a sheaf if and only if $\F_s$ is a sheaf for all $s.$
\end{proposition}

\begin{proof}
This follows from \eqref{stackequiv1}, \eqref{stackequiv2} and the fact that the fibered categories $\mathbbm{T}'$ and $\underline{\mathbbm{T}}$ \eqref{stack1} are stacks.
\end{proof}

\begin{remark}
Proposition \ref{iranlem145} remains true if we replace $\calE(X'/\frakS)_{/\calI'}$ and $\underline{\calE}(X/\frakS)_{\underline{\calI}}$ with the sites $\calE_{\mathrm{strict}}(X'/\frakS)_{/\calI'}$ and $\underline{\calE}_{\mathrm{strict}}(X/\frakS)_{/\underline{\calI}}$ respectively.
\end{remark}

\begin{parag}
Let $\calI_X$ be a sheaf of $X_{\text{ét}}$ and consider the sheaves $\calI'$ and $\underline{\calI}$ defined in \ref{paragMuzan194}.
Let $(U,\frakT,u)$ be an object of $\calE(X'/\frakS)$ (resp. $\underline{\calE}(X/\frakS)$). Consider the morphism of topoi $\alpha_{(U,\frakT,u)}$ \eqref{alphatopos} and denote by $\calI_U$ the restriction of $\calI_X$ on $U.$ We have
\begin{alignat*}{2}
\alpha_{(U,\frakT,u)}^{-1}\calI' &= \calI' \circ \alpha_{(U,\frakT,u)}=\calI_U\\
(\text{resp.}\ \alpha_{(U,\frakT,u)}^{-1}\underline{\calI} &= \underline{\calI} \circ \alpha_{(U,\frakT,u)}=\calI_U).
\end{alignat*}
The morphism $\alpha_{(U,\frakT,u)}$ then induces a morphism of topoi
\begin{equation}\label{deltatop}
\delta_{(U,\frakT,u),\calI_X}:U_{\text{ét}/\calI_U} \ra \widetilde{\calE}(X'/\frakS)_{/\calI'} \left (\text{resp.\ }\widetilde{\underline{\calE}}(X/\frakS)_{/\underline{\calI}} \right ).
\end{equation}
Its inverse image functor is composition by the following functor, that we abusively denote by $\delta_{(U,\frakT,u),\calI_X}:$
$$
\delta_{(U,\frakT,u),\calI_X}:
\begin{array}[t]{clc}
\text{ét}_{/U\times \calI_X} & \ra & \calE(X'/\frakS)_{/\calI'} \left (\text{resp.\ }\underline{\calE}(X/\frakS)_{/\underline{\calI}}\right )\\
(V,s) & \mapsto & \left (\alpha_{(U,\frakT,u)}(V),s\right ),
\end{array}
$$
where $V$ is an étale $U$-scheme and $s\in \Gamma(V,\calI_X)$ (note that we consider $\calI_X$ as sheaf of $X'_{\text{ét}}$ via $F_1:X\ra X'$) and $\alpha_{(U,\frakT,u)}$ is given in \eqref{paragIran4}. If $(U_1,\frakT_1,u_1) \ra (U_2,\frakT_2,u_2)$ is a morphism of $\calE(X'/\frakS)$ (resp. $\underline{\calE}(X/\frakS)$), we consider the functor
$$
j_{g,\calI_X}:\begin{array}[t]{cll}
\text{ét}_{/U_1\times \calI_X} & \ra & \text{ét}_{/U_2\times \calI_X} \\
(V \ra U_1\times \calI_X) & \mapsto & \left (V \ra U_1\times \calI_X \xrightarrow{g\times \op{Id}} U_2\times \calI_X \right ).
\end{array}
$$
If $V$ is an étale $U_1$-scheme then we have a morphism
$$
\beta_{(f,g)}(V):\alpha_{(U_1,\frakT_1,u_1)}(V) \ra \alpha_{(U_2,\frakT_2,u_2)}\circ j_g(V),
$$
where $\beta_{(f,g)}$ and $j_g$ are given in \ref{paragIran4}.
We thus get a morphism
\begin{equation}\label{defbeta115ind}
\beta_{(f,g),\calI_X}:\delta_{(U_1,\frakT_1,u_1)} \ra \delta_{(U_2,\frakT_2,u_2)} \circ j_{g,\calI_X}.
\end{equation}
For étale morphisms $V\ra U \ra X'$ (resp. $V\ra U \ra X$), an object $(V,\frakZ,v)$ of $\calE(U/\frakS)$ (resp. $\underline{\calE}(U/\frakS)$) and a section $s\in \Gamma(U,\calI_X),$ we have the following commutative diagram:
\begin{equation}\label{diag188}
\begin{tikzcd}
\text{ét}_{/V} \ar{rr}{\alpha_{(V,\frakZ,v)}} \ar[swap]{dd}{W\mapsto (W,s_{|W})} & & \calE(U/\frakS) \left (\text{resp.\ }\underline{\calE}(U/\frakS)\right )\ar{dd}{\alpha_s'\left (\text{resp.\ }\underline{\alpha}_s\right )} \\
 & & \\
\text{ét}_{/V\times \calI_X} \ar[swap]{rr}{\delta_{(V,\frakZ,v),\calI_X}} & & \calE(X'/\frakS)_{/\calI'} \left (\text{resp.\ }\underline{\calE}(X/\frakS)_{/\underline{\calI}}\right ).
\end{tikzcd}
\end{equation}
For a positive integer $n$ and $\calI_X=\upmu_X^n,$ we will denote $\delta_{(U,\frakT,u),\calI_X}$ simply by $\delta_{(U,\frakT,u),n}.$ If, furthermore $n=1,$ we denote $\delta_{(U,\frakT,u),n}$ simply by $\delta_{(U,\frakT,u)}.$
\end{parag}

\begin{remark}
Let $\calI_X$ be an étale sheaf on $X$ and consider the sheaves $\calI'$ and $\underline{\calI}$ defined in \ref{paragMuzan194} and the sheaves $\calI'_{\mathrm{strict}}$ and $\underline{\calI}_{\mathrm{strict}}$ defined in \ref{remShiho194}. If $(U,\frakT,u)$ is an object of $\calE_{\mathrm{strict}}(X'/\frakS)$ (resp. $\underline{\calE}_{\mathrm{strict}}(X/\frakS)$), then, by \ref{remShiho1}, the functor \eqref{deltatop} factors through $\widetilde{\calE}_{\mathrm{strict}}(X'/\frakS)_{/\calI'_{\mathrm{strict}}}$ (resp. $\widetilde{\underline{\calE}}_{\mathrm{strict}}(X/\frakS)_{/\underline{\calI}_{\mathrm{strict}}})$, yielding a morphism of topoi, abusively denoted
\begin{equation}
\delta_{(U,\frakT,u),\calI_X}^{\mathrm{strict}}:U_{\text{ét}/\calI_U} \ra \widetilde{\calE}_{\mathrm{strict}}(X'/\frakS)_{/\calI'_{\mathrm{strict}}} \left (\text{resp.\ }\widetilde{\underline{\calE}}_{\mathrm{strict}}(X/\frakS)_{/\underline{\calI}_{\mathrm{strict}}} \right ).
\end{equation}
\end{remark}

\begin{proposition}\label{inddescentdata2}
Let $\calI_X$ be an étale sheaf on $X$ an consider it as an étale sheaf on $X'$ via the exact relative Frobenius $X\ra X'.$
For an étale $X$-scheme or an étale $X'$-scheme $U,$ we denote by $\calI_U$ the restriction of $\calI_X$ on $U.$ A presheaf $\F$ on $\calE(X'/\frakS)_{/\calI'}$ (resp. $\underline{\calE}(X/\frakS)_{/\underline{\calI}}$) is equivalent to the following data:
\begin{enumerate}
\item For every object $(U,\frakT,u)$ of $\calE(X'/\frakS)$ (resp. $\underline{\calE}(X/\frakS)$), a presheaf $\F_{(U,\frakT,u)}$ on $\text{ét}_{/U\times \calI_X}.$
\item For every morphism $(f,g):(U_1,\frakT_1,u_1) \ra (U_2,\frakT_2,u_2)$ in $\calE(X'/\frakS)$ (resp. $\underline{\calE}(X/\frakS)$), a morphism of presheaves on $\text{ét}_{/U_1\times \calI_X},$
$$\gamma_{\F,(f,g)}:g^{-1}\F_{(U_2,\frakT_2,u_2)} \ra \F_{(U_1,\frakT_1,u_1)},$$
\end{enumerate}
where $g$ denotes abusively the morphism of topoi $U_{1,\text{ét}/\calI_{U_1}} \ra U_{2,\text{ét}/\calI_{U_2}}$
induced by $g:U_1\ra U_2,$
such that
\begin{enumerate}[(i)]
\item For every object $(U,\frakT,u)$ of $\calE(X/\frakS)$ (resp. $\underline{\calE}(X/\frakS)$),
$$\gamma_{\F,\op{Id}_{(U,\frakT,u)}}=\op{Id}_{\F_{(U,\frakT,u)}}.$$
\item For all composable morphisms $(f_1,g_1):(U_1,\frakT_1,u_1) \ra (U_2,\frakT_2,u_2)$ and $(f_2,g_2):(U_2,\frakT_2,u_2) \ra (U_3,\frakT_3,u_3)$ in $\calE(X'/\frakS)$ (resp. $\underline{\calE}(X/\frakS)$),
$$\gamma_{\F,(f_1,g_1)} \circ g_1^{-1}\gamma_{\F,(f_2,g_2)}=\gamma_{\F,(f_2,g_2)\circ (f_1,g_1)}.$$
\item If $(f,g)$ is a cartesian morphism then $\gamma_{\F,(f,g)}$ is an isomorphism.
\end{enumerate}
This equivalence is given as follows: for a presheaf $\F$ on $\calE(X'/\frakS)_{/\calI'}$ (resp. $\underline{\calE}(X/\frakS)_{/\underline{\calI}}$),
$$\F_{(U,\frakT,u)}=\F \circ \delta_{(U,\frakT,u),\calI_X}$$
and $\gamma_{\F,(f,g)}$ is induced by $\beta_{(f,g),\calI_X}$ \eqref{defbeta115ind}. Conversely, given a data $(\F_{(U,\frakT,u)},\gamma_{\F,(f,g)})$ as above, the presheaf is defined by
$$\F(U,\frakT,u,s)=\F_{(U,\frakT,u)}(U,s)$$
and, for a morphism $(f,g):(U_1,\frakT_1,u_1,s_1) \ra (U_2,\frakT_2,u_2,s_2),$ the morphism
$$\F(U_2,\frakT_2,u_2,s_2) \ra \F(U_1,\frakT_1,u_1,s_1)$$
is equal to $\gamma_{\F,(f,g)}(U_1,s_1).$
\end{proposition}

\begin{proof}
Similar to (\cite{SF2} 8.11).
\end{proof}

\begin{remark}
Proposition \ref{inddescentdata2} remains true if we replace $\calE(X'/\frakS)_{/\calI'}$ and $\underline{\calE}(X/\frakS)_{\underline{\calI}}$ with $\calE_{\mathrm{strict}}(X'/\frakS)_{/\calI'_{\mathrm{strict}}}$ and $\underline{\calE}_{\mathrm{strict}}(X/\frakS)_{/\underline{\calI}_{\mathrm{strict}}}$ respectively and if we replace $\calE(X'/\frakS)$ and $\underline{\calE}(X/\frakS)$ with $\calE_{\mathrm{strict}}(X'/\frakS)$ and $\underline{\calE}_{\mathrm{strict}}(X/\frakS)$ respectively.
\end{remark}

\begin{definition}\label{definddescentdata}
For a presheaf of sets $\F$ on $\calE(X'/\frakS)_{/\calI'}$ (resp. $\underline{\calE}(X/\frakS)_{/\underline{\calI}},$ $\calE_{\mathrm{strict}}(X'/\frakS)_{/\calI'_{\mathrm{strict}}},$ $\underline{\calE}_{\mathrm{strict}}(X/\frakS)_{/\underline{\calI}_{\mathrm{strict}}}$), we call \emph{indexed descent data associated with $\F$} the data $(\F_{(U,\frakT,u)},\gamma_{\F,(f,g)})$ given in \ref{inddescentdata2}.
\end{definition}

\begin{proposition}\label{prop1409}
Let $\F$ be a presheaf of sets on $\calE(X'/\frakS)_{/\calI'}$ (resp. $\underline{\calE}(X/\frakS)_{/\underline{\calI}},$ $\calE_{\mathrm{strict}}(X'/\frakS)_{/\calI'},$ $\underline{\calE}_{\mathrm{strict}}(X/\frakS)_{/\underline{\calI}}$) and $(\F_{(U,\frakT,u)},\gamma_{\F,(f,g)})$ the indexed descent data associated with $\F$ \eqref{definddescentdata}. Then $\F$ is a sheaf for the étale topology if and only if $\F_{(U,\frakT,u)}$ is a sheaf of $U_{\text{ét}/\calI_U}$ for every $(U,\frakT,u).$
\end{proposition}

\begin{proof}
The proof is similar to (\cite{SF2} 8.15).
\end{proof}

\begin{parag}\label{kraz1911}
Let $\calI_X$ be an étale sheaf on $X$ and consider the sheaves $\calI'$ and $\underline{\calI}$ defined in \ref{paragMuzan194} and the sheaves $\calI'_{\mathrm{strict}}$ and $\underline{\calI}_{\mathrm{strict}}$ defined in \ref{remShiho194}. Consider the localization morphisms of topoi
$$j_{\calI'}:\widetilde{\calE}(X'/\frakS)_{/\calI'} \ra \widetilde{\calE}(X/\frakS),\ j_{\underline{\calI}}:\widetilde{\underline{\calE}}(X/\frakS)_{/\underline{\calI}} \ra \widetilde{\underline{\calE}}(X/\frakS),$$
$$
j_{\calI'_{\mathrm{strict}}}:\widetilde{\calE}_{\mathrm{strict}}(X'/\frakS)_{/\calI'_{\mathrm{strict}}} \ra \widetilde{\calE}_{\mathrm{strict}}(X/\frakS),\ j_{\underline{\calI}_{\mathrm{strict}}}:\widetilde{\underline{\calE}}_{\mathrm{strict}}(X/\frakS)_{/\underline{\calI}_{\mathrm{strict}}} \ra \widetilde{\underline{\calE}}_{\mathrm{strict}}(X/\frakS)
$$
given in \ref{indeq2}.
We equip the topos $\widetilde{\calE}(X'/\frakS)_{/\calI'}$ (resp. $\widetilde{\underline{\calE}}(X/\frakS)_{/\underline{\calI}}$) with the ring $\Ox_{\calI'}=j_{\calI'}^{-1}\Ox_{\calE(X'/\frakS)}$ (resp. $\Ox_{\underline{\calI}}=j_{\underline{\calI}}^{-1}\Ox_{\underline{\calE}(X/\frakS)}$). Notice that, if $s\in \Gamma(U,\calI_X)$ is a local section over an étale $X'$-scheme (resp. $X$-scheme) $U,$ then
$$\alpha_s'^{-1}\Ox_{\calI'}=\rho_{X',U}^{-1}\Ox_{\calE(X'/\frakS)}=\Ox_{\calE(U/\frakS)}, \quad \left (\text{resp.}\ \underline{\alpha}_s^{-1}\Ox_{\underline{\calI}}=\underline{\rho_{X,U}^{-1}}\Ox_{\underline{\calE}(X/\frakS)}=\Ox_{\underline{\calE}(U/\frakS)}\right ).$$
If $\F$ is a module of $\widetilde{\calE}(X'/\frakS)_{/\calI'}$ (resp. $\widetilde{\underline{\calE}}(X/\frakS)_{/\underline{\calI}}$) then $\alpha_s'^{-1}\F$ (resp. $\underline{\alpha}_s^{-1}\F$) is a module on $\calE(U/\frakS)$ (resp. $\underline{\calE}(U/\frakS)$).

We similarly define the rings $\Ox_{\calI'_{\mathrm{strict}}}$ and $\Ox_{\underline{\calI}_{\mathrm{strict}}}.$ We then have similar properties to those given above.
\end{parag}

\begin{lemma}
Let $(U,\frakT,u)$ be an object of $\calE(X'/\frakS)$ (resp. $\underline{\calE}(X/\frakS)$). Consider the morphism of topoi $\delta_{(U,\frakT,u),\calI_X}$ \eqref{deltatop}. Set $\Ox=\Ox_{\calE(X'/\frakS)}$ (resp. $\Ox=\Ox_{\underline{\calE}(X/\frakS)}$) and $j=j_{\calI'}$ (resp $j=j_{\underline{\calI}}$). Then
$$
\delta_{(U,\frakT,u),\calI_X}^{-1}j^{-1}\Ox=\left (u_*\Ox_T\right )_{/\calI_X}.
$$
\end{lemma}

\begin{proof}
For $\calE(X'/\frakS),$ this follows from (\cite{SF2} (8.24.11)) and the commutative diagram
$$
\begin{tikzcd}
\text{ét}_{/U\times \calI_X} \ar{rr}{\delta_{(U,\frakT,u),\calI_X}} \ar[swap]{d}{j_{\calI_X}} & & \calE(X'/\frakS)_{/\calI'} \ar{d}{j_{\calI'}} \\
\text{ét}_{/U} \ar[swap]{rr}{\alpha_{(U,\frakT,u)}} & & \calE(X'/\frakS).
\end{tikzcd}
$$
The proof for $\underline{\calE}(X/\frakS)$ is similar.
\end{proof}

\begin{parag}
Let $(f,g):(U_1,\frakT_1,u_1) \ra (U_2,\frakT_2,u_2)$ be a morphism in $\calE(X'/\frakS)$ (resp. $\underline{\calE}(X/\frakS)$). We have a morphism of ringed topoi \eqref{Kokoftilda}
\begin{equation}
\widetilde{f}:\left (U_{1,\text{ét}},u_{1*}\Ox_{T_{1}} \right ) \ra \left ( U_{2,\text{ét}}, u_{2*}\Ox_{T_{2}} \right ).
\end{equation}
We denote by $\calI_{U_i}$ the restriction of $\calI_X$ to $U_i.$ We have the following commutative diagram of topoi
$$
\begin{tikzcd}
U_{1,\text{ét}/\calI_{U_1}} \ar{r}{g_{\calI}} \ar[swap]{d}{j_{\calI_{U_1}}} & U_{2,\text{ét}/\calI_{U_2}} \ar{d}{j_{\calI_{U_2}}} \\
U_{1,\text{ét}} \ar{r}{g} & U_{2,\text{ét}}, 
\end{tikzcd}
$$
where the upper arrow is induced by $g.$ We then have
$$
\left (g^{-1}u_{2*}\Ox_{T_2} \right )_{\calI_{U_1}} =g_{\calI}^{-1}\left ( \left (u_{2*}\Ox_{T_2} \right )_{\calI_{U_2}} \right ).
$$
It follows that the canonical morphism
$$
g^{-1}u_{2*}\Ox_{T_2}=\widetilde{f}^{-1}u_{2*}\Ox_{T_2} \ra u_{1*} \Ox_{T_1}
$$
induces
$$
g_{\calI}^{-1}\left ( \left (u_{2*}\Ox_{T_2} \right )_{\calI_{U_2}} \right ) \ra \left ( u_{1*}\Ox_{T_1} \right )_{\calI_{U_1}}.
$$
We then get a morphism of ringed topoi
\begin{equation}\label{gmu}
\widetilde{f}_{\calI}: \left (U_{1,\text{ét}/\calI_{U_1}},\left ( u_{1*}\Ox_{T_1} \right )_{\calI_{U_1}} \right ) \ra \left (U_{2,\text{ét}/\calI_{U_2}},  \left (u_{2*}\Ox_{T_2} \right )_{\calI_{U_2}}\right ).
\end{equation}
\end{parag}

\begin{proposition}\label{indlindescentdata2}
Let $\calI_X$ be an étale sheaf on $X$ an consider it as an étale sheaf on $X'$ via the exact relative Frobenius $X\ra X'.$ For an étale $X$-scheme $U,$ we denote by $\calI_U$ the restriction of $\calI_X$ on $U.$ A module $\F$ on $\calE(X'/\frakS)_{/\calI'}$ (resp. $\underline{\calE}(X/\frakS)_{/\underline{\calI}}$) is equivalent to the following data:
\begin{enumerate}
\item For every object $(U,\frakT,u)$ of $\calE(X'/\frakS)$ (resp. $\underline{\calE}(X/\frakS)$), a module $\F_{(U,\frakT,u)}$ of
$$\left (U_{\text{ét}/ \calI_U},\left (u_*\Ox_T \right )_{\calI_U}\right ).$$
\item For every morphism $(f,g):(U_1,\frakT_1,u_1) \ra (U_2,\frakT_2,u_2)$ in $\calE(X'/\frakS)$ (resp. $\underline{\calE}(X/\frakS)$), a morphism of modules of $U_{1,\text{ét}/ \calI_{U_1}},$
\begin{equation}\label{cind}
c_{\F,(f,g)}:\widetilde{f}_{\calI}^*\F_{(U_2,\frakT_2,u_2)} \ra \F_{(U_1,\frakT_1,u_1)},
\end{equation}
\end{enumerate}
where $\widetilde{f}_{\calI}$ is given in \eqref{gmu}.
such that
\begin{enumerate}[(i)]
\item For every object $(U,\frakT,u)$ of $\calE(X/\frakS)$ (resp. $\underline{\calE}(X/\frakS)$),
$$c_{\F,\op{Id}_{(U,\frakT,u)}}=\op{Id}_{\F_{(U,\frakT,u)}}.$$
\item For all composable morphisms $(f_1,g_1):(U_1,\frakT_1,u_1) \ra (U_2,\frakT_2,u_2)$ and $(f_2,g_2):(U_2,\frakT_2,u_2) \ra (U_3,\frakT_3,u_3)$ in $\calE(X'/\frakS)$ (resp. $\underline{\calE}(X/\frakS)$),
$$c_{\F,(f_1,g_1)} \circ \widetilde{f}_{1,\calI}^{*}\gamma_{\F,(f_2,g_2)}=\gamma_{\F,(f_2,g_2)\circ (f_1,g_1)}.$$
\item If $(f,g)$ is a cartesian morphism then $\gamma_{\F,(f,g)}$ is an isomorphism.
\end{enumerate}
This equivalence is given as follows: for a module $\F$ on $\calE(X'/\frakS)_{/\calI'}$ (resp. $\underline{\calE}(X/\frakS)_{/\underline{\calI}}$),
$$\F_{(U,\frakT,u)}=\F \circ \delta_{(U,\frakT,u),\calI_X}$$
and $c_{\F,(f,g)}$ is induced by $\gamma_{\F,(f,g)}$ \eqref{inddescentdata2}. Conversely, given a data $(\F_{(U,\frakT,u)},c_{\F,(f,g)})$ as above, the module is defined by
$$\F(U,\frakT,u,s)=\F_{(U,\frakT,u)}(U,s)$$
and, for a morphism $(f,g):(U_1,\frakT_1,u_1,s_1) \ra (U_2,\frakT_2,u_2,s_2),$ the morphism
$$\F(U_2,\frakT_2,u_2,s_2) \ra \F(U_1,\frakT_1,u_1,s_1)$$
is equal to $c_{\F,(f,g)}(U_1,s_1).$
\end{proposition}

\begin{proof}
The proof is similar to (\cite{SF2} 8.24).
\end{proof}

\begin{remark}\label{remShiho1920}
Proposition \ref{indlindescentdata2} remains true if we replace $\calE(X'/\frakS),$ $\underline{\calE}(X/\frakS),$ $\calI',$ $\underline{\calI}$ and $\delta_{(U,\frakT,u),\calI_X}$ with $\calE_{\mathrm{strict}}(X'/\frakS),$ $\underline{\calE}_{\mathrm{strict}}(X/\frakS),$ $\calI'_{\mathrm{strict}},$ $\underline{\calI}_{\mathrm{strict}}$ and $\delta_{(U,\frakT,u),\calI_X}^{\mathrm{strict}}$ respectively.
\end{remark}

\begin{definition}\label{defindlindescentdata}
For a module $\F$ on $\calE(X'/\frakS)_{/\calI'}$ (resp. $\underline{\calE}(X/\frakS)_{/\underline{\calI}}$), we call \emph{indexed linearized descent data associated with $\F$} the data $(\F_{(U,\frakT,u)},c_{\F,(U,\frakT,u)})$ defined in \ref{indlindescentdata2}.
\end{definition}

\begin{parag}
Let $p_1,p_2:\calI_X^2 \ra \calI_X$ be the canonical projections. These projections induce morphisms of ringed topoi, that we abusively denote by
$$
p_1,p_2:\widetilde{\calE}(X'/\frakS)_{/\calI'^2} \ra \widetilde{\calE}(X'/\frakS)_{/\calI'} \left (\text{resp.}\ \widetilde{\underline{\calE}}(X/\frakS)_{/\underline{\calI}^2} \ra \widetilde{\underline{\calE}}(X/\frakS)_{/\underline{\calI}}\right ).
$$
For modules $\F$ and $\calG$ of $\widetilde{\calE}(X'/\frakS)_{/\calI'}$ (resp. $\widetilde{\underline{\calE}}(X/\frakS)_{/\underline{\calI}}$), we denote by $\F\boxtimes \calG$ the exterior tensor product
$$
\F\boxtimes \calG=p_1^*\F \otimes_{\Ox_{\calI'^2}} p_2^*\calG \left (\text{resp.}\ \F\boxtimes \calG=p_1^*\F \otimes_{\Ox_{\underline{\calI}^2}} p_2^*\calG\right ).
$$
\end{parag}

\begin{proposition}\label{prop1916Ko}
Let $\calG$ and $\F$ be two $\Ox_{\calI'}$-modules of $\widetilde{\calE}(X'/\frakS)_{/\calI'}$ (resp. two $\Ox_{\underline{\calI}}$-modules of $\widetilde{\underline{\calE}}(X/\frakS)_{/\underline{\calI}}$) and $(\calG_s),$ $(\F_s)$ and $\left (\calG\boxtimes\F \right )_{t}$ (resp. $\left (\calG\boxtimes\F \right )_{t}$) the corresponding descent data \eqref{inddescentdata}. Then, for any local sections $s,t\in \Gamma(U,\calI_X)$ over an étale $X'$-scheme (resp. $X$-scheme) $U,$
$$\left (\calG\boxtimes\F \right )_{(s,t)}=\calG_s\otimes_{\Ox_{\calE(U/\frakS)}}\calG_t$$
$$\left (\text{resp.}\ \left (\calG\boxtimes\F \right )_{(s,t)}=\calG_s\otimes_{\Ox_{\underline{\calE}(U/\frakS)}}\calG_t\right ).$$
\end{proposition}

\begin{proof}
Let $p_1,p_2:\calI'^2 \ra \calI'$ be the canonical projections. The proposition results from the commutativity of the following diagram, for any étale $X'$-scheme $U$ and $s_1,s_2\in \Gamma(U,\calI_X):$
$$
\begin{tikzcd}
\widetilde{\calE}(U/\frakS) \ar{rr}{\alpha_{(s_1,s_2)}'} \ar[swap]{drr}{\alpha_{s_i}'} & & \widetilde{\calE}(X'/\frakS)_{/\calI'^2} \ar{d}{p_i}  \\
 & & \widetilde{\calE}(X'/\frakS)_{/\calI'} 
\end{tikzcd}
$$
and from a similar diagram for $\underline{\calE}(X/\frakS).$
\end{proof}

\begin{parag}
Suppose that $\calI_X$ is a sheaf of monoids. Let $p_1,p_2:\calI'^2\ra \calI'$ and $p_{1,X},p_{2,X}:\calI_X^2 \ra \calI_X$ be the canonical projections and $\sigma:\calI'^2\ra \calI'$ and $\sigma_X:\calI_X^2\ra \calI_X$ the addition maps. Then
$$
u_X'^{-1}p_i=p_{i,X},\ u_X'^{-1}\sigma=\sigma_X,
$$
where $u_X'$ is given in \eqref{Kokou}.
\end{parag}

\begin{proposition}\label{prop1415}
Let $\F$ be a presheaf of sets on $\calE(X'/\frakS)_{/\calI'}$ (resp. $\underline{\calE}(X/\frakS)_{/\underline{\calI}}$) and $(\F_{(U,\frakT,u)},\gamma_{\F,(f,g)})$ the indexed descent data associated with $\F$ \eqref{definddescentdata}. Then, if $\F$ has a $\calI'$-indexed (resp. $\underline{\calI}$-indexed) algebra structure, the sheaves $\F_{(U,\frakT,u)}$ have $\calI_U$-indexed algebra structures for any object $(U,\frakT,u)$ of $\calE(X'/\frakS)_{/\calI'}$ (resp. $\underline{\calE}(X/\frakS)_{/\underline{\calI}}$). The same is true for indexed modules.
\end{proposition}

\begin{proof}
Let $p_1,p_2:\calI'^2\ra \calI'$ and $p_{1,X},p_{2,X}:\calI_X^2 \ra \calI_X$ be the canonical projections and $\sigma:\calI'^2\ra \calI'$ and $\sigma_X:\calI_X^2\ra \calI_X$ the addition maps.
For any object $(U,\frakT,u)$ of $\calE(X'/\frakS),$ we have the following commutative diagram
$$
\begin{tikzcd}
U_{\text{ét}/\calI_U^2} \ar{rr}{\delta_{(U,\frakT,u),\calI_X^2}} \ar[swap]{d}{p_{i,U}\ (\text{resp.}\ \sigma_U)} & & \widetilde{\calE}(X'/\frakS)_{/\calI'^2} \ar{d}{p_{i}\ (\text{resp.}\ \sigma)} \\
U_{\text{ét}/\calI_U} \ar[swap]{rr}{\delta_{(U,\frakT,u),\calI_X}} & & \widetilde{\calE}(X'/\frakS)_{/\calI'}
\end{tikzcd}
$$
This implies that, if $\F$ has a $\calI'$-indexed algebra structure, then $\F_{(U,\frakT,u)}=\delta_{(U,\frakT,u),\calI_X}^*\F$ has a $\calI_U$-indexed algebra structure. The proof for $\underline{\calE}(X/\frakS)_{/\underline{\calI}}$ and for indexed modules is similar.
\end{proof}

\begin{proposition}\label{propKoko1915}
Let $\calA$ be a $\calI'$-indexed algebra on $\calE(X'/\frakS)_{/\calI'}$ and $\calE$ and $\F$ two $\calI'$-indexed $\calA$-modules \eqref{inddef}. For any object $(U,\frakT,u)$ of $\calE(X'/\frakS),$ there exists a canonical isomorphism
$$\calE_{(U,\frakT,u)}\circledast_{\calA_{(U,\frakT,u)}}\F_{(U,\frakT,u)} \xrightarrow{\sim} \left (\calE \circledast_{\calA}\F \right )_{(U,\frakT,u)}.$$
\end{proposition}

\begin{proof}
Let $p_1,p_2:\calI'^2\ra \calI'$ and $p_{1,X},p_{2,X}:\calI_X^2 \ra \calI_X$ be the canonical projections and $\sigma:\calI'^2 \ra \calI'$ and $\sigma_X:\calI_X^2\ra \calI_X$ the addition maps.
Applying \ref{lemtop15} to
$$
\begin{tikzcd}
U_{\text{ét}/\calI_U^2} \ar{rr}{\delta_{(U,\frakT,u),\calI_X^2}} \ar[swap]{d}{\sigma_U} & & \widetilde{\calE}(X'/\frakS)_{/\calI'^2} \ar{d}{\sigma} \\
U_{\text{ét}/\calI_U} \ar[swap]{rr}{\delta_{(U,\frakT,u),\calI_X}} & & \widetilde{\calE}(X'/\frakS)_{/\calI'},
\end{tikzcd}
$$
we get a canonical isomorphism
$$\sigma_{U!}\delta_{(U,\frakT,u),\calI_X^2}^*\left (\calE\boxtimes \F \right ) \xrightarrow{\sim} \delta_{(U,\frakT,u),\calI_X}^*\sigma_!\left (\calE\boxtimes \F \right ).$$
Then, by the commutative diagram
$$
\begin{tikzcd}
U_{\text{ét}/\calI_U^2} \ar{rr}{\delta_{(U,\frakT,u),\calI_X^2}} \ar[swap]{d}{p_{i,X}} & & \widetilde{\calE}(X'/\frakS)_{/\calI'^2} \ar{d}{p_i} \\
U_{\text{ét}/\calI_U} \ar[swap]{rr}{\delta_{(U,\frakT,u),\calI_X}} & & \widetilde{\calE}(X'/\frakS)_{/\calI'},
\end{tikzcd}
$$
we get a canonical isomorphism
\begin{equation}\label{feq012}
\sigma_{U!}\left (\calE_{(U,\frakT,u)}\boxtimes \F_{(U,\frakT,u)} \right ) \xrightarrow{\sim} \left (\sigma_!\left (\calE\boxtimes \F \right )\right )_{(U,\frakT,u)}.
\end{equation}
Let $\tau:\calI^3 \ra \calI$ and $\tau_X:\calI_X^3\ra \calI_X$ be the addition maps.
Recall that, by \eqref{exseqindtens}, we have an exact sequence of $\Ox_{\calI}$-modules
$$\tau_!\left (\calA \boxtimes \calE\boxtimes \F \right )\ra \sigma_!\left (\calE \boxtimes \F\right ) \ra \calE \circledast_{\calA}\F \ra 0.$$
Since $\delta_{(U,\frakT,u)}^*$ is right exact, we obtain the exact sequence
$$\left (\tau_!\left (\calA \boxtimes \calE\boxtimes \F \right )\right )_{(U,\frakT,u)}\ra \left (\sigma_!\left (\calE \boxtimes \F\right )\right )_{(U,\frakT,u)} \ra \left (\calE \circledast_{\calA}\F\right )_{(U,\frakT,u)} \ra 0.$$
For a sheaf $\calG$ of $\widetilde{\calE}(X'/\frakS)_{/\calI'},$ we denote $\calG_{(U,\frakT,u)}$ simply by $\calG_U.$
By \eqref{feq012} and a similar isomorphism for $\tau,$ we get a commutative diagram
$$
\begin{tikzcd}
\tau_{U!}\left (\calA_U \boxtimes \calE_U\boxtimes \F_U \right )\ar{r} \ar[sloped]{d}{\sim} & \sigma_{U!}\left (\calE_U \boxtimes \F_U\right ) \ar{r} \ar[sloped]{d}{\sim} & \calE_U  \circledast_{\calA_U }\F_U \ar{r} \ar[dashed]{d} & 0 \\
\left (\tau_!\left (\calA \boxtimes \calE\boxtimes \F \right )\right )_{(U,\frakT,u)}\ar{r} & \left (\sigma_!\left (\calE \boxtimes \F\right )\right )_{(U,\frakT,u)} \ar{r} & \left (\calE \circledast_{\calA}\F\right )_{(U,\frakT,u)} \ar{r} & 0
\end{tikzcd}
$$
with exact lines.
Hence the existence of the dashed arrow and the fact that it is an isomorphism.
\end{proof}

\subsection*{Crystals and the Cartier transform}

\begin{parag}
In this subsection, let $\calI_X$ be an étale sheaf on $X$ and consider the sheaves $\calI'$ and $\underline{\calI}$ defined in \ref{paragMuzan194} and the sheaves $\calI'_{\mathrm{strict}}$ and $\underline{\calI}_{\mathrm{strict}}$ defined in \ref{remShiho194}.
\end{parag}

\begin{definition}
Let $\F$ be an $\Ox_{\calI'}$-module of $\widetilde{\calE}(X'/\frakS)_{/\calI'}$ (resp. an $\Ox_{\underline{\calI}}$-module of $\widetilde{\underline{\calE}}(X/\frakS)_{/\underline{\calI}},$ an $\Ox_{\calI'_{\mathrm{strict}}}$-module of $\widetilde{\calE}_{\mathrm{strict}}(X'/\frakS)_{/\calI'_{\mathrm{strict}}},$ an $\Ox_{\underline{\calI}_{\mathrm{strict}}}$-module of $\widetilde{\underline{\calE}}_{\mathrm{strict}}(X/\frakS)_{/\underline{\calI}_{\mathrm{strict}}}$) \eqref{kraz1911} and $(\F_s)$ the corresponding descent data \eqref{inddescentdata}. Recall that for every local section $s\in \Gamma(U,\calI_X)$ over an étale $X'$-scheme (resp. $X$-scheme) $U,$ the sheaf $\F_s$ is a module on $\calE(U/\frakS)$ (resp. $\underline{\calE}(U/\frakS),$ $\calE_{\mathrm{strict}}(U'/\frakS),$ $\underline{\calE}_{\mathrm{strict}}(U/\frakS)$). We say that $\F$ is \emph{quasi-coherent} (resp. a \emph{crystal}) if $\F_s$ is quasi-coherent (resp. a crystal) \eqref{defcrys} for every local section $s$ of $\calI_X.$
\end{definition}

\begin{proposition}\label{equivShihoind}
The restriction functor
$$
\widetilde{\calE}(X'/\frakS) \ra \widetilde{\calE}_{\mathrm{strict}}(X'/\frakS)
$$
$$
\left ( \text{resp.}\ \widetilde{\underline{\calE}}(X/\frakS) \ra \widetilde{\underline{\calE}}_{\mathrm{strict}}(X/\frakS) \right )
$$
induces an equivalence between the subcategories of crystals
$$
\begin{Bmatrix}\text{Crystals\ of\ }\widetilde{\calE}(X'/\frakS)_{/\calI'} \end{Bmatrix} \xrightarrow{\sim} \begin{Bmatrix}\text{Crystals\ of\ }\widetilde{\calE}_{\mathrm{strict}}(X'/\frakS)_{/\calI'_{\mathrm{strict}}} \end{Bmatrix}
$$
$$
\left (\text{resp.}\ \begin{Bmatrix}\text{Crystals\ of\ }\widetilde{\underline{\calE}}(X/\frakS)_{/\underline{\calI}} \end{Bmatrix} \xrightarrow{\sim} \begin{Bmatrix}\text{Crystals\ of\ }\widetilde{\underline{\calE}}_{\mathrm{strict}}(X/\frakS)_{/\underline{\calI}_{\mathrm{strict}}} \end{Bmatrix}
\right )
$$
\end{proposition}

\begin{proof}
Note that, for any local section $s\in \Gamma \left (U,\calI_X \right )$ over an étale $X$-scheme $U,$ we clearly have the commutative diagram
$$
\begin{tikzcd}
\widetilde{\calE}(X'/\frakS)_{/\calI'} \ar[swap]{d}{\alpha_s'^{-1}} \ar{r}{R} & \widetilde{\calE}_{\mathrm{strict}}(X'/\frakS)_{/\calI'_{\mathrm{strict}}} \ar{d}{\left (\alpha_s'^{\mathrm{strict}} \right )^{-1}} \\
\widetilde{\calE}(U' /\frakS) \ar{r}{R_U}  & \widetilde{\calE}_{\mathrm{strict}}(U'/\frakS),
\end{tikzcd}
$$
where the horizontal arrows are the restriction functors and the vertical ones are the inverse image functors of the morphisms of topoi \eqref{alphakraz} and \eqref{alphakrazShiho}. It follows that, if $\F$ is a crystal of $\widetilde{\calE}(X'/\frakS)_{/\calI'},$ then
$$
R(\F)_s=R_U(\F_s).
$$
The module $R(\F)_s$ is then a crystal for all local sections $s$ and hence $R(\F)$ is a crystal. The functor $R$ then induces
$$
R:\begin{Bmatrix}\text{Crystals\ of\ }\widetilde{\calE}(X'/\frakS)_{/\calI'} \end{Bmatrix} \ra \begin{Bmatrix}\text{Crystals\ of\ }\widetilde{\calE}_{\mathrm{strict}}(X'/\frakS)_{/\calI'_{\mathrm{strict}}} \end{Bmatrix}.
$$
The fact that it is an equivalence follows from (\cite{SF2} 8.26). The proof for $\underline{\calE}(X/\frakS)$ is similar.
\end{proof}

\begin{parag}\label{AcrysShiho}
Recall that, for a logarithmic scheme $T,$ we set $\upmu_T=\ov{\calM}_T^{gp}.$ For an object $(U,\frakT,u)$ of $\underline{\calE}_{\mathrm{strict}}(X/\frakS),$ we identify the étale sites on $T$ and $\underline{T}$ via the canonical nilpotent immersion $\underline{T} \ra T.$ The strict morphism $u:\underline{T} \ra U$ and the canonical immersion $\underline{T} \ra T,$ which is strict, induce an isomorphism
$$
u^{-1}\upmu_U=u^{-1}\ov{\calM}_U^{gp} \xrightarrow{\sim} \ov{\calM}_{\underline{T}}^{gp} \xrightarrow{\sim} \ov{\calM}_T^{gp}=\upmu_T.
$$
We then have a morphism of ringed topoi, abusively denoted
$$
u:\left (T_{\text{ét}/\upmu_T},\Ox_{T,\upmu_T} \right ) \xrightarrow{\sim} \left (\underline{T}_{\text{ét}/\upmu_{T}},\Ox_{T,\upmu_T} \right ) \ra \left (U_{\text{ét}/\upmu_U},\left (u_*\Ox_T\right )_{\upmu_U} \right ).
$$
We set
\begin{equation}\label{Pandaeq2021}
\frakA_{X,(U,\frakT,u)}=u_*\calA_{T}.
\end{equation}
Let
$$
(f,g):(U_1,\frakT_1,u_1) \ra (U_2,\frakT_2,u_2)
$$
be a morphism in $\underline{\calE}_{\mathrm{strict}}(X/\frakS).$ The strict morphism $f:\frakT_1 \ra \frakT_2$ induces a strict morphism $f_1:T_1 \ra T_2,$ which yields, by (\cite{Lor2000} I 3.1), an isomorphism
$$
f_1^*\calA_{T_2} \xrightarrow{\sim} \calA_{T_1}.
$$
This is equivalent to an isomorphism
\begin{equation}\label{Pandaeq2022}
c_{\frakA_X,(f,g)}:\widetilde{f}_{\upmu}^*\frakA_{X,(U_2,\frakT_2,u_2)} \xrightarrow{\sim} \frakA_{X,(U_1,\frakT_1,u_1)},
\end{equation}
where $\widetilde{f}_{\upmu}$ is given in \eqref{gmu}. The data \eqref{Pandaeq2021} and \eqref{Pandaeq2022} satisfy the conditions of \ref{remShiho1920} and hence are the descent data of a module $\frakA_X$ of $\widetilde{\underline{\calE}}_{\mathrm{strict}}(X/\frakS)_{/\underline{\calI}_{\mathrm{strict}}}.$ Since \eqref{Pandaeq2022} is an isomorphism, $\frakA_X$ is a crystal. By \ref{equivShihoind}, it extends to a crystal $\frakA_X$ of $\widetilde{\underline{\calE}}(X/\frakS)_{/\underline{\calI}}.$
\end{parag}

\begin{proposition}\label{propcrystens}
Let $\F$ and $\calG$ be $\Ox_{\calI'}$-modules of $\widetilde{\calE}(X'/\frakS)_{/\calI'}$ (resp. an $\Ox_{\underline{\calI}}$-modules of $\widetilde{\underline{\calE}}(X/\frakS)_{/\underline{\calI}}$) \eqref{kraz1911}.
\begin{enumerate}
\item If $\F$ and $\calG$ are both quasi-coherent modules or both crystals, then so is $\F \otimes_{\Ox_{\calI'}}\calG$ (resp. $\F\otimes_{\Ox_{\underline{\calI}}}\calG$).
\item Let $h_X:\calJ_X \ra \calI_X$ be a morphism of sheaves of $X_{\text{ét}},$ $h':\calJ':=u_X'^{-1}\calJ_X \ra \calI'$ (resp. $\underline{h}:\underline{\calJ}:=\underline{u}_X^{-1}\calJ_X \ra \underline{\calI}$) the morphism induced by $h_X$ and
$$
h':\widetilde{\calE}(X'/\frakS)_{/\calJ'} \ra \widetilde{\calE}(X'/\frakS)_{/\calI'}\ \left (\widetilde{\underline{\calE}}(X/\frakS)_{/\underline{\calJ}} \ra \widetilde{\underline{\calE}}(X/\frakS)_{/\underline{\calI}} \right )
$$
the morphism of ringed topoi induced by $h'$ (resp. $\underline{h}$). If $\F$ is a quasi-coherent module or a crystal of $\widetilde{\calE}(X'/\frakS)_{/\calI'}$ (resp. $\widetilde{\underline{\calE}}(X/\frakS)_{/\underline{\calI}}$), then so is $h'^*\F$ (resp. $\underline{h}^*\F$).
\end{enumerate}
\end{proposition}

\begin{proof}
We only prove the case of $\widetilde{\calE}(X'/\frakS)_{/\calI'}$ as the other is similar. Let $(\F_s)$ and $(\calG_s)$ be descent data corresponding to $\F$ and $\calG$ respectively \eqref{inddescentdata}. Then, for every local section $s\in \Gamma(U,\calI_X)$ over an étale $X'$-scheme $U,$ we have a canonical isomorphism
$$
\left (\F \otimes_{\Ox_{\calI'}} \calG \right )_s=\alpha_s'^*\left (\F \otimes_{\Ox_{\calI'}}\calG \right )\xrightarrow{\sim} \left (\alpha_s'^*\F \right ) \otimes_{\Ox_{\calE(U/\frakS)}} \left (\alpha_s'^*\calG \right )=\F_s \otimes_{\Ox_{\calE(U/\frakS)}}  \calG_s.
$$
By the commutativity of the diagram
$$
\begin{tikzcd}
\widetilde{\calE}(U/\frakS) \ar{r}{\alpha_s'} \ar[swap]{dr}{\alpha_{h_X(s)}'} & \widetilde{\calE}(X'/\frakS)_{/\calJ'} \ar{d}{h'} \\
 & \widetilde{\calE}(X'/\frakS)_{\calI'}, 
\end{tikzcd}
$$
we have
$$
(h'^*\F)_s=\F_{h_X(s)}.
$$
The result follows.
\end{proof}

\begin{proposition}\label{propcrystensstar}
Let $\calA$ be a $\calI'$-indexed (resp. $\underline{\calI}$-indexed) algebra of $\widetilde{\calE}(X'/\frakS)_{/\calI'}$ (resp. of $\widetilde{\underline{\calE}}(X/\frakS)_{/\underline{\calI}}$) and $\F$ and $\calG$ two $\calI'$-indexed (resp. $\underline{\calI}$-indexed) $\calA$-modules \eqref{inddef}. If $\calA,$ $\F$ and $\calG$ are crystals then so is $\F \circledast_{\calA} \calG.$
\end{proposition}

\begin{proof}
We only prove the case of $\widetilde{\calE}(X'/\frakS)_{/\calI'}$ as the other is similar. Let $s\in \Gamma(U,\calI_X)$ be a section over an étale $X$-scheme $U.$ We have to prove that $\left (\F \circledast_{\calA} \calG\right )_s=\alpha_s'^*\left (\F \circledast_{\calA} \calG\right )$ is a crystal of $\widetilde{\calE}(U'/\frakS).$ Let $(f,g):(U_1,\frakT_1,u_1) \ra (U_2,\frakT_2,u_2)$ be a morphism of $\calE(U'/\frakS).$ We have to prove that the morphism
\begin{equation}\label{eqkraz1923}
c_{\alpha_s'^* \left (\F \circledast_{\calA} \calG \right )}:\widetilde{f}^* \left ( \left (\alpha_s'^* \left (\F \circledast_{\calA} \calG\right )\right )_{(U_2,\frakT_2,u_2)}\right ) \ra \left (\alpha_s'^*\left (\F \circledast_{\calA} \calG\right ) \right ) _{(U_1,\frakT_1,u_1)} 
\end{equation}
is an isomorphism, where
$$
\widetilde{f}:\left (U_{1,\text{ét}},u_{1*}\Ox_{T_{1}} \right ) \ra \left ( U_{2,\text{ét}}, u_{2*}\Ox_{T_{2}} \right )
$$
is given in \ref{paragIran4}. Recall the morphism of topoi
$$
\widetilde{f}_{\calI}: \left (U_{1,\text{ét}/\calI_{U_1}},\left ( u_{1*}\Ox_{T_1} \right )_{\calI_{U_1}} \right ) \ra \left (U_{2,\text{ét}/\calI_{U_2}},  \left (u_{2*}\Ox_{T_2} \right )_{\calI_{U_2}}\right )
$$
is given in \eqref{gmu}.
Consider the morphism of ringed topoi
$$
\Lambda_i:U_{i,\text{ét}} \ra U_{i,\text{ét}/\calI_{U_i}},
$$
induced by the continuous and cocontinuous functor
$$
\text{ét}_{/U_i} \ra \text{ét}_{/U_{i,\text{ét}}\times \calI_{U_i}},\ V \mapsto (V,s_{|V}).
$$
By \eqref{diag188}, the morphism \eqref{eqkraz1923} becomes
$$
\widetilde{f}^*\left (\Lambda_2^*\delta_{(U_2,\frakT_2,u_2),\calI_X}^*\left (\F \circledast_{\calA} \calG \right ) \right ) \ra \Lambda_1^* \delta_{(U_1,\frakT_1,u_1),\calI_X}^* \left (\F \circledast_{\calA} \calG\right ).
$$
By \ref{propKoko1915}, this is isomorphic to
$$
\widetilde{f}^*\left (\Lambda_2^*\left (\F_{(U_2,\frakT_2,u_2)} \circledast_{\calA_{(U_2,\frakT_2,u_2)}} \calG_{(U_2,\frakT_2,u_2)} \right ) \right ) \ra \Lambda_1^* \left (\F_{(U_1,\frakT_1,u_1)} \circledast_{\calA_{(U_1,\frakT_1,u_1)}} \calG_{(U_1,\frakT_1,u_1)}\right ).
$$
By the commutativity of the diagram
$$
\begin{tikzcd}
U_{1,\text{ét}} \ar{rr}{\widetilde{f}} \ar[swap]{d}{\Lambda_1}  & & U_{2,\text{ét}} \ar{d}{\Lambda_2} \\
U_{1,\text{ét}/\calI_{U_1}} \ar{rr}{\widetilde{f}_{\calI}} & & U_{2,\text{ét}/\calI_{U_2}},
\end{tikzcd}
$$
it is isomorphic to
$$
\Lambda_1^*\widetilde{f}_{\calI}^* \left (\F_{(U_2,\frakT_2,u_2)} \circledast_{\calA_{(U_2,\frakT_2,u_2)}} \calG_{(U_2,\frakT_2,u_2)} \right )  \ra \Lambda_1^* \left (\F_{(U_1,\frakT_1,u_1)} \circledast_{\calA_{(U_1,\frakT_1,u_1)}} \calG_{(U_1,\frakT_1,u_1)}\right ).
$$
By \ref{Koko829}, it is isomorphic to
$$
\Lambda_1^*\widetilde{f}_{\calI}^*\F_{(U_2,\frakT_2,u_2)} \circledast_{\Lambda_1^*\widetilde{f}_{\calI}^*\calA_{(U_2,\frakT_2,u_2)}} \Lambda_1^*\widetilde{f}_{\calI}^*\calG_{(U_2,\frakT_2,u_2)}   \ra \Lambda_1^* \F_{(U_1,\frakT_1,u_1)} \circledast_{\Lambda_1^*\calA_{(U_1,\frakT_1,u_1)}} \Lambda_1^*\calG_{(U_1,\frakT_1,u_1)}.
$$
Since $\calA,$ $\F$ and $\calG$ are crystals, we have isomorphisms
\begin{alignat*}{2}
\Lambda_1^*\widetilde{f}_{\calI}^*\calA_{(U_2,\frakT_2,u_2)} & \xrightarrow{\sim} \Lambda_1^* \calA_{(U_1,\frakT_1,u_1)},\\
\Lambda_1^*\widetilde{f}_{\calI}^*\F_{(U_2,\frakT_2,u_2)} & \xrightarrow{\sim} \Lambda_1^* \F_{(U_1,\frakT_1,u_1)},\\
\Lambda_1^*\widetilde{f}_{\calI}^*\calG_{(U_2,\frakT_2,u_2)} & \xrightarrow{\sim} \Lambda_1^* \calG_{(U_1,\frakT_1,u_1)}.
\end{alignat*}
The result follows.
\end{proof}

\begin{parag}\label{paragEST1}
In the rest of this section, we take $\calI_X=\upmu_X=\ov{\calM}_X^{gp}.$ Let $V$ be an étale $X$-scheme.
We identify the small étale sites of $V$ and $V'$ via the exact relative Frobenius $V\ra V'$ and hence consider $\upmu_V=\ov{\calM}_V^{gp}$ as a sheaf of $V'_{\text{ét}}.$ We set
$$\upmu'_V=u_V'^{-1}\upmu_V=\upmu_V\circ u_V',\ \underline{\upmu}_V=\underline{u}_V^{-1}\upmu_V=\upmu_V\circ \underline{u}_V.$$
For $V=X,$ we denote $\upmu'_V$ (resp. $\underline{\upmu}_V$) simply by $\upmu'$ (resp. $\underline{\upmu}$).
Let $n$ be a positive integer, $U$ an étale $X$-scheme. Since $u_X'^{-1}$ and $\underline{u}_X^{-1}$ commute with finite products, we have
$$
\upmu'^n=u_X'^{-1}\left (\upmu_X^n\right )=\upmu_X^n\circ u_X',\ \underline{\upmu}^n=\underline{u}_X^{-1}\left (\upmu_X^n\right )=\upmu_X^n \circ \underline{u}_X.
$$
We easily prove that the morphism of topoi \eqref{eqrhoEST1} (resp. \eqref{eqrhoEST2})
$$
\underline{\rho_{X,V}}:\widetilde{\underline{\calE}}(V/\frakS) \ra \widetilde{\underline{\calE}}(X/\frakS)\ \left (\text{resp.}\ \rho_{X',V'}:\widetilde{\calE}(V'/\frakS) \ra \widetilde{\calE}(X'/\frakS)\right )
$$
induces a morphism of topoi
\begin{equation}\label{eqREST}
\underline{R_{X,V}}:\widetilde{\underline{\calE}}(V/\frakS)_{/\underline{\upmu}_V} \ra \widetilde{\underline{\calE}}(X/\frakS)_{/\underline{\upmu}}\ \left (\text{resp.}\ R_{X',V'}:\widetilde{\calE}(V'/\frakS)_{/\upmu'_V} \ra \widetilde{\calE}(X'/\frakS)_{\upmu'}\right ).
\end{equation}
For a local section $s\in \Gamma \left ( W,\upmu_X \right )$ over an étale $V$-scheme $W,$ we clearly have the commutative diagram
$$
\begin{tikzcd}
\widetilde{\underline{\calE}}(W/\frakS) \ar{drr} \ar{d} & & \\
\widetilde{\underline{\calE}}(V/\frakS)_{\underline{\upmu}_V} \ar[swap]{rr}{\underline{R_{X,V}}} & & \widetilde{\underline{\calE}}(X/\frakS)_{/\underline{\upmu}},
\end{tikzcd}
$$
where the vertical and oblique arrows are given in \eqref{alphakraz}. By a similar diagram for $R_{X',V'},$ we deduce that the inverse image functors $\underline{R_{X,V}}^{-1}$ and $R_{X',V'}^{-1}$ preserve crystals.
By \eqref{stackequiv1} and \eqref{stackequiv2}, the morphisms $\underline{R_{X,V}}$ and $R_{X',V'}$ identify with the localization morphisms
\begin{equation}\label{eqRlocEST}
\underline{R_{X,V}}:\widetilde{\underline{\calE}}(X/\frakS)_{/\underline{u}_X^{-1}\left (V^a\right ) \times \underline{\upmu}} \ra \widetilde{\underline{\calE}}(X/\frakS)_{/\underline{\upmu}},\ R_{X',V'}:\widetilde{\calE}(V'/\frakS)_{/u_X'^{-1}\left (V'^a\right ) \times \upmu'} \ra \widetilde{\calE}(X'/\frakS)_{\upmu'}.
\end{equation}
It follows, by \ref{Locprop1}, that the inverse image functors $\underline{R_{X,V}}^{-1}$ and $R_{X',V'}^{-1}$ preserve structures of indexed algebras and indexed modules.
\end{parag}

\begin{parag}\label{paragEST2}
Consider the split fibered $\widetilde{\calE}(X'/\frakS)_{/\upmu'}$-category $\mathbbm{E}'$ (resp. $\widetilde{\underline{\calE}}(X/\frakS)_{/\underline{\upmu}}$-category $\underline{\mathbbm{E}}$) defined, for every sheaf $\F$ of $\widetilde{\calE}(X'/\frakS)_{/\upmu'}$ (resp. $\widetilde{\underline{\calE}}(X/\frakS)_{/\underline{\upmu}}$), by the localized category $\left (\widetilde{\calE}(X'/\frakS)_{/\upmu'} \right )_{/\F}$ (resp. $\left (\widetilde{\underline{\calE}}(X/\frakS)_{/\underline{\upmu}}\right )_{/\F}$) and, for every morphism $\varphi:\F \ra \calG,$ by the inverse image functor of the morphism of topoi
\begin{equation}
\left (\widetilde{\calE}(X'/\frakS)_{/\upmu'} \right )_{/\F} \ra \left (\widetilde{\calE}(X'/\frakS)_{/\upmu'} \right )_{/\calG} \left( \text{resp.}\ \left (\widetilde{\underline{\calE}}(X/\frakS)_{\underline{\upmu}}\right )_{/\F} \ra \left (\widetilde{\underline{\calE}}(X/\frakS)_{/\underline{\upmu}}\right )_{/\calG}\right )
\end{equation}
induced by $\varphi.$ By (\cite{Giraud} II 3.4.4), $\mathbbm{E}'$ (resp. $\underline{\mathbbm{E}}$) is a stack.
\end{parag}

\begin{parag}
Recall the functor $\rho:\underline{\calE}(X/\frakS) \ra \calE(X'/\frakS)$ \eqref{era3rhoShiho} and the functors $u_X'$ and $\underline{u}_X$ defined in \ref{paragu}. By definition, these functors fit into the commutative diagram
$$
\begin{tikzcd}
\underline{\calE}(X/\frakS) \ar[swap]{d}{\underline{u}_X} \ar{r}{\rho} & \calE(X'/\frakS) \ar{d}{u_X'} \\
\text{ét}_{/X} \ar{r}{\sim} & \text{ét}_{/X'},
\end{tikzcd}
$$
where the lower equivalence of sites is induced by the exact relative Frobenius $F_1:X\ra X'.$
Since these functors are continuous and cocontinuous, we deduce that the diagram
\begin{equation}\label{diag14151}
\begin{tikzcd}
\widetilde{\underline{\calE}}(X/\frakS) \ar{rr}{C_{X/\frakS}} \ar[swap]{d}{\underline{u}_X} & & \widetilde{\calE}(X'/\frakS) \ar{d}{u_X'} \\
X_{\text{ét}} \ar{rr}{\sim} & & X'_{\text{ét}},
\end{tikzcd}
\end{equation}
where $C_{X/\frakS}$ is given in \eqref{era3C} and $u_X'$ and $\underline{u}_X$ are defined in \eqref{Kokou}, is commutative up to a canonical isomorphism.
Recall that $\upmu'=u_X'^{-1}F_{1*}\upmu_X.$ Then
$$C_{X/\frakS}^{-1}\upmu'=\underline{u}_X^{-1}F_1^{-1}F_{1*}\upmu_X=\underline{u}_X^{-1}\upmu_X=\underline{\upmu}.$$
Therefore \eqref{diag14151} induces a commutative diagram
\begin{equation}\label{diagkraz2}
\begin{tikzcd}
\widetilde{\underline{\calE}}(X/\frakS)_{/\underline{\upmu}} \ar{rr}{C_{X/\frakS,\upmu}} \ar{d}{\underline{u}_{X,\underline{\upmu}}} & & \widetilde{\calE}(X'/\frakS)_{/\upmu'} \ar{d}{u'_{X,\upmu'}} \\
X_{\text{ét}/\upmu_X} \ar{rr}{\sim} & & X'_{\text{ét}/F_{1*}\upmu_X}.
\end{tikzcd}
\end{equation}
The continuous and cocontinuous functor $\rho$ also induces a continuous and cocontinuous functor
$$
\rho_{\upmu}:\underline{\calE}(X/\frakS)_{/\underline{\upmu}} \ra \calE(X'/\frakS)_{/\upmu'}.
$$
The inverse image functor $C_{X/\frakS,\upmu}^{-1}$ is then given by composition with $\rho_{\upmu}.$
For every local section $s$ of $\upmu_X$ over an étale $X$-scheme $U,$ we have a bigger commutative diagram:
\begin{equation}\label{diagkraz19253}
\begin{tikzcd}
\widetilde{\underline{\calE}}(U/\frakS) \ar{dd}{\underline{u}_X} \ar{rr}{C_{U/\frakS}} \ar{dr}{\underline{\alpha}_s} & & \widetilde{\calE}(U'/\frakS) \ar[dashed]{dd}[pos=0.20]{u_X'} \ar{dr}{\alpha_s'} &  \\
 & \widetilde{\underline{\calE}}(X/\frakS)_{/\underline{\upmu}} \ar{rr}[pos=0.20]{C_{X/\frakS,\upmu}} \ar{dd}[pos=0.20]{\underline{u}_{X,\underline{\upmu}}} & & \widetilde{\calE}(X'/\frakS)_{/\upmu'} \ar{dd}{u'_{X,\upmu'}} \\
U_{\text{ét}} \ar{dr} \ar[dashed]{rr}[pos=0.20]{\sim} & & U'_{\text{ét}} \ar[dashed]{dr} &  \\
 & X_{\text{ét}/\upmu_X} \ar{rr}{\sim} & & X'_{\text{ét}/\upmu_X},
\end{tikzcd}
\end{equation}
where $U_{\text{ét}} \ra X_{\text{ét}/\upmu_X}$ and $U'_{\text{ét}} \ra X'_{\text{ét}/\upmu_X}$ are the localization morphisms corresponding to
$$
s\in \Gamma(U,\upmu_X)=\Gamma(U',F_{1*}\upmu_X).
$$
\end{parag}

\begin{proposition}\label{commCXU}
Let $U$ be an étale $X$-scheme. Then, the diagram
$$
\begin{tikzcd}
\widetilde{\underline{\calE}}(U/\frakS)_{/\underline{\upmu}_U} \ar{d} \ar{rr}{C_{U/\frakS,\upmu}} & & \widetilde{\calE}(U'/\frakS)_{/\upmu'_U} \ar{d} \\
\widetilde{\underline{\calE}}(X/\frakS)_{/\underline{\upmu}} \ar[swap]{rr}{C_{X/\frakS,\upmu}} & & \widetilde{\calE}(X'/\frakS)_{/\upmu'},
\end{tikzcd}
$$
where $\upmu'_U$ and $\underline{\upmu}_U$ are defined in \ref{paragEST1}, the horizontal arrows are defined in \eqref{diagkraz2} and the vertical arrows are \eqref{eqREST}, is commutative.
\end{proposition}

\begin{proof}
This follows from the commutativity of
$$
\begin{tikzcd}
\widetilde{\underline{\calE}}(U/\frakS) \ar[swap]{d}{\underline{\rho_{X,U}}} \ar{rr}{C_{U/\frakS}} & & \widetilde{\calE}(U'/\frakS) \ar{d}{\rho_{X',U'}} \\
\widetilde{\underline{\calE}}(X/\frakS) \ar[swap]{rr}{C_{X/\frakS}} & & \widetilde{\calE}(X'/\frakS).
\end{tikzcd}
$$
\end{proof}

\begin{proposition}\label{ommek4}
Let $\frakE$ be a module of $\widetilde{\calE}(X'/\frakS)_{/\upmu'}.$ For any object $(U,\frakT,u)$ of $\underline{\calE}(X/\frakS),$ there exists a canonical isomorphism
$$
\left (C_{X/\frakS,\upmu}^{-1} \frakE \right )_{(U,\frakT,u)}\xrightarrow{\sim} \frakE_{\rho(U,\frakT,u)}.
$$
\end{proposition}

\begin{proof}
This follows from the commutativity of
$$
\begin{tikzcd}
\text{ét}_{/U\times \upmu_X} \ar{rr}{\delta_{(U,\frakT,u)}} \ar[swap]{d}{(V,s) \mapsto (V',s)} & & \underline{\calE}(X/\frakS)_{/\underline{\upmu}} \ar{d}{\rho_{\upmu}} \\
\text{ét}_{/U'\times \upmu_X} \ar[swap]{rr}{\delta_{\rho(U,\frakT,u)}} & & \calE(X'/\frakS)_{/\upmu'}.
\end{tikzcd}
$$
\end{proof}

\section{Indexed logarithmic Cartier transform}

\begin{parag}
In this section, we fix a perfect field $k$ of positive characteristic $p$ and denote by $W(k)$ its ring of Witt vectors. We equip $\op{Spf}W(k)$ with the trivial logarithmic structure. If a gothic letter $\frakT$ denotes a logarithmic $p$-adic formal scheme (\cite{SF1} 7), we denote by $\frakT_n,$ for any positive integer $n,$ the logarithmic scheme obtained from $\frakT$ by reduction modulo $p^n$ (\cite{SF1} 7.8). The corresponding roman letter $T$ will denote the logarithmic scheme $\frakT_1.$

Let $\frakS$ be a logarithmic formal scheme, log flat and locally of finite type over $\op{Spf}W(k),$ and $X\ra S$ a log smooth morphism of logarithmic schemes, which is of Cartier type.

For an étale $X$-scheme $U,$ we equip $U$ with the logarithmic structure $\left (\calM_{X}\right )_{|U},$ restriction of $\calM_X$ on $U,$ so that $U\ra X$ is strict.
For every logarithmic scheme $T,$ we denote by $\upmu_T$ the sheaf $\ov{\calM}_T^{gp}.$ We consider the monoids $\upmu_X,$ $\upmu'$ and $\underline{\upmu}$ defined in \ref{paragEST1}.

We will sometimes need to assume one of following two hypothesis: 
\begin{itemize}
\item[$(H_1)$] The morphism $X\ra S$ lifts to a log smooth morphism $f:(\frakX,Q) \ra (\frakS,P)$ of framed logarithmic formal schemes and $X'$ lifts to a framed logarithmic formal scheme $(\frakX',Q')$ over $(\frakS,P),$ such that $\frakX' \ra \frakS$ is log smooth.
\item[$(H_2)$] The hypothesis $(H_1)$ is satisfied and the exact relative Frobenius $F_1:X\ra X'$ lifts to an $(\frakS,P)$-morphism of framed logarithmic formal schemes $F:(\frakX,Q) \ra (\frakX',Q').$
\end{itemize}
When we suppose that $(H_1)$ is satisfied, we consider the logarithmic formal schemes $R_{\frakX',1}$ and $Q_{\frakX}$ \eqref{paragIran5} and we denote by $R_1'$ and $Q_1$ their respective special fibers and by $\calR_1'$ and $\calQ_1$ the corresponding Hopf algebras. Also denote by $\lambda_Q:\underline{Q_1} \ra X$ and $\lambda_{R'}:R_1' \ra X'$ the morphisms given in (\cite{SF2} 6.5) and (\cite{SF1} 11.4). When we suppose that $(H_2)$ is satisfied, we consider the morphism of Hopf algebras $\calV:\calR_1' \ra \calQ_1$ induced by $\nu:Q_{\frakX}\ra R_{\frakX',1}$ \eqref{paragIran6}.
We start by studying the crystals in the indexed logarithmic Oyama topoi, introduced in the previous section, and then construct an indexed version of the Cartier transform.
\end{parag}

\begin{proposition}\label{equivRQind}
Assume $(H_1)$ is satisfied. There exists a canonical equivalence of categories between the category of crystals of $\widetilde{\underline{\calE}}(X/\frakS)_{/\underline{\upmu}}$ (resp. $\widetilde{\calE}(X'/\frakS)_{/\upmu'}$) and the category of $\Ox_{X,\upmu_X}$-modules equipped with a $\calQ_{1,\upmu_X}$-stratification (resp. $\Ox_{X',\upmu_X}$-modules equipped with an $\calR_{1,\upmu_X}'$-stratification). In addition, these equivalences preserve structures of $\upmu_X$-indexed algebras and $\upmu_X$-indexed modules.
\end{proposition}

\begin{proof}
A crystal $\mathfrak{E}$ of $\widetilde{\underline{\calE}}(X/\frakS)_{/\underline{\upmu}}$ is equivalent to its descent data $(\mathfrak{E}_s)$ (\ref{inddescentdata}). By definition, for any étale $X$-scheme $U$ and $s\in \Gamma(U,\upmu_X),$ the $\Ox_{\underline{\calE}(U/\frakS)}$-module $\mathfrak{E}_s$ is a crystal of $\widetilde{\underline{\calE}}(U/\frakS).$ Let $\frakU$ be the unique étale formal $\frakX$-scheme such that $U=\frakU\times_{\frakX}X.$ Consider the logarithmic formal scheme $Q_{\frakU}$ obtained from $\frakU \ra \frakS$ just as $Q_{\frakX}$ is obtained from $\frakX \ra \frakS.$ By \ref{equivRQ}, $\mathfrak{E}_s$ corresponds to an $\Ox_U$-module $\F_s$ equipped with a $Q_{\frakU}$-stratification $\epsilon_s.$ Since $\calQ_{\frakU}=\calQ_{\frakX|\frakU}$ then, by \ref{indremark73}, the objects $(\F_s,\epsilon_s)$ correspond to an $\Ox_{X,\upmu_X}$-module $\F$ equipped with a $\calQ_1$-stratification $\epsilon.$ This proves the equivalence.

Suppose now that $\mathfrak{E}$ has a structure of a $\upmu_X$-indexed algebra given by a morphism
$$\pi:\mathfrak{E}\boxtimes \mathfrak{E} \ra \sigma^*\mathfrak{E},$$
where $\sigma:\upmu^2 \ra \upmu$ is the addition map. By \ref{propcrystens}, $\mathfrak{E}\boxtimes \mathfrak{E}$ and $\sigma^*\mathfrak{E}$ are crystals. For every local sections $s,t$ of $\upmu_X$ over an étale $X$-scheme $U,$ by applying $\underline{\alpha}_{(s,t)}^{-1}$ and \ref{prop1916Ko}, we obtain a morphism of crystals
$$\pi_{(s,t)}:\mathfrak{E}_s\otimes_{\Ox_{\underline{\calE}(U/\frakS)}}\mathfrak{E}_t \ra \mathfrak{E}_{s+t}.$$
This corresponds, by \ref{equivRQ}, to a morphism
$$\F_s \otimes_{\Ox_U}\F_t \ra \F_{s+t}$$
of $\Ox_U$-modules with $Q_{\frakU}$-stratification. These morphisms correspond, by \ref{indremark73}, to a morphism
$$\F\boxtimes \F\ra \sigma_X^*\F,$$
where $\sigma_X:\upmu_X^2 \ra \upmu_X$ is the addition map. This defines a $\upmu_X$-indexed algebra structure on $\F.$
The proof for $\widetilde{\calE}(X'/\frakS)_{/\upmu'}$ and for indexed modules is similar.
\end{proof}

\begin{parag}\label{parag1419}
Assume $(H_1)$ is satisfied. We have canonical projections $q_1,q_2:(X,Q_{\frakX},\lambda_Q) \ra (X,\frakX,i_X)$ of $\underline{\calE}(X/\frakS),$ where $i_X:\underline{X} \ra X$ is the canonical immersion, and $q_1',q_2':(X',R_{\frakX',1},\lambda_{R'}) \ra (X',\frakX',\op{Id}_{X'})$ of $\calE'(X/\frakS).$
Let $\F$ be a module of $\widetilde{\calE}(X'/\frakS)_{/\upmu'}.$ The morphism $c_{\F,q_i'}$ \eqref{cind} is equal to
$$
c_{\F,q_i'}:\left (\lambda_{R'*}\Ox_{R_1'} \right )_{\upmu_X} \otimes_{\Ox_{X',\upmu_X}} \F_{(X',\frakX',\op{Id}_{X'})} \ra \F_{(X',R_{\frakX',1},\lambda_{R'})},
$$
where $\Ox_{X'} \ra \lambda_{R'*}\Ox_{R_1'}$ is induced by $q_i'.$ Similarly, if $\F$ is a module of $\widetilde{\underline{\calE}}(X/\frakS)_{/\underline{\upmu}}$ the morphism $c_{\F,q_i}$ is equal to
$$
c_{\F,q_i}:\left (\lambda_{Q*}\Ox_{Q_1} \right )_{\upmu_X} \otimes_{\Ox_{X,\upmu_X}} \F_{(X,\frakX,i_{X})} \ra \F_{(X,Q_{\frakX},\lambda_{Q})}.
$$
Recall that $\calR_1'=\lambda_{R'*}\Ox_{R_1'}$ and $\calQ_1=\lambda_{Q*}\Ox_{Q_1}.$
By commutativity of \eqref{diag188}, the equivalence given in \ref{equivRQind} can be explicitly given as follows:
To a crystal $\mathfrak{E}$ of $\widetilde{\underline{\calE}}(X/\frakS)_{/\underline{\upmu}}$ (resp. $\widetilde{\calE}(X'/\frakS)_{/\upmu'}$), corresponds the $\Ox_{X,\upmu_X}$-module $\mathfrak{E}_{(X,\frakX,i_X)}$ (resp. the $\Ox_{X',\upmu_X}$-module $\mathfrak{E}_{(X',\frakX',\op{Id}_{X'})}$) equipped with the $\calQ_{1,\upmu_X}$-stratification
$$\epsilon:\calQ_{1,\upmu_X} \otimes_{\Ox_{X,\upmu_X}} \mathfrak{E}_{(X,\frakX,i_X)} \xrightarrow{c_{\mathfrak{E},q_2}} \mathfrak{E}_{(X,Q_{\frakX},\lambda_Q)} \xrightarrow{c_{\mathfrak{E},q_1}^{-1}} \mathfrak{E}_{(X,\frakX,i_X)} \otimes_{\Ox_{X,\upmu_X}}\calQ_{1,\upmu_X}.$$
(resp. the $\calR_{1,\upmu_X}'$-stratification
$$\epsilon':\calR'_{1,\upmu_X} \otimes_{\Ox_{X',\upmu_X}} \mathfrak{E}_{(X',\frakX',\op{Id}_{X'})} \xrightarrow{c_{\mathfrak{E},q_2'}} \mathfrak{E}_{(X',R_{\frakX',1},\lambda_{R'})} \xrightarrow{c_{\mathfrak{E},q_1'}^{-1}} \mathfrak{E}_{(X',\frakX',\op{Id}_{X'})})\otimes_{\Ox_{X',\upmu_X}}\calR'_{1,\upmu_X})$$
\end{parag}

\begin{proposition}\label{propEST1831}
Assume $(H_1)$ is satisfied. Let $U$ be an étale $X$-scheme. The diagrams
$$
\begin{tikzcd}
\begin{Bmatrix}
\mathrm{Crystals\ of\ }\widetilde{\underline{\calE}}(X/\frakS)_{/\underline{\upmu}}
\end{Bmatrix}
\ar{r}{\sim} \ar{d} &
\begin{Bmatrix}
\Ox_{X,\upmu_X}\text{-}\mathrm{modules\ equipped} \\
\mathrm{with\ a\ }\calQ_{1,\upmu_X}\text{-}\mathrm{stratification}
\end{Bmatrix}
 \ar{d} \\
\begin{Bmatrix}
\mathrm{Crystals\ of\ }\widetilde{\underline{\calE}}(U/\frakS)_{/\underline{\upmu}_U}
\end{Bmatrix} 
\ar{r}{\sim} &
\begin{Bmatrix}
\Ox_{U,\upmu_U}\text{-}\mathrm{modules\ equipped} \\
\mathrm{with\ a\ }\calQ_{1|U,\upmu_U}\text{-}\mathrm{stratification}
\end{Bmatrix}
\end{tikzcd}
$$
and
$$
\begin{tikzcd}
\begin{Bmatrix}
\mathrm{Crystals\ of\ }\widetilde{\calE}(X'/\frakS)_{/\upmu'}
\end{Bmatrix}
\ar{r}{\sim} \ar{d} &
\begin{Bmatrix}
\Ox_{X',\upmu_X}\text{-}\mathrm{modules\ equipped} \\
\mathrm{with\ a\ }\calR'_{1,\upmu_X}\text{-}\mathrm{stratification}
\end{Bmatrix}
 \ar{d} \\
\begin{Bmatrix}
\mathrm{Crystals\ of\ }\widetilde{\calE}(U'/\frakS)_{/\upmu'_U}
\end{Bmatrix}
\ar{r}{\sim} &
\begin{Bmatrix}
\Ox_{U',\upmu_U}\text{-}\mathrm{modules\ equipped} \\
\mathrm{with\ a\ }\calR'_{1|U',\upmu_U}\text{-}\mathrm{stratification}
\end{Bmatrix},
\end{tikzcd}
$$
where the horizontal arrows are the equivalences given in \ref{equivRQind}, the right vertical arrows are the restriction functors and the left arrows are induced by the inverse image functors of \eqref{eqREST},
are commutative.
\end{proposition}

\begin{proof}
We just check the commutativity of the first diagram. The second is similar. For any local section $s\in \Gamma\left (V,\upmu_X \right )$ over an étale $X$-scheme $V,$ we have the diagram
$$
\begin{tikzcd}
\begin{Bmatrix}
\mathrm{Crystals\ of\ }\widetilde{\underline{\calE}}(X/\frakS)_{/\underline{\upmu}}
\end{Bmatrix}
\ar{r}{\sim} \ar{d} &
\begin{Bmatrix}
\Ox_{X,\upmu_X}\text{-}\mathrm{modules\ equipped} \\
\mathrm{with\ a\ }\calQ_{1,\upmu_X}\text{-}\mathrm{stratification}
\end{Bmatrix}
 \ar{d} \\
\begin{Bmatrix}
\mathrm{Crystals\ of\ }\widetilde{\underline{\calE}}(U/\frakS)_{/\underline{\upmu}_U}
\end{Bmatrix} 
\ar{r}{\sim} \ar[swap]{d}{\alpha_s^{-1}} &
\begin{Bmatrix}
\Ox_{U,\upmu_U}\text{-}\mathrm{modules\ equipped} \\
\mathrm{with\ a\ }\calQ_{1|U,\upmu_U}\text{-}\mathrm{stratification}
\end{Bmatrix} \ar{d}{(\F,\epsilon ) \mapsto (\F_s,\epsilon_s)} \\
\begin{Bmatrix}
\mathrm{Crystals\ of\ }\widetilde{\underline{\calE}}(V/\frakS)
\end{Bmatrix} 
\ar{r}{\sim} &
\begin{Bmatrix}
\Ox_{V}\text{-}\mathrm{modules\ equipped} \\
\mathrm{with\ a\ }\calQ_{1|V}\text{-}\mathrm{stratification}
\end{Bmatrix}.
\end{tikzcd}
$$
The lower and outer squares are commutative by definition of the equivalences \ref{equivRQind}. This implies the commutativity of the upper square.
\end{proof}

\begin{parag}\label{defQstrat}
Here, we give an alternate description of the crystal $\frakA_X$ defined in \ref{AcrysShiho} under the hypothesis $(H_1).$ Assume $(H_1)$ is satisfied. Let $q_1,q_2:Q_{\frakX} \ra \frakX$ and $q_1',q_2':R_{\frakX',1} \ra \frakX'$ be the canonical projections. By definition of $Q_{\frakX},$ the projections $q_1$ and $q_2$ are strict and so they induce, by (\cite{Lor2000} I 3.1 page 273), isomorphisms of $\upmu_X$-indexed algebras
\begin{equation}\label{Khaminei2151}
c_i:q_i^*\calA_X \xrightarrow{\sim} \calA_{Q_1}.
\end{equation}
We get a $\calQ_{1,\upmu_X}$-stratification
\begin{equation}\label{eqeAQ}
\epsilon_{\calA_X,Q}:q_2^*\calA_X \xrightarrow{c_2} \calA_{Q_1} \xrightarrow{c_1^{-1}} q_1^*\calA_X.
\end{equation}
Note that $\epsilon_{\calA_X,Q}$ is a morphism of $\upmu_X$-indexed algebras.
We denote by $\frakA$ the crystal of $\widetilde{\underline{\calE}}(X/\frakS)_{/\underline{\upmu}}$ corresponding to $\left (\calA_X,\epsilon_{\calA_X,Q} \right )$ \eqref{equivRQind}. By \ref{equivRQind}, $\frakA$ is a $\upmu$-indexed algebra. By construction,
\begin{equation}\label{eq14201}
\frakA_{(X,\frakX,i_X)}=\calA_X,\ \frakA_{(X,Q_{\frakX},\lambda_Q)}=\calA_{Q_1}.
\end{equation}
The crystal $\frakA$ is hence equal to the crystal $\frakA_X$ defined in \ref{AcrysShiho}.
\end{parag}

\begin{remark}\label{remkraz1}
By definition, the stratification $\epsilon_{\calA_X,Q}$ is an isomorphism of $\upmu_X$-indexed algebras.
\end{remark}

\begin{parag}\label{paragBglue}
There exist étale coverings $\left ( \frakS_i \ra \frakS \right )_{i\in I}$ and $\left (X_{ij} \ra X_i:=X\times_SS_i \right )_{j\in I_i}$ satisfying:
\begin{itemize}
\item $\frakS_i$ has a chart $\frakS_i \ra B\langle P_i \rangle.$
\item $X_{ij} \ra S_i$ has a chart $A_1[Q_{ij}] \ra A_1[P_i]$ such that the morphism $X_{ij} \ra S_i\times_{A_1[P_i]}A_1[Q_{ij}],$ induced by $X_{ij} \ra S_i$ and the chart $X_{ij} \ra A_1[Q_{ij}],$ is étale and strict and $P_i^{gp} \ra Q_{ij}^{gp}$ is injective and the torsion subgroup of its cokernel has a finite order coprime with $p.$
\end{itemize}
Let $Q_{ij}'$ and $Q_{ij}''$ be the monoids defined from $P_i \ra Q_{ij}$ as in (\cite{SF1} 3.1) and $X_{ij}''$ the fiber product, in the category of fine logarithmic schemes, of $X_{ij}\ra S_i$ and the Frobenius $F_{S_i}:S_i \ra S_i.$ We then have a chart $X_{ij}'' \ra A_1[Q_{ij}''],$ $X_{ij}'=X_{ij}''\times_{A_1[Q_{ij}'']}A_1[Q_{ij}']$ and the canonical morphism $P_i \ra Q_{ij}'$ is a chart of $X_{ij}' \ra S_i.$
For $i\in I$ and $j\in I_i,$ we have seen, in \ref{paragu}, that the functor
\begin{equation}\label{eqShiho1}
\begin{array}[t]{clc}
\calE(X_{ij}'/\frakS) & \ra & \text{ét}_{/X_{ij}'} \\
(U,\frakT,u) & \mapsto & U
\end{array}
\end{equation}
is continuous and cocontinuous and hence induces a morphism of topoi
$
u_{ij}:\widetilde{\calE}(X_{ij}'/\frakS) \ra X'_{{ij},\text{ét}},
$
such that $u_{ij}^{-1}$ is given by composition with \eqref{eqShiho1}. We set
$
\upmu_{ij}'=u_{ij}^{-1} \left (\upmu_{X_{ij}} \right ) = u_{ij}^{-1} \left (\ov{\calM}_{X_{ij}}^{gp} \right ).
$
Similarly, for $i,j\in I,$ $k\in I_i$ and $l\in I_j,$ we have a morphism of topoi
$$
u_{ikjl}:\widetilde{\calE}\left ( X_{ik}'\times_{X'}X_{jl}'/\frakS \right ) \ra \left (X_{ik}'\times_{X'}X_{jl}' \right )_{\text{ét}}
$$
and we set
$
\upmu_{ikjl}'=u_{ikjl}^{-1} \left (\upmu_{X_{ik}\times_XX_{jl}} \right ).
$
By the lemma \ref{lemShiho208} below, the morphism $X_{ij}' \ra S_i\times_{A_1[P_i]}A_1[Q_{ij}']$ is étale and strict and so there exists a unique étale and strict morphism of logarithmic formal schemes $\frakX'_{ij} \ra \frakS_i\times_{B\langle P_i \rangle}B\langle Q_{ij}' \rangle,$ fitting into a cartesian square
$$
\begin{tikzcd}
X_{ij}' \ar{r} \ar{d} & \frakX_{ij}' \ar{d} \\
S_i \times_{A_1[P_i]}A_1[Q_{ij}'] \ar{r} & \frakS_i\times_{B\langle P_i \rangle }B\langle Q_{ij}' \rangle .
\end{tikzcd}
$$
We have the dilatation $R_{\frakX_{ij}',1},$ defined from $\left ( \frakX_{ij}',Q_{ij}' \right ) \ra \left ( \frakS_i, P_i \right ),$ as in \ref{paragIran5}. Let $R_{ij}'$ be the logarithmic scheme obtained by reduction of $R_{\frakX_{ij}',1}$ modulo $p$ and $\calR_{ij}'$ its structural ring. Now consider the $\Ox_{X_{ij}',\upmu_{X_{ij}}}$-module $\calB_{X_{ij}/S}$ \eqref{Bgroupe2} equipped with the stratification
$$
\calR'_{ij,\upmu_{X_{ij}}} \otimes_{\Ox_{X_{ij}',\upmu_{X_{ij}}}} \calB_{X_{ij}/S} \xrightarrow{\sim} \calB_{X_{ij}/S} \otimes_{\Ox_{X_{ij}',\upmu_{X_{ij}}}} \calR'_{{ij},\upmu_{X_{ij}}}
$$
that exchanges factors (note that the two projections $R_{ij}' \ra X'_{ij}$ are equal by (\cite{SF1} 11.4)). By \ref{equivRQind}, it corresponds to a crystal $\frakB_{ij}$ of $\widetilde{\calE}\left ( X_{ij}' / \frakS \right )_{/\upmu_{ij}'}.$ By \ref{equivRQind}, $\frakB_{ij}$ is a $\upmu_{ij}$-indexed algebra.
We prove in \ref{propBglue} that the crystals $\frakB_{ij}$ glue together into a crystal $\frakB$ of $\widetilde{\calE}\left (X'/\frakS \right )_{/\upmu'}.$
\end{parag}

\begin{lemma}\label{lemShiho208}
Keep the notation of \ref{paragBglue}. For $i\in I$ and $j\in I_i,$ the morphism $X_{ij}' \ra S_i\times_{A_1[P_i]}A_1[Q_{ij}'],$ induced by $X_{ij}' \ra S_i$ and the chart $X_{ij}' \ra A_1[Q_{ij}'],$ is étale and strict.
\end{lemma}

\begin{proof}
Consider the cartesian squares in the category of fine logarithmic schemes:
$$
\begin{tikzcd}
T_{ij} \ar{r} \ar{d} & S_i \times_{A_1[P_i]}A_1[Q_{ij}] \ar{r} \ar{d} & A_1[Q_{ij}] \ar{d} \\
S_i \ar[swap]{r}{F_{S_i}} & S_i \ar{r} & A_1[P_i].
\end{tikzcd}
$$
The composition
$
S_i \xrightarrow{F_{S_i}} S_i \rightarrow A_1[P_i]
$
is equal to
$
S_i \ra A_1[P_i] \xrightarrow{F_{A_1[P_i]}} A_1[P_i].
$
It follows that we also have the following cartesian squares in the category of fine logarithmic schemes:
$$
\begin{tikzcd}
T_{ij} \ar{r} \ar{d} & A_1[Q_{ij}''] \ar{r} \ar{d} & A_1[Q_{ij}] \ar{d} \\
S_i \ar{r} & A[P_i] \ar[swap]{r}{F_{A_1[P_i]}} & A_1[P_i].
\end{tikzcd}
$$
The morphism $X_{ij}\ra S_i \times_{A_1[P_i]}A_1[Q_{ij}]$ is étale and strict hence so is its base change by $F_{S_i}$ in the category of fine logarithmic schemes, which is $X_{ij}''\ra T_{ij}=S_i \times_{A_1[P_i]}A_1[Q_{ij}''].$ The result then follows by base change by the canonical $A_1[Q_{ij}'] \ra A_1[Q_{ij}''].$
\end{proof}

\begin{proposition}\label{propBglue}
Keep the notation of \ref{paragBglue}. For $i\in I,$ let
$$
Res_{ij}:\widetilde{\calE}\left (X_{ij}' / \frakS \right )_{/\upmu'_{ij}}  \ra  \widetilde{\calE} \left (X' / \frakS \right )_{/\upmu'}
$$
be the morphism of topoi induced by $\rho_{X',X_{ij}'}$ \eqref{eqrhoEST2}. There exists a unique crystal $\frakB_X$ of $\widetilde{\calE}(X'/\frakS)_{/\upmu'}$ such that, for all $i\in I$ and $j\in I_i,$ there exists a canonical isomorphism
$$
Res_{ij}^{-1}\frakB_X \xrightarrow{\sim} \frakB_{ij}.
$$
In addition, $\frakB_X$ is naturally equipped with a structure of a $\upmu'$-indexed algebra.
\end{proposition}

\begin{proof}
For $i,j\in I$ and $k\in I_i,$ $l\in I_j,$ the morphism of topoi $\rho_{X_{ik}',X_{ik}'\times_{X'}X_{jl}'}$ \eqref{eqrhoEST2} induces a morphism of topoi
$$
Res_{ikjl}: \widetilde{\calE}\left (X_{ik}' \times_{X'} X_{jl}' / \frakS \right )_{/\upmu'_{ikjl}}  \ra  \widetilde{\calE} \left (X_{ik}' / \frakS \right )_{/\upmu'_{ik}}
$$
such that $Res_{ikjl}^{-1}$ is given by composition with
\begin{equation}\label{eqShiho2}
\begin{array}[t]{clc}
\calE\left (X_{ik}' \times_{X'} X_{jl}' / \frakS \right )_{/\upmu_{ikjl}'} & \ra & \calE \left (X_{ik}' / \frakS \right )_{/\upmu_{ik}'} \\
(U,\frakT,u,s) & \mapsto & \left (\rho_{X_{ik}',X_{ik}'\times_{X'}X_{jl}'}\left (U,\frakT,u \right ),s \right ),
\end{array}
\end{equation}
where $(U,\frakT,u)$ is an object of $\calE\left (X_{ik}' \times_{X'} X_{jl}' / \frakS \right )$ and $s\in \Gamma(U,\upmu_X).$
We construct an isomorphism
$$
Res_{ikjl}^{-1} \frakB_{ik} \xrightarrow{\sim} Res_{jlik}^{-1}\frakB_{jl}.
$$
Let $(U,\frakT,u)$ be an object of $\calE \left (X_i' \times_{X'} X_j' / \frakS \right ).$ We abusively denote $\rho_{X_{ik}',X_{ik}'\times_{X'}X_{jl}'}(U,\frakT,u)$ by $(U,\frakT,u).$ By lemma \ref{lemShiho209} below,
$$
\left (Res_{ikjl}^{-1}\frakB_{ik} \right )_{(U,\frakT,u)}=\frakB_{ik,(U,\frakT,u)}. 
$$
Suppose that $U$ is affine, $\frakT$ is then affine and there exists a morphism $\frakT \ra \frakX_{ik}'$ fitting into the commutative diagram
$$
\begin{tikzcd}
U \ar{r} & X_{ik}'\times_{X'}X_{jl}' \ar{r} & \frakX_{ik}' \ar{d} \\
T \ar{u} \ar{r} & \frakT \ar{ur} \ar{r} & \frakS.
\end{tikzcd}
$$
We obtain a morphism
$
(U,\frakT,u) \ra (X_{ik}',\frakX_{ik}',\op{Id}_{X_{ik}'})
$
of $\calE\left ( X_{ik}'/\frakS \right ).$ The corresponding morphism of ringed topoi
$$
\varphi_{ik}: \left ( U_{\text{ét}/\upmu_U},\left (u_*\Ox_T \right )_{\upmu_U} \right ) \ra \left ( X_{{ik},\text{ét}/\upmu_{X_{ik}}}',\Ox_{X_{ik}',\upmu_{X_{ik}}}\right )
$$
does not depend on the choice of the lifting $\frakT \ra \frakX_{ik}'.$ Since $\frakB_{ik}$ is a crystal, we have an isomorphism
\begin{equation}\label{eqShiho3}
\varphi_{ik}^*\calB_{X_{ik}/S}=\varphi_{ik}^*\frakB_{{ik},(X_{ik}',\frakX_{ik}',\op{Id}_{X_{ik}'})} \xrightarrow{\sim} \frakB_{{ik},(U,\frakT,u)}.
\end{equation}
Similarly, we have an isomorphism
\begin{equation}\label{eqShiho4}
\varphi_{jl}^*\calB_{X_{jl}/S} \xrightarrow{\sim} \frakB_{{jl},(U,\frakT,u)}.
\end{equation}
Since $U \ra X_{ik}'$ and $U\ra X_{jl}'$ factor through $U \ra X_{ik}'\times_{X'}X_{jl}',$ the morphisms of ringed topoi $\varphi_{ik}$ and $\varphi_{jl}$ factor through
$$
\varphi_{ikjl}:\left ( U_{\text{ét}/\upmu_U},\left (u_*\Ox_T \right )_{\upmu_U} \right ) \ra \left ( \left ( X_{ik}'\times_{X'}X_{jl}' \right )_{\text{ét}/\upmu_{X_{ik}\times_XX_{jl}}},\Ox_{X_{ik}'\times_{X'}X_{jl}',\upmu_{X_{ik}\times_XX_{jl}}}\right ).
$$
By \ref{lemBres}, the isomorphisms \eqref{eqShiho3} and \eqref{eqShiho4} induce then isomorphisms
$$
\varphi_{ikjl}^*\calB_{X_{ik}\times_XX_{jl}/S} \xrightarrow{\sim} \frakB_{ik,(U,\frakT,u)},\quad 
\varphi_{ikjl}^*\calB_{X_{ik}\times_XX_{jl}/S} \xrightarrow{\sim} \frakB_{jl,(U,\frakT,u)}.
$$
We obtain an isomorphism
$
\frakB_{ik,(U,\frakT,u)} \xrightarrow{\sim} \varphi_{ikjl}^*\calB_{X_{ik}\times_XX_{jl}/S} \xrightarrow{\sim} \frakB_{jl,(U,\frakT,u)}.
$
These isomorphisms satisfy the cocyclicity condition, hence the result. If $\sigma_{ik}':\upmu_{ik}'^2 \ra \upmu_{ik}'$ and $\sigma':\upmu'^2\ra \upmu'$ are the addition maps, then the morphisms
$
\frakB_{ik} \boxtimes \frakB_{ik} \ra \sigma_{ik}'^*\frakB_{ik},
$
defining the $\upmu_{ik}'$-indexed algebra structure on $\frakB_{ik},$ induce
$
\frakB_X\boxtimes \frakB_X \ra \sigma'^*\frakB_X.
$
This morphism defines the $\upmu'$-indexed algebra structure on $\frakB_X.$
\end{proof}

\begin{lemma}\label{lemShiho209}
Keep the notation of \ref{propBglue}. For any sheaf $\F$ of $\widetilde{\calE}(X'/\frakS)_{/\upmu'}$ and any object $(U,\frakT,u)$ of $\calE(X'_{ik}/\frakS),$ we have
$$
\left (Res_{ik}^{-1}\F \right )_{(U,\frakT,u)}=\F_{\rho_{X',X_{ik}'}\left (U,\frakT,u \right )}.
$$
\end{lemma}

\begin{proof}
This follows from the obvious commutativity of
$$
\begin{tikzcd}
\text{ét}_{/U\times \upmu_U} \ar{rr}{\delta_{(U,\frakT,u)}} \ar[swap]{drr}{\delta_{\rho_{X',X_{ik}'}(U,\frakT,u)}} & & \widetilde{\calE}(X_{ik}'/\frakS)_{/\upmu_{ik}'} \ar{d}{Res_{ik}} \\
 & & \widetilde{\calE}(X'/\frakS)_{/\upmu'}.
\end{tikzcd}
$$
\end{proof}

\begin{proposition}\label{AXUESTiso}
Consider an étale $X$-scheme $U$ and the morphisms $\underline{R_{X,U}}$ and $R_{X',U'}$ \eqref{eqREST}.
Then we have canonical isomorphisms
$$
\underline{R_{X,U}}^{-1}\frakA_X \xrightarrow{\sim} \frakA_U,\ R_{X',U'}^{-1}\frakB_X \xrightarrow{\sim}\frakB_U.
$$
\end{proposition}

\begin{proof}
By \ref{propEST1831}, we have the commutative diagram
$$
\begin{tikzcd}
\begin{Bmatrix}
\mathrm{Crystals\ of\ }\widetilde{\underline{\calE}}(X/\frakS)_{/\underline{\upmu}}
\end{Bmatrix}
\ar{r}{\sim} \ar[swap]{d}{\underline{R_{X,U}}^{-1}} &
\begin{Bmatrix}
\Ox_{X,\upmu_X}\text{-}\mathrm{modules\ equipped} \\
\mathrm{with\ a\ }\calQ_{1,\upmu_X}\text{-}\mathrm{stratification}
\end{Bmatrix}
 \ar{d} \\
\begin{Bmatrix}
\mathrm{Crystals\ of\ }\widetilde{\underline{\calE}}(U/\frakS)_{/\underline{\upmu}_U}
\end{Bmatrix} 
\ar{r}{\sim} &
\begin{Bmatrix}
\Ox_{U,\upmu_U}\text{-}\mathrm{modules\ equipped} \\
\mathrm{with\ a\ }\calQ_{1|U,\upmu_U}\text{-}\mathrm{stratification}
\end{Bmatrix}.
\end{tikzcd}
$$
The crystal $\frakA_X$ is an object of the upper left category. It corresponds to the stratified $\Ox_{X,\upmu_X}$-module $\left (\calA_X,\epsilon_{\calA_X,Q}\right ).$ This is sent by the right vertical arrow to $\left (\calA_U,\epsilon_{\calA_U,Q}\right ).$ This corresponds by the lower equivalence to $\frakA_U.$ It follows that we have a canonical isomorphism
$$
\underline{R_{X,U}}^{-1}\frakA_X \xrightarrow{\sim} \frakA_U.
$$
The isomorphism for $\frakB_X$ is proved similarly.
\end{proof}

\begin{lemma}\label{lemkraz}
Let $X\ra Y$ be the exact diagonal immersion, $X \xrightarrow{\Delta} U \ra Y$ a factorization into a closed immersion followed by an open one, $\calI$ the ideal of $\Delta,$ $p_1,p_2:U \ra X$ the canonical projections, $\ov{x} \ra X$ a geometric point, and $a_1,\hdots a_d$ generators of $\calI_{\ov{x}}$ that reduce to a basis of the $\Ox_{X,\ov{x}}$-module $\calI_{\ov{x}}/\calI_{\ov{x}}^2.$ Equip $\calI$ with a structure of $\Ox_X$-module via the first projection $p_1.$ For $I\in \N^d,$ set 
$
a^I=\prod_{i=1}^da_i^{I_i}.
$
Then, for any positive integer $n$ and $b\in \calI_{\ov{x}},$ there exist $c_I\in \Ox_{X,\ov{x}}$ such that
$$
b-\sum_{0<|I|\le n}c_Ia^I\in \calI_{\ov{x}}^{n+1}.
$$
\end{lemma}

\begin{proof}
The canonical morphism $U\ra X\times_S^{\op{log}}X$ is log étale and the projections $X\times^{\op{log}}_SX\ra X$ are log smooth. It follows that the projections $U \ra X$ are log smooth. Since they are strict, they are smooth and $\Delta$ is a regular immersion. It follows that we have a canonical isomorphism
\begin{equation}\label{SymgrKoko}
S^{\bullet}\left (\calI/\calI^2\right )\xrightarrow{\sim} \bigoplus_{k\ge 0} \calI^k/\calI^{k+1}.
\end{equation}
There exist $c_1,\hdots,c_d\in \Ox_{X,\ov{x}}$ such that
$
b-\sum_{i=1}^dc_ia_i \in \calI_{\ov{x}}^2.
$
By \eqref{SymgrKoko}, $\left (a^I\right )_{I\in \N^d, |I|=2}$ is a basis of the $\Ox_{X,\ov{x}}$-module $\calI_{\ov{x}}^2/\calI_{\ov{x}}^3.$ Then there exist $c_I\in \Ox_{X,\ov{x}}$ such that
$$
b-\sum_{i=1}^dc_ia_i-\sum_{I\in \N^d, |I|=2}c_Ia^I\in \calI_{\ov{x}}^3.
$$
The result is then proved by induction.
\end{proof}

\begin{proposition}\label{propkraz1932}
Assume $(H_1)$ is satisfied. Let $s\in \Gamma(V,\calM^{gp}_X)$ be a section over an étale $X$-scheme $V$ and denote by $\ov{s}$ its image in $\Gamma(V,\upmu_X).$ Let $X\ra Y$ be the exact diagonal immersion, $X \xrightarrow{\Delta} U \ra Y$ a factorization into a closed immersion followed by an open one, $\calI$ the ideal of $\Delta,$ $p_1,p_2:U \ra X$ the canonical projections, $g:Q_1 \ra U$ and $\iota:X\ra Q_1$ the canonical morphisms and $q_1=p_1\circ g,$ $q_2=p_2\circ g.$ Consider a geometric point $\ov{x} \ra V.$
Then, for every $m\in \calM_{X,\ov{x}},$ there exists $a \in \calI_{\Delta(\ov{x})}$ such that
\begin{equation}\label{Kokooooh}
p_{2,\Delta(\ov{x})}^{\flat}m-p_{1,\Delta(\ov{x})}^{\flat}m=\alpha_{Y,\Delta(\ov{x})}^{-1}(1+a).
\end{equation}
In addition, if $e_s$ is the basis \eqref{eqAbase} of the $\Ox_V$-module $\calA_{X,\ov{s}}$ and $b$ is a local section of $\Ox_V,$ then
\begin{equation}\label{eqepsilonAQ2}
\epsilon_{\calA_X,Q,\ov{s}}(1\otimes be_s) = \left (e_s\otimes g^{\#}\left (p_2^{\#}(b)(1+a)\right )\right ).
\end{equation}
Moreover, we suppose that the hypothesis (\cite{SF2} 6.12) is satisfied. We have a basis $\left (\eta^I\eta_{(q)}^J \right )_{\substack{I\in \llbracket 0,p-1 \rrbracket^d \\ J\in \N^d}}$ of the $\Ox_X$-module $\calQ_1.$ Let $\left (\varphi_{I,J} \right )_{\substack{I\in \llbracket 0,p-1 \rrbracket^d \\ J\in \N^d}}$ be its dual basis. Then, for any local section $x$ of $\calA_{X,\ov{s}},$ we have
\begin{equation}\label{eqepsilonAQ}
\epsilon_{\calA_X,Q,\ov{s}}(1\otimes x) = \sum_{I\in \llbracket 0,p-1 \rrbracket^d} (\varphi_{I,0}\cdot x) \otimes \eta^I,
\end{equation}
where $\varphi_{I,0}\cdot x$ is the action of $\varphi_{I,0}$ on $x$ via the $\calQ_1$-stratification $\epsilon_{\calA_X,Q,\ov{s}}.$
\end{proposition}

\begin{proof}
Let $\upmu_X=\ov{\calM}_X^{gp}$ and $\upmu_{Q_1}=\ov{\calM}_{Q_1}^{gp}.$ The projection $q_i:Q_1 \ra X$ induces a morphism $
q_i^{-1}\calA_X \ra \calA_{Q_1},
$
of $Q_{1,\text{ét}/q_i^{-1}\upmu_X}$
and then a morphism $
\calA_X\ra \iota^{-1}\calA_{Q_1},
$ of $X_{\text{ét}/\upmu_X}.$
Taking the fibers over $\ov{s}\in \Gamma(V,\upmu_X)$ on both sides and then the stalks over a geometric point $\ov{x} \ra V,$ we get
$$
c_i:\calA_{X,\ov{s},\ov{x}} \ra \left (\iota^{-1}\calA_{Q_1}\right )_{\ov{s},\ov{x}}
$$
Consider $\ov{s}$ as an element of $\upmu_{X,\ov{x}}\xrightarrow{\sim} \upmu_{Q_1,\iota(\ov{x})}$ (the isomorphism comes from the fact that $\iota$ is strict). Recall the description $\calA_{X,\ov{s},\ov{x}}=\Ox_{X,\ov{x}}\times \calM_{X,\ov{x},\ov{s}}^{gp}/\sim$ (\eqref{eq8422b} and \eqref{eq8422c}) where $\calM_{X,\ov{x},\ov{s}}^{gp}$ is the fiber of $\calM_{X,\ov{x}}^{gp}$ over $\ov{s}.$ Similarly
$$
\left (\iota^{-1}\calA_{Q_1}\right )_{\ov{s},\ov{x}}=\left (\Ox_{Q_1,\iota(\ov{x})} \times \calM_{Q_1,\iota(\ov{x}),\ov{s}}^{gp} \right )/\sim.
$$
We have, for all $m\in \calM_{X,\ov{x},\ov{s}}^{gp},$
$$
c_i(1,m)=(1,q_{i,\iota(\ov{x})}^{\flat}m).
$$
To lighten notation, for a morphism $f$ of sheaves of $X_{\text{ét}}$ (resp. $Y_{\text{ét}},$ $Q_{1,\text{ét}}$), we denote $f_{\ov{x}}$ (resp. $f_{\Delta(\ov{x})},$ $f_{\iota(\ov{x})}$) simply by $f.$ For every $m\in \calM_{X,\ov{x}} \subset \calM_{X,\ov{x}}^{gp},$ by the exact sequence (\cite{SF1} (4.3.1)), there exists a local section $a$ of $\Delta^{-1}\calI$ such that, in $\calM_{Y,\Delta(\ov{x})}^{gp},$
\begin{equation}
p_2^{\flat}m-p_1^{\flat}m=\alpha_Y^{-1}(1+a).
\end{equation}
It follows that, in $\calM^{gp}_{Q_1,\iota(\ov{x})},$ we have
$$
q_2^{\flat}m-q_1^{\flat}m=g^{\flat}(p_2^{\flat}m-p_1^{\flat}m)=g^{\flat}\alpha_Y^{-1}(1+a)=\alpha_{Q_1}^{-1}g^{\#}(1+a).
$$
If $m$ is over $\ov{s},$ then, in $\left (\iota^{-1}\calA_{Q_1}\right )_{\ov{s},\ov{x}},$ we have
$
(1,q_2^{\flat}m)=(g^{\#}(1+a),q_1^{\flat}m).
$
We finally get, in $\left (\calA_{X,\ov{s}}\otimes_{\Ox_U}\calQ_1\right )_{\ov{x}},$ the equality
$
\epsilon_{\calA_X,Q,\ov{s}}(1\otimes(1,m)) = ((1,m)\otimes g^{\#}(1+a)).
$
If $b$ is a local section of $\Ox_U,$ then
\begin{equation}
\begin{alignedat}{2}
\epsilon_{\calA_X,Q,\ov{s}}(1\otimes(b,s)) &= \epsilon_{\calA_X,Q,\ov{s}}(q_2^{\#}(b)\otimes(1,s)) = q_2^{\#}(b) \epsilon_{\calA_X,Q,\ov{s}}(1\otimes(1,s)) \\
&= q_2^{\#}(b) \left ((1,s)\otimes g^{\#}(1+a)\right ) = \left ((1,s)\otimes g^{\#}\left (p_2^{\#}(b)(1+a)\right )\right ).
\end{alignedat}
\end{equation}
Suppose that the hypothesis of (\cite{SF2} 6.12) is satisfied. For $1\le i\le d,$ we have, in $\calQ_1,$
$$
\eta_i^p=p\eta_{i(q)}=0.
$$
Then $\calI^{pd} \Ox_{Q_1}=0.$
Therefore, by \ref{lemkraz}, $g^{\#}(a)$ is an $\Ox_X$-linear combination of $\eta^{I}:=\prod_{i=1}^d\eta_i^{I_i}$ for $I\in \llbracket 0,p-1 \rrbracket^d$ (recall that the Hopf algebra $\calQ_1$ is considered as an $\Ox_X$-module via the first projection). It follows that, for any local section $x$ of $\calA_{X,\ov{s}},$ there exist local sections $x_I \in \calA_{X,\ov{s}}$ such that
\begin{equation*}
\epsilon_{\calA_X,Q,\ov{s}}:\begin{array}[t]{clc}
\calQ_1 \otimes_{\Ox_U} \calA_{X,\ov{s}} & \ra & \calA_{X,\ov{s}} \otimes_{\Ox_U} \calQ_1 \\
1\otimes x & \mapsto & \sum_{I\in \llbracket 0,p-1 \rrbracket^d} x_I \otimes \eta^I.
\end{array}
\end{equation*}
By definition, the action $\varphi_{I,J}\cdot x$ is the image of $1\otimes x$ by the composition
$$
\calQ_1 \otimes_{\Ox_U}\calA_{X,\ov{s}} \xrightarrow{\epsilon_{\calA_X,Q,\ov{s}}} \calA_{X,Q,\ov{s}} \otimes_{\Ox_U} \calQ_1 \xrightarrow{\op{Id}\otimes \varphi_{I,J}} \calA_{X,Q,\ov{s}}.
$$
It is then clear that $x_I=\varphi_{I,0}\cdot x.$
\end{proof}

\begin{lemma}\label{lem1924Koko}
Assume $(H_1)$ and the hypothesis (\cite{SF2} 6.12) are satisfied and consider the basis
$\left (\eta^I\eta_{(q)}^J\right )_{\substack{I\in \llbracket 0,p-1\rrbracket^d\\ J\in \N^d}}$
of the $\Ox_X$-module $\calQ_1,$ where
$
\eta^I=\prod_{i=1}^d\eta_i^{I_i},\ \eta_{(q)}^J=\prod_{i=1}^d\eta_{(q)}^{J_i}.
$
Let $\left (\varphi_{I,J} \right )$ be its dual basis and extend the notation $\varphi_{I,J}$ by setting $\varphi_{I,J}=0$ if $I\in \Z^d$ and $I\notin \llbracket 0,p-1\rrbracket^d.$
Consider the ring structure on $\calQ_1^{\vee}$ induced by the Hopf algebra structure on $\calQ_1$ (\cite{SF2} 6.16). Let $\epsilon_i$ be the multi-index of $\N^d$ whose all coefficients are zero except for the $i^{th}$ which is equal to $1.$ Then
$$
\varphi_{I,0}\cdot \varphi_{\epsilon_i,0}=I_i\varphi_{I,0}+(I_i+1)\varphi_{I+\epsilon_i,0}-\varphi_{I-(p-1)\epsilon_i,\epsilon_i}.
$$
In particular, if $J\in \llbracket 0,p-1\rrbracket^d$ and $1\le i\le d$ such that $J_i\ge 1,$ then
\begin{equation}\label{eq1925Koko}
\varphi_{J-\epsilon_i,0}\cdot \varphi_{\epsilon_i,0}=(J_i-1)\varphi_{J-\epsilon_i,0}+J_i\varphi_{J,0}.
\end{equation}
\end{lemma}

\begin{proof}
By definition, the product $\varphi_{I,0}\cdot \varphi_{\epsilon_i,0}$ is equal to the composition
$$
\calQ_1 \xrightarrow{\delta} \calQ_1 \otimes_{\Ox_X}\calQ_1 \xrightarrow{\op{Id}\otimes \varphi_{\epsilon_i,0}} \calQ_1 \xrightarrow{\varphi_{I,0}} \Ox_X,
$$
where $\delta$ is given in (\cite{SF2} 6.16). Let $\alpha \in \llbracket 0,p-1 \rrbracket ^d$ and $\beta\in \N^d.$ Let $\calK \subset \calQ_1$ be the ideal generated by $\eta_1,\hdots,\eta_d,\eta_{1(q)},\hdots,\eta_{d(q)}.$ The linear form $\varphi_{\epsilon_i,0}:\calQ_1 \ra \Ox_X$ factors clearly through the canonical projection $\pi:\calQ_1 \ra \calQ_1/\calK^2.$ By (\cite{SF2} 6.16), we have
\begin{alignat*}{2}
\delta \left (\eta^{\alpha}\eta_{(q)}^{\beta} \right ) &= \prod_{j=1}^d\left (1\otimes \eta_j+\eta_j\otimes 1+\eta_j\otimes \eta_j\right )^{\alpha_j} \\
& \times \prod_{j=1}^d\left (1\otimes \eta_{j(q)}+\sum_{0<b+c<p} \frac{(-1)^{b+c}}{b+c}\begin{pmatrix}b+c \\ b \end{pmatrix} \eta_j^{b+c}\otimes \eta_j^{p-b} +\eta_{j(q)}\otimes 1\right )^{\beta_j}.
\end{alignat*}
Then, in $\calQ_1 \otimes_{\Ox_X} \left ( \calQ_1/\calK^2\right ),$ we have
\begin{alignat*}{2}
\left (\pi \otimes \op{Id}_{\calQ_1} \right )\circ \delta \left (\eta^{\alpha}\eta_{(q)}^{\beta} \right ) &= \prod_{j=1}^d\left (1\otimes \eta_j+\eta_j\otimes 1+\eta_j\otimes \eta_j\right )^{\alpha_j} \\
& \times \prod_{j=1}^d\left (1\otimes \eta_{j(q)}- \eta_j^{p-1}\otimes \eta_j +\eta_{j(q)}\otimes 1\right )^{\beta_j}.
\end{alignat*}
By the binomial formula, we have
\begin{alignat*}{2}
\left (1\otimes \eta_j+\eta_j\otimes 1+\eta_j\otimes \eta_j\right )^{\alpha_j} &= \sum_{\substack{r,s,t\in \N \\ r+s+t=\alpha_j}} \frac{\alpha_j!}{r!s!t!} \eta_j^{s+t}\otimes \eta_j^{r+t}, \\
\left (1\otimes \eta_{j(q)}- \eta_j^{p-1}\otimes \eta_j +\eta_{j(q)}\otimes 1\right )^{\beta_j} &= \sum_{\substack{u,v,w\in \N \\ u+v+w=\beta_j}}\frac{\beta_j!}{u!v!w!}(-1)^v \times \left (\eta_{j(q)}^{w}\eta_j^{(p-1)v} \right )\otimes\left ( \eta_{j(q)}^{u}\eta_j^{v} \right ).
\end{alignat*}
We have the same equalities with multi-indices:
\begin{alignat*}{2}
\prod_{j=1}^d\left (1\otimes \eta_j+\eta_j\otimes 1+\eta_j\otimes \eta_j\right )^{\alpha_j}&=\sum_{\substack{r,s,t\in\N^d \\ r+s+t=\alpha}} \frac{\alpha!}{r!s!t!} \eta^{s+t}\otimes \eta^{r+t}, \\
\prod_{j=1}^d\left (1\otimes \eta_{j(q)}- \eta_j^{p-1}\otimes \eta_j +\eta_{j(q)}\otimes 1\right )^{\beta_j} &= \sum_{\substack{u,v,w\in \N^d \\ u+v+w=\beta}}\frac{\beta!}{u!v!w!}(-1)^{|v|} \times \left (\eta_{(q)}^{w}\eta^{(p-1)v} \right )\otimes\left ( \eta_{(q)}^{u}\eta^{v} \right ).
\end{alignat*}
Suppose that $\alpha_i \ge 1$ and $\beta_i\ge 1.$
Then, the only terms whose image doesn't vanish by $\op{Id}\otimes \varphi_{\epsilon_i,0}$ are those that either correspond to
$$
\begin{cases}r+t=\epsilon_i\\
u=v=0, \end{cases}\ \text{or}\ 
\begin{cases}
r+t=0\\
u=0\\
v=\epsilon_i. 
\end{cases}
$$
These cases are equivalent to
$$
\begin{cases}r=\epsilon_i, t=0, s=\alpha-\epsilon_i\\
u=v=0, w=\beta,\end{cases}, \begin{cases}t=\epsilon_i, r=0, s=\alpha-\epsilon_i\\
u=v=0, w=\beta,\end{cases}\ \text{or}\ 
\begin{cases}
r=t=0, s=\alpha\\
u=0\\
v=\epsilon_i\\
w=\beta-\epsilon_i.
\end{cases}
$$
They yield
$
\alpha_i\left (\eta^{\alpha-\epsilon_i}\otimes \eta_i+\eta^{\alpha}\otimes \eta_i \right )\left (\eta_{(q)}^{\beta}\otimes 1\right )-\beta_i(\eta^{\alpha}\otimes 1)\left (\left (\eta^{(p-1)\epsilon_i}\eta_{(q)}^{\beta-\epsilon_i} \right )\otimes \eta_i \right ).
$
If $\alpha_i=0$ and $\beta_i\ge 1,$ the only term that doesn't vanish by $\op{Id}\otimes \varphi_{\epsilon_i,0}$ is
$$
-\beta_i(\eta^{\alpha}\otimes 1)\left (\left (\eta^{(p-1)\epsilon_i}\eta_{(q)}^{\beta-\epsilon_i} \right )\otimes \eta_i \right ).
$$
If $\alpha_i \ge 1$ and $\beta_i=0,$ the only term that doesn't vanish by $\op{Id}\otimes \varphi_{\epsilon_i,0}$ is
$$
\alpha_i\left (\eta^{\alpha-\epsilon_i}\otimes \eta_i+\eta^{\alpha}\otimes \eta_i \right )\left (\eta_{(q)}^{\beta}\otimes 1\right ).
$$
If $\alpha_i=\beta_i=0,$ all terms vanish by $\op{Id}\otimes \varphi_{\epsilon_i,0}.$
It follows that
$$
\left (\varphi_{I,0}\cdot \varphi_{\epsilon_i,0} \right )\left (\eta^{\alpha}\eta_{(q)}^{\beta} \right )=\begin{cases}I_i+1\ \text{if}\ \alpha=I+\epsilon_i\ \text{and}\ \beta=0,\\
I_i\ \text{if}\ \alpha=I\ \text{and}\ \beta=0,\\
-1\ \text{if}\ \alpha=I-(p-1)\epsilon_i\ \text{and}\ \beta=\epsilon_i.
\end{cases}
$$
The result follows.
\end{proof}

\begin{proposition}\label{Kokoprop1925}
Assume $(H_1)$ is satisfied. The stratification $\epsilon_{\calA_X,Q}$ \eqref{eqeAQ} is $\calB_{X/S}$-linear, where $\calB_{X/S}$ is defined in \ref{Bgroupe2}.
\end{proposition}

\begin{proof}
Let $s\in \Gamma(U,\calM_X^{gp})$ be a local section of $\calM_X^{gp}$ over an étale $X$-scheme $U$ and $\ov{s}\in \Gamma(U,\upmu_X)$ its image. By definition, $\epsilon_{\calA_X,Q}$ is a morphism of $\upmu_X$-indexed algebras \eqref{remkraz1}. It is thus sufficient to prove that, for any local section $b$ of $\calB_{X/S,\ov{s}},$
$$
\epsilon_{\calA_X,Q,\ov{s}}(1\otimes b)=b\otimes 1.
$$
The stratification $\epsilon_{\calA_X,Q,\ov{s}}$ defines a $\calQ_1^{\vee}$-module structure on $\calA_{X,\ov{s}},$ given explicitly by
$$
\begin{array}[t]{clc}
\calQ_1^{\vee} \otimes_{\Ox_U} \calA_{X,\ov{s}} & \ra & \calA_{X,\ov{s}} \\
\varphi\otimes a & \mapsto & \left (\op{Id} \otimes \varphi \right )\circ \epsilon_{\calA_X,Q,\ov{s}}(1\otimes a).
\end{array}
$$
We may work étale locally on $X$ and suppose that the hypothesis (\cite{SF2} 6.12) is satisfied. We take again the notation of \ref{lem1924Koko} and denote by $\varphi_{I,J}\cdot b$ the action of $\varphi_{I,J}$ on $b.$
By \eqref{eqepsilonAQ}, we have
$$
\epsilon_{\calA_X,Q,\ov{s}}(1\otimes b)=\sum_{I\in \llbracket 0,p-1\rrbracket^d}\left (\varphi_{I,0}\cdot b\right )\otimes \eta^I.
$$
We have to prove that $\varphi_{I,0}\cdot b=0$ for all $I\in\llbracket 0,p-1\rrbracket^d$ such that $|I|>0.$ By \eqref{eq1925Koko} and by induction on $|I|,$ it is sufficient to prove that $\varphi_{\epsilon_i,0}\cdot b=0$ for $1\le i\le d.$ 
Recall that $e_s=(1,s)$ \eqref{eqAbase} is a basis for the $\Ox_U$-module $\calA_{X,\ov{s}}.$ It follows that there exists a local section $t$ of $\Ox_X$ such that
$$
b=te_s.
$$
By \eqref{eq8431Koko}, we have
$$
dt+t\op{dlog}s=0.
$$
Let $X\ra Y$ be the exact diagonal immersion, $\calI$ its ideal, $\iota:X\ra Q_1$ and $g:Q_1\ra Y$ the canonical morphisms and $p_1,p_2:Y\ra X$ and $q_1,q_2:Q_1 \ra X$ the canonical projections. By \eqref{eqepsilonAQ2} and \eqref{Kokooooh}, there exists a local section $a$ of $\Delta^{-1}\calI$ such that
$$
\begin{cases}
\epsilon_{\calA_X,Q,\ov{s}}\left (1\otimes te_s\right )=\left (e_s\otimes g^{\#}(p_2^{\#}(t)(1+a)) \right ),\\
p_2^{\flat}(s)=\alpha_Y^{-1}(1+a)+p_1^{\flat}(s).
\end{cases}
$$
We have a canonical isomorphism of $\Ox_X$-modules
$$
\Delta^{-1}\left (\calI/\calI^2\right )\xrightarrow{\sim}\omega^1_{X/S}.
$$
Let $\ov{x} \ra X$ be a geometric point.
By (\cite{SF1} 6.5), the differential $\op{dlog}s \in \omega^1_{X/S,\ov{x}}$ corresponds to the class of $a\in \calI_{\Delta(\ov{x})}$ modulo $\calI_{\Delta(\ov{x})}^2.$ The local section $dt$ corresponds to the class of $p_{2,\Delta(\ov{x})}^{\#}(t)-p_{1,\Delta(\ov{x})}^{\#}(t).$ Since $dt+t\op{dlog}s=0,$
$$
p_{2,\Delta(\ov{x})}^{\#}(t)-p_{1,\Delta(\ov{x})}^{\#}(t)+p_{2,\Delta(\ov{x})}^{\#}(t)a=p_{2,\Delta(\ov{x})}^{\#}(t)(1+a)-p_{1,\Delta(\ov{x})}^{\#}(t)
$$
belongs to $\calI_{\Delta(\ov{x})}^2.$ Since
$$
\varphi_{\epsilon_i,0}(q_{1,\iota(\ov{x})}^{\#}(t))=t\varphi_{\epsilon_i,0}(1)=0,
$$
we deduce that, for $1\le i\le d,$
$$
\varphi_{\epsilon_i,0}\left (g_{\Delta(\ov{x})}^{\#}(p_{2,\Delta(\ov{x})}^{\#}(t)(1+a))\right )=0.
$$
By \ref{eqepsilonAQ2}, we have
\begin{equation*}
\epsilon_{\calA_X,Q,\ov{s}}(1\otimes b) =\epsilon_{\calA_X,Q,\ov{s}}(1\otimes te_s) = e_s\otimes g^{\#}\left (p_2^{\#}(t)(1+a)\right ).
\end{equation*}
Then,
$$
\varphi_{\epsilon_i,0}\cdot b=e_s\otimes \varphi_{\epsilon_i,0}\left (g^{\#}\left (p_2^{\#}(t)(1+a)\right )\right )=0.
$$
This finishes the proof.
\end{proof}

\begin{remark}\label{actiondepsQ}
Assume $(H_1)$ is satisfied and consider the special fiber $h:P_0 \ra Q_1$ of the morphism (\cite{SF2} (7.2.1)). The pull-back
$$
h^*\epsilon_{\calA_X,Q}:(q_2\circ h)^*\calA_X \ra (q_1\circ h)^*\calA_X
$$
is, by definition, equal to the $P_0$-stratification $\epsilon_{\calA_X}$ defined by Montagnon in \cite{Mon} 4.4.1 page 46.
\end{remark}

\begin{proposition}\label{propkraz1934}
Assume $(H_2)$ is satisfied. Consider the morphism of Hopf algebras $\calV:\calR_1' \ra \calQ_1$ induced by $\nu:Q_{\frakX}\ra R_{\frakX',1}$ \eqref{nu}. Then, the diagram
\begin{equation}\label{diag186kraz}
\begin{tikzcd}
\begin{Bmatrix}\text{Crystals\ of\ }\widetilde{\calE}(X'/\frakS)_{/\upmu'} \end{Bmatrix} \ar{r}{C^{-1}_{X/\frakS,\upmu}} \ar[swap, sloped]{d}{\sim} & \begin{Bmatrix}\text{Crystals\ of\ }\widetilde{\underline{\calE}}(X/\frakS)_{/\underline{\upmu}} \end{Bmatrix}  \ar[sloped]{d}{\sim} \\
\begin{Bmatrix}\Ox_{X',\upmu_X}\text{-modules\ with\ an}\\ \calR_{1,\upmu_X}'\text{-stratification} \end{Bmatrix} \ar{r}{\Psi_{\upmu}} & \begin{Bmatrix}\Ox_{X,\upmu_X}\text{-modules\ with\ a}\\ \calQ_{1,\upmu_X}\text{-stratification} \end{Bmatrix},
\end{tikzcd}
\end{equation}
where the vertical equivalences are given in \ref{equivRQind}, $C_{X/\frakS,\upmu}$ is given in \eqref{diagkraz2} and $\Psi_{\upmu}(\calE',\epsilon') = \left (F_1^*\calE',\calV^*\epsilon'\right )$ for an $\Ox_{X',\upmu_X}$-module $\calE'$ equipped with an $\calR_{1,\upmu_X}'$-stratification $\epsilon',$ is 2-commutative.
\end{proposition}

\begin{proof}
Let $U$ be an étale $X$-scheme and $s\in \Gamma(U,\upmu_X)=\Gamma(U',F_{1*}\upmu_X).$ Consider the diagram
$$
\begin{tikzcd}
\begin{Bmatrix}\text{Crystals\ of}\\ \widetilde{\calE}(X'/\frakS)_{/\upmu'} \end{Bmatrix} \ar{rr}{C^{-1}_{X/\frakS,\upmu}} \ar[swap, sloped]{dd}{\sim} \ar{dr}{\alpha_s'^*} & &  \begin{Bmatrix}\text{Crystals\ of} \\ \widetilde{\underline{\calE}}(X/\frakS)_{/\underline{\upmu}} \end{Bmatrix}  \ar[pos=0.20,dashed,sloped]{dd}{\sim} \ar{dr}{\underline{\alpha}_s^*} & \\
 & \calC(U'/\frakS) \ar[pos=0.20]{rr}{C_{U/\frakS}^{-1}} \ar[pos=0.20,swap,sloped]{dd}{\sim} & & \underline{\calC}(U/\frakS) \ar[sloped] {dd}{\sim} \\
\begin{Bmatrix}\Ox_{X',\upmu_X}\text{-mod}\\ \text{with\ an}\\ \calR_{1,\upmu_X}'\text{-strat.} \end{Bmatrix} \ar[pos=0.20,dashed]{rr}{\Psi_{\upmu}} \ar{dr} & & \begin{Bmatrix}\Ox_{X,\upmu_X}\text{-mod}\\ \text{with\ a}\\ \calQ_{1,\upmu_X}\text{-strat.} \end{Bmatrix} \ar[dashed]{dr} & \\
& \begin{Bmatrix}\Ox_{U'}\text{-mod}\\ \text{with\ an}\\ \calR_{1|U}'\text{-strat.} \end{Bmatrix} \ar{rr}{\Psi_{0}} & & \begin{Bmatrix}\Ox_{U}\text{-mod}\\ \text{with\ a}\\ \calQ_{1|U}\text{-strat.} \end{Bmatrix},
\end{tikzcd}
$$
where $\alpha_s'$ and $\underline{\alpha}_s$ are given in \eqref{alphakraz}, $\calC(U'/\frakS)$ and $\underline{\calC}(U/\frakS)$ are defined in \ref{defcrys}, $\Psi_0$ is given in (\cite{SF2} 8.36) and the lower oblique lines take a module with a stratification to its fiber over $s.$
The front square is commutative by (\cite{SF2} 8.36). The upper square is commutative by definition of $C^{-1}_{X/\frakS,\upmu}$ \eqref{diagkraz19253}. The two side squares are commutative by construction of the equivalences of \ref{equivRQind} and the lower square is clearly commutative.
This being true for all local sections $s$ of $\upmu_X,$ the commutativity of the back square then follows from the fact that two $\Ox_{X,\upmu_X}$-modules with $\calQ_{1,\upmu_X}$-stratifications $(\calE,\epsilon)$ and $(\calE',\epsilon')$ are equal if and only if, for any local section $s$ of $\upmu_X,$ $(\calE_s,\epsilon_s)=(\calE'_s,\epsilon'_s)$ \eqref{indremark73}.
\end{proof}

\begin{lemma}\label{qF}
Suppose that $X\ra S$ is of Cartier type. Let $Q_1$ be the special fiber of $Q_{\frakX}$ and $\underline{Q_1}$ the logarithmic scheme-theoretic image of the Frobenius of $Q_1$ (\cite{SF2} 5). Denote by $q_1,q_2:Q_1\ra X$ the canonical projections and $F_1:X\ra X'$ the exact relative Frobenius (which is equal to the relative Frobenius since $X\ra S$ is of Cartier type). Then
$$
F_1\circ q_1=F_1\circ q_2.
$$
\end{lemma}

\begin{proof}
Since $X\ra S$ is of Cartier type, $X'$ is the base change of $X$ by the Frobenius $F_S$ of $S.$ Since $Q_1 \xrightarrow{q_i} X\ra X' \ra S,$ where $X'\ra S$ is the canonical projection, is equal to the structural morphism $Q_1 \ra S$ for $i=1,2,$ it is sufficient to prove that $Q_1 \xrightarrow{q_i} X\ra X' \ra X$ is the same for $i=1,2,$ where $X'\ra X$ is the canonical projection. This is equal to $Q_1 \xrightarrow{q_i} X\xrightarrow{F_X} X.$ We have the following commutative diagram
$$
\begin{tikzcd}
Q_1 \ar{r}{q_i} \ar{dr}{F_{Q_1}} \ar[swap]{d}& X \ar{r}{F_X} & X \\
\underline{Q_1} \ar[hook]{r} & Q_1 \ar{ur}{q_i} & 
\end{tikzcd}
$$
By (\cite{SF2} 6.5), the compositions $\underline{Q_1} \hookrightarrow Q_1 \xrightarrow{q_1} X$ and $\underline{Q_1} \hookrightarrow Q_1 \xrightarrow{q_2} X$ are equal. The result follows.
\end{proof}

\begin{proposition}
There exists a canonical morphism
$$C_{X/\frakS,\upmu}^{-1}\frakB_X \ra \frakA_X,$$
where $\frakB_X$ and $\frakA_X$ are defined in \ref{defQstrat} and \ref{propBglue} and $C_{X/\frakS,\upmu}$ is defined in \eqref{diagkraz2}.
\end{proposition}

\begin{proof}
Assume $(H_2)$ is satisfied. Consider the morphism of Hopf algebras $\calV:\calR_1' \ra \calQ_1$ induced by the morphism $\nu:Q_{\frakX}\ra R_{\frakX',1}$ given in \eqref{nu}. By \ref{propkraz1934}, we have a 2-commutative diagram
\begin{equation}\label{diag186}
\begin{tikzcd}
\begin{Bmatrix}\text{Crystals\ of\ }\widetilde{\calE}(X'/\frakS)_{/\upmu'} \end{Bmatrix} \ar{r}{C^{-1}_{X/\frakS,\upmu}} \ar[swap, sloped]{d}{\sim} & \begin{Bmatrix}\text{Crystals\ of\ }\widetilde{\underline{\calE}}(X/\frakS)_{/\underline{\upmu}} \end{Bmatrix}  \ar[sloped]{d}{\sim} \\
\begin{Bmatrix}\Ox_{X',\upmu_X}\text{-modules\ with\ an}\\ \calR_{1,\upmu_X}'\text{-stratification} \end{Bmatrix} \ar{r}{\Psi_{\upmu}} & \begin{Bmatrix}\Ox_{X,\upmu_X}\text{-modules\ with\ a}\\ \calQ_{1,\upmu_X}\text{-stratification} \end{Bmatrix}
\end{tikzcd}
\end{equation}
The lower arrow $\Psi_{\upmu}$ is defined by
$$\Psi_{\upmu}(\calE',\epsilon')=\left (F_1^*\calE',\calV^*\epsilon'\right )=\left (\Ox_{X,\upmu_X}\otimes_{\Ox_{X',\upmu_X}}\calE',\calV^*\epsilon'\right ).$$
Consider the morphism
$$
\sigma_R:\begin{array}[t]{clc}
\calR_{1,\upmu_X}' \otimes_{\Ox_{X',\upmu_X}} \calB_{X/S} & \ra & \calB_{X/S} \otimes_{\Ox_{X',\upmu_X}} \calR_{1,\upmu_X}'\\
a\otimes b & \mapsto & b\otimes a.
\end{array}
$$
Clearly, $\sigma_R$ is an $\calR_{1,\upmu_X}'$-stratification on the $\Ox_{X',\upmu_X}$-module $\calB_{X/S}.$
By \ref{qF}, we have a canonical isomorphism
$$
\calQ_{1,\upmu_X} \otimes_{\Ox_{X',\upmu_X}}\calB_{X/S} 
\xrightarrow{\sim} \calB_{X/S}\otimes_{\Ox_{X',\upmu_X}}\calQ_{1,\upmu_X}.
$$
Consider the composition
\begin{equation}\label{eqsigmaQkraz}
\begin{alignedat}{2}
\sigma_Q:
\calQ_{1,\upmu_X} \otimes_{\Ox_{X,\upmu_X}} (F_1^*\calB_{X/S}) &\xrightarrow{\sim} \calQ_{1,\upmu_X} \otimes_{\Ox_{X',\upmu_X}}\calB_{X/S} \\
&\xrightarrow{\sim} \calB_{X/S}\otimes_{\Ox_{X',\upmu_X}}\calQ_{1,\upmu_X} \\
&\xrightarrow{\sim} (F_1^*\calB_{X/S}) \otimes_{\Ox_{X,\upmu_X}} \calQ_{1,\upmu_X}.
\end{alignedat}
\end{equation}
Clearly, $\sigma_Q$ is a $\calQ_{1,\upmu_X}$-stratification on the $\Ox_{X,\upmu_X}$-module $(F_1^*\calB_{X/S})$ and we have
$$
\Psi_{\upmu}\left (\calB_{X/S},\sigma_R\right )=\left (F_1^*\calB_{X/S},\sigma_Q\right ).
$$
Note that $\Psi_{\upmu}\left (\calB_{X/S},\sigma_R\right )$ is independant of the choice of the lifting $F.$
By \ref{Kokoprop1925}, the diagram
$$
\begin{tikzcd}
\calQ_{1,\upmu_X} \otimes_{\Ox_{X',\upmu_X}} \calB_{X/S} \ar{d} \ar{r}{\sigma_Q} & \calB_{X/S} \otimes_{\Ox_{X',\upmu_X}} \calQ_{1,\upmu_X} \ar{d} \\
\calQ_{1,\upmu_X}\otimes_{\Ox_{X,\upmu_X}} \calA_X \ar{r}{\epsilon_{\calA_X,Q}} & \calA_X \otimes_{\Ox_{X,\upmu_X}}\calQ_{1,\upmu_X},
\end{tikzcd}
$$
where the vertical arrows are the canonical ones, is commutative.
It follows that the canonical morphism
$$\Ox_{X,\upmu_X}\otimes_{\Ox_{X',\upmu_X}}\calB_{X/S} \ra \calA_X$$
is compatible with the $\calQ_{1,\upmu_X}$-stratifications. We deduce a canonical morphism, independant of the lifting $F,$ of $\upmu$-indexed algebras
$C_{X/\frakS,\upmu}^{-1}\frakB_X \ra \frakA_X.$
\end{proof}

\begin{parag}
The functor \eqref{diagkraz2}
$$
C^{-1}_{X/\frakS,\upmu}:\begin{Bmatrix}\text{Crystals\ of\ }\widetilde{\calE}(X'/\frakS)_{/\upmu'} \end{Bmatrix} \ra  \begin{Bmatrix}\text{Crystals\ of\ }\widetilde{\underline{\calE}}(X/\frakS)_{/\underline{\upmu}} \end{Bmatrix}
$$
induces a functor
\begin{equation}\label{CESTeq}
C^{-1}_{X/\frakS,\upmu}:\begin{Bmatrix}\text{Crystals\ of\ }\upmu'\text{-indexed}\\ \frakB_X\text{-modules\ of\ }\widetilde{\calE}(X'/\frakS)_{/\upmu'} \end{Bmatrix} \ra \begin{Bmatrix}\text{Crystals\ of\ }\underline{\upmu}\text{-indexed}\\ C_{X/\frakS,\upmu}^{-1}\frakB_X\text{-modules\ of\ }\widetilde{\underline{\calE}}(X/\frakS)_{/\underline{\upmu}} \end{Bmatrix}.
\end{equation}
By \ref{propcrystensstar}, we have a well-defined linearization functor
\begin{equation}\label{ffL}
L_{\mathrm{crys}}:\begin{array}[t]{clc}
\begin{Bmatrix}\text{Crystals\ of\ }\underline{\upmu}\text{-indexed}\\ C_{X/\frakS,\upmu}^{-1}\frakB_X\text{-modules\ of\ }\widetilde{\underline{\calE}}(X/\frakS)_{/\underline{\upmu}} \end{Bmatrix} & \ra & \begin{Bmatrix}\text{Crystals\ of\ }\underline{\upmu}\text{-indexed}\\ \frakA_X\text{-modules\ of\ }\widetilde{\underline{\calE}}(X/\frakS)_{/\underline{\upmu}} \end{Bmatrix} \\
\mathfrak{E} & \mapsto & \frakA_X \circledast_{C_{X/\frakS,\upmu}^{-1}\frakB_X}\mathfrak{E}.
\end{array}
\end{equation}
Set
\begin{equation}\label{ffC}
C^{-1}=L_{\mathrm{crys}}\circ C_{X/\frakS,\upmu}^{-1}:\begin{array}[t]{clc}
\begin{Bmatrix}\text{Crystals\ of\ }\upmu'\text{-indexed}\\ \frakB_X\text{-modules\ of\ }\widetilde{\calE}(X'/\frakS)_{/\upmu'} \end{Bmatrix} & \ra & \begin{Bmatrix}\text{Crystals\ of\ }\underline{\upmu}\text{-indexed}\\ \frakA_X\text{-modules\ of\ }\widetilde{\underline{\calE}}(X/\frakS)_{/\underline{\upmu}} \end{Bmatrix} \\
\mathfrak{E} & \mapsto & \frakA_X \circledast_{C_{X/\frakS,\upmu}^{-1}\frakB_X}C_{X/\frakS,\upmu}^{-1}\mathfrak{E}.
\end{array}
\end{equation}
\end{parag}

\begin{definition}~  
\begin{enumerate}
\item A $\upmu_X$-indexed $\calB_{X/S}$-module $\calE'$ equipped with an $\calR_{1,\upmu_X}'$-stratification
$$\epsilon':\calR_{1,\upmu_X}'\otimes_{\Ox_{X',\upmu_X}}\calE' \ra \calE'\otimes_{\Ox_{X',\upmu_X}}\calR_{1,\upmu_X}'$$
is said to be \emph{admissible} if $\epsilon'$ is $\calB_{X/S}$-linear.
\item Let $\frakA$ be a $\upmu_X$-indexed algebra equipped with a $\calQ_{1,\upmu_X}$-stratification $\epsilon_{\frakA}$ and $\calE$ a $\upmu_X$-indexed $\frakA$-module equipped with a $\calQ_{1,\upmu_X}$-stratification
$$\epsilon:\calQ_{1,\upmu_X}\otimes_{\Ox_{X,\upmu_X}}\calE \ra \calE\otimes_{\Ox_{X,\upmu_X}}\calQ_{1,\upmu_X}.$$
We have a morphism
\begin{alignat*}{2}
&\epsilon_{\frakA}\boxtimes \op{Id}_{\calE}=p_1^*\epsilon_{\frakA}\otimes p_2^*\op{Id}_{\calE}:\\ 
&\calQ_{1,\upmu_X^2}\otimes_{\Ox_{X,\upmu_X^2}} p_1^*\frakA \otimes_{\Ox_{X,\upmu_X^2}}p_2^*\calE &\ra p_1^*\frakA \otimes_{\Ox_{X,\upmu_X^2}}\calQ_{1,\upmu_X^2}\otimes_{\Ox_{X,\upmu_X^2}}p_2^*\calE,
\end{alignat*}
which identifies with
$$
\epsilon_{\frakA}\boxtimes \op{Id}_{\calE}:\calQ_{1,\upmu_X^2}\otimes_{\Ox_{X,\upmu_X^2}}\left (\frakA \boxtimes  \calE\right ) \ra \frakA \boxtimes \left (\calQ_{1,\upmu_X}\otimes_{\Ox_{X,\upmu_X}}\calE\right ).
$$
Similarly, we have a morphism
$$
\op{Id}_{\frakA} \boxtimes \epsilon:\frakA \boxtimes (\calQ_{1,\upmu_X}\otimes_{\Ox_{X,\upmu_X}} \calE) \ra (\frakA \boxtimes \calE)\otimes_{\Ox_{X,\upmu_X^2}}\calQ_{1,\upmu_X^2}.
$$
The stratification $\epsilon$ is said to be \emph{$\frakA$-admissible} if it is compatible with $\epsilon_{\frakA}$ i.e. if the diagram
\begin{equation}\label{admindQstratdiag}
\begin{tikzcd}
\calQ_{1,\upmu_X^2}\otimes_{\Ox_{X,\upmu_X^2}}(\frakA\boxtimes \calE) \ar{d}{\epsilon_{\frakA}\boxtimes \op{Id}_{\calE}} \ar{rr}{\op{Id}_{\calQ_{1,\upmu_X^2}}\otimes \pi} & & \calQ_{1,\upmu_X^2}\otimes_{\Ox_{X,\upmu_X^2}} \sigma_X^*\calE \ar{dd}{\sigma_X^*\epsilon} \\
\frakA \boxtimes (\calQ_{1,\upmu_X}\otimes_{\Ox_{X,\upmu_X}} \calE) \ar{d}{\op{Id}_{\frakA}\boxtimes \epsilon} & & \\
(\frakA \boxtimes \calE)\otimes_{\Ox_{X,\upmu_X^2}}\calQ_{1,\upmu_X^2}\ar{rr}{\pi\otimes \op{Id}_{\calQ_{1,\upmu_X^2}}} & & \sigma_X^*\calE\otimes_{\Ox_{X,\upmu_X^2}}\calQ_{1,\upmu_X^2}
\end{tikzcd}
\end{equation}
is commutative, where $\sigma_X:\upmu_X^2 \ra \upmu_X$ is the addition map and $\pi:\frakA \boxtimes \calE \ra \sigma_X^*\calE$ is the morphism defining the $\upmu_X$-indexed $\frakA$-module structure on $\calE.$
If $\frakA=\calA_X$ and $\epsilon_{\frakA}=\epsilon_{\calA_X,Q}$ \eqref{defQstrat}, then we say that $\epsilon$ is \emph{admissible}.
\end{enumerate}
\end{definition}

\begin{proposition}\label{equivRQAB}
Assume $(H_1)$ is satisfied. The canonical equivalence of categories between the category of crystals of $\widetilde{\underline{\calE}}(X/\frakS)_{/\underline{\upmu}}$ (resp. $\widetilde{\calE}(X'/\frakS)_{/\upmu'}$) and the category of $\Ox_{X,\upmu_X}$-modules equipped with a $\calQ_{1,\upmu_X}$-stratification (resp. $\Ox_{X',\upmu_X}$-modules equipped with an $\calR_{1,\upmu_X}'$-stratification), given in \ref{equivRQind}, induces a canonical equivalence of categories between the full subcategory of crystals of $\upmu$-indexed $\frakA_X$-modules of $\widetilde{\underline{\calE}}(X/\frakS)_{/\underline{\upmu}}$ (resp. crystals of $\upmu'$-indexed $\frakB_X$-modules of $\widetilde{\calE}(X'/\frakS)_{/\upmu'}$) and the full subcategory of $\upmu_X$-indexed $\calA_X$-modules equipped with an admissible $\calQ_{1,\upmu_X}$-stratification (resp. $\upmu_X$-indexed $\calB_{X/S}$-modules equipped with an admissible $\calR_{1,\upmu_X}'$-stratification).
\end{proposition}

\begin{proof}
Let $\frakE$ be a crystal of $\widetilde{\underline{\calE}}(X/\frakS)_{/\underline{\upmu}}$ and $(\F,\epsilon)$ the corresponding $\Ox_{X,\upmu_X}$-module with a $\calQ_{1,\upmu_X}$-stratification. We take again the notation of \ref{parag1419}. To prove the proposition, it is sufficient to prove that $\frakE$ is equipped with a structure of $\upmu$-indexed $\frakA_X$-module if and only if $\epsilon$ is admissible. Suppose that $\frakE$ is a $\upmu$-indexed $\frakA_X$-module. By \ref{parag1419}, $\F=\frakE_{(X,\frakX,i_X)},$ where $i_X:\underline{X} \ra X$ is the canonical immersion, and $\epsilon$ is the composition
$$\epsilon:\calQ_{1,\upmu_X} \otimes_{\Ox_{X,\upmu_X}} \mathfrak{E}_{(X,\frakX,i_X)} \xrightarrow{c_{\mathfrak{E},q_2}} \mathfrak{E}_{(X,Q_{\frakX},\lambda_Q)} \xrightarrow{c_{\mathfrak{E},q_1}^{-1}} \mathfrak{E}_{(X,\frakX,i_X)}\otimes_{\Ox_{X,\upmu_X}}\calQ_{1,\upmu_X}.$$
By \ref{prop1415} and \eqref{eq14201}, $\F$ is a $\upmu_X$-indexed $\calA_X$-module. Let $\sigma:\upmu^2\ra \upmu$ and $\sigma_X:\upmu_X^2 \ra \upmu_X$ be the addition maps and $\frakA_X \boxtimes \frakE \ra \sigma^*\frakE$ the morphism defining the $\upmu$-indexed $\frakA_X$-algebra structure on $\frakE.$ This morphism and the projections $q_1$ and $q_2$ yield the commutative diagram
$$
\begin{tikzcd}
\calQ_{1,\upmu_X^2} \otimes_{\Ox_{X,\upmu_X^2}}\left (\frakA_X \boxtimes \frakE \right )_{(X,\frakX,i_X)} \ar[swap]{d}{c_{\frakA_X\boxtimes \frakE,q_2}} \ar{r} & \calQ_{1,\upmu_X^2} \otimes_{\Ox_{X,\upmu_X^2}}\left ( \left (\sigma^*\frakE\right )_{(X,\frakX,i_X)}\right ) \ar{d}{c_{\sigma^*\frakE,q_2}} \\
\left (\frakA_X\boxtimes\frakE \right )_{(X,Q_{\frakX},\lambda_Q)} \ar{r} \ar[swap]{d}{c_{\frakA_X\boxtimes \frakE,q_1}^{-1}} & \left (\sigma^*\frakE \right )_{(X,Q_{\frakX},\lambda_Q)} \ar{d}{c_{\sigma^*\frakE,q_1}^{-1}} \\
\left (\frakA_X \boxtimes \frakE \right )_{(X,\frakX,i_X)} \otimes_{\Ox_{X,\upmu_X^2}}\calQ_{1,\upmu_X^2}  \ar{r} & \left ( \left (\sigma^*\frakE\right )_{(X,\frakX,i_X)}\right )\otimes_{\Ox_{X,\upmu_X^2}}\calQ_{1,\upmu_X^2}.
\end{tikzcd}
$$
This proves that $\epsilon$ is admissible.

Conversly, suppose that $\F$ is a $\upmu_X$-indexed $\calA_X$-module and that $\epsilon$ is admissible. The stratifications $\epsilon$ and $\epsilon_{\calA_X,Q}$ \eqref{eqeAQ} canonically induce stratifications on $\calA_X\boxtimes \F$ and $\sigma_X^*\F.$ The crystals corresponding to $\calA_X \boxtimes \F$ and $\sigma_X^*\F$ are $\frakA_X \boxtimes \frakE$ and $\sigma^*\frakE$ respectively. The admissibility of $\epsilon$ proves that the morphism $\calA_X\boxtimes \F \ra \sigma_X^*\F,$ defining the $\upmu_X$-indexed $\calA_X$-algebra structure on $\F,$ is compatible with the stratifications of $\calA_X\boxtimes \F$ and $\F$ and hence induces, by \ref{equivRQind}, a morphism of crystals
$$
\frakA_X \boxtimes \frakE \ra \sigma^*\frakE.
$$
This yields the desired $\upmu$-indexed $\frakA_X$-module structure on $\frakE.$ These two constructions are quasi-inverse to each other. The proof for crystals of $\widetilde{\calE}(X'/\frakS)_{/\upmu'}$ is similar.
\end{proof}

\begin{lemma}\label{ommek1}
Let $\calP$ be a presheaf of $\Ox_X$-modules, $\calE$ an $\Ox_X$-module and $\F$ the presheaf
$$
\F:U\mapsto \calE(U) \otimes_{\Ox_X(U)}\calP(U).
$$
Then, there exists a canonical isomorphism
$$
a\F \xrightarrow{\sim} \calE \otimes_{\Ox_X} a\calP,
$$
where $a$ is the associated sheaf functor.
\end{lemma}

\begin{proof}
The canonical morphism of presheaves $\calP \ra a\calP$ induces a canonical morphism of presheaves
$$
 \F \ra \calE \otimes_{\Ox_X}a\calP,
$$
and then a canonical morphism of $\Ox_X$-modules
\begin{equation}\label{omek2}
a\F \ra \calE \otimes_{\Ox_X} a\calP.
\end{equation}
Conversly, since $a$ is left exact (\cite{SGA43} II 3.4), the canonical $\Ox_X$-bilinear morphism of presheaves
$$
\calE \times  \calP \ra \F \ra a\F
$$
factors through $\calE \times a\calP.$ It then induces a canonical $\Ox_X$-linear morphism
\begin{equation}\label{omek3}
\calE\otimes_{\Ox_X}a\calP \ra a\F.
\end{equation}
By checking the stalks, the morphisms \eqref{omek2} and \eqref{omek3} are inverse to each other.
\end{proof}

\begin{proposition}\label{propdraft1}
Assume $(H_1)$ is satisfied. Let $(\calE_1,\epsilon_1)$ and $(\calE_2,\epsilon_2)$ be two $\Ox_{X,\upmu_X}$-modules equipped with $\calQ_{1,\upmu_X}$-stratifications.
Then $\epsilon=p_1^*\epsilon_1 \otimes_{\calQ_{1,\upmu_X^2}} p_2^*\epsilon_2$ (note that we tensor over $\calQ_{1,\upmu_X^2}$) is a $\calQ_{1,\upmu_X^2}$-stratification on $\calE=\calE_1 \boxtimes \calE_2.$
\end{proposition}

\begin{proof}
First, we have a commutative diagram
$$
\begin{tikzcd}
\left (\calQ_{1,\upmu_X^2} \otimes_{\Ox_{X,\upmu_X^2}} p_1^*\calE_1 \right ) \otimes_{\calQ_{1,\upmu_X^2}} \left (\calQ_{1,\upmu_X^2} \otimes_{\Ox_{X,\upmu_X^2}} p_2^*\calE_2 \right ) \ar{d}{\epsilon}  \ar{rr}{\sim} & & \calQ_{1,\upmu_X^2} \otimes_{\Ox_{X,\upmu_X^2}} \left (\calE_1 \boxtimes \calE_2 \right ) \ar{d}\\
\left (p_1^*\calE_1 \otimes_{\Ox_{X,\upmu_X^2}} \calQ_{1,\upmu_X^2} \right ) \otimes_{\calQ_{1,\upmu_X^2}} \left (p_2^*\calE_2 \otimes_{\Ox_{X,\upmu_X^2}} \calQ_{1,\upmu_X^2} \right ) \ar{rr}{\sim} & & \left (\calE_1 \boxtimes \calE_2 \right )\otimes_{\Ox_{X,\upmu_X^2}} \calQ_{1,\upmu_X^2},
\end{tikzcd}
$$
where the horizontal isomorphisms are the canonical ones. This way, we consider $\epsilon$ as an isomorphism
$$
\epsilon:\calQ_{1,\upmu_X^2} \otimes_{\Ox_{X,\upmu_X^2}} \left (\calE_1 \boxtimes \calE_2 \right ) \ra \left (\calE_1 \boxtimes \calE_2 \right )\otimes_{\Ox_{X,\upmu_X^2}} \calQ_{1,\upmu_X^2}.
$$
To prove that it is a stratification, the only difficult part is to check the cocyclicity condition. The cocyclicity condition satisfied by $\epsilon_i$ implies the commutativity of the diagram
$$
\begin{tikzcd}
\calQ_{1,\upmu_X^2} \otimes \calQ_{1,\upmu_X^2} \otimes p_i^* \calE_i \ar{rr}{\delta_{\upmu_X^2}^*p_i^*\epsilon_i} \ar[swap]{dr}{\op{Id}\otimes p_i^*\epsilon_i} & & p_i^*\calE_i \otimes \calQ_{1,\upmu_X^2}\otimes \calQ_{1,\upmu_X^2}, \\
 & \calQ_{1,\upmu_X^2} \otimes p_i^*\calE_i \otimes \calQ_{1,\upmu_X^2} \ar[swap]{ur}{p_i^*\epsilon_i \otimes \op{Id}} &
\end{tikzcd}
$$
where the tensor products are over $\Ox_{X,\upmu_X^2}$ and $\delta$ is the comultiplication map of the Hopf algebra $\calQ_1,$ given in (\cite{SF2} 6.16). Tensoring both diagrams for $i=1,2$ over $\calQ_{1,\upmu_X^2},$ we get the cocycle condition for $\epsilon.$
\end{proof}

\begin{proposition}\label{propdraft2}
Assume $(H_1)$ is satisfied. Let $(\calE,\epsilon)$ be an $\Ox_{X,\upmu_X^2}$-module equipped with a $\calQ_{1,\upmu_X^2}$-stratification. Then $\sigma_!\epsilon$ is a $\calQ_{1,\upmu_X}$-stratification on $\sigma_!\calE.$ 
\end{proposition}

\begin{proof}
By (\cite{SGA43} 12.11 b), we have canonical isomorphisms
$$
\sigma_! \left (\calQ_{1,\upmu_X} \otimes  \calE \right ) \xrightarrow{\sim} \calQ_{1,\upmu_X} \otimes  \sigma_!\calE,\ \sigma_! \left (\calE \otimes \calQ_{1,\upmu_X} \right ) \xrightarrow{\sim} \sigma_!\calE \otimes \calQ_{1,\upmu_X}.
$$
Via these isomorphisms, $\sigma_!\epsilon$ identifies with an isomorphism
$$
\sigma_!\epsilon:\calQ_{1,\upmu_X} \otimes  \sigma_!\calE \ra \sigma_!\calE \otimes \calQ_{1,\upmu_X}.
$$
We start by checking the cocycle condition: applying $\sigma_!$ to the cocyclicity diagram of $\epsilon,$ we get a commutative diagram
$$
\begin{tikzcd}
\sigma_! \left (\calQ_{1,\upmu_X^2} \otimes \calQ_{1,\upmu_X^2} \otimes  \calE \right )\ar{rr}{\sigma_!\delta_{\upmu_X^2}^*\epsilon} \ar[swap]{dr}{\sigma_! \left (\op{Id}\otimes \epsilon\right )} & & \sigma_! \left (\calE \otimes \calQ_{1,\upmu_X^2}\otimes \calQ_{1,\upmu_X^2}\right ), \\
 & \sigma_! \left (\calQ_{1,\upmu_X^2} \otimes \calE \otimes \calQ_{1,\upmu_X^2}\right ) \ar[swap]{ur}{\sigma_! \left (\epsilon \otimes \op{Id}\right )} &
\end{tikzcd}
$$
where the tensor products are over $\Ox_{X,\upmu_X^2}$ and $\delta$ is the comultiplication map of the Hopf algebra $\calQ_1,$ given in (\cite{SF2} 6.16). By (\cite{SGA43} 12.11 b), we have a canonical isomorphism
$$
\sigma_! \left (\calQ_{1,\upmu_X^2} \otimes \calQ_{1,\upmu_X^2} \otimes  \calE \right ) \xrightarrow{\sim} \calQ_{1,\upmu_X} \otimes \calQ_{1,\upmu_X} \otimes  \sigma_!\calE,
$$
where the tensor products in the right-hand side are over $\Ox_{X,\upmu_X}.$ By similar isomorphisms for the other terms of the diagram, we deduce the commutativity of the diagram
$$
\begin{tikzcd}
\calQ_{1,\upmu_X} \otimes \calQ_{1,\upmu_X} \otimes  \sigma_!\calE \ar{rr}{\sigma_!\delta_{\upmu_X^2}^*\epsilon} \ar[swap]{dr}{\op{Id}\otimes \sigma_!\epsilon} & & \sigma_!\calE \otimes \calQ_{1,\upmu_X}\otimes \calQ_{1,\upmu_X}. \\
 & \calQ_{1,\upmu_X} \otimes \sigma_!\calE \otimes \calQ_{1,\upmu_X} \ar[swap]{ur}{\sigma_!\epsilon \otimes \op{Id}} &
\end{tikzcd}
$$
It remains to prove that $\sigma_!\delta_{\upmu_X^2}^*\epsilon$ identifies with $\delta_{\upmu_X}^*\sigma_!\epsilon.$
We have a commutative diagram of ringed topoi
$$
\begin{tikzcd}
\left (X_{\text{ét}/\upmu_X^2}, \calQ_{1,\upmu_X^2}\otimes \calQ_{1,\upmu_X^2} \right ) \ar[swap,bend right=80]{dd}{j_{\upmu_X^2}} \ar{rr}{\delta_{\upmu_X^2}} \ar[swap]{d}{\sigma} & & \left (X_{\text{ét}/\upmu_X^2}, \calQ_{1,\upmu_X^2} \right ) \ar{d}{\sigma} \ar[bend right=-80]{dd}{j_{\upmu_X^2}} \\
\left (X_{\text{ét}/\upmu_X}, \calQ_{1,\upmu_X} \otimes \calQ_{1,\upmu_X} \right ) \ar[swap]{d}{j_{\upmu_X}} \ar{rr}{\delta_{\upmu_X}} & & \left (X_{\text{ét}/\upmu_X}, \calQ_{1,\upmu_X} \right ) \ar{d}{j_{\upmu_X}} \\
\left (X_{\text{ét}}, \calQ_{1}\otimes \calQ_{1} \right ) \ar{rr}{\delta} & & \left (X_{\text{ét}}, \calQ_{1} \right ).
\end{tikzcd}
$$
By \ref{lemtop15} (2), $\sigma_!\delta_{\upmu_X^2}^*\epsilon$ identifies with $\delta_{\upmu_X}^*\sigma_!\epsilon.$ The cocyclicity condition of $\sigma_!\epsilon$ is hence proved. To prove that $\sigma_!\epsilon$ is a stratification, it remains to check that
$$
\pi_{\upmu_X}^*\sigma_!\epsilon=\op{Id}_{\sigma_!\calE},
$$
where $\pi:\Ox_X\ra \calQ_1$ is the counit of the Hopf algebra $\calQ_1.$ For that, it is sufficient to prove that $\pi_{\upmu_X}^*\sigma_!\epsilon$ identifies with $\sigma_!\pi_{\upmu_X^2}^*\epsilon.$ This is again proved using \ref{lemtop15} (2) and the commutative diagram
$$
\begin{tikzcd}
\left (X_{\text{ét}/\upmu_X^2},  \calQ_{1,\upmu_X^2} \right ) \ar{rr}{\pi_{\upmu_X^2}} \ar[swap]{d}{\sigma} & & \left (X_{\text{ét}/\upmu_X^2}, \Ox_{X,\upmu_X^2} \right ) \ar{d}{\sigma} \\
\left (X_{\text{ét}/\upmu_X}, \calQ_{1,\upmu_X} \right ) \ar{rr}{\pi_{\upmu_X}} & & \left (X_{\text{ét}/\upmu_X}, \Ox_{X,\upmu_X} \right ).
\end{tikzcd}
$$
\end{proof}

\begin{corollaire}\label{cordraft1}
Assume $(H_1)$ is satisfied. Let $(\calE_1,\epsilon_1)$ and $(\calE_2,\epsilon_2)$ be two $\Ox_{X,\upmu_X}$-modules equipped with $\calQ_{1,\upmu_X}$-stratifications and $\epsilon=p_1^*\epsilon_1 \otimes_{\calQ_{1,\upmu_X^2}} p_2^*\epsilon_2$ the stratification on $\calE=\calE_1 \boxtimes \calE_2$ as in \ref{propdraft1}. Then $\sigma_!\epsilon$ is a $\calQ_{1,\upmu_X}$-stratification on $\sigma_!\calE.$
\end{corollaire}

\begin{proof}
This is a consequence of \ref{propdraft1} and \ref{propdraft2}.
\end{proof}

\begin{proposition}\label{propdraft3}
Assume $(H_1)$ is satisfied. Let $\frakA$ be a $\upmu_X$-indexed algebra and $(\calE_1,\epsilon_1)$ and $(\calE_2,\epsilon_2)$ two $\upmu_X$-indexed $\frakA$-modules equipped with $\calQ_{1,\upmu_X}$-stratifications and $\epsilon=p_1^*\epsilon_1 \otimes_{\calQ_{1,\upmu_X^2}} p_2^*\epsilon_2$ the stratification on $\calE=\calE_1 \boxtimes \calE_2$ as in \ref{propdraft1}. If $\epsilon_1$ and $\epsilon_2$ are $\frakA$-linear, then the stratification $\sigma_!\epsilon$ \eqref{cordraft1} induces a $\calQ_{1,\upmu_X}$-stratification
$$
\epsilon_1 \circledast_{\calQ_{1,\upmu_X}\otimes_{\Ox_{X,\upmu_X}}\frakA}\epsilon_2:
\calQ_{1,\upmu_X} \otimes_{\Ox_{X,\upmu_X}}\left (\calE_1\circledast_{\frakA}\calE_2 \right ) \ra \left (\calE_1\circledast_{\frakA}\calE_2 \right )\otimes_{\Ox_{X,\upmu_X}}\calQ_{1,\upmu_X}
$$
on $\calE_1\circledast_{\frakA}\calE_2.$
\end{proposition}

\begin{proof}
Let $s\in \Gamma(U,\upmu_X)$ be a local section over an étale $X$-scheme $U.$ By \ref{indpropshriek} (3), the fiber $\sigma_!\left (\calE_1 \boxtimes \calE_2 \right )_s$ is the sheaf associated to the presheaf $\calP$ on $\text{ét}_{/U}$ defined by
$$
\calP:V \mapsto \bigoplus_{\substack{\alpha,\beta \in \Gamma(V,\upmu_X) \\ \alpha+\beta=s_{|V}}} \calE_{1,\alpha}(V)\otimes_{\Ox_{X}(V)}\calE_{2,\beta}(V).
$$
By \ref{ommek1}, the fiber $\left (\calQ_{1,\upmu_X} \otimes_{\Ox_{X,\upmu_X}} \sigma_!\left (\calE_1\boxtimes \calE_2 \right ) \right )_s=\calQ_{1|U} \otimes_{\Ox_U} \sigma_!\left (\calE_1 \boxtimes \calE_2\right )_s$ is then the sheaf associated to the presheaf on $\text{ét}_{/U}$ defined by
$$
V \mapsto \calQ_1(V) \otimes_{\Ox_X(V)}\bigoplus_{\substack{\alpha,\beta \in \Gamma(V,\upmu_X) \\ \alpha+\beta=s_{|V}}} \calE_{1,\alpha}(V)\otimes_{\Ox_{X}(V)}\calE_{2,\beta}(V).
$$
The morphism $\left (\sigma_!\epsilon \right )_s$ is then associated to the morphism of presheaves
$$
\varphi_s:\calQ_{1|U} \otimes_{\Ox_U}\calP \ra \calP\otimes_{\Ox_U} \calQ_{1|U}
$$
defined, for every étale $U$-scheme $V,$ by
$$
\begin{tikzcd}
\calQ_1(V) \otimes_{\Ox_X(V)} \left (\displaystyle\bigoplus_{\substack{\alpha,\beta \in \Gamma(V,\upmu_X) \\ \alpha+\beta=s_{|V}}} \calE_{1,\alpha}(V)\otimes_{\Ox_{X}(V)}\calE_{2,\beta}(V) \right ) \ar{d} & 1 \otimes x \otimes y \ar[mapsto]{d} \\
\left (\displaystyle\bigoplus_{\substack{\alpha,\beta \in \Gamma(V,\upmu_X) \\ \alpha+\beta=s_{|V}}} \calE_{1,\alpha}(V)\otimes_{\Ox_{X}(V)}\calE_{2,\beta}(V) \right ) \otimes_{\Ox_X(V)} \calQ_1(V) & \epsilon_{1,\alpha}(1\otimes x)\otimes \epsilon_{2,\beta}(1\otimes y),
\end{tikzcd}
$$
where $x$ and $y$ are local sections of $\calE_{1,\alpha}$ and $\calE_{2,\beta}$ over $V$ and
$$
\epsilon_{1,\alpha}(1\otimes x)\otimes \epsilon_{2,\beta}(1\otimes y)
$$
is a local section of
$$
\left (\calE_{1,\alpha}\otimes_{\Ox_U} \calQ_{1|U}\right ) \otimes_{\calQ_{1|U}} \left (\calE_{2,\beta} \otimes_{\Ox_U} \calQ_{1|U} \right ) \xrightarrow{\sim} \left (\calE_{1,\alpha}\otimes_{\Ox_U} \calE_{2,\beta} \right )\otimes_{\Ox_U} \calQ_{1|U}.
$$
For every $\alpha,\beta,\gamma \in \Gamma(V,\upmu_X)$ such that $\alpha+\beta+\gamma=s_{|V}$ and $x\in \calE_{1,\alpha}(V),$ $y\in \calE_{2,\beta}(V),$ $a\in \frakA_{\gamma}(V),$ we have
\begin{alignat*}{2}
\varphi_s(1\otimes x\otimes (by)) &= \epsilon_{1,\alpha}(1\otimes x)\otimes \epsilon_{2,\beta+\gamma}(1\otimes ay) \\
&=\epsilon_{1,\alpha}(1\otimes x)\otimes a\epsilon_{2,\beta}(1\otimes y), \\
\varphi_s(1\otimes (bx)\otimes y) &= \epsilon_{1,\alpha+\gamma}(1\otimes ax)\otimes \epsilon_{2,\beta}(1\otimes y) \\
&= \left (a\epsilon_{1,\alpha}(1\otimes x)\right )\otimes \epsilon_{2,\beta}(1\otimes y).
\end{alignat*}
We deduce that, when composing $\varphi_s$ with the canonical morphism
$$\calP\otimes_{\Ox_U} \calQ_{1|U} \ra \left (\calE_1 \circledast_{\frakA}\calE_2 \right )_s \otimes \calQ_{1|U},$$
the images of $\varphi_s(1\otimes x\otimes (by))$ and $\varphi_s(1\otimes (bx)\otimes y)$ are equal.
It follows, by \ref{indpropshriek} (3), that $\sigma_!\epsilon$ induces a $\calQ_{1,\upmu_X}$-linear isomorphism
\begin{equation}
\epsilon_1 \circledast_{\calQ_{1,\upmu_X}\otimes_{\Ox_{X,\upmu_X}}\frakA}\epsilon_2:
\calQ_{1,\upmu_X} \otimes_{\Ox_{X,\upmu_X}}\left (\calE_1\circledast_{\frakA}\calE_2 \right ) \ra \left (\calE_1\circledast_{\frakA}\calE_2 \right )\otimes_{\Ox_{X,\upmu_X}}\calQ_{1,\upmu_X}.
\end{equation}
The fact that it is a stratification follows from the fact that $\sigma_!\epsilon$ is.
\end{proof}

\begin{proposition}\label{propdraft4}
Assume $(H_1)$ is satisfied. Let $\frakB \ra \frakA$ a morphism of $\upmu_X$-indexed algebras and $(\calE,\epsilon)$ a $\upmu_X$-indexed $\frakB$-module equipped with a $\frakB$-linear $\calQ_{1,\upmu_X}$-stratification. Suppose that $\frakA$ is equipped with a $\frakB$-linear $\calQ_{1,\upmu_X}$-stratification $\epsilon_{\frakA}$ satisfying
$$
\epsilon_{\frakA}(1\otimes (ab))=\epsilon_{\frakA}(1\otimes a)\epsilon_{\frakA}(1\otimes b)
$$
for any local sections $a$ and $b$ of $\frakA.$ Then the stratification
$$
\epsilon_1=\epsilon_{\frakA} \circledast_{\frakB \otimes_{\Ox_{X,\upmu_X}} \calQ_{1,\upmu_X}} \epsilon: \calQ_{1,\upmu_X} \otimes_{\Ox_{X,\upmu_X}} \left (\frakA \circledast_{\frakB}\calE \right ) \xrightarrow{\sim} \left (\frakA \circledast_{\frakB}\calE \right )  \otimes_{\Ox_{X,\upmu_X}} \calQ_{1,\upmu_X},
$$
defined in \ref{propdraft3}, is $\frakA$-admissible.
\end{proposition}

\begin{proof}
Let $\epsilon_2=p_1^*\epsilon_{\frakA} \otimes_{\calQ_{1,\upmu_X^2}} p_2^*\epsilon$ be the stratification on $\calE_2=\frakA \boxtimes \calE$ defined as in \ref{propdraft1}. Since $\epsilon_{\frakA}$ and $\epsilon$ are $\frakB$-linear, the stratification $\sigma_!\epsilon_2$ \eqref{cordraft1} induces the $\calQ_{1,\upmu_X}$-stratification $\epsilon_1$ on $\frakA\circledast_{\frakB}\calE.$
We have to prove the commutativity of the diagram \eqref{admindQstratdiag}
\begin{equation*}
\begin{tikzcd}
\calQ_{1,\upmu_X^2}\otimes_{\Ox_{X,\upmu_X^2}}(\frakA\boxtimes \left (\frakA \circledast_{\frakB}\calE \right )) \ar{d}{\epsilon_{\frakA}\boxtimes \op{Id}_{\left (\frakA \circledast_{\frakB}\calE \right )}} \ar{rr}{\op{Id}_{\calQ_{1,\upmu_X^2}}\otimes \pi} & & \calQ_{1,\upmu_X^2}\otimes_{\Ox_{X,\upmu_X^2}} \sigma^*\left (\frakA \circledast_{\frakB}\calE \right ) \ar{dd}{\sigma^*\epsilon_1} \\
\frakA \boxtimes (\calQ_{1,\upmu_X}\otimes_{\Ox_{X,\upmu_X}} \left (\frakA \circledast_{\frakB}\calE \right )) \ar{d}{\op{Id}_{\frakA}\boxtimes \epsilon_1} & & \\
(\frakA\boxtimes \left (\frakA \circledast_{\frakB}\calE \right ))\otimes_{\Ox_{X,\upmu_X^2}}\calQ_{1,\upmu_X^2}\ar{rr}{\pi\otimes \op{Id}_{\calQ_{1,\upmu_X^2}}} & & \sigma^*\left (\frakA \circledast_{\frakB}\calE \right )\otimes_{\Ox_{X,\upmu_X^2}}\calQ_{1,\upmu_X^2},
\end{tikzcd}
\end{equation*}
where $\pi$ is the morphism defining the $\upmu_X$-indexed $\frakA$-algebra structure on $\frakA \circledast_{\frakB} \calE.$
Since $\sigma^*$ is right exact, it is sufficient to prove the commutativity of the diagram
$$
\begin{tikzcd}
\calQ_{1,\upmu_X^2} \otimes_{\Ox_{X,\upmu_X^2}} \left (\frakA \boxtimes\sigma_!\left (\frakA\boxtimes \calE \right )\right ) \ar[swap]{d}{\epsilon_{\frakA}\boxtimes \op{Id}} \ar{r} & \calQ_{1,\upmu_X^2} \otimes_{\Ox_{X,\upmu_X^2}}  \sigma^*\sigma_{!}\left (\frakA \boxtimes \calE\right ) \ar{dd}{\sigma^*\sigma_{!}\epsilon_{2}} 
\\
 \frakA \boxtimes \left (\calQ_{1,\upmu_X} \otimes_{\Ox_{X,\upmu_X}} \sigma_!\left (\frakA\boxtimes\calE\right )\right ) \ar[swap]{d}{\op{Id}\boxtimes \sigma_!\epsilon_{2}} & 
\\
 \left (\frakA \boxtimes \left (\sigma_!\left (\frakA\boxtimes\calE\right )\right )\right )  \otimes_{\Ox_{X,\upmu_X^2}}  \calQ_{1,\upmu_X^2} \ar{r} &  \sigma^*\sigma_{!}\left (\frakA \boxtimes \calE\right ) \otimes_{\Ox_{X,\upmu_X^2}} \calQ_{1,\upmu_X^2},
\end{tikzcd}
$$
where the upper and lower arrows are induced by the $\upmu_X$-indexed $\frakA$-algebra structure on $\sigma_!\left (\frakA\boxtimes \calE\right )$ \eqref{krazzzz}.
For that, it is sufficient to prove the commutativity of the fiber of the previous diagram over all local sections of $\upmu_X^2.$ Let $s,t\in \Gamma(U,\upmu_X)$ be local sections over an étale $X$-scheme $U.$ It is then sufficient to prove the commutativity of
\begin{equation}\label{diagdraft2}
\begin{tikzcd}
\calQ_{1|U} \otimes \frakA_s \otimes \left (\sigma_!\left (\frakA\boxtimes \calE \right )\right )_t \ar[swap]{d}{\epsilon_{\frakA,s}\otimes \op{Id}} \ar{r} & \calQ_{1|U} \otimes \left (\sigma_{!}\left (\frakA \boxtimes \calE\right )\right )_{s+t} \ar{dd}{\left (\sigma_{!}\epsilon_{2}\right )_{s+t}} 
\\
 \frakA_s \otimes \calQ_{1|U} \otimes \left (\sigma_!\left (\frakA\boxtimes\calE\right )\right )_t \ar[swap]{d}{\op{Id}\otimes \left (\sigma_!\epsilon_{2}\right )_t} & 
\\
 \frakA_s \otimes \left (\sigma_!\left (\frakA\boxtimes\calE\right )\right )_t  \otimes \calQ_{1|U} \ar{r} &  \left (\sigma_{!}\left (\frakA \boxtimes \calE\right )\right )_{s+t} \otimes \calQ_{1|U},
\end{tikzcd}
\end{equation}
where the tensor products are over $\Ox_U.$
By \ref{indpropshriek} (3), the fiber $\sigma_!\left (\frakA \boxtimes \calE \right )_s$ is the sheaf associated to the presheaf $\calP_s$ on $\text{ét}_{/U}$ defined by
$$
\calP_s:V \mapsto \bigoplus_{\substack{\alpha,\beta \in \Gamma(V,\upmu_X) \\ \alpha+\beta=s_{|V}}} \frakA_{\alpha}(V)\otimes_{\Ox_{X}(V)}\calE_{\beta}(V).
$$
If $V$ is an étale $U$-scheme, $\alpha,\beta \in \Gamma(V,\upmu_X)$ and $b$ and $y$ are local sections of $\frakA_{\alpha}$ and $\calE_{\beta}$ respectively, then
\begin{alignat*}{2}
\epsilon_{2,(\alpha,\beta)}(1\otimes b\otimes y) = \epsilon_{\frakA,\alpha}(1\otimes b) \otimes \epsilon_{\beta}(1\otimes y) &\in \left ( \frakA_{\alpha} \otimes_{\Ox_V} \calQ_{1|V} \right ) \otimes_{\calQ_{1|V}} \left ( \calE_{\beta} \otimes_{\Ox_V} \calQ_{1|V} \right )\\ &\xrightarrow{\sim} \left (\frakA_{\alpha} \otimes_{\Ox_V} \calE_{\beta} \right )\otimes_{\Ox_V}\calQ_{1|V}.
\end{alignat*}
Then
$$
\left (\sigma_!\epsilon_2 \right )_{\alpha+\beta}(1\otimes b \otimes y)=\epsilon_{\frakA,\alpha}(1\otimes b) \otimes \epsilon_{\beta}(1\otimes y).
$$
Let $\ov{x} \ra U$ be a geometric point. We have then a canonical isomorphism of stalks
$$
\left (\sigma_!\left (\frakA \boxtimes \calE \right )\right )_{t,\ov{x}} \xrightarrow{\sim} \calP_{t,\ov{x}}.
$$
For simplicity, when tensoring over $\Ox_{X,\ov{x}},$ we use the notation $\otimes$ instead of $\otimes_{\Ox_{X,\ov{x}}}.$
To prove the commutativity of \eqref{diagdraft2}, it is sufficient to prove the commutativity of the diagram
\begin{equation}\label{diagdraft1}
\begin{tikzcd}
\calQ_{1,\ov{x}} \otimes \frakA_{s,\ov{x}} \otimes \calP_{t,\ov{x}} \ar[swap]{d}{\epsilon_{\frakA,s,\ov{x}}\otimes \op{Id}} \ar{r} & \calQ_{1,\ov{x}} \otimes \calP_{s+t,\ov{x}} \ar{dd}{\left (\sigma_{!}\epsilon_{2}\right )_{s+t,\ov{x}}} 
\\
 \frakA_{s,\ov{x}} \otimes \calQ_{1,\ov{x}} \otimes \calP_{t,\ov{x}} \ar[swap]{d}{\op{Id}\otimes \left (\sigma_!\epsilon_{2}\right )_{t,\ov{x}}} & 
\\
 \frakA_{s,\ov{x}} \otimes \calP_{t,\ov{x}}  \otimes \calQ_{1,\ov{x}} \ar{r} &  \calP_{s+t,\ov{x}} \otimes \calQ_{1,\ov{x}}.
\end{tikzcd}
\end{equation}
If $V$ is an étale $U$-scheme, $\alpha,\beta \in \Gamma(V,\upmu_X)$ such that $\alpha+\beta=t_{|V},$ and if $a,$ $b$ and $y$ are local sections of $\frakA_{s},$ $\frakA_{\alpha}$ and $\calE_{\beta}$ respectively, then the images of the element $1\otimes a \otimes b\otimes y$ of $\calQ_{1,\ov{x}} \otimes \frakA_{s,\ov{x}} \otimes \calP_{t,\ov{x}}$ by the arrows of the previous diagram are as follows:
$$
\begin{tikzcd}
1\otimes a \otimes b \otimes y \ar[mapsto]{r} \ar[mapsto]{d} & 1\otimes (ab) \otimes y \ar[mapsto]{dd} \\
\epsilon_{\frakA,s,\ov{x}}(1\otimes a) \otimes b \otimes y \ar[mapsto]{d} & \\
\begin{matrix}\epsilon_{\frakA,s,\ov{x}} (1\otimes a) \otimes \epsilon_{2,(\alpha,\beta),\ov{x}}(1\otimes b \otimes y) \\
= \epsilon_{\frakA,s,\ov{x}}(1\otimes a) \otimes \epsilon_{\frakA,\alpha,\ov{x}}(1\otimes b) \otimes \epsilon_{\beta,\ov{x}}(1\otimes y). \end{matrix} \ar[mapsto]{r} & \begin{matrix} \left (\epsilon_{\frakA,s,\ov{x}}(1\otimes a) \cdot \epsilon_{\frakA,\alpha,\ov{x}}(1\otimes b) \right ) \otimes \epsilon_{\beta,\ov{x}}(1\otimes y) \\
= \epsilon_{\frakA,s+\alpha,\ov{x}}(1\otimes ab) \otimes \epsilon_{\beta,\ov{x}}(1\otimes y), \end{matrix}
\end{tikzcd}
$$
where
\begin{alignat*}{2}
\epsilon_{\frakA,s,\ov{x}} (1\otimes a) \otimes \epsilon_{2,(\alpha,\beta),\ov{x}}(1\otimes b \otimes y) &\in \left (\frakA_{s,\ov{x}} \otimes \calQ_{1,\ov{x}} \right ) \otimes_{\calQ_{1,\ov{x}}} \left (\frakA_{\alpha,\ov{x}} \otimes \calE_{\beta,\ov{x}} \otimes \calQ_{1,\ov{x}} \right ) \\
& \xrightarrow{\sim} \left (\frakA_{s,\ov{x}}\otimes \frakA_{\alpha,\ov{x}} \otimes \calE_{\beta,\ov{x}} \right ) \otimes \calQ_{1,\ov{x}}.
\end{alignat*}
This proves the commutativity of \eqref{diagdraft1} and hence the $\frakA$-admissibility of $\epsilon.$
\end{proof}

\begin{parag}\label{ffLcrys}
Assume $(H_1)$ is satisfied. Let $\calE'$ be a $\upmu_X$-indexed $F_1^*\calB_{X/S}$-module equipped with a $\calB_{X/S}$-linear $\calQ_{1,\upmu_X}$-stratification
$$
\epsilon':\calQ_{1,\upmu_X}\otimes_{\Ox_{X,\upmu_X}}\calE' \ra \calE' \otimes_{\Ox_{X,\upmu_X}} \calQ_{1,\upmu_X}.
$$
Recall the stratification
$
\epsilon_{\calA_X,Q}:\calQ_{1,\upmu_X}\otimes_{\Ox_{X,\upmu_X}}\calA_X \ra \calA_X \otimes_{\Ox_{X,\upmu_X}} \calQ_{1,\upmu_X}
$
given in \eqref{eqeAQ}.
By \ref{propdraft3} we have a stratification
$
\epsilon=\epsilon_{\calA_X,Q} \circledast_{\calQ_{1,\upmu_X}\otimes_{\Ox_{X,\upmu_X}}F_1^*\calB_{X/S}}\epsilon'.
$
Its source is
$$
\left (\calQ_{1,\upmu_X} \otimes_{\Ox_{X,\upmu_X}} \calA_X \right ) \circledast_{\calQ_{1,\upmu_X}\otimes_{\Ox_{X,\upmu_X}}F_1^*\calB_{X/S}} \left (\calQ_{1,\upmu_X}\otimes_{\Ox_{X,\upmu_X}} \calE' \right ).
$$
By \ref{kraz827} and \eqref{indiso1}, this source is canonically isomorphic to
$$
\calQ_{1,\upmu_X} \otimes_{\Ox_{X,\upmu_X}} \left (\calA_X \circledast_{F_1^*\calB_{X/S}} \calE'\right ) = \left (\calQ_{1,\upmu_X} \otimes_{\Ox_{X,\upmu_X}} \calA_X \right )\circledast_{F_1^*\calB_{X/S}} \calE'.
$$
Similarly, the target is isomorphic to
$
\calA_X \circledast_{F_1^*\calB_{X/S}}\calE'\otimes_{\Ox_{X,\upmu_X}}\calQ_{1,\upmu_X},
$
considered as a $\upmu_X$-indexed module over $\calQ_{1,\upmu_X}\otimes_{\Ox_{X,\upmu_X}}F_1^*\calB_{X/S}$ via the morphism
$$
\sigma_Q:\calQ_{1,\upmu_X}\otimes_{\Ox_{X,\upmu_X}}F_1^*\calB_{X/S} \ra F_1^*\calB_{X/S} \otimes_{\Ox_{X,\upmu_X}}\calQ_{1,\upmu_X}
$$
given in \eqref{eqsigmaQkraz}.
By \ref{propdraft4}, $\epsilon$ is an admissible $\calQ_{1,\upmu_X}$-stratification.
We obtain a linearization functor
\begin{equation}\label{eqLstrat1942f}
\begin{array}[t]{clc}
L_{\mathrm{strat}}:\begin{Bmatrix}\underline{\upmu}\text{-indexed\ }F_1^*\calB_{X/S}\text{-modules\ with\ a}\\ \calB_{X/S}\text{-linear}\\Q_{1,\upmu_X}\text{-stratification} \end{Bmatrix} & \ra & \begin{Bmatrix}\underline{\upmu}\text{-indexed\ }\calA_X\text{-modules\ with}\\ \text{an\ admissible}\\Q_{1,\upmu_X}\text{-stratification} \end{Bmatrix} \\
 & & \\
(\calE',\epsilon') & \mapsto & \left (\calA_X\circledast_{F_1^*\calB_{X/S}}\calE',\epsilon \right ).
\end{array}
\end{equation}
\end{parag}

\begin{theorem}\label{thmKoko20125}
Assume $(H_2)$ is satisfied. Keep the notation of \ref{parag1419}. Consider the diagram
\begin{equation}\label{diag187}
\begin{tikzcd}
\begin{Bmatrix}\mathrm{Crystals\ of\ }\upmu'\text{-}\mathrm{indexed}\\ \frakB_X\text{-}\mathrm{modules\ of\ }\widetilde{\calE}(X'/\frakS)_{/\upmu'} \end{Bmatrix} \ar{r}{C^{-1}} \ar[swap,sloped]{d}{\sim} & \begin{Bmatrix}\mathrm{Crystals\ of\ }\underline{\upmu}\text{-}\mathrm{indexed}\\ \frakA_X\text{-}\mathrm{modules\ of\ }\widetilde{\underline{\calE}}(X/\frakS)_{/\underline{\upmu}} \end{Bmatrix} \ar[sloped]{d}{\sim}\\
\begin{Bmatrix}\upmu_X\text{-}\mathrm{indexed\ }\calB_{X/S}\text{-}\mathrm{modules\ with}\\ \mathrm{an\ admissible}\\ \calR_{1,\upmu_X}'\text{-}\mathrm{stratification} \end{Bmatrix} \ar{r}{L_{\mathrm{strat}} \circ \Psi_{\upmu}} & \begin{Bmatrix}\upmu_X\text{-}\mathrm{indexed\ }\calA_X\text{-}\mathrm{modules\ with}\\ \mathrm{an\ admissible}\\Q_{1,\upmu_X}\text{-}\mathrm{stratification} \end{Bmatrix},
\end{tikzcd}
\end{equation}
where the vertical equivalences are given in \ref{equivRQAB} and $C^{-1},$ $L_{\mathrm{strat}}$ and $\Psi_{\upmu}$ are given respectively in \eqref{ffC}, \ref{ffLcrys} and \eqref{diag186}.
The diagram \eqref{diag187} is 2-commutative.
\end{theorem}

\begin{proof}
By definition of $C^{-1}$ \eqref{ffC}, the diagram \eqref{diag187} factors as follows:
$$
\begin{tikzcd}
\begin{Bmatrix}\mathrm{Crystals\ of\ }\\ \upmu'\text{-indexed}\\ \frakB_X\text{-modules\ of\ }\\ \widetilde{\calE}(X'/\frakS)_{/\upmu'} \end{Bmatrix} \ar{r}{C_{X/\frakS,\upmu}^{-1}} \ar[bend right=-30]{rr}{C^{-1}} \ar[swap,sloped]{d}{\sim} & \begin{Bmatrix}\mathrm{Crystals\ of\ }\\ \underline{\upmu}\text{-indexed}\\ C_{X/\frakS,\upmu}^{-1}\frakB_X\text{-modules\ of\ }\\ \widetilde{\underline{\calE}}(X/\frakS)_{/\underline{\upmu}} \end{Bmatrix} \ar{r}{L_{\mathrm{crys}}} \ar[swap,sloped]{d}{\sim} \ar{d}{\omega} & \begin{Bmatrix}\mathrm{Crystals\ of\ } \\ \underline{\upmu}\text{-indexed}\\ \frakA_X\text{-modules\ of\ } \\ \widetilde{\underline{\calE}}(X/\frakS)_{/\underline{\upmu}} \end{Bmatrix} \ar[swap, sloped]{d}{\sim} \ar{d}{\omega'} \\
\begin{Bmatrix}\upmu_X\text{-indexed\ }\\ \calB_{X/S}\text{-modules\ with}\\ \mathrm{an\ admissible}\\ \calR_{1,\upmu_X}'\text{-stratification} \end{Bmatrix} \ar{r}{\Psi_{\upmu}} \ar[swap,bend right=30]{rr}{L_{\mathrm{strat}} \circ \Psi_{\upmu}} & \begin{Bmatrix}\upmu_X\text{-indexed\ } \\ F_1^*\calB_{X/S}\text{-modules\ with\ a}\\ \calB_{X/S}\text{-linear}\\ \calQ_{1,\upmu_X}\text{-stratification} \end{Bmatrix} \ar{r}{L_{\mathrm{strat}}} & \begin{Bmatrix}\upmu_X\text{-indexed\ } \\ \calA_X\text{-modules\ with}\\ \mathrm{an\ admissible}\\Q_{1,\upmu_X}\text{-stratification} \end{Bmatrix}.
\end{tikzcd}
$$
The 2-commutativity of the left square was proved in \ref{propkraz1934}. It is hence sufficient to prove the 2-commutativity of the right square. For simplicity, when we tensor over $\Ox_{X,\upmu_X},$ we use the notation $\otimes$ instead of $\otimes_{\Ox_{X,\upmu_X}}.$
Let $\frakE$ be an object of the upper middle category. Keep the notation of \ref{parag1419}. By definition, we have
$$
L_{\mathrm{crys}}(\frakE)=\frakA_X \circledast_{C_{X/\frakS,\upmu}^{-1}\frakB_X}\mathfrak{E},
$$
and
$$
\omega'\left (\frakA_X \circledast_{C_{X/\frakS,\upmu}^{-1}\frakB_X}\mathfrak{E} \right )= \left ( \calE',\epsilon'\right ),
$$
where $\calE'=\left (\frakA_X \circledast_{C_{X/\frakS,\upmu}^{-1}\frakB_X}\mathfrak{E}\right )_{(X,\frakX,i_X)}$ and $\epsilon'$ is the composition
$$
\epsilon':\calQ_{1,\upmu_X} \otimes \calE' \xrightarrow{c_{\frakA_X \circledast_{C_{X/\frakS,\upmu}^{-1}\frakB_X}\frakE,q_2}} \left (\frakA_X \circledast_{C_{X/\frakS,\upmu}^{-1}\frakB_X}\mathfrak{E}\right )_{(X,Q_{\frakX},\lambda_Q)} \xrightarrow{c_{\frakA_X \circledast_{C_{X/\frakS,\upmu}^{-1}\frakB_X}\mathfrak{E},q_1}^{-1}}\calE'  \otimes \calQ_{1,\upmu_X}.
$$
On the other hand, we have
$$
\omega(\frakE)=\left (\F,\epsilon \right ),
$$
where $\F=\frakE_{(X,\frakX,i_X)}$ and $\epsilon$ is the composition
$$
\epsilon:\calQ_{1,\upmu_X} \otimes \F \xrightarrow{c_{\mathfrak{E},q_2}} \mathfrak{E}_{(X,Q_{\frakX},\lambda_Q)} \xrightarrow{c_{\mathfrak{E},q_1}^{-1}} \F \otimes \calQ_{1,\upmu_X}.
$$
By \eqref{eqLstrat1942f}, we have
$$
L_{\mathrm{strat}}\left (\F,\epsilon \right )=\left (\calA_X\circledast_{F_1^*\calB_{X/S}}\F,\epsilon_{\calA_X,Q}\circledast_{\calQ_{1,\upmu_X}\otimes F_1^*\calB_{X/S}}\epsilon\right ),
$$
where the target of
$$
\epsilon_{\calA_X,Q}\circledast_{\calQ_{1,\upmu_X}\otimes F_1^*\calB_{X/S}}\epsilon
$$
is considered as a $\upmu_X$-indexed module over $\calQ_{1,\upmu_X}\otimes_{\Ox_{X,\upmu_X}}F_1^*\calB_{X/S}$ via the morphism 
$$
\sigma_Q:\calQ_{1,\upmu_X}\otimes F_1^*\calB_{X/S} \ra F_1^*\calB_{X/S} \otimes \calQ_{1,\upmu_X}$$
given in \eqref{eqsigmaQkraz}.
We have to prove the existence of an isomorphism
\begin{equation}\label{eqtakriz185555}
\left (\calE',\epsilon' \right ) \xrightarrow{\sim} \left (\calA_X\circledast_{F_1^*\calB_{X/S}}\F,\epsilon_{\calA_X,Q}\circledast_{\calQ_{1,\upmu_X}\otimes F_1^*\calB_{X/S}}\epsilon\right ).
\end{equation}
By \ref{propKoko1915}, there exists canonical isomorphisms
$$
\calE'=\left (\frakA_X \circledast_{C_{X/\frakS,\upmu}^{-1}\frakB_X}\mathfrak{E}\right )_{(X,\frakX,i_X)} \xrightarrow{\sim} \calA_X\circledast_{ \left ( C_{X/\frakS,\upmu}^{-1}\frakB_X\right )_{(X,\frakX,i_X)}} \F,
$$
$$
\left (\frakA_X \circledast_{C_{X/\frakS,\upmu}^{-1}\frakB_X}\mathfrak{E}\right )_{(X,Q_{\frakX},\lambda_Q)} \xrightarrow{\sim} \frakA_{X,(X,Q_{\frakX},\lambda_Q)} \circledast_{\left ( C_{X/\frakS,\upmu}^{-1}\frakB_X \right )_{(X,Q_{\frakX},\lambda_Q)}}\frakE_{(X,Q_{\frakX},\lambda_Q)}.
$$
By \ref{ommek4}, we have canonical isomorphisms
$$
\left ( C_{X/\frakS,\upmu}^{-1}\frakB_X\right )_{(X,\frakX,i_X)} \xrightarrow{\sim} \frakB_{X,\rho(X,\frakX,i_X)}, \quad
\left ( C_{X/\frakS,\upmu}^{-1}\frakB_X \right )_{(X,Q_{\frakX},\lambda_Q)} \xrightarrow{\sim} \frakB_{X,\rho(X,Q_{\frakX},\lambda_Q)}.
$$
The lifting $F:\frakX \ra \frakX'$ and the morphism $\nu:Q_{\frakX} \ra R_{\frakX',1}$ induce morphisms (\cite{SF2} (8.36.24)), abusively denoted $F$ and $\nu,$
$$
F:\rho(X,\frakX,i_X) = (X',\frakX,F_1) \ra (X',\frakX',\op{Id}_{X'}),
$$
$$
\nu:\rho(X,Q_{\frakX},\lambda_Q) \ra (X',R_{\frakX',1},\lambda_{R'}),
$$
such that the diagram
$$
\begin{tikzcd}
\rho(X,Q_{\frakX},\lambda_Q) \ar{r}{\nu} \ar[swap]{d}{\rho(q_i)} & (X',R_{\frakX',1},\lambda_{R'}) \ar{d}{r_i} \\
\rho(X,\frakX,i_X) \ar[swap]{r}{F} & (X',\frakX',\op{Id}_{X'})
\end{tikzcd}
$$
is commutative, where $r_1$ and $r_2$ are the canonical projections.
Since $\frakB_X$ is a crystal, we have isomorphisms
$$
c_{\frakB_X,F}:F_1^*\calB_{X/S}=F_1^*\frakB_{X,(X',\frakX',\op{Id}_{X'})} \xrightarrow{\sim} \frakB_{X,\rho(X,\frakX,i_X)},
$$
$$
c_{\frakB_X,\nu}:\calQ_{1,\upmu_X} \otimes_{\calR'_{1,\upmu_X}} \frakB_{X,(X',R_{\frakX',1},\lambda_{R'})} \xrightarrow{\sim} \frakB_{X,\rho(X,Q_{\frakX},\lambda_Q)},
$$
$$
c_{\frakB_X,r_2}:\calR_{1,\upmu_X}' \otimes_{\Ox_{X',\upmu_X}} \calB_{X/S} \xrightarrow{\sim} \frakB_{X,(X',R_{\frakX',1},\lambda_{R'})},
$$
$$
c_{\frakB_X,r_1}:  \calB_{X/S}\otimes_{\Ox_{X',\upmu_X}} \calR_{1,\upmu_X}'\xrightarrow{\sim} \frakB_{X,(X',R_{\frakX',1},\lambda_{R'})}.
$$
We deduce canonical isomorphisms
$$
\calE'=\left (\frakA_X \circledast_{C_{X/\frakS,\upmu}^{-1}\frakB_X}\mathfrak{E}\right )_{(X,\frakX,i_X)} \xrightarrow{\sim} \calA_X\circledast_{F_1^*\calB_{X/S}} \F,
$$
\begin{alignat*}{2}
\left (\frakA_X \circledast_{C_{X/\frakS,\upmu}^{-1}\frakB_X}\mathfrak{E}\right )_{(X,Q_{\frakX},\lambda_Q)} &\xrightarrow{\sim} \frakA_{X,(X,Q_{\frakX},\lambda_Q)} \circledast_{\calQ_{1,\upmu_X} \otimes_{\calR'_{1,\upmu_X}} \frakB_{X,(X',R_{\frakX',1},\lambda_{R'})}}\frakE_{(X,Q_{\frakX},\lambda_Q)} \\
&\xrightarrow{\sim}  \frakA_{X,(X,Q_{\frakX},\lambda_Q)} \circledast_{\calQ_{1,\upmu_X} \otimes_{\Ox_{X',\upmu_X}} \calB_{X/S}}\frakE_{(X,Q_{\frakX},\lambda_Q)},
\end{alignat*}
fitting into the commutative diagram
$$
\begin{tikzcd}
\calQ_{1,\upmu_X} \otimes \calE' \ar{r}{\sim} \ar[swap]{dd}{c_{\frakA_X \circledast_{C_{X/\frakS,\upmu}^{-1}\frakB_X}\frakE,q_2}} &
\calQ_{1,\upmu_X} \otimes \left (\calA_X\circledast_{F_1^*\calB_{X/S}} \F\right ) \ar[sloped]{d}{\sim} \\
 &
\left (\calQ_{1,\upmu_X} \otimes \calA_X \right )\circledast_{\calQ_{1,\upmu_X} \otimes F_1^*\calB_{X/S}}  \left ( \calQ_{1,\upmu_X} \otimes \F\right ) \ar{d}{c_{\frakA_X,q_2}\circledast c_{\frakE,q_2}} \\
\left (\frakA_X \circledast_{C_{X/\frakS,\upmu}^{-1}\frakB_X}\mathfrak{E}\right )_{(X,Q_{\frakX},\lambda_Q)} \ar{r}{\sim} \ar[swap]{dd}{c_{\frakA_X \circledast_{C_{X/\frakS,\upmu}^{-1}\frakB_X}\mathfrak{E},q_1}^{-1}}  &
 \frakA_{X,(X,Q_{\frakX},\lambda_Q)} \circledast_{\calQ_{1,\upmu_X} \otimes F_1^*\calB_{X/S}}\frakE_{(X,Q_{\frakX},\lambda_Q)}  \ar{d}{c_{\frakA_X,q_1}^{-1}\circledast c_{\frakE,q_1}^{-1}} \\
&
\left ( \calA_X \otimes \calQ_{1,\upmu_X} \right )\circledast_{ F_1^*\calB_{X/S} \otimes \calQ_{1,\upmu_X} }  \left (  \F \otimes \calQ_{1,\upmu_X} \right ) \ar[sloped]{d}{\sim} \\
\calE' \otimes \calQ_{1,\upmu_X} \ar{r}{\sim} & \left (\calA_X\circledast_{F_1^*\calB_{X/S}} \F \right )\otimes \calQ_{1,\upmu_X}.
\end{tikzcd}
$$
We deduce the desired isomorphism of stratified $\upmu_X$-indexed $\calA_X$-modules \eqref{eqtakriz185555}.
\end{proof}

\begin{parag}\label{kraz1935}
In this paragraph, we assume $(H_1)$ is satisfied and we define a $\upmu_X$-indexed algebra structure on $\calQ_{1,\upmu_X}^{\vee} \otimes_{\Ox_{X,\upmu_X}} \calA_X.$
Since $\calA_X$ is an invertible $\Ox_{X,\upmu_X}$-module, we have a canonical isomorphism
\begin{equation}\label{isoQHom}
\calQ_{1,\upmu_X}^{\vee} \otimes_{\Ox_{X,\upmu_X}} \calA_X \xrightarrow{\sim} \mathscr{Hom}_{\Ox_{X,\upmu_X}}\left ( \calQ_{1,\upmu_X},\calA_X \right ).
\end{equation}
Note that, by (\cite{SGA43} IV 12.3), we have a canonical isomorphism
$
\calQ_{1,\upmu_X}^{\vee} \xrightarrow{\sim} \left (\calQ_1^{\vee} \right )_{\upmu_X}.
$
We define a multiplication on $\mathscr{Hom}_{\Ox_{X,\upmu_X}}\left ( \calQ_{1,\upmu_X},\calA_X \right )$ as follows: for local sections $s,t\in \Gamma(U,\upmu_X)$ over an étale $X$-scheme $U,$ and global sections $f$ and $g$ of
$$\mathscr{Hom}_{\Ox_{X,\upmu_X}}\left ( \calQ_{1,\upmu_X},\calA_X \right )_s\xrightarrow{\sim} \mathscr{Hom}_{\Ox_{U}}\left ( \calQ_{1|U},\calA_{X,s} \right ),$$
$$\mathscr{Hom}_{\Ox_{X,\upmu_X}}\left ( \calQ_{1,\upmu_X},\calA_X \right )_t\xrightarrow{\sim}\mathscr{Hom}_{\Ox_{U}}\left ( \calQ_{1|U},\calA_{X,t} \right )$$
respectively (where the isomorphisms follow from (\cite{SGA43} IV 12.3)). We abusively denote $\calQ_{1|U}$ simply by $\calQ_1.$ Define $f\circ g$ to be the composition
$$
\calQ_1 \xrightarrow{\delta} \calQ_1 \otimes_{\Ox_U} \calQ_1 \xrightarrow{\op{Id}\otimes g} \calQ_1 \otimes_{\Ox_U} \calA_{X,t} \xrightarrow{\epsilon_{\calA_X,Q,t}} \calA_{X,t}\otimes_{\Ox_U} \calQ_1  \xrightarrow{\op{Id}\otimes f} \calA_{X,t} \otimes_{\Ox_U} \calA_{X,s} \ra \calA_{X,s+t}.
$$
The only difficulty we encounter when proving that this multiplication defines an indexed algebra structure on $\mathscr{Hom}_{\Ox_{X,\upmu_X}}\left ( \calQ_{1,\upmu_X},\calA_X \right )$ and hence on $\calQ_{1,\upmu_X}^{\vee} \otimes_{\Ox_{X,\upmu_X}} \calA_X,$ is the associativity. This is proved exactly the same way as in (\cite{Ber96} 2.3.5).

Suppose that the hypothesis (\cite{SF2} 6.12) is satisfied. We have a basis $\left (\eta^K\eta_{(q)}^L \right )_{\substack{K\in \llbracket 0,p-1 \rrbracket^d\\ L\in \N^d}}$ for the $\Ox_X$-module $\calQ_1.$ We denote by $\left (\varphi_{K,L} \right )_{\substack{K\in \llbracket 0,p-1 \rrbracket^d \\ L\in \N^d}}$ its dual basis. Let $a$ be a local section of $\calA_{X,s}$ and $(I,J) \in \llbracket 0,p-1 \rrbracket^d \times \N^d.$ Then $\varphi_{I,J}\otimes 1$ is a section of $\calQ_{1}^{\vee} \otimes_{\Ox_{U}} \calA_{X,0}.$ Let us compute the product
$$
(\varphi_{I,J} \otimes 1)\cdot (1 \otimes a).
$$
By definition, $1\otimes a$ corresponds, by \eqref{isoQHom}, to the morphism
$$
\theta_a:\begin{array}[t]{clc}
\calQ_1 & \ra & \calA_{X,s} \\
\eta^{\alpha} \eta_{(q)}^{\beta} & \mapsto & \begin{cases} a\ \text{if }\alpha=\beta=0,\\ 0\ \text{else.} \end{cases} 
\end{array}
$$
By \eqref{isoQHom}, the product $(\varphi_{I,J} \otimes 1)\cdot (1 \otimes a)$ corresponds then to the composition
$$
\calQ_1 \xrightarrow{\delta} \calQ_1 \otimes_{\Ox_U} \calQ_1 \xrightarrow{\op{Id}\otimes \theta_a} \calQ_1 \otimes_{\Ox_U} \calA_{X,s} \xrightarrow{\epsilon_{\calA_X,Q,s}} \calA_{X,s}\otimes_{\Ox_U} \calQ_1  \xrightarrow{\op{Id}\otimes \varphi_{I,J}} \calA_{X,s}.
$$
By (\cite{SF2} 6.16), $\eta^{\alpha}\eta_{(q)}^{\beta}$ is sent, by $\delta,$ to
\begin{alignat*}{2}
&\prod_{i=1}^d\left (1\otimes \eta_i+\eta_i\otimes 1 + \eta_i\otimes \eta_i \right )^{\alpha_i} \\
\times & \prod_{i=1}^d \left (1\otimes \eta_{i(q)}+\sum_{0<b+c<p} \frac{(-1)^{b+c}}{b+c}\begin{pmatrix}b+c \\ b \end{pmatrix} \eta_i^{b+c}\otimes \eta_i^{p-b} +\eta_{i(q)}\otimes 1 \right )^{\beta_i}.
\end{alignat*}
By developing the right hand side, this is sent by $\op{Id}\otimes \theta_a$ to
$
\eta^{\alpha}\eta_{(q)}^{\beta} \otimes a.
$
By \eqref{eqepsilonAQ}, this is sent, by $\epsilon_{A_X,Q,s},$ to
$
\sum_{K\in \llbracket 0,p-1 \rrbracket^d} \left (\varphi_{K,0}\cdot a\right )\otimes \eta^{K+\alpha}\eta_{(q)}^{\beta}.
$
This is then sent, by $\op{Id}\otimes \varphi_{I,J},$ to
$
\left (\varphi_{I-\alpha,0}\cdot a\right )\delta_{\beta,J}.
$
Recall that, for $1\le i\le d,$ we have $\eta_i^p=p\eta_{i(q)}=0.$
We conclude that
\begin{equation}\label{eq20311}
(\varphi_{I,J} \otimes 1)\cdot (1 \otimes a)=\sum_{K+L=I}\left (\varphi_{K,0}\cdot a \right )\varphi_{L,J}.
\end{equation}
\end{parag}

\begin{proposition}\label{prop1932Ko}
Assume $(H_1)$ is satisfied. Let $(\calE,\epsilon)$ be a $\upmu_X$-indexed $\calA_X$-module equipped with a $\calQ_{1,\upmu_X}$-stratification and consider the corresponding $\calQ_{1,\upmu_X}^{\vee}$-module structure on $\calE.$ This structure extends to a structure of $\left (\calQ_{1,\upmu_X}^{\vee} \otimes_{\Ox_{X,\upmu_X}} \calA_X\right )$-module if and only if $\epsilon$ is admissible.
\end{proposition}

\begin{proof}
Denote by $\sigma:\upmu_X^2 \ra \upmu_X$ the addition map and by $\pi:\calA_X\boxtimes \calE \ra \sigma^*\calE$ the morphism defining the $\upmu_X$-indexed $\calA_X$-module structure on $\calE.$ Keep the notation of \ref{kraz1935}. Let $s,t\in \Gamma(U,\upmu_X)$ be local sections over an étale $X$-scheme $U$ and $a$ and $x$ local sections of $\calA_{X,s}$ and $\calE_t$ respectively. We have
$$
\epsilon_{s+t}(1\otimes (ax)) = \sum_{\substack{I \in \llbracket 0,p-1 \rrbracket^d \\ J\in \N^d}} \left (\varphi_{I,J}\cdot (ax) \right )\otimes \eta^I\eta_{(q)}^J.
$$
By \eqref{eqepsilonAQ}, we have
$$
\epsilon_{\calA_X,Q,s}(1\otimes a)=\sum_{K\in \llbracket 0,p-1 \rrbracket^d} \left (\varphi_{K,0}\cdot a\right )\otimes \eta^K.
$$
The image of $1\otimes a\otimes x$ by the composition
$$
\calQ_1 \otimes_{\Ox_U} \calA_{X,s} \otimes_{\Ox_U} \calE_t \xrightarrow{\epsilon_{\calA_X,Q,s}\otimes \op{Id}} \calA_{X,s}\otimes_{\Ox_U}\calQ_1\otimes_{\Ox_U}\calE_t \xrightarrow{\op{Id}\otimes \epsilon_t} \calA_{X,s}\otimes_{\Ox_U}\calE_t\otimes_{\Ox_U}\calQ_1
$$
is then equal to
$$
\sum_{\substack{I,K \in \llbracket 0,p-1 \rrbracket^d \\ J\in \N^d}} \left (\varphi_{K,0}\cdot a \right )\otimes \left (\varphi_{I,J}\cdot x\right ) \otimes \eta^{I+K}\eta_{(q)}^J.
$$
This is sent, by the morphism $\pi_{(s,t)}\otimes \op{Id}_{\calQ_1}:\calA_{X,s}\otimes_{\Ox_U}\calE_t\otimes_{\Ox_U}\calQ_1 \ra \calE_{s+t}\otimes_{\Ox_U}\calQ_1,$ to
$$
\sum_{\substack{I,K \in \llbracket 0,p-1 \rrbracket^d \\ J\in \N^d}} \left (\varphi_{K,0}\cdot a \right ) \left (\varphi_{I,J}\cdot x\right ) \otimes \eta^{I+K}\eta_{(q)}^J.
$$
Recall that, for $1\le i\le d,$ we have $\eta_i^p=p\eta_{i(q)}=0.$
It follows that $\epsilon$ is admissible if and only if, for all $I\in\llbracket 0,p-1 \rrbracket^d$ and $J\in \N^d,$
$$
\varphi_{I,J}\cdot (ax)=\sum_{K+L=I}\left (\varphi_{K,0}\cdot a \right ) \left (\varphi_{L,J}\cdot x\right ).
$$
We conclude by \eqref{eq20311}.
\end{proof}

\begin{definition}
Let
$$
S=\left (S^{\bullet}\calT_{X'/S} \right )_{\upmu_X},\ \widehat{\Gamma}=\left (\widehat{\Gamma}^{\bullet}\calT_{X'/S} \right )_{\upmu_X}.
$$
Recall the $\upmu_X$-indexed algebra $\widetilde{\calD}_{X/S}=\calA_X \otimes_{\Ox_{X,\upmu_X}} \calD_{X/S,\upmu_X}$ \eqref{Dtilda1713} and the $\Ox_X$-algebra $\calD_{X/S}^{\gamma}=\calD_{X/S} \otimes_{S^{\bullet}\calT_{X'/S}} \w{\Gamma}^{\bullet} \calT_{X'/S}.$
We denote
$$
\widetilde{\calD}_{X/S}^{\gamma}=\widetilde{\calD}_{X/S}\otimes_S\widehat{\Gamma}=\calA_X\otimes_{\Ox_{X,\upmu_X}}\calD_{X/S,\upmu_X}^{\gamma}.
$$
\end{definition}

\begin{parag}\label{ommek5}
Assume $(H_1)$ is satisfied. By \ref{indremark73}, a locally PD-nilpotent $\calD_{X/S,\upmu_X}^{\gamma}$-module $\calE$ \eqref{deflocnilind} is equivalent to the data, for every local section $s$ of $\upmu_X$ over an étale $X$-scheme $U,$ of a locally PD-nilpotent $\calD_{U/S}^{\gamma}$-module $\calE_s$ satisfying the conditions given in \ref{indremark73}. By (\cite{SF2} 7.17), the data $\calE_s$ is equivalent to the data of an $\Ox_U$-module $\calE_s$ equipped with a $\calQ_{1|U}$-stratification $\epsilon_s.$ This is in turn equivalent to an $\Ox_{X,\upmu_X}$-module $\calE$ equipped with a $\calQ_{1,\upmu_X}$-stratification. We obtain an equivalence of categories
\begin{equation}\label{starrr1}
\begin{Bmatrix}\mathrm{Locally\ PD}\text{-}\mathrm{nilpotent}\\\calD_{X/S,\upmu_X}^{\gamma}\text{-}\mathrm{modules} \end{Bmatrix} \xrightarrow{\sim} \begin{Bmatrix}\Ox_{X,\upmu_X}\text{-}\mathrm{modules\ equipped\ with\ a\ }\calQ_{1,\upmu_X}\text{-}\mathrm{stratification} \end{Bmatrix}.
\end{equation}
Similarly, we have an equivalence of categories
\begin{equation}\label{starrr2}
\begin{Bmatrix}\mathrm{Locally\ PD}\text{-}\mathrm{nilpotent}\\ \left (\w{\Gamma}^{\bullet}\calT_{X'/S}\right )_{\upmu_X}\text{-}\mathrm{modules} \end{Bmatrix} \xrightarrow{\sim} \begin{Bmatrix}\Ox_{X',\upmu_X}\text{-}\mathrm{modules\ equipped\ with\ an\ }\calR'_{1,\upmu_X}\text{-}\mathrm{stratification} \end{Bmatrix}.
\end{equation}
\end{parag}

\begin{proposition}\label{propKoko19129}
Assume $(H_1)$ is satisfied. The equivalences \eqref{starrr1} and \eqref{starrr2} induce equivalences of categories
$$
\begin{Bmatrix}\mathrm{Locally\ PD}\text{-}\mathrm{nilpotent}\\\widetilde{\calD}_{X/S}^{\gamma}\text{-}\mathrm{modules} \end{Bmatrix} \xrightarrow{\sim} \begin{Bmatrix}\upmu_X\text{-}\mathrm{indexed\ }\calA_X\text{-}\mathrm{modules\ equipped}\\ \mathrm{with\ an\ admissible\ }\calQ_{1,\upmu_X}\text{-}\mathrm{stratification} \end{Bmatrix},
$$
$$
\begin{Bmatrix}\mathrm{Locally\ PD}\text{-}\mathrm{nilpotent}\\\left (\left (\w{\Gamma}^{\bullet}\calT_{X'/S}\right )_{\upmu_X}\otimes_{\Ox_{X',\upmu_X}}\calB_{X/S} \right )\text{-}\mathrm{modules} \end{Bmatrix} \xrightarrow{\sim} \begin{Bmatrix}\upmu_X\text{-}\mathrm{indexed\ }\calB_{X/S}\text{-}\mathrm{modules\ equipped}\\ \mathrm{with\ an\ admissible\ }\calR'_{1,\upmu_X}\text{-}\mathrm{stratification} \end{Bmatrix}.
$$
\end{proposition}

\begin{proof}
The first equivalence results from \ref{prop1932Ko} and (\cite{SF2} (7.14.26)). The second one follows from the isomorphism (\cite{SF2} (7.5.9)) and the $\calB_{X/S}$-linearity of admissible $\calR'_{1,\upmu_X}$-stratifications.
\end{proof}

\begin{theorem}\label{THMXXaa}
Assume $(H_2)$ is satisfied. The diagram
\begin{equation}
\begin{tikzcd}
\begin{Bmatrix}\upmu_X\text{-}\mathrm{indexed\ }\calB_{X/S}\text{-}\mathrm{modules\ with}\\ \mathrm{an\ admissible}\\ \calR_{1,\upmu_X}'\text{-}\mathrm{stratification} \end{Bmatrix} \ar[swap,sloped]{d}{\sim} \ar{d}{\omega'} \ar{r}{L_{\mathrm{strat}} \circ \Psi_{\upmu}} & \begin{Bmatrix}\upmu_X\text{-}\mathrm{indexed\ }\calA_X\text{-}\mathrm{modules\ with}\\ \mathrm{an\ admissible}\\Q_{1,\upmu_X}\text{-}\mathrm{stratification} \end{Bmatrix} \ar[sloped]{d}{\sim} \ar[swap]{d}{\omega} \\
\begin{Bmatrix}\mathrm{Locally\ PD}\text{-}\mathrm{nilpotent\ } \upmu_X\text{-}\mathrm{indexed}\\\left (\left (\w{\Gamma}^{\bullet}\calT_{X'/S}\right )_{\upmu_X}\otimes_{\Ox_{X',\upmu_X}}\calB_{X/S} \right )\text{-}\mathrm{modules} \end{Bmatrix} \ar{r}{\Phi_{\upmu}} & \begin{Bmatrix}\mathrm{Locally\ PD}\text{-}\mathrm{nilpotent\ }\upmu_X\text{-}\mathrm{indexed}\\\widetilde{\calD}_{X/S}^{\gamma}\text{-}\mathrm{modules} \end{Bmatrix},
\end{tikzcd}
\end{equation}
where the vertical arrows are given in \ref{propKoko19129} and $L_{\mathrm{strat}},$ $\Psi_{\upmu}$ and $\Phi_{\upmu}$ are given respectively in \eqref{eqLstrat1942f}, \eqref{diag186} and \ref{propKoko1717}, is 2-commutative.
\end{theorem}

\begin{proof}
We may suppose that the hypothesis (\cite{SF2} 6.12) is satisfied. Let $(\partial_1',\hdots,\partial_d') \in \calT_{X'/S}$ be the dual basis of $(\op{dlog}m_1',\hdots,\op{dlog}m_d')$ and denote by $\partial_I$ the differential operators of $\calD_{X/S},$ obtained from $\op{dlog}m_1,\hdots,\op{dlog}m_d$ (\cite{SF1} 4.4). For $I\in \N^d,$ set
\begin{alignat*}{1}
\partial'^{[I]}=\prod_{i=1}^d\partial_i'^{[I_i]} \in \Gamma^{\bullet}\calT_{X'/S},\ \eta_{(r)}'^I=\prod_{i=1}^d\eta_{i(r)}'^{I_i}\in \calR_1', \\
\eta^I=\prod_{i=1}^d\eta_i^{I_i}\in \calQ_1,\ \eta_{(q)}^I=\prod_{i=1}^d\eta_{i(q)}^{I_i}\in \calQ_1.
\end{alignat*}
The family $\left (\eta_{(r)}'^{I} \right )_{I\in \N^d}$ is a basis for the $\Ox_X$-module $\calR_1'.$
Let $\left (\varphi_I' \right )_{I\in \N^d}$ be its dual basis. We also have a basis $\left (\eta^K\eta_{(q)}^L \right )_{\substack{K\in \llbracket 0,p-1 \rrbracket^d\\ L\in \N^d}}$ for the $\Ox_X$-module $\calQ_1.$ We denote by $\left (\varphi_{K,L} \right )_{\substack{K\in \llbracket 0,p-1 \rrbracket^d \\ L\in \N^d}}$ its dual basis. Let $\nu:Q_1 \ra R_1'$ be the morphism induced by \eqref{nu} and $\calV:\calR_1' \ra \calQ_1$ the corresponding morphism of Hopf algebras.
Before starting the proof, let us recall isomorphisms that we will need:
\begin{itemize}
\item By (\cite{SGA43} IV 12.3), there exists canonical isomorphisms
$$
\left (\calR'^{\vee}_1 \right )_{\upmu_X} \xrightarrow{\sim} \left (\calR'_{1,\upmu_X} \right )^{\vee},\ \left (\calQ^{\vee}_1 \right )_{\upmu_X} \xrightarrow{\sim} \left (\calQ_{1,\upmu_X} \right )^{\vee}.
$$
\item We have an isomorphism (\cite{SF2} (7.14.26))
\begin{equation}\label{eqmuzanf1}
\calD_{X/S,\upmu_X}^{\gamma} = \calD_{X/S,\upmu_X} \otimes_{\left (S^{\bullet}\calT_{X'/S} \right )_{\upmu_X}} \left (\w{\Gamma}^{\bullet}\calT_{X'/S}\right )_{\upmu_X}\xrightarrow{\sim} \calQ_{1,\upmu_X}^{\vee},
\end{equation}
by which, $(-1)^{|J|}I! \partial_I \otimes \partial'^J$ corresponds to $\varphi_{I,J}.$ 
\item We have an isomorphism (\cite{SF2} (7.5.9))
\begin{equation}\label{eqMuzanf2}
\left (\w{\Gamma}^{\bullet}\calT_{X'/S}\right )_{\upmu_X} \xrightarrow{\sim} \calR_{1,\upmu_X}'^{\vee},
\end{equation}
by which, the local section $\partial'^{[I]}$ corresponds to $\varphi'_{I}.$
\end{itemize}
Let $(\calE',\epsilon')$ be an object of the upper left category i.e. $\calE'$ is a $\upmu_X$-indexed $\calB_{X/S}$-module equipped with an admissible $\calR_{1,\upmu_X}'$-stratification
$$
\epsilon':\calR_{1,\upmu_X}' \otimes_{\Ox_{X',\upmu_X}} \calE' \ra \calE' \otimes_{\Ox_{X',\upmu_X}} \calR_{1,\upmu_X}'.
$$
We have to prove the existence of an isomorphism
$$
\omega \circ L_{\mathrm{strat}} \circ \Psi_{\upmu} \left ( \calE',\epsilon' \right ) \xrightarrow{\sim} \Phi_{\upmu} \circ \omega' \left ( \calE' , \epsilon' \right ).
$$
We start by expliciting the action of $\widetilde{\calD}_{X/S}^{\gamma}$ on both of the these modules.
Let $s,t,u\in \Gamma(U,\upmu_X)$ be sections of $\upmu_X$ over an étale $X$-scheme $U$ and $a,$ $b$ and $x'$ local sections of $\calA_{X,s},$ $\calB_{X/S,u}$ and $\calE'_t$ respectively.
The image of $(\calE',\epsilon')$ by $\omega'$ is the $\upmu_X$-indexed $\left (\left (\w{\Gamma}^{\bullet}\calT_{X'/S}\right )_{\upmu_X}\otimes_{\Ox_{X',\upmu_X}}\calB_{X/S} \right )$-module $\calE'.$ Let us explicit the actions of $\left (\w{\Gamma}^{\bullet}\calT_{X'/S}\right )_{\upmu_X}$ and $\calB_{X/S}$ on $\calE':$
\begin{itemize}
\item First, the action of $\calB_{X/S}$ on $\calE'$ comes from the structure of $\upmu_X$-indexed $\calB_{X/S}$-module on $\calE'.$
\item Second, the action of $\left (\w{\Gamma}^{\bullet}\calT_{X'/S}\right )_{\upmu_X}$ on $\calE'$ is given by
$$
\left (\w{\Gamma}^{\bullet}\calT_{X'/S}\right )_{\upmu_X} \times \calE' \xrightarrow{\sim} \calR_{1,\upmu_X}'^{\vee}\times \calE' \ra \calE',
$$
where the first isomorphism is induced by \eqref{eqMuzanf2} and the second morphism is
$$
\begin{array}[t]{clc}
\calR_{1,\upmu_X}'^{\vee} \times \calE' & \ra & \calE' \\
(\varphi,x') & \mapsto & \left (\op{Id}_{\calE'} \otimes \varphi \right ) \circ \epsilon'(1\otimes x').
\end{array}
$$
By the $\calB_{X/S}$-linearity of $\epsilon',$ we have
\begin{alignat*}{2}
\varphi \cdot (bx') &= \left (\op{Id}_{\calE'}\otimes \varphi \right )\circ \epsilon'(1\otimes (bx')) \\
&=\left (\op{Id}_{\calE'} \otimes \varphi\right )(b \epsilon'(1\otimes x')) \\
&= b\left (\op{Id}_{\calE'} \otimes \varphi\right )( \epsilon'(1\otimes x')) = b\left (\varphi'\cdot x' \right ).
\end{alignat*}
By the isomorphism \eqref{eqMuzanf2}, $\partial'^{[I]}$ corresponds to $\varphi'_I.$
The action of
$$
\partial'^{[I]}\otimes b\in \left (\left (\w{\Gamma}^{\bullet}\calT_{X'/S}\right )_{\upmu_X}\otimes_{\Ox_{X',\upmu_X}}\calB_{X/S} \right )_u=\w{\Gamma}^{\bullet}\calT_{U'/S}\otimes_{\Ox_{U'}}\calB_{X/S,u}
$$
on $x'\in \calE'_t$ is then given by
\begin{equation}\label{eq20z}
\left (\partial'^{[I]}\otimes b\right )\cdot x'=b\left (\varphi_I'\cdot x' \right ) \in \calE'_{u+t}.
\end{equation}
\end{itemize}
Let $\calE=\Phi_{\upmu}(\calE').$ Then $\calE=\calA_X \circledast_{\calB_{X/S}}\calE'$ (\eqref{ren6} and \ref{kraz827}). Let us explicit the action of $\widetilde{\calD}^{\gamma}_{X/S}$ on $\calE.$ For that, we explicit the actions of $\calD_{X/S,\upmu_X}^{\gamma}$ and $\calA_X$ on $\calE:$
\begin{itemize}
\item The action of $\calA_X$ comes from the canonical structure of $\upmu_X$-indexed $\calA_X$-module on $\calE.$
\item By \ref{indproptensprod}, we can consider $a\otimes x'$ as a local section of $\calE_{s+t}.$ Let $P_0$ be the special fiber of $P_{\frakX/\frakS,0}$ and $p_1,p_2:P_0 \ra X$ the canonical projections. By \eqref{actionpartialax}, the local sections $c_1,\hdots,c_d$ of $\Ox_X,$ appearing in \eqref{remark29}, such that the action of $\partial_{\epsilon_i}\in \calD_{X/S}$ on $a\otimes x'$ is given by
\begin{equation}\label{eq20zz}
\partial_{\epsilon_i}\cdot (a\otimes x')=(\partial_{\epsilon_i}\cdot a) \otimes x' +a\otimes (\partial'_{i}\cdot x')+\sum_{j=1}^d\partial_{\epsilon_i}(p_2^{\#}c_j-p_1^{\#}c_j) a \otimes (\partial'_j\cdot x'),
\end{equation}
where $\partial_{\epsilon_i}\cdot a$ denotes the action of $\partial_{\epsilon_i} \in \calD_{X/S}$ on $a\in \calA_{X,s}$ given by the canonical connection $d_{\calA_X}.$
By \eqref{actionpartialax}, the action $\partial'^{[L]} \cdot (a\otimes x')$ of $\partial'^{[L]} \in \Gamma^{\bullet} \calT_{X'/S}$ is given, for $L\in \N^d,$ by
\begin{equation}\label{eq20zztt}
\partial'^{[L]} \cdot (a\otimes x')=a \otimes \left ( \partial'^{[L]} \cdot x' \right ).
\end{equation}
\end{itemize}
To conclude, the module structure on $\Phi_{\upmu} \circ \omega' \left (\calE',\epsilon' \right )$ is given by
\begin{equation}\label{Muzanf3}
\begin{alignedat}{2}
\partial_{\epsilon_i}\cdot (a\otimes x') &= (\partial_{\epsilon_i}\cdot a) \otimes x' +a\otimes (\partial'_{i}\cdot x')+\sum_{j=1}^d\partial_{\epsilon_i}(p_2^{\#}c_j-p_1^{\#}c_j) a \otimes (\partial'_j\cdot x'), \\
\partial'^{[L]} \cdot (a\otimes x') &= a \otimes \left ( \partial'^{[L]} \cdot x' \right ).
\end{alignedat}
\end{equation}
Let us now determine $\omega \circ L_{\mathrm{strat}}\circ \Psi_{\upmu}\left (\calE',\epsilon'\right ):$
We consider $\calQ_{1,\upmu_X} \otimes_{\Ox_{X,\upmu_X}} \calA_X$ as a $\upmu_X$-indexed $\left (\calR_{1,\upmu_X}' \otimes_{\Ox_{X',\upmu_X}} \calB_{X/S} \right )$-module via $\nu:Q_1 \ra R_1'$ and the canonical embedding $\calB_{X/S} \hookrightarrow \calA_X.$
The image of $(\calE',\epsilon')$ by $\Psi_{\upmu}$ \eqref{diag186} is 
$$
\Psi_{\upmu}\left (\calE',\epsilon'\right )=\left (F_1^*\calE',\calV^*\epsilon' \right ).
$$
This is sent, by $L_{\text{strat}}$ \eqref{eqLstrat1942f}, to
$$
\left (\calA_X \circledast_{F_1^*\calB_{X/S}} F_1^*\calE',\epsilon_{\calA_X,Q}\circledast_{\calQ_{1,\upmu_X} \otimes_{\Ox_{X',\upmu_X}}\calB_{X/S} } \calV^*\epsilon' \right ).
$$
By \ref{kraz827}, we have canonical isomorphisms
\begin{alignat*}{2}
\calA_X \circledast_{F_1^*\calB_{X/S}} F_1^*\calE' \xrightarrow{\sim} & \calA_X \circledast_{\calB_{X/S}} \calE', \\
\epsilon_{\calA_X,Q}\circledast_{\calQ_{1,\upmu_X} \otimes_{\Ox_{X',\upmu_X}}\calB_{X/S} } \calV^*\epsilon'
 \xrightarrow{\sim} & \epsilon_{\calA_X,Q}\circledast_{\calV^*\left (\calR_{1,\upmu_X}' \otimes_{\Ox_{X',\upmu_X}}\calB_{X/S}\right ) } \calV^*\epsilon'\\
\xrightarrow{\sim} &  \epsilon_{\calA_X,Q} \circledast_{\calR_{1,\upmu_X}' \otimes_{\Ox_{X',\upmu_X}} \calB_{X/S} } \epsilon'.
\end{alignat*}
We conclude that the image of $(\calE',\epsilon')$ by $L_{\mathrm{strat}} \circ \Psi_{\upmu}$ is $(\calE,\epsilon)$ where
$$
\begin{cases}
\calE=\calA_X \circledast_{\calB_{X/S}}\calE' \\
\epsilon = \epsilon_{\calA_X,Q} \circledast_{\calR_{1,\upmu_X}' \otimes_{\Ox_{X',\upmu_X}} \calB_{X/S} } \epsilon'.
\end{cases}
$$
By \eqref{eqepsilonAQ}, we have
\begin{alignat*}{2}
\epsilon_{\calA_X,Q,s}(1\otimes a) &= \sum_{K \in \llbracket 0,p-1 \rrbracket^d} (\varphi_{K,0}\cdot a)\otimes \eta^K, \\
\epsilon_t'(1\otimes x') &= \sum_{I\in \N^d}(\varphi'_I\cdot x')\otimes \eta_{(r)}'^I.
\end{alignat*}
It follows that
\begin{equation}\label{eq2030}
\epsilon_{s+t} (1\otimes a\otimes x')=\sum_{\substack{K \in \llbracket 0,p-1 \rrbracket^d \\ I\in \N^d}} (\varphi_{K,0}\cdot a)\otimes (\varphi_I'\cdot x') \otimes \eta^K \calV \left (\eta_{(r)}^I \right ).
\end{equation}
By (\cite{SF2} (6.18.9)), we have
\begin{alignat*}{2}
\calV \left (\eta_{i(r)} \right ) &= \eta_{i(q)} + \sum_{k=1}^{p-1}\frac{(-1)^{k+1}}{k}\eta_i^k + q_2^{\#}c_i-q_1^{\#}c_i \\
&= \eta_{i(q)} + \sum_{k=1}^{p-1}\frac{(-1)^{k+1}}{k}\eta_i^k + \sum_{\substack{\alpha \in \llbracket 0,p-1 \rrbracket^d \\ \beta\in \N^d}} \varphi_{\alpha,\beta}(q_2^{\#}c_i-q_1^{\#}c_i)\eta^{\alpha}\eta_{(q)}^{\beta},
\end{alignat*}
where $q_1,q_2:Q_1 \ra X$ are the canonical projections. Note that the local sections $c_i$ are the same in \eqref{eq20zz}.
Let $X\ra Y$ be the exact diagonal immersion. The projection $q_i$ factors, by definition, through $Y.$ It follows, by \ref{lemkraz}, that $\varphi_{\alpha,\beta}(q_2^{\#}c_i-q_1^{\#}c_i)=0$ whenever $\beta \neq 0$ and if $(\alpha,\beta)=(0,0).$ Then
\begin{equation}\label{eqfinalMuzan33}
\calV \left (\eta_{i(r)} \right ) = \eta_{i(q)} + \sum_{k=1}^{p-1}\frac{(-1)^{k+1}}{k}\eta_i^k + \sum_{\substack{\alpha \in \llbracket 0,p-1 \rrbracket^d \\ \alpha\neq 0}} \varphi_{\alpha,0}(q_2^{\#}c_i-q_1^{\#}c_i)\eta^{\alpha}.
\end{equation}
Denote by $\calK_0 \subset \calQ_1$ and $\calK_1 \subset \calQ_1$ the ideals generated by $\eta_1,\hdots,\eta_d$ and by $\eta_1,\hdots,\eta_d,\eta_{1(q)}, \hdots, \eta_{d(q)}$ respectively. Also denote by $\pi_0:\calQ_1 \ra \calQ_1/\calK_0,$ $\pi_1:\calQ_1 \ra \calQ_1/\calK_1^2$ the canonical projections. By \eqref{eqfinalMuzan33}, we have 
$
\pi_0 \circ \calV\left (\eta_{i(r)} \right ) = \eta_{i(q)}.
$
It follows that, for $L\in \N^d,$
$$
\varphi_{0,L}\left ( \eta^K \calV \left ( \eta_{(r)}^I \right ) \right ) = \begin{cases}
0\ \text{if\ }(K,I) \neq (0,L) \\
1\ \text{if\ }(K,I)=(0,L).
\end{cases}
$$
Then, by \eqref{eq2030},
\begin{equation}\label{eqkraz25}
\left ( \op{Id} \otimes \varphi_{0,L} \right )\circ \epsilon_{s+t}(1 \otimes a \otimes x')=a \otimes \left ( \varphi_L' \cdot x' \right ).
\end{equation}
We also have, by \eqref{eqfinalMuzan33},
$$
\pi_1\circ \calV\left (\eta_{i(r)} \right ) = \eta_{i(q)} + \eta_i + \varphi_{\epsilon_1,0}(q_2^{\#}c_1-q_1^{\#}c_1)\eta_1+ \hdots +\varphi_{\epsilon_d,0}(q_2^{\#}c_d-q_1^{\#}c_d)\eta_d.
$$ 
By \eqref{eq2030}, we have
\begin{equation}\label{eq2030kraz}
\begin{alignedat}{2}
\left (\op{Id}\otimes \pi_1 \right )\circ \epsilon_{s+t} (1\otimes a\otimes x') =& \sum_{1\le i\le d } (\varphi_{\epsilon_i,0}\cdot a)\otimes x' \otimes \eta_i + \sum_{1\le i \le d}a\otimes \left (\varphi_{\epsilon_i}'\cdot x' \right )\otimes \pi_1 \circ \calV\left (\eta_{i(r)} \right ) \\
=& \sum_{1\le i\le d} (\varphi_{\epsilon_i,0}\cdot a)\otimes x' \otimes \eta_i + a\otimes \left (\varphi_{\epsilon_i}'\cdot x' \right )\otimes \left (\eta_{i(q)}+\eta_i \right ) \\
&+ \sum_{1\le i,j\le d} a\otimes \left (\varphi_{\epsilon_i}'\cdot x' \right )\otimes \left (\varphi_{\epsilon_j,0}\left (q_2^{\#}c_j-q_1^{\#}c_j\right )\eta_j \right ).
\end{alignedat}
\end{equation}
By \eqref{eqkraz25} and \eqref{eq2030kraz}, we get
\begin{equation}\label{eq20zzz}
\begin{alignedat}{2}
\varphi_{\epsilon_i,0}\cdot (a\otimes x') &= (\varphi_{\epsilon_i,0}\cdot a)\otimes x' + a\otimes (\varphi_{\epsilon_i}'\cdot x')+\sum_{j=1}^d \varphi_{\epsilon_j,0}(q_2^{\#}c_j-q_1^{\#}c_j) a\otimes (\varphi'_{\epsilon_j}\cdot x'),\\
\varphi_{0,L} \cdot (a\otimes x') &= a \otimes (\varphi'_{L} \cdot x').
\end{alignedat}
\end{equation}
Since $\varphi_{I,J}$ corresponds to $(-1)^{|J|}I! \partial_I \otimes \partial'^{[J]}$ by \eqref{eqmuzanf1} and $\varphi_I'$ corresponds to $\partial'^{[I]}$ by \eqref{eqMuzanf2}, we get
\begin{equation}\label{eq20zzzv2}
\begin{alignedat}{2}
\partial_{\epsilon_i}\cdot (a\otimes x') &= (\partial_{\epsilon_i}\cdot a)\otimes x' + a\otimes (\partial_i'\cdot x')+\sum_{j=1}^d \varphi_{\epsilon_j,0}(q_2^{\#}c_j-q_1^{\#}c_j) a\otimes (\partial_j'\cdot x'),\\
\partial'^{[L]} \cdot (a\otimes x') &= (-1)^{|L|}a \otimes (\partial'^{[L]} \cdot x').
\end{alignedat}
\end{equation}
We have thus determined $\omega \circ L_{\mathrm{strat}} \circ \Psi_{\upmu} \left ( \calE' , \epsilon' \right ).$

As we see, there are two differences between \eqref{Muzanf3} and \eqref{eq20zzzv2}: a difference of $(-1)^{|L|}$ in the second equality and, in the first equality, we have $\partial_{\epsilon_i}(p_2^{\#}c-p_1^{\#}c)$ on one hand and $\varphi_{\epsilon_j,0}(q_2^{\#}c-q_1^{\#}c)$ in the other. To fix the first difference, we introduce the antipode isomorphism
$$
\sigma:\begin{array}[t]{clc}
\left (\w{\Gamma}^{\bullet}\calT_{X'/S}\right )_{\upmu_X} & \xrightarrow{\sim} & \left (\w{\Gamma}^{\bullet}\calT_{X'/S}\right )_{\upmu_X} \\
\partial'_i & \mapsto & -\partial'_i,
\end{array}
$$
and the functor $T$ defined by base change by $\sigma$
$$
T: \calM \mapsto \calM \otimes_{\left (\w{\Gamma}^{\bullet}\calT_{X'/S}\right )_{\upmu_X},\sigma} \left (\w{\Gamma}^{\bullet}\calT_{X'/S}\right )_{\upmu_X}.
$$
Since $\sigma$ is an isomorphism, we have an isomorphism
$$
\Phi_{\upmu} \circ \omega' \left ( \calE' , \epsilon' \right ) \xrightarrow{\sim} T \circ \Phi_{\upmu} \circ \omega' \left ( \calE' , \epsilon' \right ).
$$
In addition, $\sigma \left ( \partial'^{[L]} \right ) = (-1)^{|L|}\partial'^{[L]}.$ It follows that the actions of $\partial_{\epsilon_i}$ and $\partial'^{[L]}$ on the local section $a\otimes x'$ of $T \circ \Phi_{\upmu} \circ \omega' \left ( \calE' , \epsilon' \right )$ are given by
\begin{equation}\label{eq20zzzv3}
\begin{alignedat}{2}
\partial_{\epsilon_i}\cdot (a\otimes x') &= (\partial_{\epsilon_i}\cdot a)\otimes x' + a\otimes (\partial_i'\cdot x')+\sum_{j=1}^d \varphi_{\epsilon_j,0}(q_2^{\#}c_j-q_1^{\#}c_j) a\otimes (\partial_j'\cdot x'),\\
\partial'^{[L]} \cdot (a\otimes x') &= a \otimes (\partial'^{[L]} \cdot x').
\end{alignedat}
\end{equation}
The last thing to check is that, for a local section $c$ of $\Ox_X,$
$
\partial_{\epsilon_i}(p_2^{\#}c-p_1^{\#}c) = \varphi_{\epsilon_j,0}(q_2^{\#}c-q_1^{\#}c).
$
Recall the morphism of Hopf algebras $u:\calQ_1 \ra \calP_0$ (\cite{SF2} (7.3.2)). By definition of $u,$ we have $u \circ q_i^{\#}=p_i^{\#}.$ Thus
$
\partial_{\epsilon_i}(p_2^{\#}c-p_1^{\#}c)=\partial_{\epsilon_i} \circ u \left ( q_2^{\#}c-q_1^{\#}c \right ).
$
It is hence sufficient to prove that, as linear forms on $\calQ_1,$ we have
$
\partial_{\epsilon_i} \circ u = \varphi_{\epsilon_i,0}.
$
This is immediate since $u(\eta_j)=\eta_j$ for all $1\le j\le d$ (\cite{SF2} (7.3.4)).

By \eqref{Muzanf3} and \eqref{eq20zzzv3}, we conclude that we have an isomorphism
$$
T \circ \Phi_{\upmu} \circ \omega' \left ( \calE' , \epsilon' \right ) \xrightarrow{\sim} \omega \circ L_{\mathrm{strat}} \circ \Psi_{\upmu} \left ( \calE' , \epsilon' \right ).
$$
\end{proof}

\begin{corollaire}\label{THMXX}
Assume $(H_2)$ is satisfied. The diagram
\begin{equation}
\begin{tikzcd}
\begin{Bmatrix}\mathrm{Crystals\ of\ }\upmu'\text{-}\mathrm{indexed}\\ \frakB_X\text{-}\mathrm{modules\ of\ }\widetilde{\calE}(X'/\frakS)_{/\upmu'} \end{Bmatrix} \ar{r}{C^{-1}} \ar[swap,sloped]{d}{\sim} & \begin{Bmatrix}\mathrm{Crystals\ of\ }\underline{\upmu}\text{-}\mathrm{indexed}\\ \frakA_X\text{-}\mathrm{modules\ of\ }\widetilde{\underline{\calE}}(X/\frakS)_{/\underline{\upmu}} \end{Bmatrix} \ar[sloped]{d}{\sim}\\
\begin{Bmatrix}\upmu_X\text{-}\mathrm{indexed\ }\calB_{X/S}\text{-}\mathrm{modules\ with}\\ \mathrm{an\ admissible}\\ \calR_{1,\upmu_X}'\text{-}\mathrm{stratification} \end{Bmatrix} \ar[swap,sloped]{d}{\sim} \ar{d}{\omega'} \ar{r}{L_{\mathrm{strat}} \circ \Psi_{\upmu}} & \begin{Bmatrix}\upmu_X\text{-}\mathrm{indexed\ }\calA_X\text{-}\mathrm{modules\ with}\\ \mathrm{an\ admissible}\\Q_{1,\upmu_X}\text{-}\mathrm{stratification} \end{Bmatrix} \ar[sloped]{d}{\sim} \ar[swap]{d}{\omega} \\
\begin{Bmatrix}\mathrm{Locally\ PD}\text{-}\mathrm{nilpotent\ } \upmu_X\text{-}\mathrm{indexed}\\\left (\left (\w{\Gamma}^{\bullet}\calT_{X'/S}\right )_{\upmu_X}\otimes_{\Ox_{X',\upmu_X}}\calB_{X/S} \right )\text{-}\mathrm{modules} \end{Bmatrix} \ar{r}{\Phi_{\upmu}} & \begin{Bmatrix}\mathrm{Locally\ PD}\text{-}\mathrm{nilpotent\ }\upmu_X\text{-}\mathrm{indexed}\\\widetilde{\calD}_{X/S}^{\gamma}\text{-}\mathrm{modules} \end{Bmatrix},
\end{tikzcd}
\end{equation}
where the upper vertical equivalences are given in \ref{equivRQAB}, the lower vertical arrows are given in \ref{propKoko19129} and $C^{-1},$ $L_{\mathrm{strat}},$ $\Psi_{\upmu}$ and $\Phi_{\upmu}$ are given respectively in \eqref{ffC}, \eqref{eqLstrat1942f}, \eqref{diag186} and \ref{propKoko1717}, is 2-commutative. In addition, the functor $C^{-1}$ is an equivalence of categories.
\end{corollaire}

\begin{proof}
The commutativity of the upper square was proved in \ref{thmKoko20125}. The lower square's commutativity was proved in \ref{THMXXaa}. Since $\Phi_{\upmu}$ is an equivalence of categories \eqref{propKoko1717}, so is $C^{-1}.$
\end{proof}

\begin{lemma}\label{lemEST3}
Assume $(H_1)$ is satisfied. Let $U$ be an étale $X$-scheme. Consider the diagram
$$
\begin{tikzcd}
\begin{Bmatrix}\text{Crystals\ of\ }\underline{\upmu}\text{-indexed}\\ C_{X/\frakS,\upmu}^{-1}\frakB_X\text{-modules\ of\ }\widetilde{\underline{\calE}}(X/\frakS)_{/\underline{\upmu}} \end{Bmatrix} \ar[swap]{dd}{\mathfrak{E} \mapsto \frakA_X \circledast_{C_{X/\frakS,\upmu}^{-1}\frakB_X}\mathfrak{E}} \ar{r}{\underline{R_{X,U}}^{-1}} & \begin{Bmatrix}\text{Crystals\ of\ }\underline{\upmu}_U\text{-indexed}\\ C_{U/\frakS,\upmu}^{-1}\frakB_U\text{-modules\ of\ }\widetilde{\underline{\calE}}(U/\frakS)_{/\underline{\upmu}_U} \end{Bmatrix} \ar{dd}{\mathfrak{E}  \mapsto  \frakA_U \circledast_{C_{U/\frakS,\upmu}^{-1}\frakB_U}\mathfrak{E}} \\ 
 & \\
\begin{Bmatrix}\text{Crystals\ of\ }\underline{\upmu}\text{-indexed}\\ \frakA_X\text{-modules\ of\ }\widetilde{\underline{\calE}}(X/\frakS)_{/\underline{\upmu}} \end{Bmatrix} \ar[swap]{r}{\underline{R_{X,U}}^{-1}} & \begin{Bmatrix}\text{Crystals\ of\ }\underline{\upmu}_U\text{-indexed}\\ \frakA_U\text{-modules\ of\ }\widetilde{\underline{\calE}}(U/\frakS)_{/\underline{\upmu}_U} \end{Bmatrix},
\end{tikzcd}
$$
where the horizontal arrows are given in \eqref{eqREST} and the linearization functors are given in \eqref{ffL}. Then, this diagram is commutative.
\end{lemma}

\begin{proof}
For simplicity, we denote $R_{X',U'}$ and $\underline{R_{X,U}}$ by $R'$ and $\underline{R}$ respectively. Let $\frakE$ be a crystal of $C_{X/\frakS,\upmu}^{-1}\frakB_X$-modules. We have to prove the existence of a canonical isomorphism
\begin{equation}\label{almostthere1}
\underline{R}^{-1} \left ( \frakA_X \circledast_{C_{X/\frakS,\upmu}^{-1}\frakB_X}\mathfrak{E} \right ) \xrightarrow{\sim} \frakA_U \circledast_{C_{U/\frakS,\upmu}^{-1}\frakB_U}\underline{R}^{-1}\mathfrak{E}.
\end{equation}
In \ref{paragEST1}, we have seen that $\underline{R}$ is a localization morphism. By \ref{ren3}, we deduce a canonical isomorphism
$$
\underline{R}^{-1} \left ( \frakA_X \circledast_{C_{X/\frakS,\upmu}^{-1}\frakB_X}\mathfrak{E} \right ) \xrightarrow{\sim} \left ( \underline{R}^{-1}  \frakA_X \right ) \circledast_{\underline{R}^{-1}C_{X/\frakS,\upmu}^{-1}\frakB_X} \left (  \underline{R}^{-1} \mathfrak{E} \right ).
$$
By \ref{AXUESTiso}, we have canonical isomorphisms
$$
\underline{R}^{-1}\frakA_X \xrightarrow{\sim} \frakA_U,
$$
$$
\underline{R}^{-1}C_{X/\frakS,\upmu}^{-1}\frakB_X \xrightarrow{\sim} C_{U/\frakS,\upmu}^{-1}R'^{-1}\frakB_X \xrightarrow{\sim} C_{U/\frakS,\upmu}^{-1}\frakB_U.
$$
The isomorphism \eqref{almostthere1} follows.
\end{proof}

\begin{lemma}\label{lemEST2}
Assume $(H_1)$ is satisfied. Let $U$ be an étale $X$-scheme. Consider the diagram
$$
\begin{tikzcd}
\begin{Bmatrix}\text{Crystals\ of\ }\upmu'\text{-indexed}\\ \frakB_X\text{-modules\ of\ }\widetilde{\calE}(X'/\frakS)_{/\upmu'} \end{Bmatrix} \ar{r}{R_{X',U'}^{-1}} \ar[swap]{d}{C_{X/\frakS,\upmu}^{-1}} & \begin{Bmatrix}\text{Crystals\ of\ }\upmu'_U\text{-indexed}\\ \frakB_U\text{-modules\ of\ }\widetilde{\calE}(U'/\frakS)_{/\upmu'_U} \end{Bmatrix} \ar{d}{C_{U/\frakS,\upmu}^{-1}} \\
\begin{Bmatrix}\text{Crystals\ of\ }\underline{\upmu}\text{-indexed}\\ C_{X/\frakS,\upmu}^{-1}\frakB_X\text{-modules\ of\ }\widetilde{\underline{\calE}}(X/\frakS)_{/\underline{\upmu}} \end{Bmatrix} \ar[swap]{dd}{\mathfrak{E} \mapsto \frakA_X \circledast_{C_{X/\frakS,\upmu}^{-1}\frakB_X}\mathfrak{E}} \ar[swap]{r}{\underline{R_{X,U}}^{-1}} & \begin{Bmatrix}\text{Crystals\ of\ }\underline{\upmu}_U\text{-indexed}\\ C_{U/\frakS,\upmu}^{-1}\frakB_U\text{-modules\ of\ }\widetilde{\underline{\calE}}(U/\frakS)_{/\underline{\upmu}_U} \end{Bmatrix} \ar{dd}{\mathfrak{E}  \mapsto  \frakA_U \circledast_{C_{U/\frakS,\upmu}^{-1}\frakB_U}\mathfrak{E}} \\ 
 & \\
\begin{Bmatrix}\text{Crystals\ of\ }\underline{\upmu}\text{-indexed}\\ \frakA_X\text{-modules\ of\ }\widetilde{\underline{\calE}}(X/\frakS)_{/\underline{\upmu}} \end{Bmatrix} \ar[swap]{r}{\underline{R_{X,U}}^{-1}} & \begin{Bmatrix}\text{Crystals\ of\ }\underline{\upmu}_U\text{-indexed}\\ \frakA_U\text{-modules\ of\ }\widetilde{\underline{\calE}}(U/\frakS)_{/\underline{\upmu}_U} \end{Bmatrix},
\end{tikzcd}
$$
where the horizontal arrows are given in \eqref{eqREST}, $C_{X/\frakS,\upmu}^{-1}$ and $C_{U/\frakS,\upmu}^{-1}$ are given in \eqref{CESTeq} and the linearization functors are given in \eqref{ffL}. Then, this diagram is commutative.
\end{lemma}

\begin{proof}
The commutativity of the upper square follows from \ref{commCXU} and the commutativity of the lower square follows from \ref{lemEST3}.
\end{proof}

\begin{theorem}\label{THMXXYY}
The functor \eqref{ffC} 
$$
C^{-1}:
\begin{Bmatrix}\text{Crystals\ of\ }\upmu'\text{-indexed}\\ \frakB_X\text{-modules\ of\ }\widetilde{\calE}(X'/\frakS)_{/\upmu'} \end{Bmatrix}  \ra  \begin{Bmatrix}\text{Crystals\ of\ }\underline{\upmu}\text{-indexed}\\ \frakA_X\text{-modules\ of\ }\widetilde{\underline{\calE}}(X/\frakS)_{/\underline{\upmu}} \end{Bmatrix}
$$
is an equivalence of categories.
\end{theorem}

\begin{proof}
We can work étale locally on $\frakS$ and hence suppose that $\frakS$ has a chart $\frakS \ra B\langle P\rangle.$
Let $\left ( U_i \ra X \right )_{i\in I}$ be an étale covering such that, for every $i\in I,$ there exists a log smooth logarithmic formal scheme $\frakU_i$ over $\frakS_i$ fitting into a cartesian diagram
$$
\begin{tikzcd}
U_i \ar{r} \ar{d} & \frakU_i \ar{d} \\
S \ar{r} & \frakS.
\end{tikzcd}
$$
We can suppose that $\frakU_i$ is affine and $\frakU_i \ra \frakS$ has a chart $P \ra Q_i$ such that $P^{gp} \ra Q_i^{gp}$ is injective, its cokernel has a finite order coprime woth $p$ and the morphism $\frakU_i \ra \frakS\times_{B\langle P \rangle}B\langle Q_i \rangle,$ induced by $\frakU_i \ra \frakS$ and the chart $\frakU_i \ra B\langle Q_i \rangle,$ is strict and étale. Consider the monoid $Q_i'$ defined from $P\ra Q_i$ as in (\cite{SF2} 3.1). We have a chart $P \ra Q_i'$ of $U_i' \ra S$ and, by \ref{lemShiho208}, the morphism $U_i' \ra S\times_{A_1[P]}A_1[Q_i'],$ induced by $U_i' \ra S$ and the chart $U_i' \ra A_1[Q_i'],$ is strict and étale.
Then, there exists, for every $i\in I,$ a unique strict étale logarithmic formal scheme $\frakU'_i$ over $\frakS \times_{B\langle P\rangle} B\langle Q_i' \rangle,$ fitting into a cartesian square
$$
\begin{tikzcd}
U_i' \ar{r} \ar{d} & \frakU_i' \ar{d} \\
\frakS \times_{A_1[ P]} A_1[ Q_i'] \ar{r} & \frakS \times_{B\langle P\rangle} B\langle Q_i' \rangle.
\end{tikzcd}
$$
For every $i\in I,$ there exists a lift $\frakU_i \ra \frakU_i'$ fitting into the commutative diagram
$$
\begin{tikzcd}
U_i' \ar{rrr} & & & \frakU_i' \ar{d} \\
U_i \ar{u} \ar{r} & \frakU_i \ar{urr} \ar{r} & \frakS \times_{B\langle P\rangle} B\langle Q_i \rangle \ar{r} & \frakS \times_{B\langle P\rangle} B\langle Q_i' \rangle.
\end{tikzcd}
$$
By \ref{lemEST2}, we have, for every $i\in I,$ a commutative diagram
$$
\begin{tikzcd}
\begin{Bmatrix}\text{Crystals\ of\ }\upmu'\text{-indexed}\\ \frakB_X\text{-modules\ of\ }\widetilde{\calE}(X'/\frakS)_{/\upmu'} \end{Bmatrix} \ar{r}{R_i'^{-1}} \ar[swap]{dd}{C^{-1}} & \begin{Bmatrix}\text{Crystals\ of\ }\upmu'_{U_i}\text{-indexed}\\ \frakB_{U_i}\text{-modules\ of\ }\widetilde{\calE}(U_i'/\frakS)_{/\upmu'_{U_i}} \end{Bmatrix} \ar{dd}{C_i^{-1}} \\
 & \\
\begin{Bmatrix}\text{Crystals\ of\ }\underline{\upmu}\text{-indexed}\\ \frakA_X\text{-modules\ of\ }\widetilde{\underline{\calE}}(X/\frakS)_{/\underline{\upmu}} \end{Bmatrix} \ar[swap]{r}{\underline{R}_i^{-1}} & \begin{Bmatrix}\text{Crystals\ of\ }\underline{\upmu}_{U_i}\text{-indexed}\\ \frakA_{U_i}\text{-modules\ of\ }\widetilde{\underline{\calE}}(U_i/\frakS)_{/\underline{\upmu}_{U_i}} \end{Bmatrix},
\end{tikzcd}
$$
where $R_i'=R_{X',U_i'}$ and $\underline{R}_i=\underline{R_{X,U_i}}.$
By \ref{THMXX}, the functor $C_i^{-1}$ is an equivalence of categories.
For $i,j\in I,$ we have a commutative diagram
\begin{equation}\label{eqdiagESTn}
\begin{tikzcd}
\begin{Bmatrix}\text{Crystals\ of\ }\upmu'_{U_i}\text{-indexed}\\ \frakB_{U_i}\text{-modules\ of\ }\widetilde{\calE}(U_i'/\frakS)_{/\upmu'_{U_i}} \end{Bmatrix} \ar{r}{R_{ij}'^{-1}} \ar[swap]{dd}{C_i^{-1}} & \begin{Bmatrix}\text{Crystals\ of\ }\upmu'_{U_i\times_XU_j}\text{-indexed}\\ \frakB_{U_i\times_XU_j}\text{-modules\ of\ }\widetilde{\calE}(U_i'\times_{X'}U_j'/\frakS)_{/\upmu'_{U_i\times_XU_j}} \end{Bmatrix} \ar{dd}{C_{ij}^{-1}=C_{ji}^{-1}} \\
 & \\
\begin{Bmatrix}\text{Crystals\ of\ }\underline{\upmu}\text{-indexed}\\ \frakA_{U_i}\text{-modules\ of\ }\widetilde{\underline{\calE}}(U_i/\frakS)_{/\underline{\upmu}_{U_i}} \end{Bmatrix} \ar[swap]{r}{\underline{R}_{ij}^{-1}} & \begin{Bmatrix}\text{Crystals\ of\ }\underline{\upmu}_{U_i\times_XU_j}\text{-indexed}\\ \frakA_{U_i\times_XU_j}\text{-modules\ of\ }\widetilde{\underline{\calE}}(U_i\times_XU_j/\frakS)_{/\underline{\upmu}_{U_i\times_XU_j}} \end{Bmatrix},
\end{tikzcd}
\end{equation}
where $R_{ij}'=R_{U_i',U_i'\times_{X'}U_j'}$ and $\underline{R}_{ij}=\underline{R_{U_i,U_i\times_XU_j}}.$

We start by proving that $C^{-1}$ is essentially surjective. Let $\frakE$ be a crystal of $\underline{\upmu}$-indexed $\frakA_X$-modules. Since $C_i^{-1}$ is essentially surjective, there exists a crystal $\frakF_i$ of $\upmu'_{U_i}$-indexed $\frakB_{U_i}$-modules such that
$$
C_i^{-1}\frakF_i=\underline{R}_i^{-1}\frakE.
$$
There exists isomorphisms
$$
\underline{R}_{ij}^{-1}C_i^{-1}\frakF_i=\underline{R}_{ij}^{-1} \underline{R}_i^{-1} \frakE \xrightarrow{\sim} \underline{R}_{ji}^{-1}\underline{R}_j^{-1}\frakE=\underline{R}_{ji}^{-1}C_j^{-1}\frakF_j.
$$
satisfying the cocycle conditions.
We deduce, by \eqref{eqdiagESTn}, isomorphisms
$$
C_{ij}^{-1}R_{ij}'^{-1}\frakF_i \xrightarrow{\sim} C_{ij}^{-1}R_{ji}'^{-1}\frakF_j
$$
satisfying the cocycle conditions. Since $C_{ij}^{-1}$ is an equivalence of categories, we obtain isomorphisms
$$
R_{ij}'^{-1}\frakF_i \xrightarrow{\sim} R_{ji}'^{-1}\frakF_j
$$
satisfying the cocyle conditions. By the stack $\mathbbm{E}'$ \eqref{paragEST2}, the crystals $\frakF_i$ glue together into a crystal $\frakF$ of $\upmu'$-indexed $\frakB_X$-modules such that $R_i'^{-1}\frakF$ is canonically isomorphic to $\frakF_i.$
We then have
$$
\underline{R}_i^{-1}C^{-1}\frakF=C_i^{-1}R_i'^{-1}\frakF\xrightarrow{\sim} C_i^{-1}\F_i=\underline{R}_i^{-1}\frakE.
$$
We deduce an isomorphism
$$
C^{-1}\frakF \xrightarrow{\sim} \frakE,
$$
hence the essential surjectivity of $C^{-1}.$ The fullness is proved in a similar way by descending morphisms.

We now prove that $C^{-1}$ is faithful. Let $f,g:\frakF_1 \ra \frakF_2$ be a morphism of crystals of $\upmu'$-indexed $\frakB_X$-modules such that $C^{-1}(f)=C^{-1}(g).$ We get
$$
\underline{R}_i^{-1}C^{-1}(f)=\underline{R}_i^{-1}C^{-1}g.
$$
This is equivalent to
$$
C_i^{-1}R_i'^{-1}(f)=C_i^{-1}R'^{-1}_i(g).
$$
Since $C_i^{-1}$ is faithful, we get
$$
R_i'^{-1}(f)=R_i'^{-1}(g).
$$
This is true for all $i\in I.$ By the stack $\mathbbm{E}'$ \eqref{paragEST2}, we deduce that $f=g.$
\end{proof}

\end{document}